%% file: 1._main.tex
\documentclass[9pt]{article}

\input{header.tex}
\title{
Equivariant Lagrangian Floer theory\\ 
on compact toric manifolds
} 
\author{Yao Xiao}
\date{}
\begin{document}

\maketitle

\begin{abstract}
    We define an equivariant Lagrangian Floer theory for Lagrangian torus fibers in a compact symplectic toric manifold equipped with a subtorus action. We show that the set of all Lagrangian torus fibers with weak bounding cochain data whose equivariant Lagrangian Floer cohomology is non-zero can be identified with a rigid analytic space. We prove that the set of these Lagrangian torus fibers is the tropicalization of the rigid analytic space. This provides a way to locate them in the moment polytope. 
Moreover, we prove that the dimension of such a rigid analytic space is equal to the dimension of the subtorus when the symplectic manifold is $\proj^n$ or has complex dimension less than or equal to $2$. 
We also show that the Lagrangian submanifolds with non-trivial equivariant Floer cohomology are non-displaceable by $G$-equivariant Hamiltonian diffeomorphisms. 
\end{abstract}

\tableofcontents

\section{Introduction}
\label{Section intro}

Lagrangian Floer theory on a symplectic manifold, developed by Floer \cite{MR0965228} and generalized by Fukaya-Oh-Ohta-Ono \cite{fooobook1, FOOObook2},  encodes the intersection information of its
Lagrangian submanifolds and is the building block in the construction of the Fukaya category. 
The Lagrangian submanifolds (with weak
bounding cochain data) with non-vanishing Lagrangian Floer cohomology, in many cases (for instance, in the compact Fano toric case \cite{MR3868001, ganatra2019automatically}), are generators of the Fukaya category of the symplectic manifolds, but are hard to compute in general. 
Fukaya-Oh-Ohta-Ono \cite{I, II, III} studied the Lagrangian Floer theory on compact toric manifolds and provided an algorithm for locating such Lagrangian torus fibers. 
Various constructions related to equivariant Floer theory have been studied by Austin-Braam \cite{AustinBraam}, Seidel-Smith \cite{SS}, Woodward \cite{Woodward, WoodwardQuantumKirwan} and Woodward-Xu \cite{WoodwardXu}, Hendricks-Lipshitz-Sarkar \cite{HLS1,HLS2}, Cho-Hong \cite{ChoHong}, Bao-Honda\cite{BaoHonda}, etc. 
We study the case where the action is by a continuous compact Lie group. 

Motivated by Fukaya-Oh-Ohta-Ono's \cite{I, II, III} and Daemi-Fukaya's \cite{daemifukaya} works, 
we define an equivariant Lagrangian Floer theory on compact toric manifolds with a subtorus action. 
To locate the Lagrangian submanifolds with non-trivial equivariant Lagrangian Floer cohomology, we apply tropical analytic geometry to the rigid analytic space $\Crit_G^{\Delta}(\po)$ consisting of all such pairs of the form $(L(u),b)$, each consisting of a Lagrangian torus fiber and a weak bounding cochain whose $G$-equivariant Lagrangian Floer cohomology is non-zero. 
We prove that the dimension of the rigid analytic space $\Crit_G^{\Delta}(\po)$ is the dimension of the subtorus $G$ in cases where the compact toric manifold is a complex projective space $\proj^n$ or of complex dimension less than or equal to $2$. 
Moreover, we show that the Lagrangian submanifolds with non-trivial equivariant Floer cohomology is non-displaceable by $G$-equivariant Hamiltonian diffeomorphisms.
\subsection{The setup}
\label{Section setup}
Let $(X,\omega, T^n, \mu)$ be a compact symplectic toric manifold. More specifically, $X$ is a (real) $2n$-dimensional manifold with symplectic form $\omega$ such that there is an effective Hamiltonian torus action on $(X,\omega)$ by $T^n$ and $\mu: X\to \tn^*$ is an associated moment map, where $\tn$ is the Lie algebra of $T^n$ and $\tn^*$ is the dual of $\tn$. 
Let $\Delta = \mu(X)$ be the moment polytope for the $T^n$-action. 
For each $u\in \interior \Delta$, $L(u):=\mu^{-1}(u)$ is a $T^n$-invariant Lagrangian torus with a $T^n$-invariant relatively spin structure induced by the $T^n$-action. 
Let $J$ be a $T^n$-invariant almost complex structure on $X$ compatible with $\omega$.    

Let $G = K T^r \subset T^n$ be a compact $r$-dimensional connected subtorus of $T^n$ given by a $(n\times r)$-matrix $K$ with integer coefficients of rank $r$. The $G$-action is induced from the $T^n$-action on $X$. 
We identify 
\[ \tn^*\cong M_{\R}\cong (\R^n)^*,  \]
where $M_{\R}$ is the dual vector space of 
an $n$-dimensional $R$-vector space 
$N_{\R}$. 
Let $N\cong \Z^n$ be a full-rank lattice in $N_{\R}$ and $M := \Hom_{\Z}(N,\Z)$ be the dual lattice of $N$ such that $M_{\R} \cong M\otimes_{\Z} \R$ and $N_{\R} \cong N\otimes_{\Z} \R$. 
For every $u\in \interior \Delta$, $\mu^{-1}(u)$ is diffeomorphic to $T^n$. 
We identify 
\[ \tn \cong N_{\R}\cong \R^n \qand H_1(L(u),\Z)\cong N\cong \Z^n, \quad H^1(L(u),\Z) \cong M \cong \Z^n. \]
Fix an integral basis $\{e_1^*,\ldots, e_n^*\}$ for the lattice $N$. 
Let $\{e_1,\ldots, e_n\}$ be the dual basis of $M$. 
Let the $a_{i,j} \in \Z$ be such that 
\[H^1(L(u), \Z) = \Span_{\Z} \left\{\alpha_i =  \sum_{j=1}^n a_{i,j}e_j \, \middle|\, 1\leq i \leq n\right\} , \]  where 
$H_{T^r}^1(L(u),\Z) = \Span_{\Z}\{ \alpha_{r+1}, \ldots, \alpha_n\}$. 

We consider the following coefficient rings. 
Define the \textbf{universal Novikov field} by
\begin{equation}
\label{universal Novikov field}
\nove  
= 
\left\{
\sum\limits_{i\in \N }  a_i T^{\lambda_i} e^{n_i}
    \;\middle| \;
    \begin{aligned} 
    &  \lambda_i \in \R, a_i \in \C , \text{ and }   n_i\in \Z  \quad \forall i \in \N \\
    &  \lambda_0 < \lambda_1< \cdots,  \quad \lim\limits_{i\to \infty }\lambda_i = \infty 
    \end{aligned} 
 \right\}. 
\end{equation} A non-Archimedean \textbf{valuation} function $\val: \nove \to \R\cup \{\infty\}$ on $\nove$ is defined as follows. 
\begin{equation}
\label{valuation universal novikov field}
y=\sum\limits_{i\in \N }  a_i T^{\lambda_i} e^{n_i} 
\mapsto 
\val (y) := \begin{cases}
\min \{ \lambda_i\mid i\in \Z, a_i\ne 0\} \quad &\text{if }y \ne 0 \\
\infty \quad &\text{if }y =0
\end{cases}     
\end{equation}
The \textbf{universal Novikov ring} is given by 
\begin{equation}
\label{universal Novikov ring}
\novringe  = \left\{ y\in \nove \mid  \val(y) \geq 0 \right\}.     
\end{equation}
We also consider the energy-zero parts of $\nove$ and $\novringe$. 
Define 
\begin{equation}\label{Novikov field}
\Lambda  = \left\{ \sum\limits_{i\in \N} a_i T^{\lambda_i} 
    \;\middle| \;
    \begin{aligned} 
    &  \lambda_i \in \R \text{ and }  a_i \in \C \quad \forall i \in \N
    \\ &  \lambda_0 < \lambda_1< \cdots,  \quad \lim\limits_{i\to \infty }\lambda_i = \infty
 \end{aligned} \right\}    
\end{equation}
to be the \textbf{Novikov field} and 
\begin{equation}
\label{Novikov ring}
\Lambda_0  
=  \left\{ \sum\limits_{i\in \N} a_i T^{\lambda_i} \in \Lambda 
    \;\middle| \; \lambda_i\geq 0  \quad \forall i \in \N
   \right\}     
\end{equation} 
to be the Novikov ring. 
The rings $\Lambda, \Lambda_0$ carries a valuation function
\begin{equation}
\label{valuation novikov field}
y=\sum\limits_{i\in \N }  a_i T^{\lambda_i} 
\mapsto 
\val (y) := \begin{cases}
\min \{ \lambda_i\mid i\in \Z, a_i\ne 0\} \quad &\text{if }y \ne 0 \\
\infty \quad &\text{if }y =0
\end{cases}.   
\end{equation}
The valuations \autoref{valuation universal novikov field} and \autoref{valuation novikov field} induce a non-Archimedean norm 
$y\mapsto |y|_R:=\exp(-\val(y))$ on $R$, where $R \in \{ \nove, \novringe, \Lambda, \Lambda_0\}$ and $\exp$ is the exponential map with Euler's number as the base. 
Moreover, we have an exponential map $\exp: \Lambda_0 \to  \Lambda_0$ defined as follows. Every $b \in \Lambda_0 $ can be decomposed as $b = b_0+ b_+$, where
$b_0\in \C$ and $b_+$ satisfies $\val(b_+)>0$. 
We define $\exp (b) = e^{b_0 } \sum_{n\in \N} \frac{b_+^n}{n!}$, where $e^{b_0 }$ is the usual exponential of complex numbers. 

\subsection{Main results}
We summarize our main results below. 
We denote by \, $\widehat{\ }$\,  the completion of the corresponding objects with respect to the energy filtration. 
\begin{enumerate}[1.]
\item \textbf{Definition of a $T^r$-equivariant $\Ainf$-algebra associated with a Lagrangian torus fiber. } 
\begin{thm}[Theorem \ref{mk is Ainf}] 
For each $u\in \interior \Delta$, we can associate to $L(u)$ a $G$-equivariant filtered $\Ainf$-algebra 
\begin{equation}\label{Ainf-algebra, intro}
\lrp{\Omega_G(L(u))\complete{\otimes}\novringe, \{\m_k^G\}_{k\in\N}},      
\end{equation}
where 
$\Omega_G(L(u))$ 
is the Cartan model for the $G$-equivariant differential forms on $L(u)$.    
\end{thm}

\item \textbf{Deformation of the $T^r$-equivariant $\Ainf$-algebra by a weak bounding cochain}. 
\begin{thm} [Theorem \ref{mkb is Ainf}] 
For each $u\in \interior \Delta$ and each
$b\in H^1(L(u),\Lambda_0)
$, 
there is a filtered  $\Ainf$-algebra 
\begin{equation}\label{Ainf-algebra, deformed}
\lrp{\Omega_G(L(u))\complete{\otimes}\novringe, \{(\m_k^G)^b\}_{k\in \N} }  
\end{equation}
 associated with $(L(u),b)$.    
\end{thm}
Moreover, we show in Corollary \ref{m1b squares to 0} that
$(\m_1^G)^b\circ (\m_1^G)^b = 0$. 
This leads to the definition of a $T^r$-equivariant Lagrangian Floer cohomology $ HF_G((L, b),(L, b), \novringe)$.

\item \textbf{Spectral sequences}. 
In Section \ref{Section spectral sequences},
we define a spectral sequence to compute the $G$-equivariant Lagrangian Floer cohomology. 
\begin{thm}[Theorem \ref{spectral sequence statements}]
There exists a spectral sequence such that its $E_2\cong H_G^*(L(u),\R)\otimes \novringe$
and it converges to $HF_G((L(u), b),(L(u), b), \novringe)$.     
\end{thm}
\item 
\textbf{The set of 
Lagrangian torus fibers with non-trivial $G$ -equivariant Lagrangian Floer cohomology is the tropicalization of a rigid analytic space}. 

\begin{lem}
[Corollary \ref{weak bounding cochains are where partial derivatives in the normal directions vanish}]
$(\m_1^G)^b=0$ on the $E_2$-page of the spectral sequence 
if and only if $(\nabla \po)_b(\alpha) =0$ for all $\alpha \in H_G^1(L(u),\Lambda_0)$. 
\end{lem}
Moreover, by degree counting, 
for any $\alpha \in H_G^1(L(u),\Lambda_0)$, $(\m_1^G)^b(\alpha)$ is a non-zero degree $0$ form, which is a multiple of $1$. 
Thus, if $(\m_1^G)^b \ne 0$ on $H_G(L(u),\Lambda_0)$, 
we have 
\[ 
HF_G((L(u), b),(L(u), b), \nove) = 0.
\] 

The above implies the following. 
\begin{thm}[Theorem \ref{main theorem}]
Let $f_i \in \Lambda[y_1^{\pm 1}, \ldots, y_n^{\pm 1}]$ be the derivative of $\po$ with respect to the $\alpha_i$-th direction. 
Then
\begin{align*}
 \set{ (u,b ) \in \bigcup_{u\in \interior \Delta} \{u\}\times H^1(L(u), \Lambda_0/(2\pi i \Z))}{  HF_G((L(u), b),(L(u), b), \nove)\ne 0}    
\end{align*} 
is isomorphic to 
\[ \Crit_G^{\Delta}(\po): = \set{ (y_1,\ldots, y_n)\in (\Lambda^*)^n }{ 
\begin{aligned}
  & f_{i}(y_1,\ldots, y_n)=0 
  \quad \forall r+1\leq i \leq n \\
  & \lrp{\val(y_1),\ldots, \val(y_n)} \in \interior \Delta
\end{aligned}  
} .
\]  
Here we identify
\[  
\lrp{ u_1,\ldots, u_n, 
\sum_{i=1}^n x_ie_i}  
\text{ with } 
(y_1,\ldots, y_n) = \lrp{ \exp(x_1)T^{u_1}, \ldots,  \exp(x_n)T^{u_n}}\in (\Lambda^*)^n.  
\]
\end{thm}

\begin{thm}[Theorem \ref{CritGPO is a rigid analytic space}]
$\Crit_G^{\Delta}(\po)$ is a rigid analytic space over $\Lambda$. 
\end{thm}

Since $\val(y_i) = u_i$, we can locate the Lagrangian torus fibers with non-zero $G$-equivariant Lagrangian Floer cohomology by tropicalizing the rigid analytic space $\Crit_G^{\Delta}(\po)$. 
For instance, Figure \ref{fig: S2 times S2, a,b non-zero, intro} shows the Lagrangian torus fibers which have non-zero $S^1$-equivariant (for a generic $S^1$ in $T^2$) Lagrangian Floer cohomology in the $T^2$-moment polytope of $S^2\left(\frac{c}{2}\right)\times S^2\left(\frac{d}{2}\right)$. This is explained in Example \ref{S2 times S2 example}. 

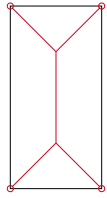
\begin{figure}[!h]
    \centering
\tikzset{every picture/.style={line width=0.75pt}} 
\resizebox{0.1\textwidth}{!}{

\begin{tikzpicture}[x=0.75pt,y=0.75pt,yscale=-1,xscale=1]

\draw   (250,50) -- (350,50) -- (350,250) -- (250,250) -- cycle ;
\draw [color={rgb, 255:red, 208; green, 2; blue, 27 }  ,draw opacity=1 ]   (251.66,51.66) -- (300,100) ;
\draw [shift={(250,50)}, rotate = 45] [color={rgb, 255:red, 208; green, 2; blue, 27 }  ,draw opacity=1 ][line width=0.75]      (0, 0) circle [x radius= 3.35, y radius= 3.35]   ;
\draw [color={rgb, 255:red, 208; green, 2; blue, 27 }  ,draw opacity=1 ]   (348.34,51.66) -- (300,100) ;
\draw [shift={(350,50)}, rotate = 135] [color={rgb, 255:red, 208; green, 2; blue, 27 }  ,draw opacity=1 ][line width=0.75]      (0, 0) circle [x radius= 3.35, y radius= 3.35]   ;
\draw [color={rgb, 255:red, 208; green, 2; blue, 27 }  ,draw opacity=1 ]   (300,100) -- (300,200) ;
\draw [color={rgb, 255:red, 208; green, 2; blue, 27 }  ,draw opacity=1 ]   (251.66,248.34) -- (300,200) ;
\draw [shift={(250,250)}, rotate = 315] [color={rgb, 255:red, 208; green, 2; blue, 27 }  ,draw opacity=1 ][line width=0.75]      (0, 0) circle [x radius= 3.35, y radius= 3.35]   ;
\draw [color={rgb, 255:red, 208; green, 2; blue, 27 }  ,draw opacity=1 ]   (348.34,248.34) -- (300,200) ;
\draw [shift={(350,250)}, rotate = 225] [color={rgb, 255:red, 208; green, 2; blue, 27 }  ,draw opacity=1 ][line width=0.75]      (0, 0) circle [x radius= 3.35, y radius= 3.35]   ;
\end{tikzpicture}
}
\caption{The Lagrangian torus fibers corresponding to the points on the red line segments have non-trivial $S^1$-equivariant Lagrangian Floer cohomology, for a generic subtorus $S^1$ in $T^2$. }
\label{fig: S2 times S2, a,b non-zero, intro}
\end{figure}

\item 
\textbf{The rigid analytic space $\Crit_G^{\Delta}(\po)$ has dimension $r = \dim G$ when $X$ is $\proj^n$ or a compact symplectic toric manifold of complex dimension $n\leq 2$. }
\begin{prop}[Proposition \ref{cpn rigid space dimension} and Proposition \ref{hypersurface rigid dimension}]
If $X$ is $\proj^n$ or a compact symplectic toric manifold of complex dimension $n\leq 2$, 
$\Crit_G(\po)$ 
 is a rigid analytic space of dimension $r = \dim G$. 
\end{prop}
We include some examples in Section \ref{examples}. 
\item \textbf{The Lagrangian submanifolds with non-trivial equivariant Lagrangian Floer cohomology are not displaceable by $G$-equivariant Hamiltonian diffeomorphisms.}
As a consequence of Theorem \ref{Hamiltonian isotopy invariance}, we prove the following equivariant non-displaceability result. 
\begin{thm}
    If $HF_G((L(u), b),(L(u), b), \nove)\ne 0$, then $L(u)$ is not displaceable by any $G$-equivariant Hamiltonian diffeomorphism, namely the time-1 map of some $G$-equivariant Hamiltonian isotopy. 
\end{thm}

\item \textbf{$T^r$-equivariant Kuranishi structures on moduli spaces}. 
In Section \ref{GKur section}, we introduce the definition of $T^r$-equivariant Kuranishi structures, which the moduli spaces of pseudoholomorphic curves are equipped with. 
We also define related concepts such as $T^r$-equivariant CF-perturbations and $T^r$-equivariant integration along the fibers, which are essential in the definition of the $\Ainf$-operators. 
\end{enumerate} 

\subsection*{Acknowledgement}
This article is part of the author's Ph.D. thesis at Stony Brook University. 
I would like to express my deepest gratitude to my advisor, Kenji Fukaya, for introducing me to this beautiful research topic and providing guidance and support during the research process. 
I also want to thank Saman Habibi Esfahani, Mohamad Rabah, Joseph Rabinoff, and Amy Yiyue Zhu for helpful discussions. 

\section{The \texorpdfstring{$T^r$}{Tr}-equivariant \texorpdfstring{$\Ainf$}{A∞}-algebra associated with a Lagrangian torus fiber}  \label{Section Ainf}
Consider the setup in Section \ref{Section setup}: 
Let $u\in \interior \Delta$. 
Let $L(u) = \mu^{-1}(u)$ be a compact connected Lagrangian torus in $(X,\omega)$, $G$ be a subtorus of $T^n$ as in \ref{Section setup}, and $J$ be a $T^n$-invariant almost complex structure compatible with $\omega$. 
Equip $L$ with a $T^n$-invariant relatively spin structure and an orientation. 
For the definition of the $\Ainf$-algebra, we take the coefficient ring to be either the universal Novikov field $\nove$ or the universal Novikov ring $\novringe$, as defined in \autoref{universal Novikov field} and \autoref{universal Novikov ring}.

\subsection{The moduli space \texorpdfstring{$\Mq_{k+1} (L,J,\beta)$}{M k+1 (L,J,β)}}
\begin{defn}[Moduli space of pseudoholomorphic discs with $k+1$ boundary marked points]
\label{moduli space of discs}
For any $\beta \in \pi_2(X,L)$ and $k\in \N$, let 
\begin{equation}
\label{moduli space of discs notation}
\Mq_{k+1}(L,J,\beta):=  \set{ (\Sigma,j, \vec{z},u)}{ 
\begin{aligned}
& \Sigma \text{ is a genus $0$ nodal Riemann surface 
with connected boundary }\\
& \text{and complex structure} j;  \quad  u:(\Sigma,\partial \Sigma)\to (X,L) 
  \text{ is smooth}; \\
& du\circ j =J\circ du; 
\quad [u] = \beta \in \pi_2(X,L); \\
& (\Sigma,j, \vec{z},u) \text{ is stable};  \quad   E(u) 
<\infty\\
& \vec{z} =\left(z_0,z_1,\ldots, z_k\right)  \in (\partial \Sigma)^{k+1}, 
\text{where the $z_i$ are distinct non-nodal} \\
& \text{ boundary marked points  and the enumeration is in } \\ & \text{counterclockwise order along }\partial \Sigma 
\end{aligned}
} \Big/ \sim ,
\end{equation} 
where $(\Sigma,j, \vec{z},u) \sim (\Sigma',j', \vec{z}^{'},u')$ if and only if there exists a biholomorphism $\varphi: (\Sigma,j) \to (\Sigma',j')$ such that $u'\circ \varphi = u$ and $\varphi (z_i) = z_i'$ for all $0\leq i \leq k$. 
A biholomorphism $\varphi$ satisfying these conditions is called an \textbf{isomorphism} between $(\Sigma,j, \vec{z},u) $ and $(\Sigma',j',\vec{z}',u')$. 

For any element $x = [ \Sigma,j, \vec{z},u] \in  \Mq_{k+1} (L,J,\beta)$, 
we define its \textbf{automorphism group} by 
\[ \Aut x 
= 
\set{\varphi: (\Sigma,j) \to (\Sigma,j)}
{ \begin{aligned}
& \varphi \text{ is a biholomorphism} \\
& u \circ \varphi = u \\ 
& \varphi(z_i) = z_i \quad \forall 0\leq i\leq k 
\end{aligned}}. 
\]
\begin{rmk}
By an abuse of notation, we also use $\vec{z}$ to denote the ordered subset $\{z_0,\ldots, z_k\}$ of $\partial \Sigma$. 
And, whenever $I$ is another set,  the elements in $\Vec{z}\cap I$ are ordered by the original enumeration in $\Vec{z}$. 
\end{rmk}
\end{defn}
 
For each $0\leq i \leq k$, let 
\[ \evaluation_{i,\beta}: \Mq_{k+1} (L,J,\beta) \to L, \qquad [\Sigma, j, u ,\Vec{z}] \mapsto u(z_i) \] be the evaluation map at the $i$-th marked point. 
$G$ acts on $ \Mq_{k+1} (L,J,\beta) $ by 
\[ g\cdot [\Sigma, j, u ,\Vec{z}] = [\Sigma, j, g\cdot  u,\vec{z}], \qquad \forall g\in G, \quad \forall  [\Sigma, j, u ,\Vec{z}] \in \Mq_{k+1} (L,J,\beta),   \] 
where 
\[ (g\cdot u)(w) = g\cdot \left(u(w)\right), \qquad \forall w \in \Sigma. \] 
Then the evaluation maps are $G$-equivariant:
\[ \evaluation_{i,\beta}\left( g\cdot [\Sigma, j, u ,\Vec{z}]\right) = g\cdot (u(z_i)) = g\cdot \evaluation_{i,\beta}\left( [\Sigma, j, u ,\Vec{z}]\right), \qquad \forall g\in G, \quad \forall 0\leq i\leq k. \]

We state some results about the moduli spaces of pseudoholomorphic discs, which will be essential in the definition of the $\Ainf$-algebras in Section \ref{Ainf algebra associated with (L,b)}. 
We postpone the definitions related to $G$-equivariant Kuranishi structures to Section \ref{GKur section}, and we recall the definitions related to $G$-equivariant orbifolds in Section \ref{orbifolds section}. 

\begin{prop}[$G$-equivariant Kuranishi structure on $\Mq_{k+1}(L,J,\beta)$]
\label{GKur structure on moduli space}
Let $u\in \interior \Delta$ and $L = \mu^{-1}(u)$. 
For $k\geq 1, \beta \in \pi_2(X,L)$ with $\beta\ne 0$, the moduli space
$\Mq_{k+1}(L,J, \beta)$ has a $G$-equivariant Kuranishi structure with corners:
\begin{equation}
\widehat{\mathcal{U}} =  \left(\set{ \mathcal{U}_p = (U_p,\E_p,\psi_p,s_p) }{ p=[\Sigma,j,
\Vec{z},u]\in \Mq_{k+1}(L,J,\beta)},   
\set{\vec{\alpha} _{pq}}{ p\in \Mq_{k+1}(L,J,\beta), q\in \image \psi_p } \right). 
\end{equation}
The normalized boundary of $\Mq_{k+1}(L,J,\beta)$ is a union of the fiber products  
\begin{align*}
    \partial \Mq_{k+1}(L,J,\beta) 
    & 
    = 
    \bigcup_{\substack{k_1,k_2 \geq 0 \\ k_1+k_2=k+1}} 
     \bigcup_{\substack{
    \beta_1,\beta_2 \in \pi_2(X,L)\\ \beta_1+\beta_2=\beta}} 
     \bigcup_{j=1}^{k_2} 
     \Mq_{k_1+1}(L,J,\beta_1) 
   _{\evaluation_0}{\times} _{\evaluation_j} \Mq_{k_2+1} (L,J,\beta_2), 
\end{align*}
and the systems of $G$-equivariant Kuranishi structures are compatible with the fiber product description. 
Moreover,
$\evaluation_{i}: \Mq_{k+1 }(L,J,\beta) \to L$ is a $T^n$-equivariant strongly continuous weakly submersive map. 
\end{prop}

\begin{proof}
A $T^n$-equivariant Kuranishi structure on the moduli space 
$\Mq_{k+1}(L,J, \beta)$  satisfying the properties is constructed in \cite{III} Section 4.3. 

We sketch the proof. 
Let $x = [\Sigma,j, \vec{z}, u] \in \Mq_{k+1}( L,J, \beta)$. 
We want to construct a $G$-equivariant Kuranishi chart $(U_x,\E_x,\psi_x,s_x)$ at $x$, where $U_x = V_x/\Gamma_x$ for some manifold $V_x$ and some finite group $\Gamma_x$. 
Let $\Gamma_x = \Aut x$. 
Let $G(x) = \{(g,\gamma)\in T^n\times\Gamma_x \mid gu=u\circ \gamma\}$. We say $z$ is a special point if $z$ is a nodal singularity or boundary marked point of $\Sigma$. 

Let $\Sigma = \bigcup_{\alpha\in A} \Sigma_{\alpha}$, where each extended disc component $\Sigma_{\alpha}$ consists of an irreducible disc component and all the spheres rooted on it. 
Let $\Vec{z}_{\alpha}  = \vec{z} \cap \Sigma_{\alpha}$ be the set of marked points on $\Sigma_{\alpha}$. 
For each $\alpha \in A$, 
choose a non-empty $G(x)$-invariant open subset $K_{\alpha}\subset \Sigma \setminus \partial \Sigma$ with compact closure $\overline{K}_{\alpha}$ which does not intersect $\partial \Sigma$ or nodal singularities
and let 
\[ D_{u,\alpha} : W^{1,p} \left(\Sigma_{\alpha},\partial \Sigma_{\alpha}; u^*TX, u^*TL \right)=:W_{\alpha} \to L^p \left(\Sigma_{\alpha}; u^*TX\otimes \Lambda^{0,1}\Sigma_{\alpha}\right),  \]
\[ D_{u} : W^{1,p} \left(\Sigma,\partial \Sigma; u^*TX, u^*TL \right)\to L^p \left(\Sigma; u^*TX\otimes \Lambda^{0,1}\Sigma\right)\]
be the linearization of $\delb$. 
Choose a finite-dimensional vector subspace $E_{\alpha} \subset C_c^{\infty}(K_{\alpha}, u^*TX)$ consisting of compactly supported elements for each $\alpha$ such that the following holds. 
\begin{enumerate} [i)]
    \item 
    $D_{u}\lrp{ \set{\xi 
\in W_{\alpha}
}{  \xi (z) =0  \; \forall \text{ special point }z 
 }}   + E_{\alpha}
= L^p \lrp{\Sigma_{\alpha}; u^*TX \otimes \Lambda^{0,1}\Sigma_{\alpha}}.  $
    \item $\bigoplus\limits_{\alpha\in A}  E_{\alpha} $ is $G(x)$-invariant. 
    \item 
    For each $\forall z_0 \in \partial \Sigma$, the map 
    $\operatorname{Ev}_{z_0}: D_{u}^{-1}  (\bigoplus\limits_{\alpha\in A}   E_{\alpha} ) \to T_{u(z_0)}L $ given by $v\mapsto v(z_0) $ is surjective. 
    \item If $\gamma \in \Gamma_x$ and $\Sigma_{\alpha'} =\gamma \Sigma_{\alpha}$, then $\gamma_*E_{\alpha} = E_{\alpha'}$.
\end{enumerate}
Let $E_x(x) = \bigoplus\limits_{\alpha\in A}   E_{\alpha}$. 
If $(g,\gamma)\in G(x)$, let $E_x((g,
\gamma)\cdot x) = g_*E_x(x)$. 

For each $\alpha$, choose $l_{\alpha}$ many appropriate extra interior marked points $\vec{w}_{\alpha}^+$ on $\Sigma_{\alpha}$ away from nodes to stabilize the domain of $x$. 
Let $\vec{w}^+$ be the ordered set of all such extra marked points on $\Sigma$. 
Let $v( \vec{w}^+)=( \Sigma, j , \vec{z}, \vec{w}^+)$.

Suppose $(v' =( \Sigma(v'), j' , \vec{z}', \vec{w}^{+'}),u')$ is a smooth curve with $k+1$ boundary marked points and $l$ interior marked points such that the domain $v'$ is close to $v$ in $\Mq_{k+1,l}$ and $(v',u')$ is close to $(hx, \vec{w}^+)$ for some $h\in T^n$. 
For each $v'$, there is an embedding $i_{v'}: \Sigma \setminus S  \to \Sigma(v')$, where $S$ is a neighborhood of the set of marked points and singularities. 

We decompose $\Sigma(v') = \bigcup_{r\in R}\Sigma_{r}'$ into extended disc components as well. 
Let $r \in R$.  
Let $A(r) =\{\alpha\in A\mid i_{v'} ( \Sigma_{\alpha} \setminus  
S)
\subset \Sigma_{r}'\} $. 
For each $\alpha\in A(r)$, we   
obtain a map 
$P_{r, \alpha}: E_{\alpha} \to h_*E_{\alpha} \to  C^{\infty} (\Sigma_r', (u')^*TX\otimes \Lambda^{0,1}\Sigma_{r}')$
using the convexity of the square of the distance function 
(and an exponential decay estimate) 
which allows us to 
(choose a ``closest'' $h\in T^n$ and)
define a suitable parallel transport map. 
Then we define $E_x(v',u') = \bigoplus\limits_{r\in R}\bigoplus\limits_{\alpha\in A(r)} \image P_{r, \alpha}$. 
Let  
\[ V_x = \{ (v',u') \text{ close to }(T^n\cdot x, \vec{w})\mid \delb u' \equiv 0 \,\operatorname{mod}\, E_x(v',u')\}.\] 
Let $s_x: (v',u')\mapsto \delb u'$ and $\E_x\to V_x$ be the orbibundle whose fiber is $E_x(v',u')$ at $[(v',u')]$ and $E_x(x)$ at $x$. Thus, an equivariant Kuranishi chart at $x$ is defined. 

For $(k_1,\beta_1), (k_2,\beta_2)\in \N\times \pi_2(
X,L)$, we say 
$(k_1,\beta_1)< (k_2,\beta_2)$ if either $\omega(\beta_1)<\omega(\beta_2)$ or $\omega(\beta_1)=\omega(\beta_2)$ and  $k_1<k_2$. 
For coordinate changes to be defined, we need to modify the obstruction bundles $\E_x$ inductively on $(k', \beta')$.  
Suppose for all $(k',\beta')< (k, \beta)$ we have a Kuranishi structure and, in particular, coordinate changes are defined on $\Mq_{k+1}(L,J,\beta')$. 
More specifically, we have a 
finite cover $\{U(\mathfrak{c})\mid [\mathfrak{c}]\in P_{k'}(\beta')\}$ 
of $\Mq_{k'}(L,J,\beta')$, where $P(\beta') \subset \Mq_{k+1}(L,J,\beta')/T^n$ 
and $K(\mathfrak{c})\subset \Mq_{k+1}(L,J,\beta')$ is a $T^n$-invariant closed subset of the Kuranishi neighborhood at $\mathfrak{c}$. 
The fiber of the obstruction bundle of a point in $\Mq_{k+1}(L,J,\beta')$ is given by a direct sum of (perturbations of) fibers of (some of) the $\mathcal{E}_{\mathfrak{c}}$'s  
so that the coordinate changes in $\Mq_{k+1}(L,J,\beta')$ is defined. 
Then we define the Kuranishi structure on $\Mq_{k+1}(L,J,\beta)$ by a downward induction on the number of disc components. 
For each $d>1$, let $S_d\Mq_{k+1}(L,J,\beta)$ be the set of elements with at least $d$ disc components. 

Suppose on $S_{d+1}\Mq_{k+1}(L,J,\beta)$ we have a finite cover $\{ K_{\p}\mid \p \in \Pindex \}$ 
such that each $K_{\p}$ is a $T^n$-invariant compact subset of the Kuranishi neighborhood $\psi_{\p }\lrp{ s_{\p}^{-1}(0)}$ and that 
\[ S_{d+1}\Mq_{k+1}(L,J,\beta) \subset \bigcup\limits_{\p\in \Pindex }   \interior K_{\p} . \]
As before, we find a cover of \[\mathcal{K}_d\Mq_{k+1}(L,J,\beta) = S_{d}\Mq_{k+1}(L,J,\beta)\setminus \bigcup\limits_{\p\in \Pindex}  \interior K_{\p}\] by $T^n$-invariant compact subsets $K_{x_1},\ldots, K_{x_m}$ of  Kuranishi neighborhoods of finitely many points $x_1,\ldots, x_m \in S_{d}\Mq_{k+1}(L,J,\beta) \setminus S_{d+1}\Mq_{k+1}(L,J,\beta) $ and define the fibers of the obstruction bundles by appropriate direct sums of the fibers of the obstruction bundles at the $x_i$'s. 

Then we glue the Kuranishi structures on $\mathcal{K}_d\Mq_{k+1}(L,J,\beta) $ and $ \bigcup_{\p\in \Pindex }   \interior K_{\p} $. In particular, when a point $x \in \lrp{ \bigcup _{\p\in \Pindex} \interior K_{\p} } \cap \lrp{\bigcup_{i} \interior K_{x_i} }$, we define the fibers of $\E_x$ by taking direct sums of (perturbations of) the relevant fibers of the obstruction bundles from the two types of Kuranishi structures. 
The induction construction then allows us to obtain coordinate changes. 
\end{proof}

\subsection{The \texorpdfstring{$G$}{G}-equivariant \texorpdfstring{$\Ainf$}{A∞}-algebra associated with \texorpdfstring{$(L,b)$}{(L,b)} }\label{Ainf algebra associated with (L,b)}
Let $L=\mu^{-1}(u)$ be a Lagrangian torus fiber for some $u\in \interior \Delta$. 
Let $\g $ be the Lie algebra of $G$ and $\g^*$ be its dual.  
Let $S(\g^*)$ be the symmetric algebra on $\g^*$. 
Let $\Omega(L)$ be the de Rham complex of $L$ and $ \Omega_G(L,\R) = \lrp{ \Omega(L)\otimes S(\g^*) }^G$ be the Cartan model of $L$. 
Let $d_G$ be the Cartan differential. 
Recall the definition of the universal Novikov ring in \ref{universal Novikov ring}. 
Define 
\begin{equation}
\label{Floer complex}
C_G(L, \novringe):= \Omega_G(L,\R)\widehat{\otimes}_{\R} \novringe.  
\end{equation} 
Let $C = C_G(L,  \novringe)$. 
It is a graded $\novringe$-module: $C= \complete{\bigoplus\limits_{p\in \N}}C^p$, where 
\begin{equation}
C^p = \complete{\bigoplus_{ i+2j+2n = p }} \Bigp{\Omega^{i}(L)\otimes S^{j}(\g^*)}^G \widehat{\otimes}_{\R} (\Lambda_{0}\cdot  e^n).    
\end{equation}
Define a degree on $C$ such that $\deg h = \min \set{ p\in \N}{ h \in \bigoplus \limits_{m=0}^p C^m}$. 
Denote by $C[1]$ the module determined by $C[1]^p = C^{p+1}$. 
Let $B_0 C[1] = \Lambda_{0,nov}$
and 
\[ 
B_kC[1] 
=  
C[1] \otimes \cdots \otimes C[1] \qquad \forall k> 0. \]
Let $BC[1] = \widehat{\bigoplus\limits _{k\in \N } } B_kC[1]$. 
For any $\beta\in \pi_2(X,L)$, we define 
$ \m_{k,\beta}^G: B_kC[1] \to C[1]$, the contribution of the moduli space $\Mq_{k+1}(L,J,\beta)$ by using the evaluation maps $\evaluation_{j,\beta}^G: \Mq_{k+1}(L,J,\beta)\to L$, as follows. 
We denote by $(\evaluation_{j,\beta}^G)^*$ the $G$-equivariant pullback by the evaluation map at the $k$-th marked point and $(\evaluation_{0,\beta}^G)_!$ the $G$-equivariant integration along the fiber by the evaluation map at the $0$-th marked point, which we discuss in Section \ref{integration along the fiber on Kuranishi}.  
If $\beta = 0$, define 
\begin{equation}
\label{Ainf operator beta = 0}
 \begin{dcases}
\m_{0,0}^G= 0 \quad &  \\
\m_{1,0}^G(h)= d_Gh
\quad & \forall h \in C  \\
\m_{2,0}^G (h_1\otimes h_2) = (-1)^{\deg h_1} (\evaluation_{1,0}^G)^*h_1 \wedge (\evaluation_{2,0}^G)^*h_2 \quad & \forall h_1,h_2  \in   C \\
\m_{k,0}^G=0 \quad & \forall k\geq 3 .
\end{dcases} 
\end{equation}
For $\beta\ne 0$, 
define
\begin{equation} \label{Ainf operator beta non-zero k=0}
\m_{0,\beta}^G(1) = (\evaluation_{0,\beta}^G)_! (1)   
\end{equation}
and 
\begin{equation}
\label{Ainf operator beta non-zero} 
\m_{k,\beta}^G(h_1\otimes \cdots \otimes h_k) = (\evaluation_{0,\beta}^G)_!\Bigp{(\evaluation_{ 1,\beta}^G)^* h_1\wedge \cdots \wedge (\evaluation_{k ,\beta}^G)^*h_k )} \quad  \forall k \geq 1.  
\end{equation}
And $\forall k\in \N$ we define $\m_k^G: B_kC[1]\to C[1]$ by, $\forall h_1\otimes \cdots h_k \in B_kC[1]$, 
\begin{equation}
 \label{Ainf operator}
\m_k^G(h_1\otimes \cdots \otimes h_k) = \sum_{\beta \in \pi_2(X,L)}\m_{k,\beta}^G(h_1\otimes \cdots \otimes h_k)T^{\frac{\omega(\beta)}{2\pi}} e^{\frac{I_{\mu}(\beta)}{2}},  
\end{equation}
where $\omega(\beta) = \pair{ \omega, \beta}$
and $I_{\mu}(\beta)$ is the Maslov index of $\beta$. 

\begin{defn}[Discrete submonoid]
Consider the monoid $(\R_{\geq 0}\times 2\Z, +, (0,0))$ and the projection maps 
$E:\R_{\geq 0}\times 2\Z \to \R_{\geq 0}$, $I_{\mu}: \R_{\geq 0}\to 2\Z$. 
A subset 
$S\subset \R_{\geq 0}\times 2\Z$ 
is a \textbf{discrete submonoid} if the following holds. 
  \begin{enumerate}[i)]
      \item $(S,+,(0,0))$ is a monoid. 
      \item $E(S)$ is discrete.
      \item For each $E_0 \in \R_{\geq 0}$, $S\cap E^{-1}([0,E_0])$ is a finite set. 
  \end{enumerate} 
\end{defn}

Let $L$ be a Lagrangian torus fiber of the compact symplectic toric manifold $X$ and let
\begin{equation}
\label{submonoid for the Ainf algebra} S=\{(\omega(\beta), I_{\mu}(\beta))\mid \beta\in \pi_2(X,L)\} . 
\end{equation} 
\begin{defn}[$S$-gapped curved filtered $G$-equivariant $\Ainf$-algebra]\label{gapped curved filtered equivariant Ainf algebra}
An \textbf{$S$-gapped curved\footnote{The word ``curved" means an $\m_0: \Lambda_{0,nov} \to C[1]$ is included, in contrast to the classical $\Ainf$-algebra. } filtered $G$-equivariant $\Ainf$-algebra}  is a tuple $(C,\{\m_k^G\}_{k\in \N}, G, S)$ consisting of 
\begin{itemize}
    \item a $\novringe$-module $C$, 
    \item a family of operators $\m_k^G : B_kC[1]\to C[1]$, 
    \item a compact Lie group $G$ with Lie algebra $\g$, and
    \item a discrete submonoid $S
\subset \R_{\geq 0}\times 2\Z$ 
\end{itemize} such that the following holds. 
\begin{enumerate}[i)]
    \item ($G$-equivariant) $C$ is an $S(\g^*)^G$-algebra and $BC[1]$ is an $S(\g^*)^G$-coalgebra and, $\forall k\in \N$, the operator $\m_k$ is an $S(\g^*)^G$-algebra homomorphism.     
    \item ($S$-gapped) $\forall k\in \N$, there is a decomposition 
    \[ \m_k^G 
     = \sum_{(\lambda, n)\in S} \m_{k,(\lambda, n)}^G  T^{\lambda}e^\frac{n}{2}.
   \] 
    \item (Energy filtered) There is an energy filtration on $C$ such that, $\forall p\in \N$, the following holds. 
    \begin{itemize}
        \item The filtration on $C^p$ is decreasing: $F^{\lambda}C^p \subset F^{\lambda'}C^p$ if $\lambda > \lambda' $. 
        \item $\forall \lambda'>0$, we have $T^{\lambda'}\cdot F^{\lambda}C^p \subset F^{\lambda+\lambda'}C^p $. 
         \item $C^p$ is complete with respect to the $T$-adic topology induced by the filtration. 
         \item $C^p$ has a basis whose elements are in $F^0 C^p \setminus \bigcup_{\lambda>0}F^{\lambda}C^p$. 
         \item $\m_0^G(1)\in F^{\lambda} C[1]$ for some $\lambda>0$. 
    \end{itemize}
    Moreover, for each $k\in \N$, $\m_k^G$ is  filtration-preserving: \[\m_k^G(F^{\lambda_1}C^{p_1}\otimes\cdots \otimes F^{\lambda_k}C^{p_k}) \subset F^{\lambda_1 + \cdots + \lambda_k} C^{p_1+\cdots + p_k -k+2} \] for all $(\lambda_1,\ldots, \lambda_k )\in \R_{\geq 0}^{k}$ and all $(p_1,\ldots, p_k )\in \N^k$. 
    \item ($\Ainf$-relations) The family $\{\m_k^G\}_{k\in \N}$ satisfy the following $\Ainf$- relations: For any $k\in \N\setminus \{ 0 \}$, $s\in S$, 
    \begin{align}
    \sum_{\substack{s_1,s_2\in S
    \\ s_1+s_2=s} }
   &  \sum_{\substack{
    k_1,k_2\in \N
    \\
    k_1+k_2=k+1}} \sum_{i=0}^{k_1} (-1)^{\dagger} \m_{k_1,s_1}^G ( h_1\otimes \cdots \otimes h_{i}\otimes \nonumber 
    \\ 
    & \m_{k_2,s_2}^G(h_{i+1}\otimes \cdots \otimes h_{i+k_2})\otimes
     h_{i+k_2+1} \otimes \cdots \otimes h_{k } ) =0 , 
    \end{align}
    where $\dagger = \sum\limits_{j=1}^{i} (\deg h_j + 1)$.
\end{enumerate}
\end{defn}
\begin{prop} 
\label{mk is Ainf}
Let 
$(C_G(L,\novringe), \{\m_k^G\}_{k\in \N}, G, S )$ be the data defined in 
\autoref{Floer complex}, 
\autoref{Ainf operator beta = 0}, 
\autoref{Ainf operator beta non-zero k=0}, 
\autoref{Ainf operator beta non-zero}, and 
\autoref{submonoid for the Ainf algebra}. 
It is an $S$-gapped curved filtered $G$-equivariant $\Ainf$-algebra. 
\end{prop}
\begin{proof}
For any $k\in \N$, $(\lambda, n)\in S$, let 
\[ \m_{k,(\lambda, n)}^G = \sum_{\substack{\beta \in \pi_2(X,L) \\ \omega(\beta)=\lambda , I_\mu(\beta) )= n } }\m_{k,\beta}^G. \]
Then $S$-gappedness follows. 
For each $k\in \N, \beta\in \pi_2(X,L)$, 
we have 
\begin{align*}
& \sum_{\substack{k_1,k_2 \geq 0 \\ k_1+k_2=k+1}} 
     \sum_{\substack{
    \beta_1,\beta_2 \in \pi_2(X,L)\\ \beta_1+\beta_2=\beta}} 
      \sum_{j=1}^{k_1} 
      (-1)^* \m_{k_1,\beta_1}^G(h_1\otimes \cdots h_{j-1} 
      \otimes \m_{k_2,\beta_2}^G(h_j\otimes \cdots h_{j+k_2-1})\otimes \cdots \otimes h_k
      )  \\
&  = \underbrace{\m_{1,0}^G \m_{k,\beta}^G(h_1\otimes \cdots \otimes h_k)}_{ \text{(I)} } 
+ \underbrace{\sum_{i=1}^k (-1)^* \m_{k,\beta}^G(h_1\otimes \cdots\otimes \m_{1,0}^G (h_i)\otimes\cdots  \otimes h_k)}_{ \text{(II)} } \\
& + \underbrace{\sum_{\substack{ k_1+k_2=k+1\\ \beta_1+\beta_2=\beta \\ 
 (k_1,\beta_1)\ne (1,0)\\
  (k_2,\beta_2)\ne (1,0)} }
      \sum_{j=1}^{k_1} 
      (-1)^{\dagger} \m_{k_1,\beta_1}^G(h_1\otimes \cdots h_{j-1} 
      \otimes \m_{k_2,\beta_2}^G(h_j\otimes \cdots h_{j+k_2-1})\otimes \cdots \otimes h_k
      ) }_{ \text{(III)} }, 
\end{align*}
where $* =  \sum\limits_{l=1}^{i-1} (\deg h_l + 1)$ and $\dagger =  \sum\limits_{l=1}^{j-1} (\deg h_l + 1)$. 
To show the $\Ainf$-relation, it suffices to show the sum is zero for all $k,\beta$. 
By Proposition \ref{GKur structure on moduli space} and Proposition \ref{composition formula}, (III) corresponds to $\gcorr{ \partial \Mq_{k+1}(\beta)} (h_1\otimes \cdots  \otimes h_k)$. Moreover, (I) corresponds to $d^G \circ \gcorr{ \Mq_{k+1}(\beta)} (h_1\otimes \cdots  \otimes h_k)$ and (II) corresponds to $\gcorr{ \Mq_{k+1}(\beta)} \circ d_G (h_1\otimes \cdots  \otimes h_k)$. 
Thus, by Stokes' Theorem \ref{Stokes for Gkur} (or Proposition \ref{Stokes corollary}), the sum is zero.  
Hence, by construction, $(C_G(L,\novringe), \{\m_k^G\}_{k\in \N}, G, S )$ has an $S$-gapped curved filtered $\Ainf$ structure. 
\end{proof}

\begin{defn}[$(\m_{k,\beta}^G)^b$]
Let $ k\in \N$ and $\beta\in \pi_2(X,L)$. For any $b\in H^1(L,\Lambda_0)$,  define 
$(\m_{k,\beta}^G)^b: B_k( C[1])\to C[1]$ 
by 
\begin{equation}\label{mkb}
(\m_{k,\beta}^G)^b( h_1\otimes \cdots \otimes h_k) = \exp(  \partial \beta \cap b  )  \m_{k,\beta}^G( h_1\otimes \cdots \otimes h_k) .     
\end{equation} 
And we define 
$(\m_k^G)^b: B_k (C_G(L)[1])\to C_G(L)[1]$
by 
\begin{equation}
\label{A infinity operator with bounding cochain}
 (\m_k^G)^b = \sum_{\beta\in \pi_2(X,L)}(\m_{k,\beta}^G)^b T^{\frac{\omega(\beta)}{2\pi}}e^{\frac{I_{\mu}(\beta)}{2}}.    
\end{equation}
\end{defn}

\begin{prop}
\label{mkb is Ainf}
Let $b\in H^1(L,\Lambda_0)$. 
Then 
$(C_G(L,\novringe), \{(\m_k^G)^b\}_{k\in \N})$ is an $\Ainf$-algebra.     
\end{prop}
\begin{proof}
 \begin{align*}
 & \sum_{
\substack
{
k_1,k_2 \geq 0 \\ 
k_1+k_2=k+1\\
\beta_1,\beta_2 \in \pi_2(X,L)\\ 
\beta_1+\beta_2=\beta
}
} 
      \sum_{j=1}^{k_1} 
      (-1)^* (\m_{k_1,\beta_1}^G)^b(h_1\otimes \cdots h_{j-1} 
      \otimes  
       (\m_{k_2,\beta_2}^G)^b(h_j\otimes \cdots  \otimes h_{j+k_2-1})\otimes \cdots \otimes h_k
      ) \\
= &  e^{\partial \beta \cap b}
\sum_{
\substack
{
k_1,k_2 \geq 0 \\ 
k_1+k_2=k+1\\
\beta_1,\beta_2 \in \pi_2(X,L)\\ 
\beta_1+\beta_2=\beta
}
} 
      \sum_{j=1}^{k_1} 
      (-1)^* \m_{k_1,\beta_1}^G(h_1\otimes \cdots \otimes h_{j-1} 
      \otimes \m_{k_2,\beta_2}^G(h_j\otimes \cdots \otimes  h_{j+k_2-1})\otimes \cdots \otimes h_k
      )   \\    
= &  0 \quad \text{by Proposition \ref{mk is Ainf}}. 
\end{align*} 
\end{proof}

\begin{defn}
Let $u\in \interior \Delta$. 
Define the \textbf{potential function} 
\[ \po_G: \bigcup_{u\in \interior \Delta}  \{u\}\times H^1(\mu^{-1}(u),\Lambda_0/(2\pi i \Z))\to \Lambda \] 
by 
\[ (\m_0^G)^b(1) = \po_G^u(b)e .\]

\end{defn} 
 
By Theorem \ref{potential function of a compact symplectic toric manifold} (\cite{I} Proposition 4.6), for all $u\in \interior \Delta$,  
there exists $\po^u(b)\in \Lambda_0$ such that 
\[
\po^u(b)PD[L(u)]e  
=
\exp(\partial \beta \cap b)\m_0(1)
=
\exp(\partial \beta \cap b)  \m_0^G(1)
= 
(\m_0^G)^b(1) .
\]
Thus, $\po_G(b)$ is defined and equal to $\po(b)$ for all $b\in H^1(\mu^{-1}(u),\Lambda_0)$. 
For this reason, we will omit $G$ in the notation of the potential function from now on. 

Note that we have the inclusion
\[\bigcup_{u\in \interior \Delta}  \{u\}\times H^1(\mu^{-1}(u),\Lambda_0/(2\pi i \Z)) \to (\Lambda^*)^n  \]
via 
\begin{equation}
\label{(u,b) in Lambda^n}
\left(u_1,\ldots,  u_n, \sum_{i=1}^nx_ie_i\right)\mapsto (\exp(x_1)T^{u_1},\ldots, \exp(x_n)T^{u_n})=: (y_1,\ldots, y_n).    
\end{equation} 
Then the potential function takes the form of a formal Laurent series in $y_1,\ldots, y_n$. 

\begin{cor}
\label{m1b squares to 0}
For any $u\in \interior \Delta$, $\forall b\in H^1(\mu^{-1}(u);\Lambda_0)$, 
we have 
$ (\m_1^G)^b \circ (\m_1^G)^b=0$. 
\end{cor}
\begin{proof}  
For any $x\in C_G(L,\novringe)$, by the $\Ainf$-relations (Proposition \ref{mkb is Ainf}), we have
\begin{align*}
(\m_1^G)^b \circ (\m_1^G)^b(x)
= - (\m_2^G)^b\lrp{ (\m_0^G)^b(1)\otimes x } +(-1)^{\deg x}(\m_2^G)^b (x\otimes (\m_0^G)^b(1)) 
=0. 
\end{align*}
 
\end{proof}

\cite{I} proves that one can express the potential function for a compact Fano toric manifold $(X,\omega)$ purely from the information of its moment polytope as follows.  
\begin{thm}[Theorem 4.5 \cite{I}]
\label{Fano potential}
Let $(X,\omega, T^n, \mu)$ be a compact symplectic toric Fano manifold with moment polytope 
\begin{equation} \label{moment polytope}
\Delta = \mu(X) = \bigcap_{i=1}^m\{u\in \tn^* \mid \pair{u,v_i}-\lambda_i\geq 0 \},    
\end{equation} 
where $m$ is the number of the facets of $\Delta$, $v_i = (v_{i,1}, \ldots, v_{i,n})$ is the inner normal vector of the $i$-th facet. 
We denote the affine function $\pair{u,v_i}-\lambda_i$ by $l_i(u)$. 
On 
\[ \bigcup_{u\in \interior \Delta}\{u\}\times  H^1\lrp{\mu^{-1}(u),\frac{\Lambda_0}{2\pi i \Z}}  ,\]
we have
\[ \po\left(u_1,\ldots,  u_n, \sum_{i=1}^nx_ie_i\right) =  \sum_{i=1}^m \exp( \pair{v_i,x})T^{l_i(u)}, \]
where $x=(x_1,\ldots, x_n)$. 
In particular, if we use the coordinates \autoref{(u,b) in Lambda^n}, the potential function
$\po$ defines a Laurent polynomial 
\begin{equation}\label{leading potential}
\po = \sum_{i=1}^m y_1^{v_{i,1}}\cdots y_n^{v_{i,n}}T^{-\lambda_i} \in \Lambda[y_1^{\pm 1},\ldots, y_n^{\pm 1}].     
\end{equation}

\end{thm} 
When $(X,\omega, T^n, \mu)$ is compact symplectic toric but not necessarily Fano, the same formula \autoref{leading potential} computes the leading order potential function $\po_0$ of $X$. 
\begin{thm}[Potential function of a compact symplectic toric manifold, \cite{I} Theorem 4.6]
\label{potential function of a compact symplectic toric manifold}
  Let $(X,\omega, T^n, \mu)$ be a compact symplectic toric manifold with moment polytope \autoref{moment polytope}. 
    Let $u\in \interior \Delta$ and $b\in H^1(L(u),(\Lambda_0/2\pi i))$. 
   Then there exists an index set $I\subset \N$ such that $\forall 1\leq i \leq m, j\in I$, there exist 
   $r_j \in \Q$, $e_{j}^i \in \N$, and $\rho_j>0$ satisfying the following. 
   \begin{enumerate}[i)]
       \item $\sum\limits_{i=1}^m e_{j}^i>0  \quad \forall 1\leq j\leq n$. 
       \item If we let $l_i(u) =\pair{u,v_i}-\lambda_i$, 
       \[ v_{j,k}' = \sum\limits_{i=1}^m e_{j}^i v_{i,k},\quad  l_j' = \sum\limits_{i=1}^m e_{j}^i l_i, \qand v_j' = (v_{j,1}',\ldots, v_{j,n}'),\]
       then the potential function
       is given by 
        \begin{equation} 
         \po ( u, \sum_{i=1}^nx_ie_i) - \sum_{i=1}^m \exp (\pair{v_i,x})T^{l_i(u)} = \sum_{j\in I} r_j   \exp (\pair{v_j',x})T^{l_j'(u)+\rho_j}.
       \end{equation}  
    In particular, if we use the coordinates \autoref{(u,b) in Lambda^n},
       \begin{align}
       \label{potential on compact non-Fano}
        \po 
     =     &\sum_{i=1}^m y_1^{v_{i,1}}\cdots y_n^{v_{i,n}}T^{-\lambda_i} + \sum_{j\in I } r_j \left(\prod_{i=1}^m(y_1^{ v_{i,1}}\cdots y_n^{v_{i,n}}T^{-\lambda_i})^{e_j^i}\right)
       T^{\rho_j}  
       \end{align}
   \end{enumerate}
\end{thm}

The rest of the section is devoted to the proof of Theorem \ref{main theorem}. 

Since $(\m_1^G)^b$ commutes with $d_G$, it is defined on $H^1_G(L)$.
\begin{lem}\label{m1b is partial derivative}
Let $u\in \interior \Delta$ and let
\[ b= \sum\limits_{i=1}^n c_i \alpha_i \in H^1(L(u),\Lambda_0), \quad \text{where }  \]
where $\alpha_{r+1},\ldots, \alpha_n$ generates $H_G^1(L(u),\Z)$. 
Then for any $r+1\leq i \leq n$, we have
\[ (\m_1^G)^b(\alpha_i)=\left(\frac{\partial \po^u}{\partial c_i}(b)\right) PD[L(u)]e .  \]
\end{lem}
\begin{proof}
For each $1\leq j \leq m$, let $\beta_j \in H_2(X,L(u),\Z) \cong \pi_2(X,L(u))$ be the class of the basic disc, which is a small disc transverse to $\mu^{-1} \lrp{j\text{-th facet of }\Delta}$. 
Let $\m_{k,\beta}$ be the ordinary $\Ainf$-operator on the de Rham model. 
Note that 
\[ \po^u (b)  PD[L(u)] e = \sum_{
\substack{\beta\in \pi_2(X,L(u))\\
I_\mu(\beta)=2}}\sum_{k=0}^{\infty}(\m_{k,\beta})(b\otimes \cdots \otimes b) T^{\frac{\omega(\beta)}{2\pi}}e \]
only involves moduli spaces $\Mq_{k+1}(L(u),J,\beta)$ for $I_\mu(\beta)=2$. 
On the one hand, 
\[ \deg (\m_{k,\beta}(b^{\otimes k})) = k - ((n-3+I_\mu(\beta)+k+1)-n) = 2-I_\mu(\beta) \geq 0  \]   
implies  
$I_\mu(\beta) \leq 2$. 
On the other hand, 
for every $\beta\in \pi_2(X,L(u))$, the evaluation map
$\evaluation_0: \Mq_{k+1}^{main}(L(u),J,\beta) \to L$ is a submersion. 
In particular, whenever $\Mq_{k+1}^{main}(L(u),J,\beta)\ne \emptyset$, 
its dimension is no less than $\dim L(u) = n$: 
\[  n-3+I_\mu(\beta) + 1 \geq n 
\quad \Rightarrow \quad I_\mu(\beta) \geq 2. \]
This proves the claim. 
Therefore,  for $r+1\leq i\leq n$, we have
\begin{align*}
 \left(\frac{\partial \po^u}{\partial c_i}(b)\right)  e
= &  \frac{\partial }{\partial c_i} \left(\sum_{
\substack{\beta\in \pi_2(X,L(u))\\
I_\mu(\beta)=2}}\sum_{k=0}^{\infty}(\m_{k,\beta})(b\otimes \cdots \otimes b) T^{\frac{\omega(\beta)}{2\pi}}e\right) \\
= & \sum_{
\substack{\beta\in \pi_2(X,L(u))\\
I_\mu(\beta)=2}} 
\sum_{k=1}^{\infty} \sum_{l=1}^k \m_{k,\beta} (b^{\otimes l-1} \otimes \alpha_i\otimes b^{\otimes k-l})T^{\frac{\omega(\beta)}{2\pi}}e 
 \\
= & \sum_{
\substack{\beta\in \pi_2(X,L(u))\\
I_\mu(\beta)=2}} \exp(\partial \beta\cap b)\m_{1,\beta}(\alpha_i) T^{\frac{\omega(\beta)}{2\pi}}e  
 \\
= &  \sum_{
\substack{\beta\in \pi_2(X,L(u))\\
I_\mu(\beta)=2}} \exp(\partial \beta \cap b)\m_{1,\beta}^G (\alpha_i) T^{\frac{\omega(\beta)}{2\pi}}e \\
= & (\m_1^G)^b(\alpha_i). 
\end{align*}
\end{proof}

\begin{cor}\label{weak bounding cochains are where partial derivatives in the normal directions vanish}
    $(\m_1^G)^b|_{H_G(L(u),\Lambda_0)} =0 $ if and only if 
    \[  \frac{\partial \po^u}{\partial c_i}(b) =0 \qquad \forall r+1\leq i \leq n.   \]
    
\end{cor}
 
\begin{proof}
   Since the $G$-action on $L(u)$ is free, 
   $H_G^1(L(u),\R) \cong H^1(L(u)/G,\R)$ generates $H_G(L(u),\R)\cong H(L(u)/G,\R) \cong H(T^{n-r},\R)$. 
    Thus,  $(\m_1^G)^b=0$ on $H_G(L(u),\R)$ 
    if and only if 
    $(\m_1^G)^b =0$  on $H_G^1(L(u),\R)$. 
    By Proposition \ref{m1b is partial derivative}, 
    this holds 
    if and only if 
    \[\frac{\partial \po^u}{\partial c_i}(b) =0 \qquad \forall r+1\leq i \leq n.\]
\end{proof}

\subsection{Spectral sequences}
\label{Section spectral sequences}
 Let $u \in \interior \Delta$, $L = \mu^{-1}(u)$, and $b\in H^1(L,\Lambda_0)$.  
Let
 \[ C = \Omega_G(L,\R) \complete{\otimes} \novringe, \quad \delta = (\m_1^G)^b. \] 
 Then 
 $( C[1],\delta)$ is a cochain complex. 

 Define 
 \[ 
 F^\lambda C^p = \complete{\bigoplus_{ \substack{m, n \in \N \\ m+2n = p}}} \Omega_G^m(L) \otimes (T^{\lambda}\Lambda_0\cdot e^{n}). 
 \]
 Define 
 $E: C^p \to \R$  by  $ E(x) = \lambda $
if $x\in F^{\lambda} C^p$ but $x\not \in F^{\lambda'} C^p $ for any $\lambda'>\lambda$. 

Let $\delta_{1,0}= (\m_{1,\beta=0}^G)^b  = \m_{1,\beta=0} = d_G$.  
Since non-constant pseudoholomorphic discs with boundary on $L$ have a universal energy lower bound, 
there exists $\lambda''>0$ such that, $\forall \lambda$, $\forall x\in F^\lambda C$, we have
$(\delta - \delta_{1,0} ) x\in F^{\lambda + \lambda''} C$. 
Let $0<\lambda_0<\lambda''$. 
We use the $\lambda_0$ to define a decreasing integral filtration as follows. 
For any $q\in \N$, let 
 \[ \fil^qC^p = \complete{\bigoplus_{ \substack{m, n \in \N \\ m+2n = p}}} \Omega_G^m(L) \otimes (T^{q\lambda_0}\Lambda_0\cdot e^{n}) \]
 and let $\fil^{\infty} C^p =\{0\}$. 

 \begin{defn}
     Define 
     \begin{align*}
A_r^{p,q} & := \fil^qC^p \cap \delta^{-1}(\fil^{q+r-1}C^{p+1})   \\
Z_r^{p,q} & : = A_r^{p,q}+\fil^{q+1}C^p =\fil^qC^p \cap \delta^{-1}(\fil^{q+r-1}C^{p+1}) +\fil^{q+1}C^p   \\
B_r^{p,q} & := 
\fil^qC^p \cap \delta (\fil^{q-r+2}C^{p-1})+ \fil^{q+1}C^p  \\
E_r^{p,q} & :=\frac{A_r^{p,q} }{B_r^{p,q}\cap A_r^{p,q}}. 
\end{align*} 
For any $r\geq 0$, we have 
\[
B_{r+1}^{p,q} = \fil^{q}C^p \cap \delta (\fil^{q-r+1}C^{p-1}) + \fil^{q+1}C^p \subset B_r^{p,q} 
\]
and 
\[Z_{r+1}^{p,q} = \fil^{q}C^p \cap \delta^{-1} (\fil^{q+r}C^{p+1}) + \fil^{q+1}C^p  \subset Z_{r}^{p,q}. \]
\end{defn}
It is easy to check the following. 
\begin{lem}\label{isomorphisms for equivalent definitions of Er}
Let $R$ be a commutative ring. 
Suppose $A,M, T$ are $R$-modules, where $M\subset A$ is a submodule. 
Then we have the following:
\begin{enumerate}[i)]
    \item $\dfrac{A}{A\cap (M+T)} \cong  \dfrac{A+T}{M+T}$ ; 
    \item $A\cap (M+ T) = M+A\cap T$. 
\end{enumerate}
\end{lem}

\begin{prop}\leavevmode
\begin{enumerate}[i)]
\item 
$A_r^{p,q}\cap B_r^{p,q} =  \fil^qC^p \cap \delta (\fil^{q-r+2}C^{p-1}) +  \fil^{q+1}C^p \cap \delta^{-1}(\fil^{q+r-1}C^{p+1}) $. 
\item $E_{r}^{p,q} \cong \dfrac{Z_r^{p,q}}{B_r^{p,q}}$
\end{enumerate}
\end{prop}
\begin{proof}
\begin{enumerate}[i)]
    \item Note that $\fil^qC^p \cap \delta (\fil^{q-r+2}C^{p-1}) \subset A_r^{p,q}$. 
    Thus, by Lemma \ref{isomorphisms for equivalent definitions of Er} ii), 
    \begin{align*}
   A_r^{p,q}\cap B_r^{p,q} 
   & =  \fil^qC^p \cap \delta (\fil^{q-r+2}C^{p-1}) + A_r^{p,q}\cap \fil^{q+1}C^p \\
   & =\fil^q C^p \cap \delta (\fil^{q-r+2}C^{p-1}) +  \fil^{q+1}C^p \cap \delta^{-1}(\fil^{q+r-1}C^{p+1})      
    \end{align*}
    \item Apply Lemma \ref{isomorphisms for equivalent definitions of Er} i)
    to \[ A = A_r^{p,q}, \qquad M = \fil^qC^p \cap \delta(\fil^{q-r+2}C^{p-1}), \qquad T = \fil^{q+1}C^p, \] we have
    \[ 
E_{r}^{p,q}  = \frac{A_r^{p,q} }{B_r^{p,q}\cap A_r^{p,q}}  \cong \frac{A_r^{p,q}+\fil^{q+1}C^p }{B_r^{p,q}}   \cong \frac{Z_r^{p,q}}{B_r^{p,q}}. \]
\end{enumerate}
\end{proof}

\begin{defn}[$E_{\infty}^{p,q}$] 
For a fixed pair $(p,q)$, if $r>q+2$, then 
\begin{align*}
B_{r}^{p,q}   
&  =\;  \fil^qC^p \cap \delta (\fil^{q-r+2}C^{p-1})+ \fil^{q+1}C^p \\
&  =\;  \fil^qC^p \cap \delta ( C^{p-1})+ \fil^{q+1}C^p
\end{align*}
is independent of $r$. 
Moreover, for $r>q+2$, we have inclusions
\[ 
\cdots \subset  Z_{q+5}^{p,q} 
\subset Z_{q+4}^{p,q} 
\subset Z_{q+3}^{p,q}, \]
and thus an inverse system
\[
\cdots \subset E_{q+5}^{p,q}
\subset  E_{q+4}^{p,q} 
\subset   E_{q+3}^{p,q}. \]
We define 
\[E_{\infty}^{p,q} := \lim_{\longleftarrow} E_{r}^{p,q}.\]
\end{defn}  
We will use the following lemma (proved in Proposition 6.3.9 \cite{fooobook1}) to prove Theorem \ref{spectral sequence statements}. 
\begin{lem}[Proposition 6.3.9 in \cite{fooobook1}]\label{Prop 6.3.9}
Let $C = \complete{\bigoplus\limits_{p\in \N} }C^p$ be a graded finitely generated and free module over $\Lambda_0$ such that $C$ and each $C^p$ is complete with respected to the energy filtration. 
Let $\delta: C^* \to C^{*+1}$ be a 
degree $1$ operator such that 
\begin{equation*}
\delta\circ \delta = 0 \quad \text{ and }\quad \delta (F^{\lambda}C)\subset F^{\lambda}C.     
\end{equation*}
Let $W\subset C^p$ be a finitely generated $\Lambda_0$-submodule.  
Then there exists a constant $c$ depending on $W$ but not on $\lambda$ such that 
\begin{equation}
\label{6.3.9 eqn}
\delta(W) \cap F^{\lambda}C^{p+1} \subset \delta \left(W\cap F^{\lambda-c}C^{p} \right). 
\end{equation}
\end{lem}

\begin{thm}
\label{spectral sequence statements}
Let $u\in \interior \Delta$, 
and $L = \mu^{-1}(u)$. 
Let 
$b\in H^1(L,\Lambda_0)$. 
There exists a spectral sequence with the following properties.  
\begin{enumerate}[i)]
    \item The $E_2$-page is given by $ H_G(L,\novringe)$, where 
    \begin{equation}
    E_2^{p,q}  \cong \complete{\bigoplus_{m\in \N}} H_G^{p-2m}(L,\R)\complete{\otimes}  \frac{\fil^q(\Lambda_0\cdot e^m)}{\fil^{q+1}(\Lambda_0\cdot e^m) } 
\end{equation}
\item $\forall r, p\in \N, \forall q\in \Z$, there exists a well-defined map $\delta_r: E_r^{p,q}\to E_r^{p+1,q+r-1}$ satisfying:
\begin{enumerate}[a)]
    \item $$\delta_r^{p+1,q+r-1} \circ \delta_r^{p,q}=0;$$
    \item \begin{equation}\label{Einfty}
    E_{r+1}^{p,q} \cong \frac{\ker \delta_r^{p,q}}{\image \delta_r^{p-1,q-r+1}}
    \end{equation}
    \item \[e^{\pm 1} \circ \delta_r^{p,q}=\delta_r^{p\pm 2, q}\circ e^{\pm 1}\]
\end{enumerate}
\item 
There exists some $r_0\geq 2$ with 
\begin{equation}\label{E2} E_2^{p,q} \quad \Longrightarrow\quad E_{r_0}^{p,q} \cong  E_{r_0+1}^{p,q} \cong \cdots \cong E_{\infty}^{p,q} = \frac{\fil^q HF_{G}^p\Bigp{ (L,b),(L, b),  \novringe}}{\fil^{q+1} HF_{G}^p\Bigp{ (L,b) , (L, b),  \novringe}}. \end{equation} 
\end{enumerate}
\end{thm}

\begin{proof}
\begin{enumerate}[i)]
\item We compute 
\begin{align*}
    A_2^{p,q} 
    & = \{ x\in \fil^{q}C^p
    \mid \delta x \in \fil^{q+1}C^{p+1} \}   = \ker d_G \cap \fil^{q}C^p,  
    \\
    B_2^{p,q} 
    & =  \delta( \fil^{q} C^{p-1})\cap \fil^{q}C^p 
    + \fil^{q+1}C^{p}   =  \image   d_G  \cap   \fil^{q}C^p 
    + \fil^{q+1}C^{p}.\\
 \Rightarrow \quad  
 E_2^{p,q} & = \frac{A_{2}^{p,q}}{B_2^{p,q}\cap A_{2}^{p,q}} \cong \frac{A_{2}^{p,q}+ \fil^{q+1}C^p}{B_2^{p,q}}\\
   & \cong \frac{\Bigp{ \ker d_G \cap \fil^{q}C^p + \fil^{q+1}C^p }/\fil^{q+1}C^p  }
   { \Bigp{\image d_G  \cap   \fil^{q}C^p + \fil^{q+1}C^{p}} /\fil^{q+1}C^p } \\
    & \cong \bigoplus_{m\in \N} H_G^{p-2m}(L;\R)\otimes  (\fil^q(\Lambda_0\cdot e^m)/\fil^{q+1}(\Lambda_0\cdot e^m)) 
\end{align*}

\item Define $\delta_r[x] = [\delta x] \in E_r^{p+1,q+r-1}$. Then 
\begin{align*}
\delta (A_r^{p,q} )
& =\fil^{q+r-1}C^{p+1} \cap \delta (\fil^qC^p) \\
& \subset  \fil^{q+r-1}C^{p+1} \cap \delta^{-1} (\{0\}) \\
& \subset  \fil^{q+r-1}C^{p+1} \cap \delta^{-1} (\fil^{q+2r-2}C^{p+2}) = A_r^{p+1,q+r-1} .     
\end{align*}
Also, 
\begin{align*}
 \delta (A_r^{p,q}\cap B_r^{p,q}) 
 & = \delta  \Bigp{\delta (\fil^{q-r+2} C^{p-1}) \cap\fil^qC^p +   \delta^{-1}(\fil^{q+r-1}C^{p+1})\cap  \fil^{q+1}C^p }  \\
 & = \delta \Bigp{\fil^{q+1}C^p \cap \delta^{-1}(\fil^{q+r-1}C^{p+1})  } \\
 & = \delta (\fil^{q+1}C^p ) \cap \fil^{q+r-1}C^{p+1}\\
 & \subset  \delta (\fil^{q+1}C^{p }) \cap\fil^{q+r-1}C^{p+1} +  \fil^{q+r}C^{p+1} \cap \delta^{-1}(\fil^{q+2r-2}C^{p+2}) \\
 & = A_r^{p+1,q+r-1}\cap B_r^{p+1,q+r-1} .
\end{align*}
Therefore, $\delta_r$ is well-defined.
\begin{enumerate}[a)]
\item $\delta_r^{p+1,q+r-1} \circ \delta_r^{p,q}=0$ follows from $\delta \circ \delta =0$. 
\item 
We have
\begin{equation*}
\begin{split}
& \ker \delta_r^{p,q}  \\
= & \frac{A_r^{p,q}\cap \delta^{-1}(A_r^{p+1,q+r-1}\cap B_r^{p+1,q+r-1})}{A_r^{p,q}\cap B_r^{p,q} }\\
= & \frac{A_r^{p,q}\cap \delta^{-1}( B_r^{p+1,q+r-1})}{A_r^{p,q}\cap B_r^{p,q} } \quad \text{ since } A_r^{p,q}\subset \delta^{-1}(A_r^{p+1,q+r-1})\\
= & \frac{\fil^qC^p \cap \delta^{-1}(\fil^{q+r-1}C^{p+1}) \cap \delta^{-1} \Bigp{\delta (\fil^{q+1}C^{p } )\cap \fil^{q+r-1}C^{p+1} + \fil^{q+r}C^{p+1}}}{A_r^{p,q}\cap B_r^{p,q}}
\\
= & \frac{\fil^qC^p \cap  \delta^{-1} \Bigp{\delta (\fil^{q+1}C^{p } )\cap \fil^{q+r-1}C^{p+1} + \fil^{q+r}C^{p+1}} }{A_r^{p,q}\cap B_r^{p,q}} \\
& \qquad  [ \text{ since  } \Bigp{\delta (\fil^{q+1}C^{p } )\cap \fil^{q+r-1}C^{p+1} + \fil^{q+r}C^{p+1} } \subset \fil^{q+r-1} ] \\
= & \frac{\fil^{q+1}C^p \cap \delta^{-1}(\fil^{q+r-1}C^{p+1}) + \fil^q C^p \cap \delta^{-1}(\fil^{q+r}C^{p+1})}{A_r^{p,q}\cap B_r^{p,q}}.
\end{split}
\end{equation*}
\begin{equation*}
\begin{split}
& \image \delta_{r}^{p-1,q-r+1} \\
= & \frac{\delta (A_r^{p-1,q-r+1}) + A_r^{p,q}\cap B_r^{p,q}}{ A_r^{p,q}\cap B_r^{p,q}}\\ 
= & \frac{\fil^qC^p\cap \delta (\fil^{q-r+1}C^{p-1}) + \fil^q C^p\cap  \delta (\fil^{q-r+2}C^{p-1}) + \fil^{q+1}C^p \cap \delta^{-1}(\fil^{q+r-1}C^{p+1})}{ A_r^{p,q}\cap B_r^{p,q}}\\
= & \frac{\fil^qC^p\cap \delta (\fil^{q-r+1}C^{p-1})  + \fil^{q+1}C^p \cap \delta^{-1}(\fil^{q+r-1}C^{p+1})}{ A_r^{p,q}\cap B_r^{p,q}}  . 
\end{split} 
\end{equation*}
Hence,
\begin{equation*}
\begin{split}
& \frac{\ker \delta_r^{p,q}}{\image \delta_r^{p-1,q-r+1}} \\
= & \frac{\fil^{q+1}C^p \cap \delta^{-1}(\fil^{q+r-1}C^{p+1}) + \fil^q C^p \cap \delta^{-1}(\fil^{q+r}C^{p+1})}{\fil^qC^p\cap \delta (\fil^{q-r+1}C^{p-1})  + \fil^{q+1}C^p \cap \delta^{-1}(\fil^{q+r-1}C^{p+1})}\\
 \cong &  \frac{\fil^{q}C^{p}\cap  \delta^{-1}(\fil^{q+r}C^{p+1})}{ \fil^qC^p\cap \delta(\fil^{q-r+1} C^{p-1})  + \fil^{q+1}C^p \cap \delta^{-1}( \fil^{q+r}C^{p+1}) } \quad \text{by Lemma }\ref{isomorphisms for equivalent definitions of Er} \\
= & \frac{A_{r+1}^{p,q} }{A_{r+1}^{p,q}\cap B_{r+1}^{p,q}} = E_{r+1}^{p,q}. 
\end{split} 
\end{equation*}
\item follows from the fact that $\delta$ commutes with multiplying by  $e^{\pm 1}$ 
and that 
$\deg e^{\pm 1} = \pm 2$. 
\end{enumerate}
\item 
We consider the restriction of the original spectral sequence to pages starting from $E_2$. Take $C =E_2$ and $W = E_2^{p,q}$ in Lemma \ref{Prop 6.3.9} to get a constant $c$ such that \autoref{6.3.9 eqn} holds. 

Let $r>2$ be big enough such that $(r-1)\lambda_0 - c > \lambda_0$. 
Let $r\geq r_0$. 
Then 
 \[ (q+r-1)\lambda_0-c = q\lambda_0 +(r-1)\lambda_0-c  > (q+1)\lambda_0  \]
 and thus, by \autoref{6.3.9 eqn}, 
\[
\delta(\fil^q C^{p}) \cap \fil^{q+r-1}C^{p+1} \subset \delta \left(\fil^{q+1}C^{p} \right).     
\]
Consider any $x\in A_r^{p,q} =  \fil^qC^p \cap \delta^{-1}(\fil^{q+r-1}C^{p+1}) $.   
Then $\delta x \in \delta (\fil^{q+1}C^p)$ and thus there exists 
\begin{equation}
\label{consequence of 6.3.9} y \in \fil^{q+1}C^p \cap \delta^{-1}(\fil^{q+r-1}C^{p+1}) \subset A_r^{p,q}\cap B_r^{p,q} 
\end{equation}
such that $[\delta x] = [\delta y]$. 
Thus, $\delta_r[x] = \delta_r[y]=0$.  
    Therefore, there exists a $r_0 \gg 2$ such that  we have
    \[
    E_{r_0}^{p,q} \cong  E_{r_0+1}^{p,q} \cong \cdots \cong E_{\infty}^{p,q}.
    \]

 Consider the map
 \[
 \pi_{p,q}: \fil^q H^p (C,\delta )\to E_{\infty}^{p,q} 
 \]
defined as follows. 
An element $[x] \in \fil^q H^p (C,\delta )$ is represented by 
\[ 
x\in \fil^q C^p \cap \delta^{-1}(0)\subset  \fil^q C^p \cap \delta^{-1}(\fil^{q+r-1}C^{p+1}) = A_r^{p,q}
\]  
for any $r\geq \max \{ r_0, q+2\}$. 
We define $\pi_{p,q}[x]$ to be the class represented by $x$ in $E_{\infty}^{p,q} \cong \frac{A_{r}^{p,q}}{A_{r}^{p,q}\cap B_{r}^{p,q}}$. 
    
    Suppose $x, x' \in \fil^q  H^p(C,\delta)$ with $x-x'=\delta y$ for some $y\in \fil^q C^{p-1}$.  
    Let $r\geq \max \{ r_0, q+2\} $. 
    Then 
     \[ \delta y\in \delta (\fil^q C^{p-1}) = \delta (\fil^{q-r+2}C^{p-1}) \cap \fil^q C^{p} \subset A_{r}^{p,q}\cap B_{r}^{p,q} . \] 
    This implies 
    $\pi_{p,q}[x]  = \pi_{p,q} [x']$ in $E_{r}^{p,q}$ and hence the well-definedness of $\pi_{p,q}$. 
 
    Let $[x] \in E_{\infty}^{p,q}$. Then for $r$ large we have $[x] \in E_{r}^{p,q}$ and that 
    $\exists y$ satisfying \autoref{consequence of 6.3.9}.  
    \[ [x-y]=[x]\in E_{r}^{p,q}  \qand \delta (x-y)=0.  \]
    Thus, $[x-y]\in \fil^qH^p(C, \delta)$ with $ \pi_{p,q}[x-y] = [x]$, showing the surjectivity of $\pi_{p,q}$. 
 
    Suppose $[x]\in \ker \pi_{p,q} $. 
    Then, for $r$ large enough, we have 
    \[ x \in \delta (\fil^{q-r+2}C^{p-1} ) +\fil^{q+1}C^{p} 
    = \delta ( C^{p-1} ) +\fil^{q+1}C^{p}. \]
   Hence, $[x] \in \fil^{q+1} H^p(C,\delta)$ and thus $\ker \pi_{p,q} \subset  \fil^{q+1} H^p(C,\delta)$. 
  Moreover, for large enough $r$, 
   \[ A_r^{p,q} \cap B_r^{p,q} = \delta(C^{p-1}) + \fil^{q+1} C^{p} \Rightarrow \fil^{q+1} H^p(C,\delta) \subset   \ker \pi_{p,q}.\]
  Therefore, $\ker \pi_{p,q} = \fil^{q+1} H^p(C,\delta).$
\end{enumerate}
\end{proof}

\begin{thm}
\label{main theorem}
Let $u\in \interior \Delta$ and $b\in H^1(L(u),\Lambda_0)$.
Then 
$b = \sum\limits_{i=1}^n x_ie_i  = \sum\limits_{i=1}^n c_i\alpha_i$, 
where $e_1,\ldots, e_n$ generate $H^1(L(u),\Z)\cong M$ as in Section \ref{Section setup} and $\alpha_{r+1},\ldots, \alpha_n$ generate $H_G^1(L(u),\Z)$. 
Let 
\begin{equation}
\label{CritGPO definition}
 \Crit_G^{\Delta}(\po): = \set{ (y_1,\ldots, y_n)\in (\Lambda^*)^n }{ 
 \begin{aligned}
& \frac{\partial \po}{\partial c_i}=0 \quad \forall r+1\leq i \leq n \\
& (\val(y_1),\ldots, \val(y_n)) \in \interior \Delta
\end{aligned} 
} 
\end{equation}
and let  $\mlag_G(X,\omega)$ be the set
\begin{equation}
\label{mlagG definition}  
\set{ 
(u,b) \in \bigcup_{u\in \interior \Delta} \{u\}\times H^1\lrp{L(u), \frac{\Lambda_0}{2\pi i \Z}} 
}
{ 
HF_G((L(u), b),(L(u), b),\nove) \ne 0
}. 
\end{equation}
Then the following are equivalent. 
\begin{enumerate}[i)]
\item \label{crit po}  $ (y_1,\ldots, y_n) = (e^{x_1}T^{u_1}, \ldots,e^{x_n}T^{u_n} )\in \Crit_G^{\Delta}(\po)$. 
    \item \label{isomorphic to HG}
   \begin{equation*}
       HF_G\left( (\mu^{-1}(u), b),(\mu^{-1}(u), b),\novringe\right)\cong H_G(\mu^{-1}(u), \R)\otimes_{\R} \novringe. 
   \end{equation*} 
    \item \label{non-zero HF}
    $(u,b) \in \mlag_G(X,\omega)$; i.e. \[   HF_G\left( (\mu^{-1}(u), b),(\mu^{-1}(u), b),\novringe\right)\ne  0 . \]
\end{enumerate}
\end{thm}

\begin{proof}
The equivalence of \ref{crit po} and \ref{non-zero HF} follows from Proposition \ref{weak bounding cochains are where partial derivatives in the normal directions vanish}. 
\ref{isomorphic to HG} $\Rightarrow$ \ref{non-zero HF} is trivial. 
If \ref{isomorphic to HG} does not hold, then the spectral sequence in Theorem \ref{spectral sequence statements} does not collapse at $E_2$ page. 
This is equivalent to saying 
$(\m_1^G)^b(\alpha) \ne 0$ for some element $\alpha \in H_G^1(L,\R)\otimes_{\R} \novringe$. 
By degree count, $(\m_1^G)^b(\alpha)\in \novringe$, which implies $E_3=0$ and thus $HF_G\left( (\mu^{-1}(u), b),(\mu^{-1}(u), b),\novringe\right)= 0$. 
\end{proof}

\section{\texorpdfstring{$\Crit_G^{\Delta}(\po)$}{Crit G Delta (PO)}} 
\label{Section CritG(PO)}
We denote the coordinate valuation map by 
\[ \trop: \Lambda^n \to (\R \cup \{\infty\})^n, \qquad (y_1,\ldots, y_n)\mapsto (\val(y_1), \ldots,\val(y_n) ).  \]
\subsection{
\texorpdfstring{$\Crit_G^{\Delta}(\po)$}{Crit G Delta (PO)} is a rigid analytic space}

\begin{lem}
\label{polytopal domain}
    Suppose $\Delta\subset \R^n$ is a polytope of the form \autoref{moment polytope}. 
    If $A$ is an affinoid algebra such that  $\trop^{-1}(\Delta) = \Sp A$, then $A$ is a Cohen-Macaulay ring of dimension $n$. 
    Moreover,  $\trop^{-1}(\interior \Delta)$ is a rigid analytic space. 
\end{lem}

\begin{proof}
Suppose the moment polytope $\Delta$ is defined by $m$ affine inequalities. 
\[ \Delta = \bigcap_{i=1}^m \set{ u\in M_{\R} }{\pair{u,v_i} - \lambda_i\geq 0}, \]
where $m>n$ is the number of facets of $\Delta$ and each $v_i = (v_{i,1}\ldots, v_{i,n}) \in N_{\Z}$ is the inner normal vector of the $i$-th facet. 
Denote $y^{v_i} := y_1^{v_{i,1}}\cdots y_n^{v_{i,n}}$. 
Then 
\begin{align*}
\trop^{-1}(\Delta) 
& = 
\{ (y_1,\ldots, y_n)\in \Lambda^n\mid (\val(y_1), \ldots, \val(y_n))\in \Delta\} \\
& = 
\{ (y_1,\ldots, y_n)\in \Lambda^n\mid \val(y^{v_i}) - \val(T^{\lambda_i})\geq 0 \quad \forall 1\leq i \leq m\} 
\\
& = 
\set{ (y_1,\ldots, y_n)\in \Lambda^n }
{\left| \frac{y^{v_i}}{T^{\lambda_i}} \right| \leq 1 \quad \forall 1\leq i \leq m }.
\end{align*}

By the smoothness of the Delzant polytope, the linear map 
\[ \Tilde{\varphi}: \Z^{m} \twoheadrightarrow M_{\Z} \cong \Z^n, \qquad (c_1,\ldots, c_m)\mapsto 
\begin{pmatrix}
v_{1,1}  & \cdots & v_{m,1}\\
\vdots  & \ddots &  \vdots\\
v_{1,n} & \cdots & v_{m,n}
\end{pmatrix}
\begin{pmatrix}
c_1 \\
\vdots \\
c_m
\end{pmatrix} , 
\] 
is surjective, and it induces a surjective ring homomorphism 
\begin{equation} \label{varphi}
\varphi: 
\Lambda[z_1^{\pm 1},\ldots, z_m^{\pm 1}] \twoheadrightarrow 
\Lambda[y_1^{\pm 1},\ldots, y_n^{\pm 1}], \qquad z^c \mapsto y^{\Tilde{\varphi}(c)}.     
\end{equation}
Here $z^c := z_1^{c_1}\cdots z_m^{c_m}$ if $c=(c_1,\ldots, c_m)$. 
In particular, 
$ 
\varphi(z_i) =y^{v_i}. 
$  

Let $\lambda = (\lambda_1,\ldots, \lambda_m)$. 
Let  \begin{align*}
 T_{m,\lambda} 
 & 
 = \Lambda\pair{z_1 T^{-\lambda_1}, \ldots, z_mT^{-\lambda_m}},
 \end{align*}
Indeed, by \cite{Rabinoff} Proposition 6.9, $ \trop^{-1}(\Delta) \cong  \Sp A$, where 
\begin{equation}\label{affinoid algebra for polytope}
 A = T_{m, \lambda} / (\ker \varphi)T_{m, \lambda}   
\end{equation}
 is a $\Lambda$-affinoid algebra of dimension $n$ and $\trop^{-1}(\Delta)$ is a $\Lambda$-affinoid space. 
Then 
\[ 
\trop^{-1}(\interior \Delta) \cong \{ z\in \trop^{-1}(\Delta)\mid |z_i T^{-\lambda_i}|<1 \quad \forall 1\leq i \leq m \} 
\]
is an admissible open subset of $\trop^{-1}(\Delta)$ by Proposition \cite{bosch} 5.1/Proposition 7 (see Proposition \ref{strict norm inequalities define admissible opens}).  
Therefore, $ \trop^{-1}(\interior \Delta)$ is also a rigid analytic $\Lambda$-space. 
\end{proof} 

Moreover, 
by \cite{Rabinoff} Proposition 6.9 and 
\cite{Rabinoff} Remarks 6.5 and 6.6, 
\begin{align}
\label{restricted power series on Delta}
   \Lambda \pair{\Delta} 
  &  := \set{ \sum_{c\in \Z^n} a_cy^c} { \val (a_c) + \pair{u,c} \to \infty \quad \text{as } |c|\to \infty \quad \forall u\in \Delta }
\end{align} 
 is the completion of $ \Lambda[y_1^{\pm 1}, \ldots, y_n^{\pm 1}]$ with respect to the norm $|\cdot|_{\Delta}: \Lambda[y_1^{\pm 1}, \ldots, y_n^{\pm 1}] \to \R_{\geq 0}$, which is defined by
\[ \left| \sum_{c\in \Z^n} a_cy^c \right|_{\Delta}: = \sup_{\substack{c\in \Z^n\\ u\in \Delta}} |a_c|_{\Lambda} \cdot \exp(-\pair{u,c}). \]
We can alternatively argue that $\Lambda\pair{\Delta}$ is an affinoid algebra as follows. Since $\Lambda\pair{\Delta}$ is a $\Lambda$-Banach algebra with this norm and $|y^{v_i}T^{-\lambda_i}|_{\Delta} = \sup\limits _{u\in \Delta} e^{- (\pair{u,v_i}-\lambda_i )}  \leq e^0 <1$ for all $1\leq i \leq m$, we have a continuous homomorphism $\Phi: \Lambda\pair{z_1,\ldots, z_m} \to \Lambda\pair{\Delta}$ prescribed by $z_i\mapsto y^{v_i}T^{-\lambda_i}$ according to \cite{bgr} 6.1.1/Proposition 4.  
Moreover, by the smoothness of the polytope, since $v_1,\ldots, v_m$ contains a $\Z$-basis of $\Z^n$, every $c  \in \Z^n$ is a $\Z$-linear combination $ c = \sum_{i=1}^m c(v_i)v_i$ of $v_1,\ldots, v_m$. 
Thus, if $f = \sum_{c\in\Z^n} a_{c}y^c \in \Lambda\pair{\Delta}$, then 
\[ f = \Phi \lrp{ \sum_{c\in\Z^n} a_c T^{\sum\limits_{i=1}^m c(v_i)\lambda_i} z_1^{c(v_1)}\cdots z_m^{c(v_m)}}  . \] 
As $|c(v_1)|+\cdots +|c(v_1)|\to \infty $, we have 
\[ \val \lrp{ a_c T^{ \sum\limits_{i=1}^m c(v_i)\lambda_i} }  =\val(  a_c ) + \sum_{i=1}^m c(v_i)\lambda_i\leq  \val(  a_c ) + \sum_{i=1}^m c(v_i)\pair{u,v_i}\to \infty\]  for all $u\in \Delta$. 
Therefore, $\sum_{c\in\Z^n} a_c T^{\sum\limits_{i=1}^m c(v_i)\lambda_i} z_1^{c(v_1)}\cdots z_m^{c(v_m)}\in \Lambda\pair{z_1,\ldots, z_m}$.  
This shows that $\Phi$ is a continuous epimorphism and thus $\Lambda\pair{\Delta}$ is a $\Lambda$-affinoid algebra (\cite{bgr} 6.1.1/Definition 1), commonly denoted by $\Lambda\pair{ y^{v_1}T^{-\lambda_1}, \ldots,  y^{v_m}T^{-\lambda_m}}$.

\begin{prop}\label{CritGPO is a rigid analytic space}
Let $X$ be a compact symplectic toric manifold as in Section \ref{Section setup}. Then $\Crit_G^{\Delta}(\po)$ is a rigid analytic space. Moreover, the closure of the image of the map $\trop: \Crit_G^{\Delta}(\po)\to \Delta, (y_1,\ldots, y_n)\mapsto (\val(y_1), \ldots, \val(y_n))$ is a polytopal set, i.e. a union of polytopes.  
\end{prop}
\begin{proof}
By Proposition \ref{potential function of a compact symplectic toric manifold}, $\po$ takes the form of \autoref{potential on compact non-Fano}.  
For all $1\leq i\leq n $, let 
\begin{align}\label{partial derivatives}
 f_i 
& = \sum_{j=1}^n a_{i,j} \left(y_j\frac{\partial \po}{\partial y_j} \right) \nonumber 
\\
& = \sum_{j=1}^n a_{i,j} \left(\sum_{k=1}^m v_{k,j} y_1^{v_{k,1}}\cdots y_n^{v_{k,n}}T^{-\lambda_k} +  \sum_{s}\sum_{k=1}^m r_s v_{k,j}\prod_{k=1}^m (y_1^{v_{k,1} }\cdots y_n^{v_{k,n}}T^{-\lambda_k})^{e_s^k}\right).     
\end{align}
Then  
\begin{align*}
W 
&= V \lrp{f_{r+1},\ldots, f_n} \cap \trop^{-1}(\Delta) \\
& = \{ (y_1,\ldots, y_n)\in \trop^{-1}(\Delta) \mid f_i(y_1,\ldots, y_n) = 0 \quad \forall r+1\leq i \leq n\} \\
& = \Sp \frac{\Lambda\pair{\Delta}}{\lrp{f_{r+1},\ldots, f_n}} 
\end{align*}
is a $\Lambda$-affinoid space. 

By Proposition \cite{bosch} 5.1/Proposition 7 (see Proposition \ref{strict norm inequalities define admissible opens}), 
 \[  M = W \cap \trop^{-1}(\interior \Delta) = \set{ (y_1,\ldots, y_n)\in W }{ \left|\frac{y^{v_1}}{T^{\lambda_1}}\right|<1, \ldots \left|\frac{y^{v_m}}{T^{\lambda_m}}\right|<1 } \]
is a finite intersection of admissible open subsets and thus is an admissible open subset of $W$.  
Hence, 
$(M,\hol_W|_M)$ is a rigid analytic space. 

The last claim follows from \cite{GublerNA} Proposition 5.2. 
\end{proof}

\begin{prop}[$\Crit_G^{\Delta}(\po)$ of $\proj^n$]
\label{cpn rigid space dimension}
Suppose $G = \iota(T^r)$ is an $r$-dimensional subtorus of $T^n$ acting on $(\proj^n, T^n,\omega, \mu)$, which has moment polytope  
\[ 
\Delta=\set{ (u_1, \ldots, u_n) \in \R^n }{ 
\begin{aligned}
& u_i\geq 0 \quad \forall 1\leq i\leq n,\\
&  1-\sum_{i=1}^n u_i \geq 0
\end{aligned}}
. 
\]
Then $\Crit_G^{\Delta}(\po)$ is a rigid analytic space of dimension $r$.   

\end{prop}
\begin{proof}
By Theorem \ref{Fano potential}, 
\[ \po = y_1+\cdots y_n + \frac{T}{y_1\cdots y_n}. \]
Consider 
\begin{align*}
Q 
& = \set{ (y_1,\ldots, y_n)\in (\Lambda^*)^n }{ \sum_{j=1}^n a_{i,j} \left(y_j\frac{\partial \po}{\partial y_j} \right)=0 \quad \forall r+1\leq i \leq n   } \\
& = \set{ (y_1,\ldots, y_n)\in (\Lambda^*)^n }{ \sum_{j=1}^n a_{i,j} \left(y_j - \frac{T}{y_1\cdots y_n}\right)=0 \quad \forall r+1\leq i \leq n   }. 
\end{align*}
Then 
\[ \Crit_G^{\Delta}(\po) = Q\cap \trop^{-1}(\interior\Delta) = Q \cap \set{ y\in \trop^{-1}(\Delta)}{ \left| y^{v_i} T^{-\lambda_i} \right| < 1 \quad \forall 1\leq i \leq m }.\] 

We first consider 
\begin{align*}
W & = Q\cap \trop^{-1}(\Delta)  \\ 
& \cong  \Sp\frac{\Lambda \langle y_1,\ldots, y_n, z \rangle }{\left(y_1\cdots y_n z-T,\sum\limits_{j=1}^n a_{r+1,j} (y_j-z), \ldots,\sum\limits_{j=1}^n a_{n,j} (y_j-z) \right)} 
\end{align*}
Recall that
$  A = \begin{pmatrix}
    a_{1,1} & \cdots &  a_{1,n}\\
    \vdots & \ddots & \vdots \\
    a_{n,1} & \cdots & a_{n,n}
\end{pmatrix}$
is an invertible matrix. 
Let 
\[ A^{-1} = \begin{pmatrix}
    b_{1,1} & \cdots &  b_{1,n}\\
    \vdots & \ddots & \vdots \\
    b_{n,1} & \cdots & b_{n,n}
\end{pmatrix} 
.
\]  
Change variables by setting 
\[ 
\begin{pmatrix}
    X_1 \\
    \vdots \\
    X_n\\
    X_{n+1}
\end{pmatrix} 
= 
\begin{pmatrix}
    a_{1,1} & \cdots &  a_{1,n} & 0 \\
    \vdots & \ddots & \vdots \\
    a_{n,1} & \cdots & a_{n,n}& 0 \\
    0 & \cdots & 0 & 1
\end{pmatrix}
\begin{pmatrix}
    y_1-z \\
    \vdots \\
    y_n-z\\
    z 
\end{pmatrix} . 
\]
By \cite{bgr} 6.1.1/Proposition 4, 
this defines a continuous epimorphism
\[ 
\varphi: \Lambda\pair{X_1,\ldots, X_n, X_{n+1}} \to \frac{\Lambda \langle y_1,\ldots, y_n, z \rangle }{\left(y_1\cdots y_n z-T,\sum\limits_{j=1}^n a_{r+1,j} (y_j-z), \ldots,\sum\limits_{j=1}^n a_{n,j} (y_j-z) \right)} . 
\]
Then 
\begin{align*}
W
& \cong  
\Sp \frac{\Lambda \langle X_1,\ldots,X_n, X_{n+1} \rangle }{\left(  X_{n+1}\prod_{i=1}^n(X_{n+1}+\sum\limits_{j=1}^n b_{i,j}X_j)-T,X_{r+1},\ldots, X_n\right)} \\
 & \cong  \Sp \frac{\Lambda \langle X_1,\ldots,X_r, X_{n+1} \rangle }{\left(  X_{n+1}\prod_{i=1}^n( X_{n+1}+\sum\limits_{j=1}^r b_{i,j}X_j)-T\right)}
\end{align*}
By \cite{Rabinoff} Proposition 6.9 and \cite{Rabinoff} Theorem 4.6, 
the Tate algebra 
\[T_{r+1} =\Lambda \langle X_1,\ldots,X_r, X_{n+1} \rangle \]
is a Cohen-Macaulay ring of dimension $r+1$, which is also an integral domain. 
Since $X_{n+1}\prod_{i=1}^n( X_{n+1}+\sum\limits_{j=1}^r b_{i,j}X_j)-T \ne 0$, it is not a zero divisor and thus is a regular sequence in $T_{r+1}$. 
Thus, by \cite{cohen} Theorem 2.1.2, the Krull dimension of $W$ is $r+1-1 = r$.  

Since $ \Crit_G^{\Delta}(\po) = Q \cap \trop^{-1}(\interior \Delta) \subset W$, 
$\dim \Crit_G(\po) \leq r$. 
On the other hand, 
let $\epsilon>0$ be sufficiently small. 
Consider 
\begin{equation*}\label{Delta epsilon}
 \Delta_{\epsilon} = \bigcap_{i=1}^m\{u\in \R^n \mid \pair{u,v_i}-\lambda_i\geq \epsilon\} \subset \Delta.    
\end{equation*}
\begin{align*}
  Q\cap \trop^{-1}(\Delta)
    & \supset   Q\cap \trop^{-1}(\Delta_{\epsilon})\\
    & \cong  \Sp \frac{\Lambda \pair{\frac{y_1}{T^{\epsilon}}, \ldots,\frac{y_n}{T^{\epsilon}},\frac{T^{1-\epsilon}}{y_1\cdots y_n}}}{\left(\sum\limits_{j=1}^n a_{r+1,j} (y_j - z) , \ldots,\sum\limits_{j=1}^n a_{n,j} (y_j - z) \right)}\\
     & \cong  \Sp \frac{\Lambda \pair{z_1, \ldots, z_n,  z_{n+1}}}{\left( z_1\cdots z_{n+1}- T^{1-(n+1)\epsilon}, \sum\limits_{j=1}^n a_{r+1,j} (z_j - z_{n+1}) , \ldots,\sum\limits_{j=1}^n a_{n,j} (z_j - z_{n+1}) \right)}\\
   & =: U.     
\end{align*}
Change the variables by setting 
\[ 
\begin{pmatrix}
    X_1' \\
    \vdots \\
    X_n' \\
    X_{n+1}'
\end{pmatrix} 
= 
\begin{pmatrix}
    a_{1,1} & \cdots &  a_{1,n} & 0 \\
    \vdots & \ddots & \vdots \\
    a_{n,1} & \cdots & a_{n,n}& 0 \\
    0 & \cdots & 0 & 1
\end{pmatrix}
\begin{pmatrix}
    z_1-z_{n+1} \\
    \vdots \\
    z_n-z_{n+1}\\
    z_{n+1}'
\end{pmatrix} . 
\]
Then 
\begin{align*}
U 
& \cong \Sp \frac{\Lambda \langle X_1',\ldots,X_n', X_{n+1}'\rangle }{\left(  X_{n+1}'\prod_{i=1}^n(X_{n+1}'+\sum\limits_{j=1}^n b_{i,j}X_j')-T^{1-(n+1)\epsilon}\right)}  \\
& \cong \Sp \frac{\Lambda \langle X_1',\ldots,X_r', X_{n+1}' \rangle }{\left(  X_{n+1}'\prod_{i=1}^n(X_{n+1}'+\sum\limits_{j=1}^n b_{i,j}X_j')-T^{1-(n+1)\epsilon}\right)}
\end{align*}
Since $X_{n+1}'\prod_{i=1}^n(X_{n+1}'+\sum\limits_{j=1}^n b_{i,j}X_j')-T^{1-(n+1)\epsilon}$ is not a zero divisor in the integral domain 
$\Lambda \langle X_1',\ldots,X_r',X_{n+1}' \rangle$, which is Cohen-Macaulay of dimension $r+1$, we see that $\dim U = r$. 
Since $Q\cap  \trop^{-1}( \Delta_{\epsilon}) \subset \Crit_G(\po)  \subset  Q\cap  \trop^{-1}( \Delta)$ and $\dim \lrp{Q\cap  \trop^{-1}( \Delta_{\epsilon})} = \dim \lrp{Q\cap  \trop^{-1}( \Delta)} = r$, we conclude that $\dim \Crit_G^{\Delta}(\po)=r$.  
\end{proof}

\begin{prop}
\label{hypersurface rigid dimension}
Let $X$ be a compact symplectic toric manifold of complex dimension $n$ as in \ref{potential function of a compact symplectic toric manifold}.  If $G\cong T^{n-1}$ be an $(n-1)$-dimensional subtorus of $T^n$ as in Section \ref{Section setup}, 
then $\Crit_G(\po)$ is a rigid analytic space of dimension $n-1$. 
\end{prop}
\begin{proof}
Consider 
\begin{align*}
Q 
& = \set{ (y_1,\ldots, y_n)\in (\Lambda^*)^n }{ \sum_{j=1}^n a_{n,j} \left(y_j\frac{\partial \po}{\partial y_j} \right)=0 }. 
\end{align*}
Then \[ \dim \lrp{ Q\cap  \trop^{-1}( \Delta) } = \dim \frac{\Lambda \pair{\Delta} }{\left(\sum\limits_{j=1}^n a_{n,j} \left(y_j\frac{\partial \po}{\partial y_j} \right) \right)}. \] 
Since $ \Lambda \pair{\Delta}$ is a subring of the formal power series ring, 
$\Lambda \pair{\Delta}$ is also an integral domain. 
Since $\sum\limits_{j=1}^n a_{n,j} \left(y_j\frac{\partial \po}{\partial y_j}\right) \ne 0$, it is a regular sequence in $\Lambda \pair{\Delta}$, which is a Cohen-Macaulay ring of dimension $n$ by \cite{Rabinoff} Proposition 6.9. 
Therefore, $\dim \left( Q\cap  \trop^{-1}( \Delta) \right)= n-1$. 
Let 
\begin{equation*}
 \Delta_{\epsilon} = \bigcap_{i=1}^m\{u\in \R^n \mid \pair{u,v_i}-\lambda_i\geq \epsilon\}.   
\end{equation*}
Then
$\dim \left( Q\cap  \trop^{-1}( \Delta_{\epsilon})\right) = n-1$ as well. 
Since 
\[ Q\cap  \trop^{-1}( \Delta_{\epsilon}) \subset \Crit_G(\po)  \subset  Q\cap  \trop^{-1}( \Delta),\] 
we conclude that 
$\dim \Crit_G(\po)= n-1$. 
\end{proof}

\begin{cor}
Let $X$ be a compact symplectic toric manifold of complex dimension $n\leq 2$ as in \ref{potential function of a compact symplectic toric manifold}.  If $G\cong T^r$ be a subtorus of $T^2$ as in Section \ref{Section setup} for some $0\leq r \leq n$, 
then $\Crit_G(\po)$ is a rigid analytic space of dimension $r$. 
\end{cor}
\begin{proof}
The case when $r=0$ follows from \cite{I}. 
The case when $r=1, n=2$ follows from Proposition \ref{hypersurface rigid dimension}.  
The case when $r=2, n=2$ follows from \cite{Rabinoff}
Proposition 6.9. 
\end{proof}

\subsection{Examples} \label{examples}
  
Consider the action on a toric manifold $(X^4, \omega, T^2, \mu)$ by a subtorus $G=\iota(S^1)$, where
\[ S^1 \hookrightarrow T^2, \qquad \theta \mapsto (k_1\theta, k_2\theta),  \]
for some $(k_1,k_2) \in \Z^2 \setminus \{(0,0)\}$. Since the action is free,  $H_{G}^1(L(u), \R) \cong H^1(L(u)/G, \R)\hookrightarrow H^1(L(u),\R)$ is generated by  
\[ \alpha_2 = -k_2 e_1 + k_1 e_2.\] 
Complete it to a basis 
$\{ \alpha_1,\alpha_2\}$ of $H^1(L(u),\R)$. 
Let $b = c_1\alpha_1 + c_2\alpha_2$. 
By Theorem \ref{main theorem}, 
there is a bijection 
\[ \mlag_G(\proj^2,\omega) \to 
V\left( \frac{\partial \po}{\partial c_2}\right) \cap \trop^{-1}(\interior \Delta) = :\Crit_G^{\Delta}(\po)\]
\[ \lrp{u_1,u_2, b = \sum_{i=1}^2 x_ie_i} 
\mapsto (y_1,y_2) = (e^{x_1}T^{u_1}, e^{x_2}T^{u_2}).   \]
Then 
\[ \val(y_i) = \val(e^{x_i}T^{u_i}) = \val(e^{x_i}) + \val(T^{u_i}) = u_i. \]
In particular, given $(y_1, y_2) \in \Crit_G^{\Delta}(\po)$, 
the Lagrangian associated with it is 
$\mu^{-1}(\val(y_1), \val(y_2) )$.

\subsubsection{\texorpdfstring{$S^1$}{S1}-action on \texorpdfstring{$\proj^2$}{CP2}}
\begin{eg}[$S^1$-action on $\proj^2$]
Consider $(\proj^2,\omega, T^2,\mu)$ associated with the moment polytope 
\[ \Delta=\set{ (u_1, u_2) \in \R^2 }{ 
\begin{aligned}
& u_i\geq 0 \quad \forall 1\leq i\leq 2,\\
&  1-u_1-u_2 \geq 0
\end{aligned}}
. \]
Its potential function is then given by 
\[ \po  = y_1 + y_2+ \frac{T}{y_1y_2}. \] 

Denote by $f$ the Laurent polynomial
\[ \frac{\partial \po}{\partial c_2} = -k_2y_1 \frac{\partial \po}{\partial y_1} + k_1 y_2 \frac{\partial \po}{\partial y_2} = -k_2(y_1-\frac{T}{y_1y_2}) + k_1(y_2-\frac{T}{y_1y_2}) =  -k_2y_1+ k_1y_2 + (k_2-k_1)\frac{T}{y_1y_2}   \]
and denote $Y=V(f)\cap (\Lambda^*)^2$. 
By Kapranov's theorem (\cite{maclagan} Theorem 3.1.3, also see Theorem \ref{Kapranov's Theorem}), 
\[ \overline{\trop (Y)} = V\left(\trop \left(f \right)\right).  \]
\begin{enumerate}[i)]
    \item Suppose $k_1, k_2, k_2-k_1$ are all non-zero. 
\begin{align*}
V(\trop f) 
= 
& \{ u\in \R^2\mid u_1 = 1-u_1-u_2\leq u_2\} \\
\cup  
& \{ u\in \R^2\mid u_2 = 1-u_1-u_2\leq u_1\}\\ 
\cup 
&\{ u\in \R^2\mid u_1 = u_2\leq 1-u_1-u_2\}.  
\end{align*}

$ \trop(\Crit_G^{\Delta}(\po))=   \trop (Y)\cap  \interior \Delta $ is shown in Figure \ref{fig: cp2 generic}. 
\begin{figure}[h!]
\centering
\tikzset{every picture/.style={line width=0.75pt}} 
\resizebox{0.3\textwidth}{!}{
\begin{tikzpicture}[x=0.75pt,y=0.75pt,yscale=-1,xscale=1]

\draw   (201,51) -- (348.8,201) -- (201,201) -- cycle ;
\draw [color={rgb, 255:red, 208; green, 2; blue, 27 }  ,draw opacity=1 ][line width=1.5]    (202.47,54.02) -- (250.27,152.14) ;
\draw [shift={(201,51)}, rotate = 64.03] [color={rgb, 255:red, 208; green, 2; blue, 27 }  ,draw opacity=1 ][line width=1.5]      (0, 0) circle [x radius= 4.36, y radius= 4.36]   ;
\draw [color={rgb, 255:red, 208; green, 2; blue, 27 }  ,draw opacity=1 ][line width=1.5]    (250.27,152.14) -- (203.38,198.64) ;
\draw [shift={(201,201)}, rotate = 135.24] [color={rgb, 255:red, 208; green, 2; blue, 27 }  ,draw opacity=1 ][line width=1.5]      (0, 0) circle [x radius= 4.36, y radius= 4.36]   ;
\draw [color={rgb, 255:red, 208; green, 2; blue, 27 }  ,draw opacity=1 ][line width=1.5]    (250.27,152.14) -- (345.79,199.51) ;
\draw [shift={(348.8,201)}, rotate = 26.38] [color={rgb, 255:red, 208; green, 2; blue, 27 }  ,draw opacity=1 ][line width=1.5]      (0, 0) circle [x radius= 4.36, y radius= 4.36]   ;

\draw (159,195.3) node [anchor=north west][inner sep=0.75pt]    {$( 0,0)$};
\draw (159,43.2) node [anchor=north west][inner sep=0.75pt]    {$( 0,1)$};
\draw (354,195.3) node [anchor=north west][inner sep=0.75pt]    {$( 1,0)$};
\draw (251,127.2) node [anchor=north west][inner sep=0.75pt]  {$\left(\frac{1}{3} ,\frac{1}{3}\right)$};
\end{tikzpicture}
}

\caption{Case when $k_1,k_2,k_1-k_2\ne 0$}
\label{fig: cp2 generic}
\end{figure}
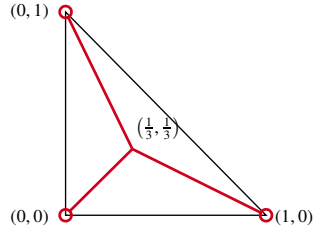

Moreover,  
\begin{align*}
W 
& :=  Y\cap  \trop ^{-1}(\Delta) \\
& = \Sp\frac{\Lambda\pair{y_1,y_2, \frac{T}{y_1y_2}}}{
\left(
-k_2y_1+k_1y_2+(k_2-k_1)\frac{T}{y_1y_2}
\right)
}\\
& \cong \Sp\frac{\Lambda\pair{y_1,y_2,z}}{
\left(
-k_2y_1+k_1y_2+(k_2-k_1)z, 
y_1y_2z-T
\right)
}
\end{align*}
 is an affinoid space, and 
\[ 
\Crit_G^{\Delta}(\po) = W \setminus \trop^{-1}\left(\{(0,0), (0,1), (1,0)\}\right). 
\]

We can compute the genus of $W$ as a rigid analytic curve as follows. 
We have a canonical reduction map 
\[ 
\rho: W
\to 
\widetilde{W} = \Spec \frac{\C[y_1,y_2]}{\lrp{(-k_2y_1+k_1y_2)y_1y_2}}.
\] 

By \cite{fresnel2012rigid} Proposition 5.6.2, since $W$ is a non-singular connected one-dimensional affinoid space, the genus of $W$ equal to the arithmetic genus of the compactification of $\widetilde{W}$. 

Let $C_1$, $C_2$, $C_3$ be the divisors corresponding to $-k_2y_1+k_1y_2=0$, $y_1=0$, $y_2=0$, respectively.  
Then by the adjunction formula 
(\cite{Hartshorne} Chapter V, Exercise 1.3),
the arithmetic genus of $\widetilde{W}$ is equal to 
\begin{align*}
  g_a(C_1+C_2+C_3) 
     = \sum_{i=1}^3 g_a(C_i) + C_1\cdot C_2 + C_2\cdot C_3 + C_1\cdot C_3 - 2 =1. 
\end{align*}

By the above argument, $g(W)$ is a rigid analytic curve of genus $1$.

\item If $k_1=0$ and $k_2\ne 0$, then  
$f = -k_2 \lrp{y_1-\frac{T}{y_1y_2}}$. 
\[ V(\trop f)  = \{ u\in \R^2\mid u_1 = 1-u_1-u_2\}. \]
$\trop\left( \Crit_G^{\Delta}(\po)\right) = \trop V(f)\cap \interior \Delta$ is shown in Figure \ref{fig: cp2 a=0}. 
\begin{figure}[h!]
\begin{center}
\tikzset{every picture/.style={line width=0.75pt}} 
\resizebox{0.3\textwidth}{!}{
\begin{tikzpicture}[x=0.75pt,y=0.75pt,yscale=-1,xscale=1]

\draw   (200,50) -- (349.8,200.6) -- (200,200.6) -- cycle ;
\draw [color={rgb, 255:red, 208; green, 2; blue, 27 }  ,draw opacity=1 ][line width=1.5]    (201.49,53) -- (273.26,197.19) ;
\draw [shift={(274.75,200.2)}, rotate = 63.54] [color={rgb, 255:red, 208; green, 2; blue, 27 }  ,draw opacity=1 ][line width=1.5]      (0, 0) circle [x radius= 4.36, y radius= 4.36]   ;
\draw [shift={(200,50)}, rotate = 63.54] [color={rgb, 255:red, 208; green, 2; blue, 27 }  ,draw opacity=1 ][line width=1.5]      (0, 0) circle [x radius= 4.36, y radius= 4.36]   ;

\draw (255,205.6) node [anchor=north west][inner sep=0.75pt]  [xscale=0.75,yscale=0.75]  {$\left(\frac{1}{2} ,0\right)$};
\draw (157,41.2) node [anchor=north west][inner sep=0.75pt]  [xscale=0.75,yscale=0.75]  {$( 0,1)$};
\end{tikzpicture}
}
\end{center}
\caption{Case when $k_1=0$}
\label{fig: cp2 a=0}
\end{figure}
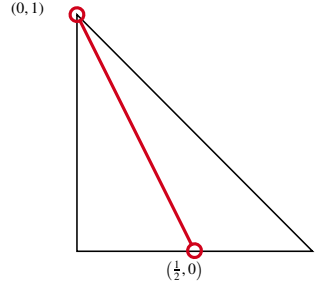

Indeed, 
\begin{align*}
\Crit_G^{\Delta}(\po)  
= &  Y\cap \trop^{-1}(\interior \Delta) \\
= & \set{ \left(y_1,y_2\right)\in (\Lambda^*)^2 }{ 
\begin{aligned}
& -k_2y_1+k_2\frac{T}{y_1y_2}=0\\
& \val(y_1) > 0, \val \left(\frac{T}{y_1^2}\right) > 0, \\
& \val \left(\frac{T}{y_1T/y_1^2}\right)>0,   
\end{aligned} 
}\\
= & \set{ \left(y_1,\frac{T}{y_1^2}\right)\in (B_{\Lambda}^1)^2 }{ 
\begin{aligned}
& \val(y_1) > 0, \val \left(\frac{T}{y_1^2}\right) > 0, \\
& 
y_1\ne 0, \frac{T}{y_1^2}\ne 0    
\end{aligned} 
}\\
= & \set{ (y_1,\frac{T}{y_1^2})\in \Lambda^2 }{ e^{-\frac{1}{2}} < |y_1| <1 }\\
\cong & \set{ y_1 \in B_{\Lambda}^1 }{ e^{-\frac{1}{2}} < |y_1| < 1 }
 \subset  \Sp \Lambda\pair{ y_1, Ty_1^{-2}} 
 \end{align*}
is an open annulus.

\item If $k_2=0$ and $k_1\ne 0$, then $f = k_1\lrp{y_2- \frac{T}{y_1y_2}}$. 
\[ V\left( \trop \left(f  \right) \right) = \{ u\in \R^2\mid u_2 = 1-u_1-u_2\}. \]
$\trop\left( \Crit_G^{\Delta}(\po)\right) = \trop V(f)\cap \interior \Delta$ is shown in Figure \ref{fig: cp2 b=0}. 

\begin{figure}[!h]
\centering 
\tikzset{every picture/.style={line width=0.75pt}} 
\resizebox{0.3\textwidth}{!}{
\begin{tikzpicture}[x=0.75pt,y=0.75pt,yscale=-1,xscale=1]

\draw   (200.8,51.2) -- (349.8,200.2) -- (200.8,200.2) -- cycle ;
\draw [color={rgb, 255:red, 208; green, 2; blue, 27 }  ,draw opacity=1 ][line width=1.5]    (204.41,129.42) -- (346.78,198.73) ;
\draw [shift={(349.8,200.2)}, rotate = 25.96] [color={rgb, 255:red, 208; green, 2; blue, 27 }  ,draw opacity=1 ][line width=1.5]      (0, 0) circle [x radius= 4.36, y radius= 4.36]   ;
\draw [shift={(201.4,127.95)}, rotate = 25.96] [color={rgb, 255:red, 208; green, 2; blue, 27 }  ,draw opacity=1 ][line width=1.5]      (0, 0) circle [x radius= 4.36, y radius= 4.36]   ;

\draw (150,105.6) node [anchor=north west][inner sep=0.75pt]    {$\left( 0,\frac{1}{2}\right)$};
\draw (358,193.6) node [anchor=north west][inner sep=0.75pt]    {$( 1,0)$};
\end{tikzpicture}
}
    \caption{Case when $k_2=0$}
    \label{fig: cp2 b=0}
\end{figure}
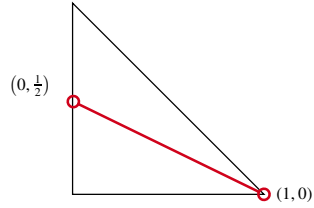

Similar to the case $k_1=0$, 
\begin{align*}
\Crit_G^{\Delta}(\po)
= & Y\cap \trop^{-1}( \interior \Delta) \\
= & \set{ \left(y_1, y_2\right)\in \Lambda^2 }{ 
\begin{aligned}
& k_1y_2-k_1\frac{T}{y_1y_2} =0 \\
& \val(y_2) >  0, \val \left(\frac{T}{y_2^2}\right) > 0, \\    
& \val \left(\frac{T}{y_2T/y_2^2 }\right) > 0, 
y_2\ne  0, \frac{T}{y_2^2}\ne 0 
\end{aligned} } \\
= & \set{ \left(\frac{T}{y_2^2}, y_2\right)\in \Lambda^2 }{ 
\begin{aligned}
& \val(y_2) >  0, \val \left(\frac{T}{y_2^2}\right) > 0, \\    
& 
y_2\ne  0, \frac{T}{y_2^2}\ne 0 
\end{aligned} } \\
 \cong 
 & \set{ y_2 \in B_{\Lambda}^1 }{ e^{-\frac{1}{2}} < |y_2| < 1 }
 \subset 
\Sp \Lambda\pair{ y_2, Ty_2^{-2}
 }
\end{align*}
is an open annulus. 

\item If $k_2-k_1 =0$ and $k_1,k_2\ne 0$, then $f = -k_1y_1+k_1y_2$. 
\[ V\left( \trop \left(f \right) \right) = \{ u\in \R^2\mid  u_1 = u_2 \}. \]
$\trop\left( \Crit_G^{\Delta}(\po)\right) = \trop V(f)\cap \interior \Delta$ is shown in Figure \ref{fig:cp2 b-a=0}.

\begin{figure}[!h]
\centering
\tikzset{every picture/.style={line width=0.75pt}} 
\resizebox{0.25\textwidth}{!}{
\begin{tikzpicture}[x=0.75pt,y=0.75pt,yscale=-1,xscale=1]

\draw   (201,50) -- (351.8,200.6) -- (201,200.6) -- cycle ;
\draw [color={rgb, 255:red, 208; green, 2; blue, 27 }  ,draw opacity=1 ][line width=1.5]    (203.37,198.23) -- (274.03,127.67) ;
\draw [shift={(276.4,125.3)}, rotate = 315.04] [color={rgb, 255:red, 208; green, 2; blue, 27 }  ,draw opacity=1 ][line width=1.5]      (0, 0) circle [x radius= 4.36, y radius= 4.36]   ;
\draw [shift={(201,200.6)}, rotate = 315.04] [color={rgb, 255:red, 208; green, 2; blue, 27 }  ,draw opacity=1 ][line width=1.5]      (0, 0) circle [x radius= 4.36, y radius= 4.36]   ;

\draw (179,206.6) node [anchor=north west][inner sep=0.75pt]  [xscale=0.7,yscale=0.7]  {$( 0,0)$};
\draw (283,74.6) node [anchor=north west][inner sep=0.75pt]  [xscale=0.7,yscale=0.7]  {$\left(\frac{1}{2} ,\frac{1}{2}\right)$};
\end{tikzpicture}
}
    \caption{Case when $k_2-k_1=0$}
    \label{fig:cp2 b-a=0}
\end{figure}
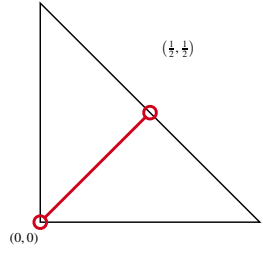

Moreover, 
\begin{align*}
\Crit_G^{\Delta}(\po)
= & Y\cap \trop^{-1}( \interior \Delta) \\
= & \set{ \left( y_1, y_1\right)\in \Lambda^2 }{ \val(y_1) >  0, \val \left(\frac{T}{ y_1^2}\right) > 0, 
y_1\ne  0
} \\
 \cong & \set{ y_1 \in B_{\Lambda}^1 }{ e^{-\frac{1}{2}} < |y_1| < 1 }
 \subset  \Sp \Lambda\pair{ y_1, Ty_1^{-2}}
\end{align*}
is an open annulus. 
\end{enumerate}

\end{eg}

\subsubsection{\texorpdfstring{$S^1$}{S1}-action on a one-point blowup of \texorpdfstring{$\proj^2$}{CP2}}
\begin{eg}[$S^1$-action on a one-point blowup of $\proj^2$]

Consider the one-point blowup $(\proj^2(1),\omega, T^2,\mu)$ of $\proj^2$ whose moment polytope is given by
\[ \Delta=\set{ (u_1, u_2) \in \R^2 }{ 
\begin{aligned}
& u_i\geq 0 \quad \forall 1\leq i\leq 2,\\
&  1-u_1-u_2 \geq 0\\
&  1-\alpha - u_2\geq 0
\end{aligned}}
. \]
Its potential function is  
\[ \po  = y_1 + y_2 + \frac{T}{y_1y_2} + \frac{T^{1-\alpha }}{y_2}. \]
\begin{align*}
\Longrightarrow 
f := \frac{\partial \po}{\partial c_2} 
& = -k_2 y_1 \frac{\partial \po}{\partial y_1} + k_1 y_2 \frac{\partial \po}{\partial y_2} \\
& = -k_2 \lrp{ y_1-\frac{T}{y_1y_2} } + k_1\left(y_2-\frac{T}{y_1y_2} - \frac{T^{1-\alpha}}{y_2}\right)  \\
& =  -k_2y_1+ k_1y_2 + (k_2-k_1)\frac{T}{y_1y_2} -k_1 \frac{T^{1-\alpha}}{y_2}.       
\end{align*}

\begin{enumerate}
\item \textbf{Suppose $k_1,k_2,k_2-k_1$ are all non-zero.}

We now consider the tropicalization of $\Crit_G^{\Delta}(\po)$. 
\[ \trop  (f) : \R^2 \to \R, \quad  u \mapsto \min \{ u_1, u_2, 1-u_1-u_2, 1-\alpha - u_2\}. \]

\begin{enumerate}[a)]
    \item $u_1= u_2\leq \min \{1-u_1-u_2, 1-\alpha - u_2\}$ 
    \[ \Longrightarrow 
    \begin{cases}
     & u_1= u_2\\ 
    & u_1 \leq 1-2u_1\\
    & u_1 \leq  1-\alpha - u_1
    \end{cases}
  \Rightarrow 
  \begin{cases}
    u_2 & = u_1, \\ 
    u_1  &  \leq \min \{\frac{1}{3},  \frac{1-\alpha}{2}\} \\
     & =  \begin{cases}
     \frac{1}{3} \quad & \text{ if } \alpha \leq \frac{1}{3}\\
      \frac{1-\alpha}{2} \quad & \text{ if } \alpha \geq \frac{1}{3}
    \end{cases}      
  \end{cases}
    \] 
    \item $u_1=1-u_1-u_2\leq \min \{ u_2, 1-\alpha - u_2\}$ 
     \[ \Longrightarrow 
    \begin{cases}
     & u_2= 1-2u_1\\ 
    & u_1 \leq 1-2u_1\\
    & u_1 \leq  1-\alpha - (1-2u_1)
    \end{cases}
    \Rightarrow 
   \begin{cases}
     & u_2= 1-2u_1\\ 
    &  \alpha\leq u_1 \leq \frac{1}{3},
    \end{cases} 
    \] 
    which can happen only if $\alpha \leq \frac{1}{3}$. 
    \item $u_1= 1-\alpha - u_2\leq \min \{1-u_1-u_2, u_2\}$ 
    \[ \Longrightarrow 
    \begin{cases}
     & u_2= 1-\alpha-u_1\\ 
    & u_1 \leq 1-u_1-(1-\alpha-u_1)\\
    & u_1 \leq  1-\alpha - u_1
    \end{cases}
  \Rightarrow 
    \begin{cases}
      u_2 & = 1-\alpha-u_1 \\ 
    u_1  &  \leq \min \{\alpha ,  \frac{1-\alpha}{2}\} \\
     & =  \begin{cases}
     \alpha \quad & \text{ if } \alpha \leq \frac{1}{3}\\
      \frac{1-\alpha}{2} \quad & \text{ if } \alpha \geq \frac{1}{3}
    \end{cases}   
      \end{cases} 
    \] 
    \item $u_2 = 1-u_1-u_2 \leq \min \{u_1, 1-\alpha - u_2\}$
     \[ \Longrightarrow 
    \begin{cases}
     & u_2= \frac{1-u_1}{2}\\ 
    &  \frac{1-u_1}{2} \leq u_1\\
    & \frac{1-u_1}{2} \leq 1-\alpha - \frac{1-u_1}{2}
    \end{cases}
     \Rightarrow 
     \begin{cases}
     u_2 & = \frac{1-u_1}{2}\\
    u_1  &  \geq \max \{\alpha , \frac{1}{3}\} \\
     & =  \begin{cases}
    \frac{1}{3} \quad & \text{ if } \alpha \leq \frac{1}{3}\\
     \alpha \quad & \text{ if } \alpha \geq \frac{1}{3}
    \end{cases}   
     \end{cases}
    \]  
     \item $u_2 = 1-\alpha - u_2 \leq \min \{u_1, 1-u_1-u_2\}$ 
      \[ \Longrightarrow 
    \begin{cases}
     & u_2= \frac{1-\alpha}{2}\\ 
    & \frac{1-\alpha}{2} \leq u_1\\
    & \frac{1-\alpha}{2} \leq 1-u_1 -\frac{1-\alpha}{2}
    \end{cases}
     \Rightarrow 
    \begin{cases}
     & u_2= \frac{1-\alpha}{2}\\ 
    & \frac{1-\alpha}{2} \leq u_1 \leq \alpha , 
    \end{cases}
    \]  
    which can happen only if $\alpha\geq \frac{1}{3}$. 
    \item $1-u_1-u_2= 1-\alpha - u_2 \leq \min \{u_1, u_2 \}$
     \[ \Longrightarrow 
    \begin{cases}
     & u_1= \alpha\\ 
    & 1-\alpha - u_2\leq \alpha\\
    & 1-\alpha - u_2 \leq u_2
    \end{cases}
    \Rightarrow 
      \begin{cases}
     u_1 & =  \alpha\\
    u_2 &  \geq \max \{ 1-2\alpha ,  \frac{1-\alpha}{2}\} \\
     & =  \begin{cases}
     1-2\alpha \quad & \text{ if } \alpha \leq \frac{1}{3}\\
      \frac{1-\alpha}{2}\quad & \text{ if } \alpha \geq \frac{1}{3}
    \end{cases}   
     \end{cases}
    \]  
\end{enumerate}

\begin{enumerate}[i)]
\item 
The case when $0< \alpha<\frac{1}{3} $ is shown in Figure \ref{fig:cp2(1) a,b,b-a non-zero, 0< alpha <1/3}. 
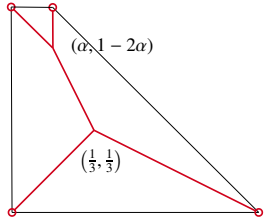
\begin{figure}[!h]
\centering
\tikzset{every picture/.style={line width=0.75pt}} 
\resizebox{0.25\textwidth}{!}{
\begin{tikzpicture}[x=0.75pt,y=0.75pt,yscale=-1,xscale=1]

\draw    (199.8,50.6) -- (200.8,300.6) ;
\draw    (500.8,300.6) -- (200.8,300.6) ;
\draw    (250.3,51.1) -- (500.8,300.6) ;
\draw [color={rgb, 255:red, 208; green, 2; blue, 27 }  ,draw opacity=1 ][line width=1.5]    (250.31,54.45) -- (250.5,100.6) ;
\draw [shift={(250.3,51.1)}, rotate = 89.77] [color={rgb, 255:red, 208; green, 2; blue, 27 }  ,draw opacity=1 ][line width=1.5]      (0, 0) circle [x radius= 4.36, y radius= 4.36]   ;
\draw [color={rgb, 255:red, 208; green, 2; blue, 27 }  ,draw opacity=1 ][line width=1.5]    (202.19,52.96) -- (250.5,100.6) ;
\draw [shift={(199.8,50.6)}, rotate = 44.6] [color={rgb, 255:red, 208; green, 2; blue, 27 }  ,draw opacity=1 ][line width=1.5]      (0, 0) circle [x radius= 4.36, y radius= 4.36]   ;
\draw [color={rgb, 255:red, 208; green, 2; blue, 27 }  ,draw opacity=1 ][line width=1.5]    (250.5,100.6) -- (300.5,200.6) ;
\draw [color={rgb, 255:red, 208; green, 2; blue, 27 }  ,draw opacity=1 ][line width=1.5]    (203.17,298.22) -- (300.5,200.6) ;
\draw [shift={(200.8,300.6)}, rotate = 314.91] [color={rgb, 255:red, 208; green, 2; blue, 27 }  ,draw opacity=1 ][line width=1.5]      (0, 0) circle [x radius= 4.36, y radius= 4.36]   ;
\draw [color={rgb, 255:red, 208; green, 2; blue, 27 }  ,draw opacity=1 ][line width=1.5]    (497.8,299.1) -- (300.5,200.6) ;
\draw [shift={(500.8,300.6)}, rotate = 206.53] [color={rgb, 255:red, 208; green, 2; blue, 27 }  ,draw opacity=1 ][line width=1.5]      (0, 0) circle [x radius= 4.36, y radius= 4.36]   ;
\draw    (199.8,50.6) -- (250.3,51.1) ;

\draw (271,87.4) node [anchor=north west][inner sep=0.75pt]  [font=\LARGE]  {$( \alpha ,1-2\alpha )$};
\draw (281,222.4) node [anchor=north west][inner sep=0.75pt]  [font=\LARGE]  {$\left(\frac{1}{3} ,\frac{1}{3}\right)$};

\end{tikzpicture}

}
\caption{Case when $k_1,k_2,k_2-k_1\ne 0$ and $0<\alpha <\frac{1}{3}$}
    \label{fig:cp2(1) a,b,b-a non-zero, 0< alpha <1/3}
\end{figure}

\item
The case when $\alpha =\frac{1}{3}$ is shown in Figure \ref{fig:cp2(1) a,b,b-a non-zero, alpha =1/3}. 

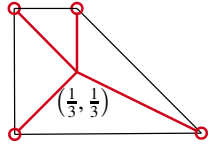
\begin{figure}[!h]
\centering
\tikzset{every picture/.style={line width=0.75pt}} 
\resizebox{0.2\textwidth}{!}{

\begin{tikzpicture}[x=0.75pt,y=0.75pt,yscale=-1,xscale=1]

\draw    (200,100) -- (200.41,201) ;
\draw    (350,200) -- (200.41,201) ;
\draw    (250.3,100) -- (350,200) ;
\draw [color={rgb, 255:red, 208; green, 2; blue, 27 }  ,draw opacity=1 ][line width=1.5]    (250.3,103.36) -- (250.3,151) ;
\draw [shift={(250.3,100)}, rotate = 90] [color={rgb, 255:red, 208; green, 2; blue, 27 }  ,draw opacity=1 ][line width=1.5]      (0, 0) circle [x radius= 4.36, y radius= 4.36]   ;
\draw [color={rgb, 255:red, 208; green, 2; blue, 27 }  ,draw opacity=1 ][line width=1.5]    (202.36,102.39) -- (250.3,151) ;
\draw [shift={(200,100)}, rotate = 45.39] [color={rgb, 255:red, 208; green, 2; blue, 27 }  ,draw opacity=1 ][line width=1.5]      (0, 0) circle [x radius= 4.36, y radius= 4.36]   ;
\draw [color={rgb, 255:red, 208; green, 2; blue, 27 }  ,draw opacity=1 ][line width=1.5]    (202.78,198.63) -- (250.3,151) ;
\draw [shift={(200.41,201)}, rotate = 314.94] [color={rgb, 255:red, 208; green, 2; blue, 27 }  ,draw opacity=1 ][line width=1.5]      (0, 0) circle [x radius= 4.36, y radius= 4.36]   ;
\draw [color={rgb, 255:red, 208; green, 2; blue, 27 }  ,draw opacity=1 ][line width=1.5]    (346.99,198.52) -- (250.3,151) ;
\draw [shift={(350,200)}, rotate = 206.17] [color={rgb, 255:red, 208; green, 2; blue, 27 }  ,draw opacity=1 ][line width=1.5]      (0, 0) circle [x radius= 4.36, y radius= 4.36]   ;
\draw    (200,100) -- (250.3,100) ;

\draw (231,162.4) node [anchor=north west][inner sep=0.75pt]  [font=\Large]  {$\left(\frac{1}{3} ,\frac{1}{3}\right)$};

\end{tikzpicture}

}
\caption{Case when $k_1,k_2,k_2-k_1\ne 0$ and $\alpha =\frac{1}{3}$}
    \label{fig:cp2(1) a,b,b-a non-zero, alpha =1/3}
\end{figure}

\item 

The case when $1/3 <\alpha <1$ is shown in Figure \ref{fig: cp2(1), a,b,b-a non-zero, 1/3 <alpha <1}. 

\begin{figure}[!h]
  \centering
\tikzset{every picture/.style={line width=0.75pt}} 
\resizebox{0.3\textwidth}{!}{
\begin{tikzpicture}[x=0.75pt,y=0.75pt,yscale=-1,xscale=1]

\draw    (200,49.1) -- (200.8,199.8) ;
\draw    (501.8,198.8) -- (200.8,199.8) ;
\draw    (350.8,49.8) -- (501.8,198.8) ;
\draw    (200,49.1) -- (350.8,49.8) ;
\draw [color={rgb, 255:red, 208; green, 2; blue, 27 }  ,draw opacity=1 ][line width=1.5]    (350.8,53.15) -- (350.8,123.8) ;
\draw [shift={(350.8,49.8)}, rotate = 90] [color={rgb, 255:red, 208; green, 2; blue, 27 }  ,draw opacity=1 ][line width=1.5]      (0, 0) circle [x radius= 4.36, y radius= 4.36]   ;
\draw [color={rgb, 255:red, 208; green, 2; blue, 27 }  ,draw opacity=1 ][line width=1.5]    (274.8,122.8) -- (203.12,197.38) ;
\draw [shift={(200.8,199.8)}, rotate = 133.86] [color={rgb, 255:red, 208; green, 2; blue, 27 }  ,draw opacity=1 ][line width=1.5]      (0, 0) circle [x radius= 4.36, y radius= 4.36]   ;
\draw [color={rgb, 255:red, 208; green, 2; blue, 27 }  ,draw opacity=1 ][line width=1.5]    (202.39,51.45) -- (274.8,122.8) ;
\draw [shift={(200,49.1)}, rotate = 44.58] [color={rgb, 255:red, 208; green, 2; blue, 27 }  ,draw opacity=1 ][line width=1.5]      (0, 0) circle [x radius= 4.36, y radius= 4.36]   ;
\draw [color={rgb, 255:red, 208; green, 2; blue, 27 }  ,draw opacity=1 ][line width=1.5]    (274.8,122.8) -- (350.8,123.8) ;
\draw [color={rgb, 255:red, 208; green, 2; blue, 27 }  ,draw opacity=1 ][line width=1.5]    (350.8,123.8) -- (498.8,197.31) ;
\draw [shift={(501.8,198.8)}, rotate = 26.41] [color={rgb, 255:red, 208; green, 2; blue, 27 }  ,draw opacity=1 ][line width=1.5]      (0, 0) circle [x radius= 4.36, y radius= 4.36]   ;

\draw (374,96.4) node [anchor=north west][inner sep=0.75pt]  [font=\LARGE]  {$\left( \alpha ,\frac{1-\alpha }{2}\right)$};
\draw (154,100.4) node [anchor=north west][inner sep=0.75pt]  [font=\LARGE]  {$\left(\frac{1-\alpha }{2} ,\frac{1-\alpha }{2}\right)$};

\end{tikzpicture}

}
    \caption{Case when $k_1,k_2,k_2-k_1\ne 0$, $\frac{1}{3} <\alpha <1$ }
    \label{fig: cp2(1), a,b,b-a non-zero, 1/3 <alpha <1}
\end{figure}
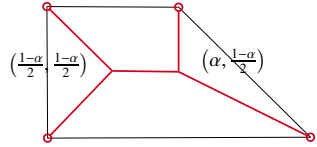
\end{enumerate}

We have 
\begin{align*}
 W : = & V(f)\cap \trop^{-1}(\Delta) \\
 = & \Sp \frac{\Lambda\pair{y_1,y_2,\frac{T}{y_1y_2}, \frac{T^{1-\alpha}}{y_2}}}{
 \lrp{-k_2y_1+ay_2+(k_2-k_1)\frac{T}{y_1y_2}-k_1  \frac{T^{1-\alpha}}{y_2}}
 } 
 \\
\cong & \Sp \frac{\Lambda\pair{y_1,y_2,z, x}}{
 \lrp{
 \begin{aligned}
 & -k_2y_1+k_1y_2+(k_2-k_1)z-k_1 x,\\
 & y_1y_2z-T, y_2x - T^{1-\alpha}
 \end{aligned}
} },
\end{align*}
and
\begin{align*} 
\Crit_G^{\Delta}(\po) 
= & W \setminus \trop^{-1} \lrp{\left\{ (0,0),(1,0),(0,1-\alpha), (\alpha,1-\alpha)\right\}}.  
\end{align*}

\item \textbf{Suppose $k_1=0$ and $k_2,k_2-k_1\ne 0$.} Then $f = -k_2\left(y_1- \frac{T}{y_1y_2}\right)$ and 
\begin{align*}
\Crit_G^{\Delta}(\po)
:= & V(f) \cap \trop^{-1}(\interior \Delta)\\
= & \set{ \left(y_1,y_2\right)\in (\Lambda^*)^2 }{ 
\begin{aligned}
& -k_2y_1+k_2\frac{T}{y_1y_2}=0\\
& \val(y_1) > 0, \val(y_2)>0, \\
& \val\lrp{\frac{T^{1-\alpha}}{y_2} }>0, \val \left(\frac{T}{y_1y_2}\right) > 0,   
\end{aligned} 
}\\
= & \set{ \left(y_1,\frac{T}{y_1^2}\right)\in (\Lambda^*)^2  }{ 
\begin{aligned}
& \val(y_1) > 0, \val \left(\frac{T}{y_1^2}\right) > 0, \\
& \val \left(\frac{T}{y_1T/y_1^2}\right)>0,   \val\lrp{\frac{T^{1-\alpha}}{T/y_1^2} }>0     
\end{aligned} 
}\\
= & \set{ \left(y_1,\frac{T}{y_1^2}\right)\in (B_{\Lambda}^1)^2 }{ 
\begin{aligned}
& |y_1|<1, |y_1|>e^{-\frac{1}{2}}, \\
& |y_1|<e^{-\frac{\alpha}{2}}
\end{aligned} 
}\\
= & \set{ \left(y_1,\frac{T}{y_1^2}\right)\in \Lambda^2 }{ 
e^{-\frac{1}{2}}<|y_1|<e^{-\frac{\alpha}{2}} 
}
\end{align*}
is an ``open" annulus. 

Moreover, we have
\[ \trop f : \R^2 \to \R, \quad  u \mapsto \min \{ u_1,   1-u_1-u_2 \}. \]
$\trop(\Crit_G^{\Delta}(\po)) = \trop V(f)\cap  \interior \Delta$ is shown in Figure \ref{fig:cp2(1), a=0}. 
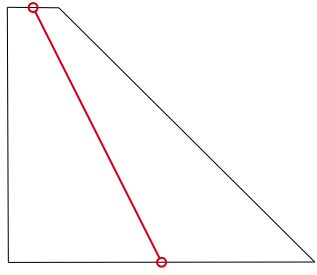
\begin{figure}[!h]
    \centering
\tikzset{every picture/.style={line width=0.75pt}} 
\resizebox{0.3\textwidth}{!}{
\begin{tikzpicture}[x=0.75pt,y=0.75pt,yscale=-1,xscale=1]

\draw    (199.8,50.6) -- (250.3,51.1) ;
\draw    (200.8,300.6) -- (500.8,300.2) ;
\draw    (199.8,50.6) -- (200.8,300.6) ;
\draw    (250.3,51.1) -- (500.8,300.2) ;
\draw [color={rgb, 255:red, 208; green, 2; blue, 27 }  ,draw opacity=1 ][line width=1.5]    (226.56,53.85) -- (349.29,297.4) ;
\draw [shift={(350.8,300.4)}, rotate = 63.26] [color={rgb, 255:red, 208; green, 2; blue, 27 }  ,draw opacity=1 ][line width=1.5]      (0, 0) circle [x radius= 4.36, y radius= 4.36]   ;
\draw [shift={(225.05,50.85)}, rotate = 63.26] [color={rgb, 255:red, 208; green, 2; blue, 27 }  ,draw opacity=1 ][line width=1.5]      (0, 0) circle [x radius= 4.36, y radius= 4.36]   ;

\end{tikzpicture}
}
    \caption{Case when $k_1=0$}
    \label{fig:cp2(1), a=0}
\end{figure}

\item \textbf{Suppose $k_2=0$ and $k_1,k_2-k_1\ne 0$.} 
Then \[ f = k_1\left(y_2-\frac{T}{y_1y_2} - \frac{T^{1-\alpha}}{y_2}\right). \]
 
We have
\[ \trop  (f) = \trop \left(y_2-\frac{T}{y_1y_2} - \frac{T^{1-\alpha}}{y_2}\right): \R^2 \to \R, \quad  u \mapsto \min \{ u_2,   1-u_1-u_2, 1-\alpha -u_2 \}. \]
Thus,  
\begin{align*}
    V(\trop f) 
    =  & \{ u_2=1-u_1-u_2\leq 1-\alpha - u_2\}  \cup   \{ 1-u_1-u_2= 1-\alpha - u_2\leq u_2 \} \\
    & \cup  \{ u_2=1-\alpha - u_2\leq 1-u_1-u_2\} \\
     =  &\{ u_2 = \frac{1-u_1}{2}, u_1 \geq \alpha\} \cup \{ u_1 = \alpha, u_2 \geq \frac{1-\alpha}{2}\} \cup  \{ u_2 = \frac{1-\alpha}{2}, u_1 \leq \alpha \} .
\end{align*}

The set $\trop(\Crit_G^{\Delta}(\po))$ is shown in Figure \ref{fig: cp2(1), b=0}. 

\begin{figure}[!h]
\centering
\tikzset{every picture/.style={line width=0.75pt}} 
\resizebox{0.3\textwidth}{!}{
\begin{tikzpicture}[x=0.75pt,y=0.75pt,yscale=-1,xscale=1]

\draw    (199.8,50.6) -- (250.3,51.1) ;
\draw    (200.8,300.6) -- (500.8,300.2) ;
\draw    (199.8,50.6) -- (200.8,300.6) ;
\draw    (250.3,51.1) -- (500.8,300.2) ;
\draw [color={rgb, 255:red, 208; green, 2; blue, 27 }  ,draw opacity=1 ][line width=1.5]    (250.8,175.6) -- (497.8,298.7) ;
\draw [shift={(500.8,300.2)}, rotate = 26.49] [color={rgb, 255:red, 208; green, 2; blue, 27 }  ,draw opacity=1 ][line width=1.5]      (0, 0) circle [x radius= 4.36, y radius= 4.36]   ;
\draw [color={rgb, 255:red, 208; green, 2; blue, 27 }  ,draw opacity=1 ][line width=1.5]    (250.31,54.45) -- (250.8,175.6) ;
\draw [shift={(250.3,51.1)}, rotate = 89.77] [color={rgb, 255:red, 208; green, 2; blue, 27 }  ,draw opacity=1 ][line width=1.5]      (0, 0) circle [x radius= 4.36, y radius= 4.36]   ;
\draw [color={rgb, 255:red, 208; green, 2; blue, 27 }  ,draw opacity=1 ][line width=1.5]    (203.65,175.6) -- (250.8,175.6) ;
\draw [shift={(200.3,175.6)}, rotate = 0] [color={rgb, 255:red, 208; green, 2; blue, 27 }  ,draw opacity=1 ][line width=1.5]      (0, 0) circle [x radius= 4.36, y radius= 4.36]   ;

\draw (267,108.4) node [anchor=north west][inner sep=0.75pt]  [font=\LARGE]  {$\left( \alpha ,\frac{1-\alpha }{2}\right)$};

\end{tikzpicture}

}
    \caption{Case when $k_2=0$}
    \label{fig: cp2(1), b=0}
\end{figure}
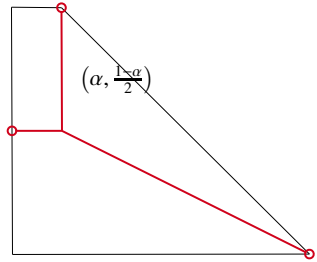
Moreover, 
\begin{align*}
 W:= &   V(f) \cap \trop^{-1}(\Delta) \\ 
 = & \Sp \frac{\Lambda \pair{ y_1, y_2, \frac{T}{y_1y_2}, \frac{T^{1-\alpha}}{y_2}} }{
 \lrp{
  a\left(y_2-\frac{T}{y_1y_2} - \frac{T^{1-\alpha}}{y_2}\right)
 }}\\
\cong & \Sp \frac{\Lambda \pair{ y_1, y_2, z, x} }{
 \lrp{
 y_2-z-x, y_1y_2z-T, y_2x-T^{1-\alpha}
 }
 },
\end{align*}
and $\Crit_G^{\Delta}(\po) = W\setminus \trop^{-1} \lrp{\left\{\lrp{0,\frac{1-\alpha}{2}}, \lrp{\alpha, 1-\alpha}, \lrp{1,0} \right\} }$. 
\item \textbf{Suppose $k_2-k_1=0$ and $k_1,k_2\ne 0$.} 
 
We have
\[ \trop  (f) = \trop (-y_1 + y_2 - \frac{T^{1-\alpha}}{y_2}): \R^2 \to \R, \quad  u \mapsto \min \{ u_1,  u_2, 1-\alpha -u_2 \}. \]
Thus,  
\begin{align*}
    V(\trop f) 
    =  & \{ u_1= u_2 \leq 1-\alpha - u_2\}  
    \cup \{ u_1=1-\alpha - u_2 \leq  u_2 \}\\ 
    &  \cup \{ u_2=1-\alpha - u_2\leq  u_1 \} \\ 
   =  & \{ u_1 =u_2, u_1\leq \frac{1-\alpha}{2}\} \cup  \{  u_2 = 1-\alpha - u_1, u_1\leq \frac{1-\alpha}{2}\}\\
   & \cup  \{ u_2 = \frac{1-\alpha}{2}, u_1\geq  \frac{1-\alpha}{2}\}
\end{align*}
The set $\trop (\Crit_G^{\Delta}(\po))=\trop V(f)\cap \interior \Delta$ in Figure \ref{fig:cp2(1), b-a=0}. 
\begin{figure}[!h]
    \centering
\tikzset{every picture/.style={line width=0.75pt}} 
\resizebox{0.25\textwidth}{!}{
\begin{tikzpicture}[x=0.75pt,y=0.75pt,yscale=-1,xscale=1]

\draw    (199.8,50.6) -- (250.3,51.1) ;
\draw    (200.8,300.6) -- (500.8,300.2) ;
\draw    (199.8,50.6) -- (200.8,300.6) ;
\draw    (250.3,51.1) -- (500.8,300.2) ;
\draw [color={rgb, 255:red, 208; green, 2; blue, 27 }  ,draw opacity=1 ][line width=1.5]    (203.17,298.22) -- (315.2,185.8) ;
\draw [shift={(200.8,300.6)}, rotate = 314.9] [color={rgb, 255:red, 208; green, 2; blue, 27 }  ,draw opacity=1 ][line width=1.5]      (0, 0) circle [x radius= 4.36, y radius= 4.36]   ;
\draw [color={rgb, 255:red, 208; green, 2; blue, 27 }  ,draw opacity=1 ][line width=1.5]    (201.98,53.15) -- (315.2,185.8) ;
\draw [shift={(199.8,50.6)}, rotate = 49.52] [color={rgb, 255:red, 208; green, 2; blue, 27 }  ,draw opacity=1 ][line width=1.5]      (0, 0) circle [x radius= 4.36, y radius= 4.36]   ;
\draw [color={rgb, 255:red, 208; green, 2; blue, 27 }  ,draw opacity=1 ][line width=1.5]    (315.2,185.8) -- (382.85,186.75) ;
\draw [shift={(386.2,186.8)}, rotate = 0.81] [color={rgb, 255:red, 208; green, 2; blue, 27 }  ,draw opacity=1 ][line width=1.5]      (0, 0) circle [x radius= 4.36, y radius= 4.36]   ;

\draw (198,155) node [anchor=north west][inner sep=0.75pt]  [font=\LARGE]  {$\left(\frac{1-\alpha }{2} ,\frac{1-\alpha }{2}\right)$};

\end{tikzpicture}

}
    \caption{Case $k_2-k_1=0$}
    \label{fig:cp2(1), b-a=0}
\end{figure}
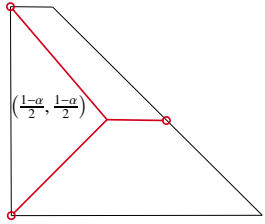
 \end{enumerate}
\end{eg}

Moreover, 
\begin{align*}
 W:= &   V(f) \cap \trop^{-1}(\Delta) \\ 
 = & \Sp \frac{\Lambda \pair{ y_1, y_2, \frac{T}{y_1y_2}, \frac{T^{1-\alpha}}{y_2}} }{(-y_1 + y_2 - \frac{T^{1-\alpha}}{y_2}) }
 \\
\cong & \Sp \frac{\Lambda \pair{ y_1, y_2, z, x} }{
 \lrp{
 -y_1 + y_2 - x, y_1y_2z-T, y_2x-T^{1-\alpha}
 }
 },
\end{align*}
and $\Crit_G^{\Delta}(\po) = W\setminus \trop^{-1} \lrp{ \left\{\lrp{0,0}, \lrp{0, 1-\alpha}, \lrp{\frac{1+\alpha}{2}, \frac{1-\alpha}{2}} \right\} }$.

\subsubsection{\texorpdfstring{$S^1$}{S1}-action on a two-point blowup of \texorpdfstring{$\proj^2$}{CP2}}
\begin{eg}[$S^1$-action on a two-point blowup of $\proj^2$]

Consider the two-point blowup $(\proj^2(2),\omega, T^2,\mu)$ of $\proj^2$ whose moment polytope is given by 
\[ \Delta=\set{ (u_1, u_2) \in \R^2 }{ 
-1\leq u_1 \leq 1,-1\leq u_2 \leq 1,u_1 +u_2 \leq 1+\alpha
}
, \]
where $-1< \alpha <1$. 
Its potential function is  
\[ \po  = Ty_1 + Ty_2 + \frac{T^{1+\alpha}}{y_1y_2} + \frac{T}{y_1} + \frac{T}{y_2}. \]
\begin{align*}
\Longrightarrow 
0 = f 
& : = \frac{\partial \po}{\partial c_2} 
= -k_2y_1 \frac{\partial \po}{\partial y_1} + k_1 y_2 \frac{\partial \po}{\partial y_2} \\
& = T\left(-k_2 \left(y_1-\frac{T^{\alpha}}{y_1y_2} - \frac{1}{y_1}\right) + k_1\left(y_2-\frac{T^{\alpha}}{y_1y_2} - \frac{1}{y_2}\right) \right) \\
& =  T\left(-k_2y_1  + k_1y_2 +(k_2-k_1)\frac{T^{\alpha}}{y_1y_2} + \frac{k_2}{y_1}- \frac{k_1}{y_2}\right) 
\end{align*}

\begin{enumerate}
\item \textbf{Suppose $k_1,k_2,k_2-k_1$ are all non-zero.}

We have
\[ \trop  (f) : \R^2 \to \R, \quad  u \mapsto \min \{ 1+ u_1, 1+u_2, 1+\alpha-u_1-u_2, 1-u_1, 1-u_2\}. \]
\begin{enumerate}[a)]
    \item $u_1=u_2\leq \min \{\alpha-u_1-u_2, -u_1,-u_2\}$. 
    \[ 
    \Longrightarrow u_2 =u_1\leq \min \{ \frac{\alpha}{3}, 0\}
     =  \begin{cases}
     \frac{\alpha}{3} \quad & \text{ if } \alpha \leq 0\\
      0 \quad & \text{ if } \alpha \geq 0
    \end{cases} . 
    \]
    \item $u_1=\alpha-u_1-u_2\leq \min \{u_2, -u_1,-u_2\} $. 
    \[ \Longrightarrow 
    \begin{cases}
     & u_2 = \alpha - 2u_1 \\ 
     & u_1 \leq \alpha-2u_1\\
    & u_1 \leq -u_1\\
     & u_1 \leq 2u_1-\alpha
    \end{cases}
    \quad \Rightarrow 
    \begin{cases}
     & u_2 = \alpha - 2u_1 \\ 
     & \alpha \leq u_1 \leq \min\{ \frac{\alpha}{3}, 0\}
    \end{cases}
    \quad \Rightarrow    \begin{cases}
     & u_2 = \alpha - 2u_1 \\ 
     & \alpha \leq u_1 \leq   \frac{\alpha}{3}\\
     & \alpha \leq 0
    \end{cases}.
    \] 
    \item $u_1=-u_1\leq \min \{u_2, \alpha-u_1-u_2, -u_2\}$.
     \[ \Longrightarrow 
    \begin{cases}
     & u_1=0 \\ 
    & u_2\geq 0\\
    & \alpha-u_2\geq 0 \\
    & -u_2 \geq 0
    \end{cases}
    \quad \Rightarrow \quad  
    \begin{cases}
     & u_1=u_2 =0 \\ 
     & \alpha\geq 0
    \end{cases}.\]
    \item $u_1=-u_2\leq \min \{\alpha-u_1-u_2, -u_1, u_2\}$. 
      \[ \Longrightarrow 
    \begin{cases}
     & u_2 = -u_1 \\ 
    & u_1 \leq \alpha \\
    & u_1\leq -u_1 
    \end{cases}
    \quad \Rightarrow \quad  
    \begin{cases}
     & u_2 = -u_1 \\ 
     & u_1\leq  \min \{0,\alpha\}
      = \begin{cases}
     \alpha \quad & \text{ if } \alpha \leq 0\\
      0 \quad & \text{ if } \alpha \geq 0
    \end{cases}
    \end{cases}.\]
     \item $u_2=\alpha-u_1-u_2\leq \min \{u_1, -u_1,-u_2\}$.
    \[ \Longrightarrow 
    \begin{cases}
     & u_2 = \frac{\alpha-u_1}{2} \\ 
    &  \frac{\alpha-u_1}{2}\leq u_1 \\
    &  \frac{\alpha-u_1}{2}\leq -u_1 \\
    & \frac{\alpha-u_1}{2}\leq -\frac{\alpha-u_1}{2}
    \end{cases}
    \quad \Rightarrow  
    \begin{cases}
     & u_2 =  \frac{\alpha-u_1}{2} \\ 
     & \max\{\alpha, \frac{\alpha}{3}\}\leq u_1\leq -\alpha 
    \end{cases}
     \quad \Rightarrow   
    \begin{cases}
     & u_2 =  \frac{\alpha-u_1}{2} \\ 
     &  \frac{\alpha}{3} \leq u_1\leq -\alpha \\
     & \alpha \leq 0
    \end{cases}
    . \] 
    \item $u_2=-u_1\leq \min \{\alpha-u_1-u_2, u_1,-u_2\}$. 
     \[ \Longrightarrow 
    \begin{cases}
     & u_2 = -u_1 \\ 
    & -u_1 \leq \alpha\\
    & -u_1\leq u_1 
    \end{cases}
    \quad \Rightarrow \quad  
    \begin{cases}
    & u_2 = -u_1 \\
    & u_1 \geq \max\{ 0,-\alpha\}   
    = \begin{cases}
     -\alpha \quad & \text{ if } \alpha \leq 0\\
      0 \quad & \text{ if } \alpha \geq 0
    \end{cases}
    \end{cases}. 
    \]
    \item $u_2=-u_2\leq \min \{\alpha-u_1-u_2, u_1, -u_1\}$. 
 \[ \Longrightarrow 
    \begin{cases}
     & u_2 = 0 \\ 
      &  0\leq u_1 \\
    &  0\leq \alpha - u_1 \\
    & 0\leq -u_1
    \end{cases}
    \quad \Rightarrow \quad  
    \begin{cases}
     & u_1=u_2=0\\
    &   \alpha \geq 0
    \end{cases}.\]
    \item $\alpha-u_1-u_2=-u_1\leq \min \{u_1, u_2,-u_2\}$. 
      \[ \Longrightarrow 
    \begin{cases}
     & u_2 = \alpha \\ 
    & -u_1 \leq u_1\\
    & -u_1 \leq \alpha\\
     & -u_1 \leq -\alpha
    \end{cases}
     \quad \Rightarrow \quad  
   \begin{cases}
     & u_2 = \alpha \\ 
    &  u_1 \geq \max\{0, \alpha, -\alpha\} 
    = 
     \begin{cases}
     -\alpha \quad & \text{ if } \alpha \leq 0\\
      \alpha \quad & \text{ if } \alpha \geq 0
    \end{cases}
    \end{cases} . 
    \] 
    \item $\alpha-u_1-u_2=-u_2\leq \min \{u_1, u_2, -u_1\}$. 
     \[ \Longrightarrow 
    \begin{cases}
     & u_1 = \alpha \\ 
    & -u_2\leq \alpha\\
    & -u_2\leq u_2\\
     & -u_2\leq -\alpha
    \end{cases}
    \quad \Rightarrow \quad  
   \begin{cases}
     & u_1 = \alpha \\ 
    & u_2 \geq \max\{0, \alpha, -\alpha\} 
    = 
     \begin{cases}
     -\alpha \quad & \text{ if } \alpha \leq 0\\
      \alpha \quad & \text{ if } \alpha \geq 0
    \end{cases}
    \end{cases}
    \] 
     \item $-u_1=-u_2\leq \min \{u_1, u_2, \alpha-u_1-u_2\}$. 
      \[ \Longrightarrow 
    \begin{cases}
     & u_1 = u_2 \\ 
    & -u_1\leq u_1\\
    & -u_1\ \leq \alpha-2u_1
    \end{cases}
    \quad \Rightarrow  
   \begin{cases}
     & u_1 = u_2 \\ 
    & u_1 \geq 0\\
     & u_1 \leq \alpha
    \end{cases} 
     \quad \Rightarrow    
   \begin{cases}
     & u_1 = u_2 \\ 
    &  0\leq u_1 \leq \alpha \\
    & \alpha \geq 0
    \end{cases} . 
    \] 
\end{enumerate}

\begin{enumerate}[1)]
\item 
The case  $-1<\alpha< 0$ is shown in Figure \ref{fig: cp2(2) a,b,b-a all non-zero, alpha<0}. 
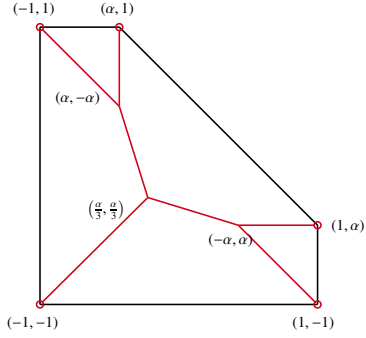
\begin{figure}[!h]
    \centering

\tikzset{every picture/.style={line width=0.75pt}} 
\resizebox{0.35\textwidth}{!}{   
\tikzset{every picture/.style={line width=0.75pt}} 

\begin{tikzpicture}[x=0.75pt,y=0.75pt,yscale=-1,xscale=1]

\draw [line width=1.5]    (150,100) -- (150,450) ;
\draw [line width=1.5]    (150,100) -- (250,100) ;
\draw [line width=1.5]    (500,450) -- (150,450) ;
\draw [line width=1.5]    (500,350) -- (500,450) ;
\draw [line width=1.5]    (250,100) -- (500,350) ;
\draw [color={rgb, 255:red, 208; green, 2; blue, 27 }  ,draw opacity=1 ][line width=1.5]    (152.38,447.64) -- (285.96,315.06) ;
\draw [shift={(150,450)}, rotate = 315.21] [color={rgb, 255:red, 208; green, 2; blue, 27 }  ,draw opacity=1 ][line width=1.5]      (0, 0) circle [x radius= 4.36, y radius= 4.36]   ;
\draw [color={rgb, 255:red, 208; green, 2; blue, 27 }  ,draw opacity=1 ][line width=1.5]    (400,350) -- (496.65,350) ;
\draw [shift={(500,350)}, rotate = 0] [color={rgb, 255:red, 208; green, 2; blue, 27 }  ,draw opacity=1 ][line width=1.5]      (0, 0) circle [x radius= 4.36, y radius= 4.36]   ;
\draw [color={rgb, 255:red, 208; green, 2; blue, 27 }  ,draw opacity=1 ][line width=1.5]    (285.96,315.06) -- (400,350) ;
\draw [color={rgb, 255:red, 208; green, 2; blue, 27 }  ,draw opacity=1 ][line width=1.5]    (400,350) -- (497.63,447.63) ;
\draw [shift={(500,450)}, rotate = 45] [color={rgb, 255:red, 208; green, 2; blue, 27 }  ,draw opacity=1 ][line width=1.5]      (0, 0) circle [x radius= 4.36, y radius= 4.36]   ;
\draw [color={rgb, 255:red, 208; green, 2; blue, 27 }  ,draw opacity=1 ][line width=1.5]    (152.37,102.37) -- (250,200) ;
\draw [shift={(150,100)}, rotate = 45] [color={rgb, 255:red, 208; green, 2; blue, 27 }  ,draw opacity=1 ][line width=1.5]      (0, 0) circle [x radius= 4.36, y radius= 4.36]   ;
\draw [color={rgb, 255:red, 208; green, 2; blue, 27 }  ,draw opacity=1 ][line width=1.5]    (250,103.36) -- (250,200) ;
\draw [shift={(250,100)}, rotate = 90] [color={rgb, 255:red, 208; green, 2; blue, 27 }  ,draw opacity=1 ][line width=1.5]      (0, 0) circle [x radius= 4.36, y radius= 4.36]   ;
\draw [color={rgb, 255:red, 208; green, 2; blue, 27 }  ,draw opacity=1 ][line width=1.5]    (250,200) -- (285.96,315.06) ;

\draw (168,180.4) node [anchor=north west][inner sep=0.75pt]  [font=\Large]  {$( \alpha ,-\alpha )$};
\draw (361,359.4) node [anchor=north west][inner sep=0.75pt]  [font=\Large]  {$( -\alpha ,\alpha )$};
\draw (208,317.4) node [anchor=north west][inner sep=0.75pt]  [font=\Large]  {$\left(\frac{\alpha }{3} ,\frac{\alpha }{3}\right)$};
\draw (114.5,67.9) node [anchor=north west][inner sep=0.75pt]  [font=\Large]  {$( -1,1)$};
\draw (107,464.4) node [anchor=north west][inner sep=0.75pt]  [font=\Large]  {$( -1,-1)$};
\draw (467,464.4) node [anchor=north west][inner sep=0.75pt]  [font=\Large]  {$( 1,-1)$};
\draw (516,341.4) node [anchor=north west][inner sep=0.75pt]  [font=\Large]  {$( 1,\alpha )$};
\draw (225,67.9) node [anchor=north west][inner sep=0.75pt]  [font=\Large]  {$( \alpha ,1)$};

\end{tikzpicture}

}
    \caption{Case when $k_1,k_2,k_2-k_1\ne 0, -1< \alpha <0$}
    \label{fig: cp2(2) a,b,b-a all non-zero, alpha<0}
\end{figure}

\item 
The case  $\alpha=0$ is shown in Figure \ref{fig: cp2(2) a,b,b-a all non-zero, alpha=0}.  

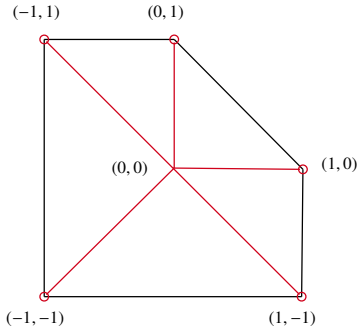
\begin{figure}[!h]
    \centering

\tikzset{every picture/.style={line width=0.75pt}} 
\resizebox{0.35\textwidth}{!}{   
\tikzset{every picture/.style={line width=0.75pt}} 

\begin{tikzpicture}[x=0.75pt,y=0.75pt,yscale=-1,xscale=1]

\draw    (200,100) -- (200,300) ;
\draw    (401,201) -- (400,300) ;
\draw    (200,100) -- (301,100) ;
\draw    (200,300) -- (400,300) ;
\draw    (301,100) -- (401,201) ;
\draw [color={rgb, 255:red, 208; green, 2; blue, 27 }  ,draw opacity=1 ]   (201.67,298.35) -- (301,200) ;
\draw [shift={(200,300)}, rotate = 315.29] [color={rgb, 255:red, 208; green, 2; blue, 27 }  ,draw opacity=1 ][line width=0.75]      (0, 0) circle [x radius= 3.35, y radius= 3.35]   ;
\draw [color={rgb, 255:red, 208; green, 2; blue, 27 }  ,draw opacity=1 ]   (201.66,101.66) -- (398.34,298.34) ;
\draw [shift={(400,300)}, rotate = 45] [color={rgb, 255:red, 208; green, 2; blue, 27 }  ,draw opacity=1 ][line width=0.75]      (0, 0) circle [x radius= 3.35, y radius= 3.35]   ;
\draw [shift={(200,100)}, rotate = 45] [color={rgb, 255:red, 208; green, 2; blue, 27 }  ,draw opacity=1 ][line width=0.75]      (0, 0) circle [x radius= 3.35, y radius= 3.35]   ;
\draw [color={rgb, 255:red, 208; green, 2; blue, 27 }  ,draw opacity=1 ]   (301,102.35) -- (301,200) ;
\draw [shift={(301,100)}, rotate = 90] [color={rgb, 255:red, 208; green, 2; blue, 27 }  ,draw opacity=1 ][line width=0.75]      (0, 0) circle [x radius= 3.35, y radius= 3.35]   ;
\draw [color={rgb, 255:red, 208; green, 2; blue, 27 }  ,draw opacity=1 ]   (300,200) -- (398.65,200.98) ;
\draw [shift={(401,201)}, rotate = 0.57] [color={rgb, 255:red, 208; green, 2; blue, 27 }  ,draw opacity=1 ][line width=0.75]      (0, 0) circle [x radius= 3.35, y radius= 3.35]   ;

\draw (251,193.4) node [anchor=north west][inner sep=0.75pt]    {$( 0,0)$};
\draw (373,310.4) node [anchor=north west][inner sep=0.75pt]    {$( 1,-1)$};
\draw (169,310.4) node [anchor=north west][inner sep=0.75pt]    {$( -1,-1)$};
\draw (174,72.4) node [anchor=north west][inner sep=0.75pt]    {$( -1,1)$};
\draw (279,72.4) node [anchor=north west][inner sep=0.75pt]    {$( 0,1)$};
\draw (414,190.4) node [anchor=north west][inner sep=0.75pt]    {$( 1,0)$};
\end{tikzpicture}
}
    \caption{Case when $k_1,k_2,k_2-k_1\ne 0, \alpha =0$}
    \label{fig: cp2(2) a,b,b-a all non-zero, alpha=0}
\end{figure}

 \item 
The case $\alpha>0$ is shown in Figure \ref{fig: cp2(2) a,b,b-a all non-zero, alpha>0}. 
\begin{figure}[!h]
    \centering
\tikzset{every picture/.style={line width=0.75pt}} 
\resizebox{0.35\textwidth}{!}{

\tikzset{every picture/.style={line width=0.75pt}} 

\begin{tikzpicture}[x=0.75pt,y=0.75pt,yscale=-1,xscale=1]

\draw    (200,100) -- (200,300) ;
\draw    (400.29,179.59) -- (400,300) ;
\draw    (200,100) -- (320,100) ;
\draw    (200,300) -- (400,300) ;
\draw    (320,100) -- (400.29,179.59) ;
\draw [color={rgb, 255:red, 208; green, 2; blue, 27 }  ,draw opacity=1 ]   (201.66,101.66) -- (398.34,298.34) ;
\draw [shift={(400,300)}, rotate = 45] [color={rgb, 255:red, 208; green, 2; blue, 27 }  ,draw opacity=1 ][line width=0.75]      (0, 0) circle [x radius= 3.35, y radius= 3.35]   ;
\draw [shift={(200,100)}, rotate = 45] [color={rgb, 255:red, 208; green, 2; blue, 27 }  ,draw opacity=1 ][line width=0.75]      (0, 0) circle [x radius= 3.35, y radius= 3.35]   ;
\draw [color={rgb, 255:red, 208; green, 2; blue, 27 }  ,draw opacity=1 ]   (201.66,298.34) -- (320,180) ;
\draw [shift={(200,300)}, rotate = 315] [color={rgb, 255:red, 208; green, 2; blue, 27 }  ,draw opacity=1 ][line width=0.75]      (0, 0) circle [x radius= 3.35, y radius= 3.35]   ;
\draw [color={rgb, 255:red, 208; green, 2; blue, 27 }  ,draw opacity=1 ]   (320,180) -- (397.65,180) ;
\draw [shift={(400,180)}, rotate = 0] [color={rgb, 255:red, 208; green, 2; blue, 27 }  ,draw opacity=1 ][line width=0.75]      (0, 0) circle [x radius= 3.35, y radius= 3.35]   ;
\draw [color={rgb, 255:red, 208; green, 2; blue, 27 }  ,draw opacity=1 ]   (320,102.35) -- (320,180) ;
\draw [shift={(320,100)}, rotate = 90] [color={rgb, 255:red, 208; green, 2; blue, 27 }  ,draw opacity=1 ][line width=0.75]      (0, 0) circle [x radius= 3.35, y radius= 3.35]   ;

\draw (321,162.4) node [anchor=north west][inner sep=0.75pt]    {$( \alpha ,\alpha )$};
\draw (251,192.4) node [anchor=north west][inner sep=0.75pt]    {$( 0,0)$};
\draw (311,72.4) node [anchor=north west][inner sep=0.75pt]    {$( \alpha ,1)$};
\draw (171,312.4) node [anchor=north west][inner sep=0.75pt]    {$( -1,-1)$};
\draw (176,72.4) node [anchor=north west][inner sep=0.75pt]    {$( -1,1)$};
\draw (409,172.4) node [anchor=north west][inner sep=0.75pt]    {$( 1,\alpha )$};
\draw (371,312.4) node [anchor=north west][inner sep=0.75pt]    {$( 1,-1)$};
\end{tikzpicture}
}
    \caption{Case when $k_1,k_2,k_2-k_1\ne 0, 0<\alpha<1$}
    \label{fig: cp2(2) a,b,b-a all non-zero, alpha>0}
\end{figure}
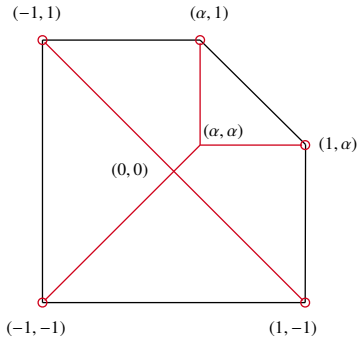
\end{enumerate}  

Then 
\begin{align*}
 W
 := 
 &   V(f) \cap \trop^{-1}(\Delta) 
 \\ 
 = 
 & \Sp \frac{\Lambda \pair{ Ty_1, Ty_2, \frac{T^{1+\alpha}}{y_1y_2},\frac{T}{y_1}, \frac{T}{y_2} } }{
 \lrp{
 T\left(-k_2y_1  + k_1y_2 +(k_2-k_1)\frac{T^{\alpha}}{y_1y_2} + \frac{k_2}{y_1}- \frac{k_1}{y_2}\right) 
 }
 }
 \\
\cong & \Sp \frac{\Lambda \pair{ z_1,z_2, z, x_1,x_2 } }{
\lrp{
 \begin{aligned}
& -k_2z_1+k_1z_2+(k_2-k_1)z +k_2x_1-k_1x_2, 
\\
& z_1z_2z-T^{3+\alpha}, x_1z_1-T^2,  x_2z_2-T^2
\end{aligned}
 } },
\end{align*}
and $\Crit_G^{\Delta}(\po) = W\cap \trop^{-1} \lrp{ \left\{(-1,-1),(-1,1),(\alpha,1),(1,\alpha), (1,-1) \right\}}$.

\item \textbf{Suppose $k_1=0$ and $k_2,k_2-k_1\ne 0$.} Then 
\[ f = -k_2T \left(y_1-\frac{T^{\alpha}}{y_1y_2} - \frac{1}{y_1}\right) \] 
and
\[ \trop f: \R^2 \to \R, \qquad (u_1,u_2)\mapsto \min \{1+ u_1, 1+\alpha-u_1-u_2, 1-u_1\}.  \]
Thus, 
\begin{align*}
\trop(V(f)) 
= &  V(\trop f)  \\
= & \{ u_1=\alpha-u_1-u_2\leq -u_1 \} \cup \{ u_1=-u_1 \leq \alpha-u_1-u_2 \} \\
& \cup \{ \alpha-u_1-u_2= -u_1\leq u_1 \} \\
= & \{ u_2 =\alpha - 2u_1, u_1\leq 0\} \cup \{ u_1=0, u_2 \leq \alpha\} \\
& \cup \{ u_2 = \alpha, u_1 \geq 0\}.
\end{align*}

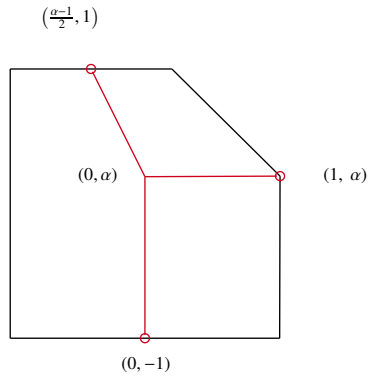
\begin{figure}[!h]
    \centering
\tikzset{every picture/.style={line width=0.75pt}} 
\resizebox{0.35\textwidth}{!}{   
\begin{tikzpicture}[x=0.75pt,y=0.75pt,yscale=-1,xscale=1]

\draw    (200,100) -- (200,300) ;
\draw    (400.29,179.59) -- (400,300) ;
\draw    (200,100) -- (320,100) ;
\draw    (200,300) -- (400,300) ;
\draw    (320,100) -- (400.29,179.59) ;
\draw [color={rgb, 255:red, 208; green, 2; blue, 27 }  ,draw opacity=1 ]   (300,180) -- (397.94,179.6) ;
\draw [shift={(400.29,179.59)}, rotate = 359.77] [color={rgb, 255:red, 208; green, 2; blue, 27 }  ,draw opacity=1 ][line width=0.75]      (0, 0) circle [x radius= 3.35, y radius= 3.35]   ;
\draw [color={rgb, 255:red, 208; green, 2; blue, 27 }  ,draw opacity=1 ]   (300,297.65) -- (300,180) ;
\draw [shift={(300,300)}, rotate = 270] [color={rgb, 255:red, 208; green, 2; blue, 27 }  ,draw opacity=1 ][line width=0.75]      (0, 0) circle [x radius= 3.35, y radius= 3.35]   ;
\draw [color={rgb, 255:red, 208; green, 2; blue, 27 }  ,draw opacity=1 ]   (261.05,102.1) -- (300,180) ;
\draw [shift={(260,100)}, rotate = 63.43] [color={rgb, 255:red, 208; green, 2; blue, 27 }  ,draw opacity=1 ][line width=0.75]      (0, 0) circle [x radius= 3.35, y radius= 3.35]   ;

\draw (249,172.4) node [anchor=north west][inner sep=0.75pt]    {$( 0,\alpha )$};
\draw (431,172.4) node [anchor=north west][inner sep=0.75pt]    {$( 1,\ \alpha )$};
\draw (281,312.4) node [anchor=north west][inner sep=0.75pt]    {$( 0,-1)$};
\draw (221,52.4) node [anchor=north west][inner sep=0.75pt]    {$\left(\frac{\alpha -1}{2} ,1\right)$};

\end{tikzpicture}
}
    \caption{Case when $k_1=0, k_2, k_2-k_1\ne 0$}
    \label{fig: cp2(2) a=0}
\end{figure}
We have 
\begin{align*}
    W  
    = V(f)\cap \trop^{-1}(\Delta) 
    & = \Sp \frac{\Lambda \pair{Ty_1, Ty_2, \frac{T^{1+\alpha}}{y_1y_2},\frac{T}{y_1}, \frac{T}{y_2}  }}{\lrp{ -k_2T \left(y_1-\frac{T^{\alpha}}{y_1y_2} - \frac{1}{y_1}\right) } } 
    \\
   &  
    \cong 
    \Sp \frac{\Lambda \pair{z_1, z_2, z, x_1,x_2}}{
    \lrp{ 
    \begin{aligned}
    & 
    z_1-z_2-x_1, z_1x_1-T^2, 
    \\
    & 
    z_2x_2-T^2, 
    z_1z_2z-T^{3+\alpha}     
    \end{aligned}
    } 
    },
\end{align*}
 and 
 \[ \Crit_G^{\Delta}(\po) = W \setminus \trop^{-1} \left\{ \lrp{\frac{\alpha -1}{2},1}, (0,-1), (1,\alpha) \right\} .\] 
\item \textbf{Suppose $k_2=0$ and $k_1,k_2-k_1\ne 0$.} Then 
\[ f =  k_1T \left(y_2-\frac{T^{\alpha}}{y_1y_2} - \frac{1}{y_2}\right)  \] 
and
\[ \trop f: \R^2 \to \R, \qquad (u_1,u_2)\mapsto \min \{ 1+u_2, 1+\alpha-u_1-u_2, 1-u_2\}.  \]
\begin{align*}
\trop(V(f)) 
= &  V(\trop f)  \\
= & \{ u_2=\alpha-u_1-u_2\leq -u_2 \} \cup \{ u_2=-u_2 \leq \alpha-u_1-u_2 \} \\
& \cup \{ \alpha-u_1-u_2= -u_2\leq u_2 \} \\
= & \{ u_2 =\frac{\alpha - u_1}{2}, u_1 \geq \alpha\} \cup \{ u_2=0, u_1\leq \alpha\} \\
& \cup \{ u_1 = \alpha, u_2 \geq 0\}
\end{align*}

The case $k_2=0, k_1,k_2-k_1\ne 0$ is shown in Figure \ref{fig: cp2(2) b=0}. 
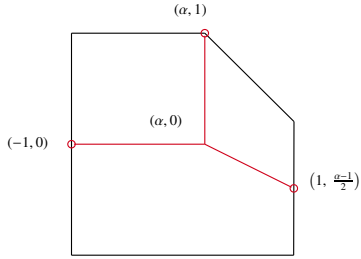
\begin{figure}[!h]
    \centering
\resizebox{0.35\textwidth}{!}{   

\tikzset{every picture/.style={line width=0.75pt}} 

\begin{tikzpicture}[x=0.75pt,y=0.75pt,yscale=-1,xscale=1]

\draw    (200,100) -- (200,300) ;
\draw    (400.29,179.59) -- (400,300) ;
\draw    (200,100) -- (320,100) ;
\draw    (200,300) -- (400,300) ;
\draw    (320,100) -- (400.29,179.59) ;
\draw [color={rgb, 255:red, 208; green, 2; blue, 27 }  ,draw opacity=1 ]   (320,200) -- (202.35,200) ;
\draw [shift={(200,200)}, rotate = 180] [color={rgb, 255:red, 208; green, 2; blue, 27 }  ,draw opacity=1 ][line width=0.75]      (0, 0) circle [x radius= 3.35, y radius= 3.35]   ;
\draw [color={rgb, 255:red, 208; green, 2; blue, 27 }  ,draw opacity=1 ]   (320,102.35) -- (320,200) ;
\draw [shift={(320,100)}, rotate = 90] [color={rgb, 255:red, 208; green, 2; blue, 27 }  ,draw opacity=1 ][line width=0.75]      (0, 0) circle [x radius= 3.35, y radius= 3.35]   ;
\draw [color={rgb, 255:red, 208; green, 2; blue, 27 }  ,draw opacity=1 ]   (398.04,238.75) -- (320,200) ;
\draw [shift={(400.15,239.8)}, rotate = 206.41] [color={rgb, 255:red, 208; green, 2; blue, 27 }  ,draw opacity=1 ][line width=0.75]      (0, 0) circle [x radius= 3.35, y radius= 3.35]   ;

\draw (269,172.4) node [anchor=north west][inner sep=0.75pt]    {$( \alpha ,0)$};
\draw (291,72.4) node [anchor=north west][inner sep=0.75pt]    {$( \alpha ,1)$};
\draw (141,192.4) node [anchor=north west][inner sep=0.75pt]    {$( -1,0)$};
\draw (412,223.4) node [anchor=north west][inner sep=0.75pt]    {$\left( 1,\ \frac{\alpha -1}{2}\right)$};

\end{tikzpicture}
}
    \caption{Case when $k_2=0, k_2-k_1\ne 0$}
    \label{fig: cp2(2) b=0}
\end{figure}
We have 
\begin{align*}
    W  
    = V(f)\cap \trop^{-1}(\Delta) 
    & = \Sp \frac{\Lambda \pair{Ty_1, Ty_2, \frac{T^{1+\alpha}}{y_1y_2},\frac{T}{y_1}, \frac{T}{y_2}  }}{\lrp{ k_1T \left(y_2-\frac{T^{\alpha}}{y_1y_2} - \frac{1}{y_2}\right) } } 
    \\
     &  
    \cong 
    \Sp \frac{\Lambda \pair{z_1, z_2, z, x_1, x_2}}{\lrp{ z_2-z-x_2, z_2x_2-T^2, } },
\end{align*}
 and 
 \[ \Crit_G^{\Delta}(\po) = W \setminus \trop^{-1} \left\{ \lrp{1,\frac{\alpha -1}{2}}, (-1,0), (\alpha,1) \right\} .\] 
\item \textbf{Suppose $k_2-k_1=0$ and $k_1,k_2\ne 0$.} 
 Then 
\[ f = T\left(-k_2y_1  + k_1y_2  + \frac{k_2}{y_1}- \frac{k_1}{y_2}\right) \]  
and 
\[ \trop f: \R^2 \to \R, \qquad (u_1,u_2)\mapsto \min \{ u_1, u_2,-u_1,-u_2\}.  \]
Thus, 
\begin{align*}
\trop(V(f)) 
= &  V(\trop f)  \\
= & \{  u_1= u_2\leq \min\{-u_1,-u_2\} \} \cup \{  u_1= -u_1\leq \min\{u_2,-u_2\} \}  \\
& \cup  \{  u_1= -u_2\leq \min\{u_2,-u_1\} \} \cup \{u_2=-u_1\leq \min\{ u_1,-u_2\}  \}  \\
& \cup  \{  u_2=-u_2\leq \min\{ u_1,-u_1\} \} \cup \{-u_1=-u_2\leq \min\{u_1,u_2\}  \}  \\
=& \{  u_1= u_2\leq 0 \} \cup \{  u_1 =u_2=0 \}  \\
& \cup  \{  u_2= -u_1, u_1\leq 0 \} \cup \{u_2 = -u_1, u_1\geq 0  \}  \\
& \cup  \{ u_2 =u_1=0\} \cup \{ u_1 = u_2\geq 0  \}.
\end{align*}
The case $k_2-k_1=0, k_1,k_2\ne 0$ is shown in Figure \ref{fig: cp2(2) b-a=0}. 
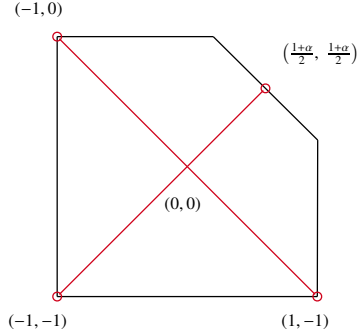
\begin{figure}[!h]
    \centering
\resizebox{0.35\textwidth}{!}{

\tikzset{every picture/.style={line width=0.75pt}} 

\begin{tikzpicture}[x=0.75pt,y=0.75pt,yscale=-1,xscale=1]

\draw    (200,100) -- (200,300) ;
\draw    (400.29,179.59) -- (400,300) ;
\draw    (200,100) -- (320,100) ;
\draw    (200,300) -- (400,300) ;
\draw    (320,100) -- (400.29,179.59) ;
\draw [color={rgb, 255:red, 208; green, 2; blue, 27 }  ,draw opacity=1 ]   (300,200) -- (201.66,298.34) ;
\draw [shift={(200,300)}, rotate = 135] [color={rgb, 255:red, 208; green, 2; blue, 27 }  ,draw opacity=1 ][line width=0.75]      (0, 0) circle [x radius= 3.35, y radius= 3.35]   ;
\draw [color={rgb, 255:red, 208; green, 2; blue, 27 }  ,draw opacity=1 ]   (201.66,101.66) -- (300,200) ;
\draw [shift={(200,100)}, rotate = 45] [color={rgb, 255:red, 208; green, 2; blue, 27 }  ,draw opacity=1 ][line width=0.75]      (0, 0) circle [x radius= 3.35, y radius= 3.35]   ;
\draw [color={rgb, 255:red, 208; green, 2; blue, 27 }  ,draw opacity=1 ]   (398.34,298.34) -- (300,200) ;
\draw [shift={(400,300)}, rotate = 225] [color={rgb, 255:red, 208; green, 2; blue, 27 }  ,draw opacity=1 ][line width=0.75]      (0, 0) circle [x radius= 3.35, y radius= 3.35]   ;
\draw [color={rgb, 255:red, 208; green, 2; blue, 27 }  ,draw opacity=1 ]   (300,200) -- (358.49,141.46) ;
\draw [shift={(360.15,139.8)}, rotate = 314.97] [color={rgb, 255:red, 208; green, 2; blue, 27 }  ,draw opacity=1 ][line width=0.75]      (0, 0) circle [x radius= 3.35, y radius= 3.35]   ;

\draw (281,222.4) node [anchor=north west][inner sep=0.75pt]    {$( 0,0)$};
\draw (161,312.4) node [anchor=north west][inner sep=0.75pt]    {$( -1,-1)$};
\draw (166,72.4) node [anchor=north west][inner sep=0.75pt]    {$( -1,0)$};
\draw (371,312.4) node [anchor=north west][inner sep=0.75pt]    {$( 1,-1)$};
\draw (371,103.4) node [anchor=north west][inner sep=0.75pt]    {$\left(\frac{1+\alpha }{2} ,\ \frac{1+\alpha }{2}\right)$};

\end{tikzpicture}

}
    \caption{Case when $k_2-k_1=0, k_1,k_2\ne 0$}
    \label{fig: cp2(2) b-a=0}
\end{figure}

We have 
\begin{align*}
    W 
    & 
    = V(f)\cap \trop^{-1}(\Delta) \cong \Sp \frac{\Lambda \pair{Ty_1, Ty_2, \frac{T^{1+\alpha}}{y_1y_2}, \frac{T}{y_1},\frac{T}{y_2}  }}{\lrp{ T\left(-k_2y_1  + k_1y_2  + \frac{k_2}{y_1}- \frac{k_1}{y_2}\right) } } 
    \\
    &  
    \cong 
    \Sp \frac{\Lambda \pair{z_1, z_2,z, x_1, x_2}}{\lrp{
    \begin{aligned}
    & -k_2z_1+k_1z_2 + k_2x_1-k_1x_2,\\ 
    & z_1x_1-T^2, z_2x_2-T^2, z_1z_2z-T^{3+\alpha}    
    \end{aligned},
    } },
\end{align*}
 and 
 \[ \Crit_G^{\Delta}(\po) = W \setminus \trop^{-1} \lrp{\left\{ \lrp{\frac{\alpha +1}{2},\frac{\alpha +1}{2}}, (-1,0), (-1,-1), (1,-1) \right\}} .\]
\end{enumerate}
\end{eg}

\subsubsection{\texorpdfstring{$S^1$}{S1}-action on \texorpdfstring{$S^2\left(\frac{c}{2}\right) \times S^2\left(\frac{d}{2}\right)$}{S2(c/2) times S2(d/2)}}
\begin{eg}[$S^1$-action on $S^2\left(\frac{c}{2}\right)\times S^2\left(\frac{d}{2}\right)$, $c<d$] 
\label{S2 times S2 example}
Denote by $S^2(r)$ the $2$-sphere with radius $r$. 
Consider $\left(S^2\left(\frac{c}{2}\right)\times S^2\left(\frac{d}{2}\right),\omega, T^2,\mu\right)$  whose moment polytope is given by
\[ \Delta=\set{ (u_1, u_2) \in \R^2 }{ 
0\leq u_1 \leq c, 0\leq u_2 \leq d
}
. \]
Its potential function is  
\[ \po  = y_1 + y_2 + \frac{T^{c}}{y_1} + \frac{T^{d}}{y_2} . \]
\begin{align*}
\Longrightarrow 
0 = f  : = & \frac{\partial \po}{\partial c_2}  = -k_2y_1 \frac{\partial \po}{\partial y_1} + k_1 y_2 \frac{\partial \po}{\partial y_2} \\
= &  -k_2\left( y_1 -\frac{T^c}{y_1}\right) + k_1\left( y_2 -\frac{T^d}{y_2}\right). 
\end{align*}

\begin{enumerate}[1.]
\item \textbf{Suppose $k_1,k_2\ne 0$.} 
We have 
\[ \trop(f): \R^2 \to \R, \qquad (u_1,u_2)\mapsto \min\{ u_1,c-u_1,u_2,d-u_2\}. \]

\begin{enumerate}[a)]
    \item $u_1 = c-u_1 \leq \min\{u_2, d-u_2\}$ 
    \[ 
    \Longrightarrow 
    \begin{cases}
       &  u_1 = \frac{c}{2} \\
       &  \frac{c}{2}\leq u_2 \\ 
       & \frac{c}{2}\leq d-u_2
    \end{cases}
    \Rightarrow 
        \begin{cases}
       &  u_1 = \frac{c}{2} \\
      &   \frac{c}{2}\leq u_2 \leq d-\frac{c}{2}
    \end{cases}
    \]
    \item $u_1 = u_2\leq \min\{c-u_1 , d-u_2\}$
     \[ 
    \Longrightarrow 
    \begin{cases}
       &  u_1 = u_2 \\
       &  u_1\leq c-u_1 \\ 
       & u_1\leq d-u_1
    \end{cases}
    \Rightarrow 
        \begin{cases}
       &  u_1=u_2 \\
      &   u_1 \leq \frac{c}{2}
    \end{cases}
    \]
    \item $u_1 = d-u_2 \leq \min\{c-u_1, u_2\}$ 
      \[ 
    \Longrightarrow 
    \begin{cases}
       &  u_2 = d-u_1 \\
       &  u_1\leq c-u_1 \\ 
       & u_1\leq d-u_1
    \end{cases}
    \Rightarrow 
        \begin{cases}
       &   u_2 = d-u_1 \\
      &   u_1 \leq \frac{c}{2}
    \end{cases}
    \]
    \item $ c-u_1 =u_2 \leq \min\{u_1 , d-u_2\}$
     \[ 
    \Longrightarrow 
    \begin{cases}
       &  u_2 = c-u_1 \\
       &  c-u_1\leq u_1 \\ 
       & c-u_1\leq d-(c-u_1)
    \end{cases}
    \Rightarrow 
        \begin{cases}
       &  u_2 = c-u_1 \\
       &  u_1 \geq \frac{c}{2}
    \end{cases}
    \]
    \item $ c-u_1 =d-u_2 \leq \min\{u_1 , u_2\}$
     \[ 
    \Longrightarrow 
    \begin{cases}
       &  u_2 = d-c+u_1 \\
       &  c-u_1\leq u_1 \\ 
       & c-u_1\leq  d-c+u_1 
    \end{cases}
    \Rightarrow 
        \begin{cases}
       &  u_2 = d-c+u_1 \\
       &  u_1 \geq  \frac{c}{2}
    \end{cases}
    \]
    \item $ u_2 =d-u_2 \leq \min\{u_1 ,c-u_1\}$
     \[ 
    \Longrightarrow 
    \begin{cases}
       &  u_2 = \frac{d}{2} \\
       &  \frac{d}{2}\leq u_1 \\ 
       & \frac{d}{2}\leq c-u_1 
    \end{cases}
    \Rightarrow 
        \begin{cases}
       &  u_2 = \frac{d}{2} \\
      &   \frac{d}{2}\leq u_1 \leq c-\frac{d}{2}
    \end{cases},
    \]
    which cannot happen because $c<d$. 
\end{enumerate}

The set $\trop(V(f))\cap \interior \Delta$ is shown in Figure \ref{fig: S2 times S2, a,b non-zero}

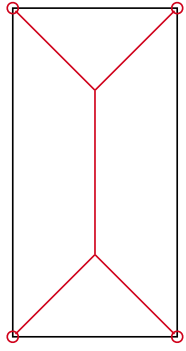
\begin{figure}[!h]
    \centering
\tikzset{every picture/.style={line width=0.75pt}} 
\resizebox{0.18\textwidth}{!}{
\begin{tikzpicture}[x=0.75pt,y=0.75pt,yscale=-1,xscale=1]

\draw   (250,50) -- (350,50) -- (350,250) -- (250,250) -- cycle ;
\draw [color={rgb, 255:red, 208; green, 2; blue, 27 }  ,draw opacity=1 ]   (251.66,51.66) -- (300,100) ;
\draw [shift={(250,50)}, rotate = 45] [color={rgb, 255:red, 208; green, 2; blue, 27 }  ,draw opacity=1 ][line width=0.75]      (0, 0) circle [x radius= 3.35, y radius= 3.35]   ;
\draw [color={rgb, 255:red, 208; green, 2; blue, 27 }  ,draw opacity=1 ]   (348.34,51.66) -- (300,100) ;
\draw [shift={(350,50)}, rotate = 135] [color={rgb, 255:red, 208; green, 2; blue, 27 }  ,draw opacity=1 ][line width=0.75]      (0, 0) circle [x radius= 3.35, y radius= 3.35]   ;
\draw [color={rgb, 255:red, 208; green, 2; blue, 27 }  ,draw opacity=1 ]   (300,100) -- (300,200) ;
\draw [color={rgb, 255:red, 208; green, 2; blue, 27 }  ,draw opacity=1 ]   (251.66,248.34) -- (300,200) ;
\draw [shift={(250,250)}, rotate = 315] [color={rgb, 255:red, 208; green, 2; blue, 27 }  ,draw opacity=1 ][line width=0.75]      (0, 0) circle [x radius= 3.35, y radius= 3.35]   ;
\draw [color={rgb, 255:red, 208; green, 2; blue, 27 }  ,draw opacity=1 ]   (348.34,248.34) -- (300,200) ;
\draw [shift={(350,250)}, rotate = 225] [color={rgb, 255:red, 208; green, 2; blue, 27 }  ,draw opacity=1 ][line width=0.75]      (0, 0) circle [x radius= 3.35, y radius= 3.35]   ;
\end{tikzpicture}
}
    \caption{Case when $k_1,k_2\ne 0$}
    \label{fig: S2 times S2, a,b non-zero}
\end{figure}
We have
\begin{align*}
W 
=  
V(f)\cap \trop^{-1}(\Delta) 
& 
= 
\Sp\frac{\Lambda\pair{y_1,y_2, \frac{T^c}{y_1}, \frac{T^d}{y_2}}}{ \lrp{ -k_2\left( y_1 -\frac{T^c}{y_1}\right) + k_1\left( y_2 -\frac{T^d}{y_2}\right)}}  
\\
&
\cong  
\Sp\frac{\Lambda\pair{y_1,y_2, x_1,x_2}}
{ \lrp{ 
\begin{aligned}
& -k_2\left( y_1 -x_1\right) + k_1\left( y_2 -x_2\right), \\  
& x_1y_1-T^c, x_2y_2-T^d
\end{aligned}
}
},
\end{align*}
and 
\[ \Crit_G^{\Delta}(\Delta) =  W \setminus \trop^{-1} \lrp{\left\{ (c,0), (0,0), (0,d) , (c,d)\right\}}  \]

\item \textbf{Suppose $k_1=0$ and $k_2 \ne 0$.}
Then \[ f =  -k_2\left( y_1 -\frac{T^c}{y_1}\right). \] 
Then 
\[ V(f) = \left\{  \left(T^{\frac{c}{2}}, y_2\right) \mid y_2\in \Lambda^*\right\} \cup \left\{  \left(-T^{\frac{c}{2}}, y_2\right) \mid y_2\in \Lambda^*\right\}, \] 
\[\trop  V(f) = \set{ (u_1,u_2)\in \R^2 }{ u_1 = \frac{c}{2} },  \]
and 
\[ \trop\lrp{ \Crit_G^{\Delta}(\po)} =  \interior \Delta \cap \trop  V(f) = \left \{\frac{c}{2} \right\}\times (0,b). \]

It is shown in Figure 
\ref{fig: S2 times S2, a=0}. 

\begin{figure}[!h]
    \centering
\tikzset{every picture/.style={line width=0.75pt}} 
\resizebox{0.18\textwidth}{!}{
\begin{tikzpicture}[x=0.75pt,y=0.75pt,yscale=-1,xscale=1]

\draw   (250,50) -- (350,50) -- (350,250) -- (250,250) -- cycle ;
\draw [color={rgb, 255:red, 208; green, 2; blue, 27 }  ,draw opacity=1 ]   (300,52.35) -- (300,247.65) ;
\draw [shift={(300,250)}, rotate = 90] [color={rgb, 255:red, 208; green, 2; blue, 27 }  ,draw opacity=1 ][line width=0.75]      (0, 0) circle [x radius= 3.35, y radius= 3.35]   ;
\draw [shift={(300,50)}, rotate = 90] [color={rgb, 255:red, 208; green, 2; blue, 27 }  ,draw opacity=1 ][line width=0.75]      (0, 0) circle [x radius= 3.35, y radius= 3.35]   ;

\end{tikzpicture}
}
    \caption{Case when $k_1=0$}
    \label{fig: S2 times S2, a=0}
\end{figure}

Indeed, 
\begin{align*}
  \Crit_G^{\Delta}(\po) 
  & =   \trop  V(f)\cap \interior \Delta\\
  & = \set{ \lrp{y_1,y_2}\in \Lambda^2 }{ 
  \begin{aligned}
  & y_1-\frac{T^c}{y_1} =0 \\
  &  0<\val(y_1)<c, 0< \val(y_2)< d   
  \end{aligned}
  } \\
    & =  
    \left\{\pm \frac{c}{2}\right\}
    \times 
    \set{ y_2 \in \B_{\Lambda}^1}
    {   
    0< \val(y_2)< d 
    } 
\end{align*}
is a union of two open annuli.

\item \textbf{Suppose $k_2=0$ and $k_1\ne 0$.} 
Then \[ f = k_1\left( y_2 -\frac{T^d}{y_2}\right).  \] 
Then 
\[ V(f) = \left\{  \left(y_1, T^{\frac{d}{2}}\right) \mid y_1\in \Lambda^*\right\} \cup \left\{ \left(y_1, -T^{\frac{d}{2}}\right) \mid y_2\in \Lambda^*\right\}, \]
\[
\trop  V(f) = \set{ (u_1,u_2)\in \R^2 } { u_2 = \frac{d}{2}} 
\]
and 
\[  \interior \Delta \cap \trop  V(f) = (0,a)\times \left\{\frac{d}{2} \right\}. \] 
It is shown in Figure 
\ref{fig: S2 times S2, b=0}.
\begin{figure}[!h]
    \centering
\tikzset{every picture/.style={line width=0.75pt}} 

\resizebox{0.18\textwidth}{!}{
\begin{tikzpicture}[x=0.75pt,y=0.75pt,yscale=-1,xscale=1]

\draw   (250,50) -- (350,50) -- (350,250) -- (250,250) -- cycle ;
\draw [color={rgb, 255:red, 208; green, 2; blue, 27 }  ,draw opacity=1 ]   (252.35,150) -- (347.65,150) ;
\draw [shift={(350,150)}, rotate = 0] [color={rgb, 255:red, 208; green, 2; blue, 27 }  ,draw opacity=1 ][line width=0.75]      (0, 0) circle [x radius= 3.35, y radius= 3.35]   ;
\draw [shift={(250,150)}, rotate = 0] [color={rgb, 255:red, 208; green, 2; blue, 27 }  ,draw opacity=1 ][line width=0.75]      (0, 0) circle [x radius= 3.35, y radius= 3.35]  ;
\end{tikzpicture}
}
    \caption{Case when $k_2=0$}
    \label{fig: S2 times S2, b=0}
\end{figure}

Indeed, 
\begin{align*}
  \Crit_G^{\Delta}(\po) 
  & =   \trop  V(f)\cap \interior \Delta\\
  & = \set{ \lrp{y_1,y_2}\in \Lambda^2 }{ 
  \begin{aligned}
  & y_2-\frac{T^d}{y_2} =0 \\
  &  0<\val(y_1)<c, 0< \val(y_2)< d   
  \end{aligned}
  } \\
    & =  
     \set{ y_1 \in \B_{\Lambda}^1}
    {   
    0< \val(y_1)< d 
    } 
      \times 
    \left\{\pm \frac{d}{2}\right\}
\end{align*}
is a union of two open annuli.
\end{enumerate}
\end{eg}

\section{\texorpdfstring{$T^r$}{Tr}-Equivariant Hamiltonian isotopy invariance}
\label{Section hamiltonian}
 We assume the notations and setup in Section \ref{Section setup}. 

\begin{defn}[Hamiltonian isotopy]
Let $H: [0,1]\times X\to \R$ be a time-dependent Hamiltonian function on $(X,\omega)$. Denote by $H_t$ the map $H(t,\act)$ and $X_{H_t}$ the Hamiltonian vector field of $H_t$ satisfying $\iota_{X_{H_t}}\omega = dH_t$. 
A \textbf{Hamiltonian isotopy} on $X$ generated by $H$ is a smooth map 
\[\psi: [0,1]\times X\to X, \qquad (t,x)\mapsto \psi_t(x) \] 
satisfying 
\[ \frac{d}{dt}\psi_t = X_{H_t}\circ \psi_t. \]
\end{defn}
\begin{defn}[$G$-equivariant Hamiltonian isotopy]
    A Hamiltonian isotopy $\psi: [0,1]\times X\to X$ of a symplectic manifold $(X,\omega)$ with a symplectic action by a compact Lie group $G$ is a \textbf{$G$-equivariant Hamiltonian isotopy} if 
    \[ \psi_t (g\cdot x ) = g\cdot \psi_t(x)\qquad \forall x\in X,\quad \forall g\in G ,\quad \forall t\in[0,1]. \] 
\end{defn}

We will define the Floer cohomology $HF_G(L,b,H, \nove)$ for a Hamiltonian function $H$ satisfying the conditions of the following theorem in Section \ref{Hamiltonian cochain complex}. 
\begin{thm}[$G$-equivariant Hamiltonian isotopy invariance]
\label{Hamiltonian isotopy invariance}
Let $(X,\omega, T^n,\mu)$ be a compact symplectic toric manifold with moment polytope $\Delta = \mu(X) \subset \R^n$. 
Let $G \hookrightarrow T^n$ be a compact $r$-dimensional connected subtorus of $T^n$ with the induced action on $X$. 
Let $u\in \interior \Delta$ and $L = \mu^{-1}(u)$. 
Let $H: [0,1]\times X\to \R$ be a $G$-invariant time-dependent smooth Hamiltonian function. 
Let
\[ \psi_{H}: [0,1]\times X\to X, \qquad (t,x)\mapsto \psi_{H}^t(x)\]
be the $G$-equivariant Hamiltonian isotopy generated by $H$ such that  $\psi_{H}^0=\operatorname{id}$ and that $\psi_{H}^1(L)\cap L$ is a finite union
\[ \psi_{H}^1(L)\cap L = \bigsqcup_{a \in \pi_0\left(\psi_{H}^1(L)\cap L\right)} R_{a}, \quad \]
where each $ R_{a} = G\cdot q_{a}$  for some $q_{a}$ in the component represented by $a \in \pi_0(\psi_{H}^1(L)\cap L)$. 
We fix our choice of $q_{a}$ for each $a\in \pi_0(\psi_{H}^1(L)\cap L)$. 
Let $u\in \interior \Delta$. Let $L = \mu^{-1}(u)$ and $b\in H^1(L,\Lambda_0)$. 
 Then, over the universal Novikov field, we have 
 \begin{equation}\label{invariance isomorphism}
    HF_G ((L,b)\, , \, (L,b), \nove )\cong HF_G (L,b, H, \nove),  
 \end{equation}
 where the right-hand-side is defined in \autoref{HF for Hamiltonian perturbation}. 
 \end{thm}

 \begin{proof}[Proof Outline] 

Before proving Theorem
\ref{Hamiltonian isotopy invariance} in detail, we outline the proof below. 
 \begin{enumerate}[i)]
     \item 
    We first construct a 
 cochain complex 
 $( CF_G(L,b,H), \delta_H^G)$ 
 for a Hamiltonian function $H$ satisfying the conditions of Theorem \ref{Hamiltonian isotopy invariance}. 
\item 
Let $C_G (L, \nove) = \Omega_G(L)\complete{\otimes} \nove$ and $\delta^G = (\m_1^G)^b$. 
We show the cochain maps
\[ 
\left(CF_G(L,b,H), \delta_{H}^G \right) \underset{\g}{\overset{\f}{\rightleftharpoons}}
\left(C_G (L, \nove), \delta^G = (\m_1^G)^b \right)
\] 
define a cochain homotopy equivalence; i.e. $\f\circ\g\sim \operatorname{id}$ and  $\g \circ \f \sim \operatorname{id}$. 
 \end{enumerate}
\end{proof} 

\subsection{The cochain complex \texorpdfstring{$\left(CF_G (L,b, H), \delta_{H}^G \right)$}{(CF G (L,b,H),d{H}{G} )} }
\label{Hamiltonian cochain complex}
Denote $\psi_{H}^1(L)$ by $L_H$. 
By our assumption,
\[ L_H \cap L = \bigsqcup_{a\in \pi_0(\psi_{H}^1(L) \cap L)} R_{a}^H , \] 
 where each $R_{a}^H = G\cdot q_{a}$ is a $G$-orbit for some $q_{a}$ in the component $a \in \pi_0(\psi_{H}^1(L)\cap L)$.  
Let $\psi_H^{-1}$ be the inverse of $\psi_H^1$. We have
\[ L \cap \psi_{H}^{-1}(L) = \psi_{H}^{-1}( \psi_{H}^1(L)\cap L ) = \psi_{H}^{-1}\left( \bigcup_{a\in \pi_0(\psi_{H}^1(L) \cap L)} R_{a}^H\right) =  \bigcup_{a\in \pi_0(\psi_{H}^1(L) \cap L)} \psi_{H}^{-1}(R_{a}^H).   \]

Define 
\[ CF_G (L,b, H) : = \bigoplus_{a\in \pi_0(L_{H} \cap L)} \Omega_G(R_{a}^{H}) \complete{\otimes}_{\R} \nove. \]
$\forall a,a'\in \pi_0(\psi_{H}^1(L) \cap L)$, define
\[ \pi_2(R_{a}^H ,R_{a'}^H) 
= 
\left\{ u:\R \times [0,1]\to X \, \middle| \, 
\begin{aligned}
& u \text{ is smooth}, \\
& u(s,0)\in L,  \quad u(s, 1)\in  L 
\quad \forall s\in \R,\\
&  \exists \, q\in \psi_H^{-1}(R_a^H), \quad  q'\in  \psi_H^{-1}(R_{a'}^H)  \quad \text{ such that  } \\
&  \lim_{s\to -\infty}u(s,t) = \psi_H^t(q) \quad \forall t\in [0,1],\\
& \lim_{s\to \infty}u(s,t) =\psi_H^t(q')  \quad \forall t\in [0,1]
\end{aligned}
\right
\}
\Big/ \sim, \] 
where $u_1\sim u_2$ if and only if $u_1$ and $u_2$ represent the same class in $\pi_2(X)$. 
 
Let $k_1,k_0\in \N$, $a,a'\in \pi_0(\psi_{H}^1(L) \cap L)$, and $B\in \pi_2(R_{a}^H ,R_{a'}^H) $. 
Let $ \mathcal{M}_{k_{1} ,k_{0}}( R_{a}^H ,R_{a'}^H, B) $ be the compactification of 
\[   \mathcal{M}^{\operatorname{\operatorname{reg}}}_{k_{1} ,k_{0}}( R_{a}^H ,R_{a'}^H , B) :=
\left\{ 
(u, \pmb{\tau_1},\pmb{\tau_0}) 
\, \middle| \, 
\begin{aligned}
& u:\R \times [0,1]\to X \text{ is smooth}, \quad [u]=B, \\
& \frac{\partial u}{\partial s} + J \left(\frac{\partial u}{\partial t} - X_{H_t}(u)\right)  = 0, \quad 
E(u) 
<\infty,\\ 
& u(s,0)\in L \quad \text{ and } \quad \quad u(s, 1)\in  L
\quad \forall s\in \R,
\\
&  \exists \, q\in \psi_H^{-1}(R_a^H), \quad  q'\in  \psi_H^{-1}(R_{a'}^H)  \quad \text{ such that  } \\
& \lim_{s\to -\infty}u(s,t) = \psi_{H}^t (q)  \quad \forall t\in [0,1] \quad \text{and} \\
&  \lim_{s\to \infty}u(s,t) = \psi_{H}^t(q') \quad \forall t\in [0,1], \\
& 
\pmb{\tau_1} =  \left( (\zeta_1, 1), \ldots, (\zeta_{k_1}, 1)\right) \in (\R\times \{1\})^{k_1},\\ 
& \text{ where }  
-\infty<\zeta_{k_1}<\cdots < \zeta_1< +\infty,
\\
&\pmb{\tau_0}= \left(  (\tau_1, 0), \ldots, (\tau_{k_0}, 0)\right) \in (\R\times \{0\})^{k_0},\\ 
& 
\text{ where } -\infty<\tau_1<\cdots < \tau_{k_0} < +\infty
\end{aligned}
\right
\}/\sim. 
\]

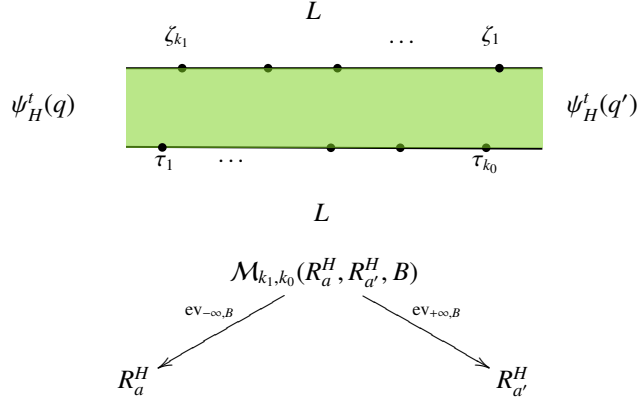
\begin{center}

\tikzset{every picture/.style={line width=0.75pt}} 

\begin{tikzpicture}[x=0.75pt,y=0.75pt,yscale=-1,xscale=1]

\draw    (98,55.67) -- (308,55.67) ;
\draw  [color={rgb, 255:red, 0; green, 0; blue, 0 }  ,draw opacity=1 ][fill={rgb, 255:red, 0; green, 0; blue, 0 }  ,fill opacity=1 ] (124.67,55.67) .. controls (124.67,56.59) and (125.41,57.33) .. (126.33,57.33) .. controls (127.25,57.33) and (128,56.59) .. (128,55.67) .. controls (128,54.75) and (127.25,54) .. (126.33,54) .. controls (125.41,54) and (124.67,54.75) .. (124.67,55.67) -- cycle ;
\draw  [color={rgb, 255:red, 0; green, 0; blue, 0 }  ,draw opacity=1 ][fill={rgb, 255:red, 0; green, 0; blue, 0 }  ,fill opacity=1 ] (168,55.67) .. controls (168,56.59) and (168.75,57.33) .. (169.67,57.33) .. controls (170.59,57.33) and (171.33,56.59) .. (171.33,55.67) .. controls (171.33,54.75) and (170.59,54) .. (169.67,54) .. controls (168.75,54) and (168,54.75) .. (168,55.67) -- cycle ;
\draw  [color={rgb, 255:red, 0; green, 0; blue, 0 }  ,draw opacity=1 ][fill={rgb, 255:red, 0; green, 0; blue, 0 }  ,fill opacity=1 ] (203,55.67) .. controls (203,56.59) and (203.75,57.33) .. (204.67,57.33) .. controls (205.59,57.33) and (206.33,56.59) .. (206.33,55.67) .. controls (206.33,54.75) and (205.59,54) .. (204.67,54) .. controls (203.75,54) and (203,54.75) .. (203,55.67) -- cycle ;
\draw    (308,96.3) -- (98,95.38) ;
\draw  [color={rgb, 255:red, 0; green, 0; blue, 0 }  ,draw opacity=1 ][fill={rgb, 255:red, 0; green, 0; blue, 0 }  ,fill opacity=1 ] (281.33,95.85) .. controls (281.34,94.93) and (280.59,94.18) .. (279.67,94.17) .. controls (278.75,94.17) and (278,94.91) .. (278,95.83) .. controls (278,96.75) and (278.74,97.5) .. (279.66,97.51) .. controls (280.58,97.51) and (281.33,96.77) .. (281.33,95.85) -- cycle ;
\draw  [color={rgb, 255:red, 0; green, 0; blue, 0 }  ,draw opacity=1 ][fill={rgb, 255:red, 0; green, 0; blue, 0 }  ,fill opacity=1 ] (238,95.85) .. controls (238,94.93) and (237.26,94.18) .. (236.34,94.17) .. controls (235.42,94.17) and (234.67,94.91) .. (234.67,95.83) .. controls (234.66,96.75) and (235.41,97.5) .. (236.33,97.51) .. controls (237.25,97.51) and (238,96.77) .. (238,95.85) -- cycle ;
\draw  [color={rgb, 255:red, 0; green, 0; blue, 0 }  ,draw opacity=1 ][fill={rgb, 255:red, 0; green, 0; blue, 0 }  ,fill opacity=1 ] (203,95.85) .. controls (203,94.93) and (202.26,94.18) .. (201.34,94.17) .. controls (200.42,94.17) and (199.67,94.91) .. (199.67,95.83) .. controls (199.66,96.75) and (200.41,97.5) .. (201.33,97.51) .. controls (202.25,97.51) and (203,96.77) .. (203,95.85) -- cycle ;
\draw  [draw opacity=0][fill={rgb, 255:red, 126; green, 211; blue, 33 }  ,fill opacity=0.53 ] (98,55.67) -- (308,55.67) -- (308,95.67) -- (98,95.67) -- cycle ;
\draw  [color={rgb, 255:red, 0; green, 0; blue, 0 }  ,draw opacity=1 ][fill={rgb, 255:red, 0; green, 0; blue, 0 }  ,fill opacity=1 ] (287.99,55.67) .. controls (288,54.75) and (287.25,54) .. (286.33,53.99) .. controls (285.41,53.99) and (284.66,54.73) .. (284.66,55.65) .. controls (284.66,56.57) and (285.4,57.32) .. (286.32,57.33) .. controls (287.24,57.33) and (287.99,56.59) .. (287.99,55.67) -- cycle ;
\draw  [color={rgb, 255:red, 0; green, 0; blue, 0 }  ,draw opacity=1 ][fill={rgb, 255:red, 0; green, 0; blue, 0 }  ,fill opacity=1 ] (117.99,95.67) .. controls (118,94.75) and (117.25,94) .. (116.33,93.99) .. controls (115.41,93.99) and (114.66,94.73) .. (114.66,95.65) .. controls (114.66,96.57) and (115.4,97.32) .. (116.32,97.33) .. controls (117.24,97.33) and (117.99,96.59) .. (117.99,95.67) -- cycle ;

\draw (116,32.4) node [anchor=north west][inner sep=0.75pt]  [font=\small]  {$\zeta_{k_{1}}$};
\draw (229,40.4) node [anchor=north west][inner sep=0.75pt]  [font=\small]  {$\cdots $};
\draw (143,100.4) node [anchor=north west][inner sep=0.75pt]  [font=\small]  {$\cdots $};
\draw (276,32.4) node [anchor=north west][inner sep=0.75pt]  [font=\small]  {$\zeta_1$};
\draw (111,98.4) node [anchor=north west][inner sep=0.75pt]  [font=\small]  {$\tau _{1}$};
\draw (271,98.4) node [anchor=north west][inner sep=0.75pt]  [font=\small]  {$\tau _{k_{0}}$};
\draw (39,66.4) node [anchor=north west][inner sep=0.75pt]    {$\psi_H^t( q)$};
\draw (319,66.4) node [anchor=north west][inner sep=0.75pt]    {$\psi_H^t( q')$};
\draw (191,122.4) node [anchor=north west][inner sep=0.75pt]    {$L$};
\draw (187,20.4) node [anchor=north west][inner sep=0.75pt]    {$L$};
\end{tikzpicture}
\end{center}
\[
\xymatrix{
 & 
 \mathcal{M}_{k_{1} ,k_{0}}(R_{a}^H ,R_{a'}^H, B) 
  \ar[dl]_{\evaluation_{-\infty,B}}
  \ar[dr]^{\evaluation_{+\infty,B}}
 &  
 \\
 R_{a}^H
& 
&
R_{a'}^H
}
\]

Define the evaluation maps as follows. 
\[ \forall 1\leq j\leq k_1, \quad \evaluation_{j,B}^{(1)}:  \mathcal{M}^{\operatorname{reg}}_{k_{1} ,k_{0}}( R_{a}^H ,R_{a'}^H , B,J) \to L, \quad (u, \pmb{\tau_1},\pmb{\tau_0}) \mapsto u(\zeta_j,1).  \]   
 \[ \forall 1\leq j\leq k_0\quad \evaluation_{j,B}^{(0)}:  \mathcal{M}^{\operatorname{reg}}_{k_{1} ,k_{0}}( R_{a}^H ,R_{a'}^H, B,J) \to L, \quad  (u, \pmb{\tau_1},\pmb{\tau_0})\mapsto u(\tau_j,0).  \]  
\[\evaluation_{-\infty,B}:   \Mq_{k_{1} ,k_{0}}( R_{a}^H ,R_{a'}^H, B) \to R_a, \quad  (u, \pmb{\tau_1},\pmb{\tau_0})\mapsto \psi_H^1\left(\lim_{s\to -\infty} u(s,0) \right).  \]
\[\evaluation_{+\infty,B}:  \Mq_{k_{1} ,k_{0}}( R_{a}^H, R_{a'}^H, B) \to R_{a'}, \quad  (u, \pmb{\tau_1},\pmb{\tau_0})\mapsto 
\psi_H^1\left(\lim_{s\to +\infty} u(s,0) \right).  \]

The evaluation maps are $G$-equivariant with respect to the $G$-action given by 
$g\cdot  (u, \pmb{\tau_1},\pmb{\tau_0}) =(l_g\circ u, \tau_1,\tau_0)$ on $\Mq_{k_{1} ,k_{0}}( R_{a}^H ,R_{a'}^H, B)$.

Let $k_1,k_0\in \N$, $a,a'\in \pi_0( \psi_{H}^1(L)\cap L)$, and $B\in \pi_2(R_{a}^H ,R_{a'}^H) $. 
Define  
\[ 
\n_{B}:  
\Omega_G(R_{a}^H)\complete{\otimes}_{\R}\nove [1]  
\to 
\Omega_G(R_{a'}^H)\complete{\otimes}_{\R}\nove [1]  
\]
by 
\[ \n_{B}\left(\eta \right) \nonumber
 =  (\evaluation_{+\infty,B}^G)_{!} (\evaluation_{-\infty,B}^G)^*\eta .  \] 
Define $\delta_H^G: CF_G (L,b, H) [1] \to CF_G (L,b, H) [1]$ such that, 
for each $a\in \pi_0(L_H\cap L)$,  
\[ 
\delta_H^G: \Omega_G(R_a^H)\complete{\otimes}\nove [1] \to CF_G (L,b, H)
\] is given by 
\begin{equation} \label{cochain map for Ham invariance}
\delta_H^G(\eta) = \sum_{a'\in \pi_0(L_H\cap L)}\sum_{B\in \pi_2(R_a^H,R_{a'}^H)}   \exp(\partial B\cap b)\n_{B}( \eta ) \exp(\partial B\cap b) T^{\frac{\omega(B)}{2\pi}}e^{\frac{I_{\mu}(B)}{2}}.     
\end{equation}

\begin{lem}\label{Kuranishi boundary for cochain map delta}
 $ \mathcal{M}_{k_{1} ,k_{0}}( R_{a}^H, R_{a'}^H, B) $ has an oriented $G$-equivariant Kuranishi structure such that $\evaluation_{-\infty,B}, \evaluation_{+ \infty,B} $ are strongly continuous and weakly submersive. 
 Moreover, its normalized boundary is a union of the following types of fiber products below. 
 \begin{enumerate}[i)]
     \item \label{broken strips}
     $\mathcal{M}_{k_{1}' ,k_{0}'}( R_{a}^H, R_{c}^H, B') _{\evaluation_{+\infty,B'}}\times _{\evaluation_{-\infty,B''}}\mathcal{M}_{k_{1}'' ,k_{0}''}( R_{c}^H,  R_{a'}^H, B'') $, where 
     \begin{itemize}
         \item $c\in \pi_0( L_H\cap L)$, 
         \item $k_1', k_1'',k_0',k_0''\in \N$ such that $k_{1}'+k_{1}'' =k_1$, $k_{0}'+k_{0}''=k_0$, 
         \item $B'\in \pi_2(R_{a}^H, R_{c}^H)$, $B''\in \pi_2(R_{c}^H, R_{a'}^H)$ such that $B'\#  B'' = B$. 
     \end{itemize} 
     \item \label{disc bubble on upper boundary}
     $\mathcal{M}_{k_{1}' ,k_{0}}( R_{a}^H, R_{a'}^H, B')      _{\evaluation_{i,B'}^{(1)}} \times_{\evaluation_{0} } \mathcal{M}_{k_{1}''}(L,J, B'')  $, 
     where 
     \begin{itemize} 
        \item $1\leq i\leq k_1'$, 
         \item $k_1', k_{1,1}, \ldots,k_{1,l} \in \N$ such that $k_{1}'+k_{1}'' =k_1 + 1$, 
         \item $B'\in \pi_2( R_{a}^H, R_{a'}^H)$, $B''\in \pi_2(X, L)$ such that 
         $B'\# B'' = B$. 
     \end{itemize} 
   \item \label{disc bubble on lower boundary}
   $\mathcal{M}_{k_{1}  ,k_{0}'}( R_{a}^H,  R_{a'}^H, B',J)
       _{ \evaluation_{i,B'}^{(0)}  } \times_{\evaluation_0}
       \mathcal{M}_{k_0''}( L,J, B'') $, 
       where
    \begin{itemize} 
        \item $1\leq i \leq k_0'$,
         \item $k_0', k_{0}''\in \N$ such that $k_{0}'+k_{0}''=k_0 + 1$, 
         \item $B'\in \pi_2( R_{a}^H, R_{a'}^H)$, $B''\in \pi_2(X, L)$ such that 
         $B'\# B'' = B$. 
     \end{itemize} 
 \end{enumerate}
 Moreover, $\evaluation_{-\infty,B}, \evaluation_{+ \infty,B} $ are strongly continuous and weakly submersive. 
 \end{lem}
\begin{proof} 
The boundary decomposition follows from \cite{fukaya2017unobstructed} Proposition 15.21. 
The construction of a $G$-invariant Kuranishi structure on $\Mq_{k_1,k_0}(R_a^H, R_{a'}^H,B)$ is similar to the proof of Proposition \ref{GKur structure on moduli space} (See \cite{III} Section 4.3.). 
\end{proof}

\begin{prop} $\delta_H^G$ is well-defined and 
$\delta_H^G\circ \delta_H^G =0$.   
\end{prop}
\begin{proof}
The proposition follows from Theorem \ref{Stokes for Gkur}, Proposition \ref{composition formula}, and Lemma \ref{Kuranishi boundary for cochain map delta}. 
In particular, in Lemma 4.1, the contribution of \ref{broken strips}, \ref{disc bubble on upper boundary} with $(k_1',k_0,B')\ne (1,0, 0)$, and \ref{disc bubble on lower boundary} with $(k_1,k_0',B')\ne (0,1, 0)$ are trivial. 
And the contributions of \ref{disc bubble on upper boundary} with $(k_1',k_0,B')= (1,0, 0)$ and \ref{disc bubble on lower boundary} with $(k_1',k_0,B')= (0,1, 0)$ cancel with each other. 
\end{proof}

Thus, we define 
\begin{equation}
\label{HF for Hamiltonian perturbation}
    HF_G(L,b,H,J,\Lambda): = \frac{\ker \delta_H^G}{\image \delta_H^G}. 
\end{equation}

\subsection{The Floer continuation map \texorpdfstring{$\f$}{f}}
Consider a smooth non-decreasing function $ \chi: \R\to [0,1]$ such that 
\begin{equation}
\label{bump function}
\chi (s)
=
\begin{cases}
0 \quad & \text{ if } s\leq -1\\
1 \quad & \text{ if }  s\geq 1
\end{cases} . 
\end{equation}

Define $F: \R\times [0,1]\times X\to R$ by
\[ 
F(s,t,x) = (1-\chi(s))H(t,x) \quad \forall (s,t,x)\in \R\times [0,1]\times X. 
\]
Then 
\[ F(s,t,x) =\begin{cases}
H(t,x)  & \text{ if }  s\leq -1\\
0 & \text{ if }  s\geq 1\\
\end{cases} . 
\]
Since $H$ is $G$-invariant, $F$ is also $G$-invariant. 

Note that $F_{s}^t = F(s,t, \cdot ): X\to \R$ defines a Hamiltonian vector field $X_{F_{s}^t}$ via
\[ dF_{s}^t (Y)= -\omega (X_{F_{s}^t}, Y) \quad \forall Y \in \Gamma(TX).  \]
Let $\psi_s^t: X\to X$ be the flow of $X_{F_{s}^t}$, namely it satisfies 
\[ \frac{d}{dt}\psi_{s}^t = X_{F_{s}^t} \circ \psi_{s}^t. \]

Recall we assumed that $L_H \cap L = \bigsqcup\limits_{a\in \pi_0(L_H,L)} R_a^H$ is a finite union. 
$\forall a\in \pi_0(L_{H} \cap L)$, define
\[ \pi_2(R_a^{H},L) 
= 
\left\{ u:\R \times [0,1]\to X \, \middle| \, 
\begin{aligned}
& u \text{ is smooth}, \\
& u(s,0)\in L \quad \forall s\in \R, \\
& u(s, 1)\in L
\quad \forall s\in \R,\\
& \exists q\in \psi_H^{-1}(R_a^H) \text{ such that }\\
& \lim_{s\to -\infty}u(s,t) = \psi_{H}^t (q) \quad \forall t\in [0,1],  \\
& \lim_{s\to \infty}u(s,t) = p \quad \forall t\in [0,1] 
\text{ for some }p\in L
\end{aligned}
\right
\}
\Big/ \sim, \] 
where $u_1\sim u_2$ if and only if $u_1$ and $u_2$ represent the same class in $\pi_2(X)$. 

Let $k_1,k_0\in \N$, $a\in \pi_0(L_{H} \cap L)$, 
and $B\in \pi_2(R_a^{H},L) $. 
Let $ \mathcal{M}_{k_{1} ,k_{0}}( R_{a}^H, L, B,J) $ be the compactification of 
\[ 
\mathcal{M}^{\operatorname{\operatorname{reg}}}_{k_{1} ,k_{0}}( R_{a}^H, L, B,J) 
:=
\left\{ 
(u, \pmb{\tau_1},\pmb{\tau_0})  
\, \middle| \, 
\begin{aligned}
& u:\R \times [0,1]\to X \text{ is smooth}, \quad [u]=B , \\
& \frac{\partial u}{\partial s} 
+ J \left(
\frac{\partial u}{\partial t} - X_{F_{s}^t} (u)\right) = 0, 
\qquad
E(u)  
<\infty,\\ 
& u(s,0)\in L \quad \text{ and } \quad \quad u(s, 1)\in L
\quad \forall s\in \R,
\\
& \lim_{s\to -\infty}u(s,t) = \psi_{H}^t (q) \quad \forall t\in [0,1] \\
& \text{ for some } q\in \psi_H^{-1}(R_a^H), \\
& \lim_{s\to \infty}u(s,t)= p    \quad \forall t\in [0,1] \text{ for some }p\in L, \\
& 
\pmb{\tau_1} =  \left( (\zeta_1, 1), \ldots, (\zeta_{k_1}, 1)\right) \in (\R\times \{1\})^{k_1},\\ 
& \text{ where }  
-\infty<\zeta_{k_1}<\cdots < \zeta_1< +\infty,
\\
&\pmb{\tau_0}= \left(  (\tau_1, 0), \ldots, (\tau_{k_0}, 0)\right) \in (\R\times \{0\})^{k_0},\\ 
& 
\text{ where } -\infty<\tau_1<\cdots < \tau_{k_0} < +\infty
\end{aligned}
\right
\}.
\]

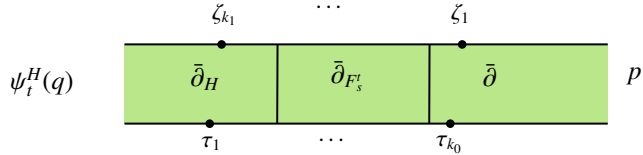
\begin{center}

\tikzset{every picture/.style={line width=0.75pt}} 

\begin{tikzpicture}[x=0.75pt,y=0.75pt,yscale=-1,xscale=1]

\draw  [draw opacity=0][fill={rgb, 255:red, 126; green, 211; blue, 33 }  ,fill opacity=0.53 ] (89.33,68.81) -- (333.07,68.81) -- (333.07,108.81) -- (89.33,108.81) -- cycle ;
\draw    (89.33,108.97) -- (333.07,108.97) ;
\draw    (89.33,68.97) -- (333.07,68.97) ;
\draw  [color={rgb, 255:red, 0; green, 0; blue, 0 }  ,draw opacity=1 ][fill={rgb, 255:red, 0; green, 0; blue, 0 }  ,fill opacity=1 ] (136.67,69) .. controls (136.67,69.92) and (137.41,70.67) .. (138.33,70.67) .. controls (139.25,70.67) and (140,69.92) .. (140,69) .. controls (140,68.08) and (139.25,67.33) .. (138.33,67.33) .. controls (137.41,67.33) and (136.67,68.08) .. (136.67,69) -- cycle ;
\draw  [color={rgb, 255:red, 0; green, 0; blue, 0 }  ,draw opacity=1 ][fill={rgb, 255:red, 0; green, 0; blue, 0 }  ,fill opacity=1 ] (258,69) .. controls (258,69.92) and (258.75,70.67) .. (259.67,70.67) .. controls (260.59,70.67) and (261.33,69.92) .. (261.33,69) .. controls (261.33,68.08) and (260.59,67.33) .. (259.67,67.33) .. controls (258.75,67.33) and (258,68.08) .. (258,69) -- cycle ;
\draw  [color={rgb, 255:red, 0; green, 0; blue, 0 }  ,draw opacity=1 ][fill={rgb, 255:red, 0; green, 0; blue, 0 }  ,fill opacity=1 ] (130.67,109) .. controls (130.67,109.92) and (131.41,110.67) .. (132.33,110.67) .. controls (133.25,110.67) and (134,109.92) .. (134,109) .. controls (134,108.08) and (133.25,107.33) .. (132.33,107.33) .. controls (131.41,107.33) and (130.67,108.08) .. (130.67,109) -- cycle ;
\draw  [color={rgb, 255:red, 0; green, 0; blue, 0 }  ,draw opacity=1 ][fill={rgb, 255:red, 0; green, 0; blue, 0 }  ,fill opacity=1 ] (252,109) .. controls (252,109.92) and (252.75,110.67) .. (253.67,110.67) .. controls (254.59,110.67) and (255.33,109.92) .. (255.33,109) .. controls (255.33,108.08) and (254.59,107.33) .. (253.67,107.33) .. controls (252.75,107.33) and (252,108.08) .. (252,109) -- cycle ;
\draw    (166.53,69.01) -- (166.37,108.62) ;
\draw    (243.2,69.01) -- (243.04,108.62) ;

\draw (30.33,79.21) node [anchor=north west][inner sep=0.75pt]    {$\psi _{t}^{H}( q)$};
\draw (132,46.57) node [anchor=north west][inner sep=0.75pt]  [font=\small]  {$\zeta _{k_{1}}{}$};
\draw (252.67,47.07) node [anchor=north west][inner sep=0.75pt]  [font=\small]  {$\zeta _{1}{}$};
\draw (184.33,48.07) node [anchor=north west][inner sep=0.75pt]  [font=\small]  {$\cdots $};
\draw (126.67,113.88) node [anchor=north west][inner sep=0.75pt]  [font=\small]  {$\tau _{1}{}$};
\draw (245.04,113.38) node [anchor=north west][inner sep=0.75pt]  [font=\small]  {$\tau _{k_{0}}{}$};
\draw (185.67,114.88) node [anchor=north west][inner sep=0.75pt]  [font=\small]  {$\cdots $};
\draw (341,79.21) node [anchor=north west][inner sep=0.75pt]    {$p$};
\draw (121.33,77.21) node [anchor=north west][inner sep=0.75pt]    {$\delb_H$};
\draw (268.67,77.21) node [anchor=north west][inner sep=0.75pt]    {$\delb$};
\draw (192,75.71) node [anchor=north west][inner sep=0.75pt]    {$ \delb_{F_{s}^{t}}$};
\draw (186.67,141.73) node [anchor=north west][inner sep=0.75pt]    {$L$};
\draw (175.33,15.93) node [anchor=north west][inner sep=0.75pt]    {$  L$};

\end{tikzpicture}
    
\end{center}

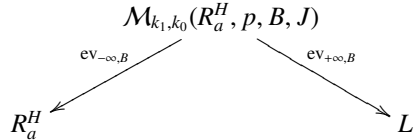
\[
\xymatrix{
 & 
 \mathcal{M}_{k_{1} ,k_{0}}( R_{a}^H, p, B,J) 
  \ar[dl]_{\evaluation_{-\infty,B}}
  \ar[dr]^{\evaluation_{+\infty,B}}
 &  
 \\
 R_{a}^H
& 
&
L
}
\]
Define the evaluation maps as follows. 
 \[ \forall 1\leq j\leq k_1, \quad \evaluation_{j,B}^{(1)}:  \mathcal{M}^{\operatorname{reg}}_{k_{1} ,k_{0}}( R_{a}^H, L, B,J) \to L, \quad (u, \pmb{\tau_1},\pmb{\tau_0})\mapsto u(\zeta_j,1).  \]   
 \[ \forall 1\leq j\leq k_0, \quad \evaluation_{j,B}^{(0)}:  \mathcal{M}^{\operatorname{reg}}_{k_{1} ,k_{0}}( R_{a}^H, L, B,J) \to L, \quad (u, \pmb{\tau_1},\pmb{\tau_0})\mapsto u(\tau_j,0).  \]  
\[\evaluation_{-\infty,B}:   \mathcal{M}^{\operatorname{reg}}_{k_{1} ,k_{0}}( R_{a}^H, L, B,J) \to R_a^H, \quad  (u, \pmb{\tau_1},\pmb{\tau_0})\mapsto \psi_H^1\left(\lim_{s\to -\infty} u(s,0) \right).  \]
\[\evaluation_{+\infty,B}:  \mathcal{M}^{\operatorname{reg}}_{k_{1} ,k_{0}}( R_{a}^H, L, B,J) \to L, \quad (u, \pmb{\tau_1},\pmb{\tau_0})\mapsto \lim_{s\to \infty} u(s,1).  \]
Let $k_1,k_0\in \N$, $a,a'\in \pi_0(L_H\cap L)$, and $B\in \pi_2(R_{a}^H ,R_{a'}^H) $. 

Define 
\[ 
\f_{B}:  \Omega_G(R_{a}^H)\complete{\otimes}_{\R}\nove [1]  \to C_G(L) [1], 
\]
by
\[ \f_{B}\left(\eta \right) \nonumber
 =  (\evaluation_{+\infty,B}^G)_{!} (\evaluation_{-\infty,B}^G)^*\eta .  \] 
Define 
\[ 
\f: C_G (L,b, H ,J)[1] \to C_G (L)[1]
\] 
such that, 
for each  $a \in \pi_0(L_H\cap L)$, 
\[ \f:  \Omega_G(R_a^H)\complete{\otimes}\nove[1] 
\to C_G(L)[1] \]  
is given by
\[ \f(\eta)
= \sum_{B\in \pi_2(R_a^{H},L)} 
\exp ( \partial B\cap b)  \f_{B} ( \eta) \exp ( \partial B\cap b)
 T^{\frac{\omega(B)}{2\pi}} e^{\frac{I_{\mu}(B)}{2}}.  \] 

In order to describe the normalized boundary of the moduli space $ \mathcal{M}_{k_{1} ,k_{0}}( R_{a}^H, L, B,J) $, we introduce some notations.
\begin{itemize}
    \item Let $p',p\in L$. Define
\[ \pi_2(p',p) 
= 
\left\{ u:\R \times [0,1]\to X \, \middle| \, 
\begin{aligned}
& u \text{ is smooth}, \\
& u(s,0)\in L \qand u(s, 1)\in L
\quad \forall s\in \R,\\
& \lim_{s\to -\infty}u(s,t) = p' \quad \forall t\in [0,1] \quad \text{ and } \\
&   \lim_{s\to \infty}u(s,t) = p \quad \forall t\in [0,1]
\end{aligned}
\right
\}
\Big/ \sim, \] 
where $u_1\sim u_2$ if and only if $u_1$ and $u_2$ represent the same class in $\pi_2(X)$. 
   \item Let $a\in \pi_0(\psi_H^1(L)\cap L)$ and $p\in L$. Define 
   \[ \pi_2(R_a^H, p):=
\left\{ u:\R \times [0,1]\to X \, \middle| \, 
\begin{aligned}
& u \text{ is smooth}, \\
& u(s,0)\in L \qand  u(s, 1)\in L
\quad \forall s\in \R,\\
& \exists  q\in \psi_H^{-1}(R_a^H) \quad \text{ such that }\\
& \lim_{s\to -\infty}u(s,t) = \psi_{H}^t (q) \quad \forall t\in [0,1] \\
& \lim_{s\to \infty}u(s,t) = p \quad \forall t\in [0,1]  
\end{aligned}
\right
\}
\Big/ \sim, 
\] 
where $u_1\sim u_2$ if and only if $u_1$ and $u_2$ represent the same class in $\pi_2(X)$. 
\item Let $k_1'',k_0''\in \N$, $a\in \pi_0(L_{H} \cap L)$, 
and $B\in \pi_2(R_a^{H},L) $. 
Let $ \mathcal{M}_{k_{1}'' ,k_{0}''}( p',p, B'',J)  $ be the compactification of 
\[  
\mathcal{M}^{\operatorname{\operatorname{reg}}}_{k_{1}'' ,k_{0}''}( p',p, B'',J) 
:=
\left\{ 
(u, \pmb{\tau_1},\pmb{\tau_0})  
\, \middle| \, 
\begin{aligned}
& u:\R \times [0,1]\to X \text{ is smooth}, \quad [u]=B , \\
& \frac{\partial u}{\partial s} 
+ J \left(
\frac{\partial u}{\partial t} - X_{F_{s}^t} (u)\right) = 0, 
\quad
E(u)  
<\infty,\\ 
& u(s,0)\in L \qand \quad u(s, 1)\in L
\quad \forall s\in \R,
\\
& \lim_{s\to -\infty}u(s,t) = \psi_{H}^t (q) \quad \forall t\in [0,1] \\ 
& \text{ for some } q\in \psi_H^{-1}(R_a^H), \\
& \lim_{s\to \infty}u(s,t)= p    \quad \forall t\in [0,1]\\ 
& \text{ for some }p\in L, \\
& 
\pmb{\tau_1} =  \left( (\zeta_1, 1), \ldots, (\zeta_{k_1''}, 1)\right) \in (\R\times \{1\})^{k_1},\\ 
& \text{ where }  
-\infty<\zeta_{k_1}<\cdots < \zeta_1< +\infty,
\\
&\pmb{\tau_0}= \left(  (\tau_1, 0), \ldots, (\tau_{k_0''}, 0)\right) \in (\R\times \{0\})^{k_0},\\ 
& 
\text{ where } -\infty<\tau_1<\cdots < \tau_{k_0} < +\infty
\end{aligned}
\right
\}.
\] 
\end{itemize}

\begin{lem}\label{Kuranishi boundary for Floer continuation map}
 $ \mathcal{M}_{k_{1} ,k_{0}}( R_{a}^H, L, B,J) $ has an oriented $G$-equivariant Kuranishi structure such that $\evaluation_{-\infty,B}, \evaluation_{+ \infty,B} $ are strongly continuous and weakly submersive. 
 Moreover, its normalized boundary is a union of the four types of fiber products below. 
 \begin{enumerate}[i)]
     \item \label{cochain map after boundary}
     $\mathcal{M}_{k_{1}' ,k_{0}'}( R_{a}^H, R_{c}^H, B',J) \times \mathcal{M}_{k_{1}'' ,k_{0}''}( R_{c}^H, L, B'',J) $, where 
     \begin{itemize}
         \item $c\in \pi_0(L_H\cap L)$, $p\in L$
         \item $k_1', k_1'',k_0',k_0''\in \N$ such that $k_{1}'+k_{1}'' =k_1$, $k_{0}'+k_{0}''=k_0$, 
         \item $B'\in \pi_2(R_{a}^H, R_{c}^H)$, $B''\in \pi_2(R_{c}^H, p)$ such that $B'\#  B'' = B$. 
     \end{itemize} 
     \item \label{boundary after cochain map}
     $\mathcal{M}_{k_{1}' ,k_{0}'}( R_{a}^H, L, B',J) _{\evaluation_{+\infty}}\times_{\evaluation_0} \mathcal{M}_{k_{1}'' +k_{0}''}( L, B'',J) $, where 
     \begin{itemize}
         \item $k_1', k_1'',k_0',k_0''\in \N$ such that $k_{1}'+k_{1}'' =k_1$, $k_{0}'+k_{0}''=k_0$, 
         \item $B'\in \pi_2(R_{a}^H, L)$, $B''\in \pi_2(X,L)$ such that 
         $B'\# B'' = B$. 
     \end{itemize} 
     \item $\mathcal{M}_{k_{1}' ,k_{0}}( R_{a}^H, L, B',J) 
     _{ \evaluation_{i,B'}^{(1)} } \times_{\evaluation_{0}}
    \mathcal{M}_{k_1''}(L,J, B'')  $, 
     where  
     \begin{itemize} 
        \item $1\leq i  \leq k_1'$,  
         \item $k_1', k_{1,1}, \ldots, k_1'' \in \N$ such that $k_{1}'+ k_1'' =k_1 + 1$, 
         \item $B'\in \pi_2( R_{a}^H, L)$, $B'' \in \pi_2(X, L)$ such that 
         $B'\# B'' = B$. 
     \end{itemize} 
       \item $\mathcal{M}_{k_{1}  ,k_{0}'}( R_{a}^H, L, B',J)
       _{ \evaluation_{i,B'}^{(0)}} \times_{\evaluation_{0}}
  \mathcal{M}_{k_{0}''}(L,J, B'') $,  where
     \begin{itemize} 
         \item $1\leq i \leq k_0'$, $1\leq j\leq k_0''$, 
         \item $k_0', k_0''\in \N$ such that $k_{0}'+k_{0}'' =k_0$, 
         \item $B'\in \pi_2( R_{a}^H, p)$, $B''\in \pi_2(X, L)$ such that 
         $B'\# B'' = B$. 
     \end{itemize} 
 \end{enumerate}
 \end{lem}
\begin{proof}
The boundary decomposition follows from \cite{fukaya2017unobstructed} Proposition 15.22. 
The construction of such a $G$-equivariant Kuranishi structure is similar to the proof of Proposition \ref{GKur structure on moduli space} (See \cite{III} Section 4.3.) 
\end{proof}
 
 \begin{cor} $\f$ is a cochain map: 
   \begin{equation}
   \f \circ \delta_{H}^G -  \delta^G \circ \f=0.  \label{cochain} 
   \end{equation}
 \end{cor}
 \begin{proof}
  This follows from Lemma \ref{Kuranishi boundary for Floer continuation map}, Stokes' Theorem \ref{Stokes for Gkur}, and the composition formula (Proposition \ref{composition formula}).   
 \end{proof}
\subsection{The map \texorpdfstring{$\g$}{g}}

Let $\chi$ be as in \autoref{bump function}. 
Define $\overline{F}: \R\times [0,1]\times X\to \R $ by 
\[ \overline{F}(s,t,x)= \overline{F}_{s}^t (x) := \chi(s)H_t(x). \]
Then $\overline{F}_{s}^t(x) = 0$ if $s\leq -1$ and $\overline{F}_{s}^t(x) = H_t(x)$ if $s\geq 1$.

Note that $\overline{F}_{s}^t = \overline{F}(s,t, \cdot ): X\to \R$ defines a Hamiltonian vector field $X_{\overline{F}_{s}^t}$ via
\[ d\overline{F}_{s}^t (Y)= -\omega (X_{\overline{F}_{s}^t}, Y) \quad \forall Y \in \Gamma(TX).  \] 

$\forall a\in \pi_0(\psi_H^1(L) \cap L)$, define
\[ \pi_2(L, R_a^{H}) 
= 
\left\{ u:\R \times [0,1]\to X \, \middle| \, 
\begin{aligned}
& u \text{ is smooth}, \\
& u(s,0)\in L \quad \text{ and } \quad \quad u(s, 1)\in L
\quad \forall s\in \R,\\
& \lim_{s\to -\infty}u(s,t) = p \quad \forall t\in [0,1] 
\text{ for some }p\in L,   \\
& \exists  q\in \psi_H^{-1}(R_a^H) \quad \text{ such that }\\
& \lim_{s\to +\infty}u(s,t) = \phi_{H}^t (q) \quad \forall t\in [0,1] 
\end{aligned}
\right
\}
\Big/ \sim, \] 
where $u_1\sim u_2$ if and only if $u_1$ and $u_2$ represent the same class in $\pi_2(X)$. 

Let $k_1,k_0\in \N$, $a\in \pi_0(L_{H} \cap L)$, 
and $B\in \pi_2(L, R_a^{H}) $. 
Let $ \mathcal{M}_{k_{1} ,k_{0}}( L,  R_{a}^H, B,J) $ be the compactification of 
\[  
\mathcal{M}^{\operatorname{\operatorname{reg}}}_{k_{1} ,k_{0}}(  L,R_{a}^H, B,J) 
:=
\left\{ 
(u, \pmb{\tau_1},\pmb{\tau_0})  
\, \middle| \, 
\begin{aligned}
& u:\R \times [0,1]\to X \text{ is smooth}, \quad [u]=B , \\
& \frac{\partial u}{\partial s} 
+ J \left(
\frac{\partial u}{\partial t} - X_{\overline{F}_{s}^t} (u)\right) = 0, 
\qquad
E(u)  
<\infty,\\ 
& u(s,0)\in L \quad \text{ and } \quad \quad u(s, 1)\in L
\quad \forall s\in \R,
\\
& \lim_{s\to -\infty}u(s,t)= p    \quad \forall t\in [0,1] \text{ for some }p\in L, \\
& \lim_{s\to +\infty}u(s,t) = \psi_{H}^t (q) \quad \forall t\in [0,1] \text{ for some } q\in \psi_H^{-1}(R_a^H), \\
& 
\pmb{\tau_1} =  \left( (\zeta_1, 1), \ldots, (\zeta_{k_1}, 1)\right) \in (\R\times \{1\})^{k_1},\\ 
& \text{ where }  
-\infty<\zeta_{k_1}<\cdots < \zeta_1< +\infty,
\\
&\pmb{\tau_0}= \left(  (\tau_1, 0), \ldots, (\tau_{k_0}, 0)\right) \in (\R\times \{0\})^{k_0},\\ 
& 
\text{ where } -\infty<\tau_1<\cdots < \tau_{k_0} < +\infty
\end{aligned}
\right
\}.
\]

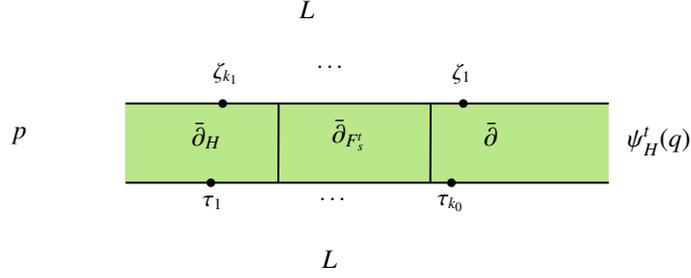
\begin{center}

\tikzset{every picture/.style={line width=0.75pt}} 

\begin{tikzpicture}[x=0.75pt,y=0.75pt,yscale=-1,xscale=1]

\draw  [draw opacity=0][fill={rgb, 255:red, 126; green, 211; blue, 33 }  ,fill opacity=0.53 ] (89.33,68.81) -- (333.07,68.81) -- (333.07,108.81) -- (89.33,108.81) -- cycle ;
\draw    (89.33,108.97) -- (333.07,108.97) ;
\draw    (89.33,68.97) -- (333.07,68.97) ;
\draw  [color={rgb, 255:red, 0; green, 0; blue, 0 }  ,draw opacity=1 ][fill={rgb, 255:red, 0; green, 0; blue, 0 }  ,fill opacity=1 ] (136.67,69) .. controls (136.67,69.92) and (137.41,70.67) .. (138.33,70.67) .. controls (139.25,70.67) and (140,69.92) .. (140,69) .. controls (140,68.08) and (139.25,67.33) .. (138.33,67.33) .. controls (137.41,67.33) and (136.67,68.08) .. (136.67,69) -- cycle ;
\draw  [color={rgb, 255:red, 0; green, 0; blue, 0 }  ,draw opacity=1 ][fill={rgb, 255:red, 0; green, 0; blue, 0 }  ,fill opacity=1 ] (258,69) .. controls (258,69.92) and (258.75,70.67) .. (259.67,70.67) .. controls (260.59,70.67) and (261.33,69.92) .. (261.33,69) .. controls (261.33,68.08) and (260.59,67.33) .. (259.67,67.33) .. controls (258.75,67.33) and (258,68.08) .. (258,69) -- cycle ;
\draw  [color={rgb, 255:red, 0; green, 0; blue, 0 }  ,draw opacity=1 ][fill={rgb, 255:red, 0; green, 0; blue, 0 }  ,fill opacity=1 ] (130.67,109) .. controls (130.67,109.92) and (131.41,110.67) .. (132.33,110.67) .. controls (133.25,110.67) and (134,109.92) .. (134,109) .. controls (134,108.08) and (133.25,107.33) .. (132.33,107.33) .. controls (131.41,107.33) and (130.67,108.08) .. (130.67,109) -- cycle ;
\draw  [color={rgb, 255:red, 0; green, 0; blue, 0 }  ,draw opacity=1 ][fill={rgb, 255:red, 0; green, 0; blue, 0 }  ,fill opacity=1 ] (252,109) .. controls (252,109.92) and (252.75,110.67) .. (253.67,110.67) .. controls (254.59,110.67) and (255.33,109.92) .. (255.33,109) .. controls (255.33,108.08) and (254.59,107.33) .. (253.67,107.33) .. controls (252.75,107.33) and (252,108.08) .. (252,109) -- cycle ;
\draw    (166.53,69.01) -- (166.37,108.62) ;
\draw    (243.2,69.01) -- (243.04,108.62) ;

\draw (30.33,79.21) node [anchor=north west][inner sep=0.75pt]    {$p$};
\draw (132,46.57) node [anchor=north west][inner sep=0.75pt]  [font=\small]  {$\zeta _{k_{1}}{}$};
\draw (252.67,47.07) node [anchor=north west][inner sep=0.75pt]  [font=\small]  {$\zeta _{1}{}$};
\draw (184.33,48.07) node [anchor=north west][inner sep=0.75pt]  [font=\small]  {$\cdots $};
\draw (126.67,113.88) node [anchor=north west][inner sep=0.75pt]  [font=\small]  {$\tau _{1}{}$};
\draw (245.04,113.38) node [anchor=north west][inner sep=0.75pt]  [font=\small]  {$\tau _{k_{0}}{}$};
\draw (185.67,114.88) node [anchor=north west][inner sep=0.75pt]  [font=\small]  {$\cdots $};
\draw (341,79.21) node [anchor=north west][inner sep=0.75pt]    {$\psi_H^t( q)$};
\draw (121.33,77.21) node [anchor=north west][inner sep=0.75pt]    {$\delb_H$};
\draw (268.67,77.21) node [anchor=north west][inner sep=0.75pt]    {$\delb$};
\draw (192,75.71) node [anchor=north west][inner sep=0.75pt]    {$ \delb_{F_{s}^{t}}$};
\draw (186.67,141.73) node [anchor=north west][inner sep=0.75pt]    {$L$};
\draw (175.33,15.93) node [anchor=north west][inner sep=0.75pt]    {$  L$};

\end{tikzpicture}
    
\end{center}

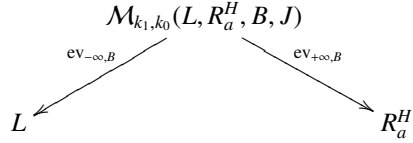
\[
\xymatrix{
 & 
\mathcal{M}_{k_{1} ,k_{0}}( L, R_{a}^H, B,J)
  \ar[dl]_{\evaluation_{-\infty,B}}
  \ar[dr]^{\evaluation_{+\infty,B}}
 &  
 \\
L 
& 
&
R_{a}^H
}
\]

Define the evaluation maps as follows. 
\[ \forall 1\leq j\leq k_1, \quad \evaluation_{j,B}^{(1)}:  \mathcal{M}^{\operatorname{reg}}_{k_{1} ,k_{0}}( L, R_{a}^H, B,J) \to L, \quad (u, \pmb{\tau_1},\pmb{\tau_0})\mapsto u(\zeta_j,1).  \]   
 \[ \forall 1\leq j\leq k_0, \quad \evaluation_{j,B}^{(0)}:  \mathcal{M}^{\operatorname{reg}}_{k_{1} ,k_{0}}( L, R_{a}^H, B,J) \to L, \quad (u, \pmb{\tau_1},\pmb{\tau_0})\mapsto u(\tau_j,0).  \]  
\[\evaluation_{-\infty,B}:   \mathcal{M}^{\operatorname{reg}}_{k_{1} ,k_{0}}( L, R_{a}^H, B,J)\to L, \quad  (u, \pmb{\tau_1},\pmb{\tau_0})\mapsto \lim_{s\to -\infty} u(s,1).  \]
\[\evaluation_{+\infty,B}:  \mathcal{M}^{\operatorname{reg}}_{k_{1} ,k_{0}}( L, R_{a}^H, B,J) \to R_{a}^H, \quad (u, \pmb{\tau_1},\pmb{\tau_0})\mapsto \psi_H^1\left(\lim_{s\to +\infty} u(s,0) \right) .  \]

Let $a\in \pi_0(L_H\cap L)$, and $B\in \pi_2(L, R_{a}^H) $. 

Define 
\[ 
\g_{B}:   C_G(L) [1]\to 
\Omega_G(R_{a}^H)\complete{\otimes}_{\R}\nove [1] , \]
by
\[ \g_{B}\left(\eta \right) \nonumber
 =  (\evaluation_{+\infty,B}^G)_{!} (\evaluation_{-\infty,B}^G)^*\eta .  \] 
Define 
\[ 
\g: C_G (L)[1] \to C_G (L,b, H ,J)[1]
\] such that, 
for each  $a \in \pi_0(L_H\cap L)$, 
\[ \g:  
 C_G(L)[1] \to \Omega_G(R_a^H)\complete{\otimes}\nove[1] \]  
is given by
\[ \g(\eta)
= \sum_{B\in \pi_2(L, R_a^{H})} 
\exp ( \partial B\cap b)  \g_{B} ( \eta) \exp ( \partial B\cap b)
 T^{\frac{\omega(B)}{2\pi}} e^{\frac{I_{\mu}(B)}{2}}.  \] 

\subsection{Proof of Hamiltonian isotopy invariance}

Define a smooth function $\Tilde{F}:[0,\infty)\times \R\times [0,1]\times X\to \R$ as follows. 
Fix a large constant, say $10$. 

For $\theta\geq 10$, let 
\[
\Tilde{F}(\theta, s, t, x)
= \Tilde{F}_{\theta,s}^{t}(x)
= \begin{cases}
F_{s+\theta}^t(x) \quad \text{ if } s \leq 0\\
\overline{F}_{s-\theta}^t(x)\quad \text{ if } s \geq 0\\
\end{cases}. 
\]

\begin{center}

\tikzset{every picture/.style={line width=0.75pt}} 

\begin{tikzpicture}[x=0.75pt,y=0.75pt,yscale=-1,xscale=1]

\draw  [draw opacity=0][fill={rgb, 255:red, 184; green, 233; blue, 134 }  ,fill opacity=1 ] (10,10) -- (300,10) -- (300,50) -- (10,50) -- cycle ;
\draw    (10,10) -- (300,10) ;
\draw    (10,50) -- (300,50) ;
\draw    (130,10) -- (130,50) ;
\draw    (180,10) -- (180,50) ;
\draw    (230,10) -- (230,50) ;
\draw    (80,10) -- (80,50) ;

\draw (36,22.4) node [anchor=north west][inner sep=0.75pt]  [font=\small]  {$H$};
\draw (251,22.4) node [anchor=north west][inner sep=0.75pt]  [font=\small]  {$H$};
\draw (91,22.4) node [anchor=north west][inner sep=0.75pt]  [font=\small]  {$F_{s+\theta}^{t}$};
\draw (191,22.4) node [anchor=north west][inner sep=0.75pt]  [font=\small]  {$\overline{F}_{s-\theta}^{t}$};
\draw (146,22.4) node [anchor=north west][inner sep=0.75pt]    {$0$};
\draw (61,62.4) node [anchor=north west][inner sep=0.75pt]  [font=\tiny]  {$-( \theta+1)$};
\draw (218,62.4) node [anchor=north west][inner sep=0.75pt]  [font=\tiny]  {$\theta+1$};
\draw (169,62.4) node [anchor=north west][inner sep=0.75pt]  [font=\tiny]  {$\theta-1$};
\draw (111,62.4) node [anchor=north west][inner sep=0.75pt]  [font=\tiny]  {$-( \theta-1)$};
\end{tikzpicture}

\end{center}
Let $\Tilde{\chi}: \R\to[0,1]$ be a smooth nondecreasing function such that 
\[ \Tilde{\chi}(\theta) 
= \begin{cases}
    0 \quad & \text{ if } \theta\leq 1 \\
    1 \quad & \text{ if } \theta\geq 9 \\
\end{cases}. \] 
For $0 \leq \theta < 10$, define 
\[
\Tilde{F}(\theta, s, t, x)
= (1-\Tilde{\chi}(\theta))H_t(x) + \Tilde{\chi}(\theta)\tilde{F}_{10,s}^t(x).  
\] 
Consider 
\[
\pi_2^{\theta}(R_a^{H},R_{a'}^{H}) = \left\{ u:\R \times [0,1]\to X \, \middle| \, 
\begin{aligned}
& u \text{ is smooth}, \\
& u(s,0)\in L \quad \text{ and } \quad \quad u(s, 1)\in \psi_{\theta,s}^1(L)
\quad \forall s\in \R,\\
&  \lim_{s\to -\infty}u(s,t) = \psi_0^t(q_{a}) \quad \forall t\in [0,1],\\
& \lim_{s\to \infty}u(s,t) =\psi_1^t(q_{a'})  \quad \forall t\in [0,1]
\end{aligned}
\right
\}
\Big/ \sim, \] 
where $u_1\sim u_2$ if and only if $u_1$ and $u_2$ represent the same class in $\pi_2(X)$. 

Let $a,a'\in \pi_0(\psi_H^1(L)  \cap L)$. Define 
\begin{align*}
\pi_2^{[0,\infty]}(R_a^{H},R_{a'}^{H}) 
= 
\left(
\pi_2(R_a^{H},R_{a'}^{H}) \cup \bigcup_{\theta\in [0,\infty)}  \pi_2^{\theta}(R_a^{H},R_{a'}^{H})
\right)
\Big/ \sim, 
\end{align*}
where $[u_1]\sim [u_2]$ 
if and only if 
$[u_1]=[u_2]$ in $\pi_2(X)$.

Let $\theta\in [0,\infty)$, $k_1,k_0\in \N$, $a,a'\in \pi_0(\psi_H^1(L) \cap L)$, and $B\in \pi_2^{[0,\infty]}(R_a^{H},R_{a'}^{H}) $. 
Define 
\begin{align*}
 & \mathcal{M}^{\theta,reg}_{k_{1} ,k_{0}}( R_{a}^H, R_{a'}^H, B,J) \\
: = & \left\{ 
(u, \tau_1,\tau_0) 
\, \middle| \, 
\begin{aligned}
& u:\R \times [0,1]\to X \text{ is smooth}, \quad [u]=B , \\
& \frac{\partial u}{\partial s} + J\left(\frac{\partial u}{\partial t} - X_{\Tilde{F}_{\theta,s}^{t}}(u)\right) = 0, \\
& E(u) =  \norm{\frac{\partial u}{\partial s}  }^2 +  \norm{\frac{\partial u}{\partial t} - X_{\Tilde{F}_{\theta,s}^{t}}(u) }^2<\infty,\\ 
& u(s,0)\in L \quad \text{ and } \quad \quad u(s, 1)\in L
\quad \forall s\in \R,
\\
&  \exists \, q\in \psi_H^{-1}(R_a^H), \quad  q'\in  \psi_H^{-1}(R_{a'}^H)  \quad \text{ such that  } \\
& \lim_{s\to -\infty}u(s,t) = \psi_{H}^t (q) \quad \text{and} \quad  \lim_{s\to \infty}u(s,t) = \psi_{H}^t(q') \quad \forall t\in [0,1], \\
& \tau_0= \left(  (\tau_1, 0), \ldots, (\tau_{k_0}, 0)\right) \in (\R\times \{0\})^{k_0},\\ 
& 
\text{ where } -\infty<\tau_{0,1}<\cdots < \tau_{0,k_0} < +\infty, 
\\
& 
\tau_1 =  \left( (\zeta_1, 1), \ldots, (\zeta_{k_1}, 1)\right) \in (\R\times \{1\})^{k_1},\\ 
& \text{ where }  
-\infty<\zeta_{k_1}<\cdots < \zeta_1 < +\infty
\end{aligned}
\right
\}.
\end{align*}

When $\theta = \infty $, let
\[ 
\mathcal{M}^{+\infty,reg}_{k_{1} ,k_{0}}( R_{a}^H, R_{a'}^H, B,J) 
:=
\bigcup_{B'\#B'' = B}
\bigcup_{\substack{k_1'+k_1'' = k_1\\ k_0'+k_0'' = k_0}}
\mathcal{M}_{k_{1}' ,k_{0}'}( R_{a}^H, L, B',J) \times \mathcal{M}_{k_{1}'' ,k_{0}''}( L, R_{a'}^H, B'',J),  
\]  
where $ \mathcal{M}_{k_{1}'' ,k_{0}''}(  L,R_{a}^H, B,J) $ is the compactification of 
\begin{align*}
& \mathcal{M}^{\operatorname{\operatorname{reg}}}_{k_{1} ,k_{0}}(  L, R_{a}^H, B,J) 
:=
& \left\{ 
(u, \pmb{\tau_1},\pmb{\tau_0})  
\, \middle| \, 
\begin{aligned}
& u:\R \times [0,1]\to X \text{ is smooth}, \quad [u]=B , \\
& \frac{\partial u}{\partial s} 
+ J \left(
\frac{\partial u}{\partial t} - X_{F_{s}^t} (u)\right) = 0, 
\qquad
E(u)  
<\infty,\\ 
& u(s,0)\in L \quad \text{ and } \quad \quad u(s, 1)\in L
\quad \forall s\in \R,
\\
& \lim_{s\to -\infty}u(s,t) = p    \quad \forall t\in [0,1] \text{ for some } p\in L \\
& \lim_{s\to \infty}u(s,t)
\psi_{H}^t (q) \quad \forall t\in [0,1] \text{ for some } q\in \psi_H^{-1}(R_a^H),\\
& 
\pmb{\tau_1} =  \left( (\zeta_1, 1), \ldots, (\zeta_{k_1}, 1)\right) \in (\R\times \{1\})^{k_1},\\ 
& \text{ where }  
-\infty<\zeta_{k_1}<\cdots < \zeta_1< +\infty,
\\
&\pmb{\tau_0}= \left(  (\tau_1, 0), \ldots, (\tau_{k_0}, 0)\right) \in (\R\times \{0\})^{k_0},\\ 
& 
\text{ where } -\infty<\tau_1<\cdots < \tau_{k_0} < +\infty
\end{aligned}
\right
\}.
\end{align*}
Let 
\[ \mathcal{M}^{ [0,+\infty],reg}_{k_{1} ,k_{0}}( R_{a}^H, R_{a'}^H, B,J) 
= 
\bigcup_{\theta\in [0,+\infty]} \left(\{\theta \}\times \mathcal{M}^{\theta,reg}_{k_{1} ,k_{0}}( R_{a}^H, R_{a'}^H, B,J)\right)  \]
and let
$ \mathcal{M}^{[0,+\infty]}_{k_{1} ,k_{0}}( R_{a}^H, R_{a'}^H, B,J) $ be its compactification. 

\[
\xymatrix{
 & 
\mathcal{M}^{[0,+\infty]}_{k_{1} ,k_{0}}( R_{a}^H, R_{a'}^H, B,J)
  \ar[dl]_{\evaluation_{-\infty,B}}
  \ar[dr]^{\evaluation_{+\infty,B}}
 &  
 \\
R_a^H
& 
&
 R_{a'}^H
}
\]

Define the evaluation maps as follows. 
 \[ \forall 1\leq j\leq k_1, \quad \evaluation_{j,B}^{(1)}:    \mathcal{M}^{[0,+\infty),reg}_{k_{1} ,k_{0}}( R_{a}^H, R_{a'}^H, B,J) \to L, \quad (\theta, (u, \pmb{\tau_1},\pmb{\tau_0}))\mapsto u(\zeta_j,1) .  \]   
\[ \forall 1\leq j\leq k_0, \quad \evaluation_{j,B}^{(0)}:    \mathcal{M}^{[0,+\infty),reg}_{k_{1} ,k_{0}}( R_{a}^H, R_{a'}^H, B,J) \to L, \quad (\theta, (u, \pmb{\tau_1},\pmb{\tau_0}))\mapsto u(\tau_j,0).  \]  
\[\evaluation_{-\infty,B}:     \mathcal{M}^{[0,+\infty),reg}_{k_{1} ,k_{0}}( R_{a}^H, R_{a'}^H, B,J) \to R_{a}^H, \quad (\theta, (u, \pmb{\tau_1},\pmb{\tau_0}))\mapsto \psi_H^1\left(\lim_{s\to -\infty} u(s,0) \right) .  \]
\[\evaluation_{+\infty,B}:    \mathcal{M}^{[0,+\infty),reg}_{k_{1} ,k_{0}}( R_{a}^H, R_{a'}^H, B,J) \to R_{a'}^H, \quad (\theta, (u, \pmb{\tau_1},\pmb{\tau_0}))\mapsto \psi_H^1\left(\lim_{s\to +\infty} u(s,0) \right).   \]
 
Let $a,a'\in \pi_0(L_H\cap L)$, and $B\in \pi_2^{[0,\infty]}( R_{a}^H, R_{a'}^H) $. 

Define 
\[ 
\Theta_{B}:  \Omega_G(R_{a}^H)\complete{\otimes}_{\R}\nove [1]\to 
\Omega_G(R_{a'}^H)\complete{\otimes}_{\R}\nove [1] , \]
by
\[ \Theta_{B}\left(\eta \right) \nonumber
 =  (\evaluation_{+\infty,B}^G)_{!} (\evaluation_{-\infty,B}^G)^*\eta .  \] 
Define 
\[ \Theta: CF_G(L,b,H,J)[1]\to  CF_G(L,b,H,J)[1]
\] 
such that, 
for each $a,a'\in \pi_0(\psi_H^1(L) \cap L)$, 
\[ \Theta:  \Omega_G(R_a^H)\complete{\otimes}_{\R}\nove [1]   \to \Omega_G(R_{a'}^H)\complete{\otimes}_{\R}\nove[1]\]
is given by
\[ \Theta(\eta) = \sum_{B\in \pi_2^{[0,\infty]}( R_{a}^H, R_{a'}^H)} \exp ( \partial B\cap b)\Theta_{B} (\eta ) \exp ( \partial B\cap b)  T^{\frac{\omega(B)}{2\pi}} e^{\frac{I_{\mu}(B)}{2}}. 
\] 

\begin{lem}\label{Kuranishi boundary for cochain homotopic to identity}
 $\mathcal{M}^{[0,+\infty]}_{k_{1} ,k_{0}}( R_{a}^H, R_{a'}^H, B,J)$ 
 has an oriented $G$-equivariant Kuranishi structure  such that $\evaluation_{-\infty,B}, \evaluation_{+ \infty,B} $ are strongly continuous and weakly submersive. 
 Moreover, its normalized boundary is a union of the following types of fiber products below.  
\begin{enumerate}[i)]
     \item 
    \label{s=-infty}
     $\mathcal{M}_{k_{1}' ,k_{0}'}( R_{a}^H, R_{c}^H, B',J) _{\evaluation_{\infty}}\times_{\evaluation_{-\infty}} \mathcal{M}^{[0,+\infty]}_{k_{1}'' ,k_{0}''}( R_{c}^H, R_{a'}^H, B'',J) $, where 
     \begin{itemize}
         \item $c\in \pi_0(L_H\cap L)$, 
         \item $k_1', k_1'',k_0',k_0''\in \N$ such that $k_{1}'+k_{1}'' =k_1$, $k_{0}'+k_{0}''=k_0$, 
         \item $B'\in \pi_2(R_{a}^H, R_{c}^H)$, $B''\in \pi_2(R_{c}^H, R_{a'}^H)$ such that $B'\#  B'' = B$. 
     \end{itemize} 
   
\item 
\label{s=infty}
$\mathcal{M}^{[0,+\infty]}_{k_{1}' ,k_{0}'}( R_{a}^H, R_{c}^H, B',J) _{\evaluation_{\infty}}\times_{\evaluation_{-\infty}} \mathcal{M}_{k_{1}'' ,k_{0}''}( R_{c}^H, R_{a'}^H, B'',J) $, where 
    \begin{itemize}
         \item $c\in \pi_0(L_H\cap L)$, 
         \item $k_1', k_1'',k_0',k_0''\in \N$ such that $k_{1}'+k_{1}'' =k_1$, $k_{0}'+k_{0}''=k_0$, 
         \item $B'\in \pi_2(R_{a}^H, R_{c}^H)$, $B''\in \pi_2(R_{c}^H, R_{a'}^H)$ such that 
         $B'\# B'' = B$. 
     \end{itemize} 
    
\item \label{t=1}
$\mathcal{M}^{[0,+\infty]}_{k_{1} ,k_{0}'}( R_{a}^H, R_{a'}^H, B',J) _{\evaluation_{1,i,B'} } \times_{\evaluation_0} 
    \mathcal{M}_{k_{0}''}( L, B'',J) $, where 
    \begin{itemize}
      \item $1\leq i \leq k_1'$,
     \item $k_1',k_1'' \in \N$ such that $k_{1}'+k_1'' =k_1 + 1$, 
     \item $B'\in \pi_2( R_{a}^H, R_{a'}^H)$, $B''\in \pi_2(X, L)$ such that 
      $B'\# B'' = B$. 
    \end{itemize}
\item 
    \label{t=0}
$\mathcal{M}^{[0,+\infty]}_{k_{1}' ,k_{0}}( R_{a}^H, R_{a'}^H, B',J) _{ \evaluation_{0,i,B'} } \times_{\evaluation_0}  \mathcal{M}_{k_{0}''}( L, B'',J) $, where 
    \begin{itemize} 
         \item $1\leq i \leq k_0'$, 
         \item $k_0', k_0''\in \N$ such that $k_{0}'+k_{0}'' =k_0+1$, 
         \item $B'\in \pi_2( R_{a}^H, R_{a'}^H)$, $B''\in \pi_2(X, L)$ such that 
         $B'\# B'' = B$. 
     \end{itemize}
      \item  \label{theta = infty}
     $\mathcal{M}_{k_{1}' ,k_{0}'}( R_{a}^H, L, B',J) _{\evaluation_{\infty}}\times_{\evaluation_{-\infty}} \mathcal{M}_{k_{1}'' ,k_{0}''}( L, R_{a'}^H, B'',J)$, where 
     \begin{itemize} 
          \item $k_1', k_1'',k_0',k_0''\in \N$ such that $k_{1}'+k_{1}'' =k_1$, 
         \item $B'\in \pi_2( R_{a}^H, L)$, $B''\in \pi_2( L, R_{a'}^H)$ such that 
         $B'\# B'' = B$. 
     \end{itemize}
\item \label{theta = 0}
$\widetilde{\mathcal{M}}_{k_{1}  ,k_{0}}( R_{a}^H, R_{a'}^H, B,J) $, 
       whose quotient space by the free $\R$-action  is $\mathcal{M}_{k_{1}  ,k_{0}}( R_{a}^H, R_{a'}^H, B,J) $.
   
 \end{enumerate}

\end{lem}
\begin{proof}
The cases where bubbling happens at $s = -\infty$ and $s = + \infty$ correspond to \ref{s=-infty} and \ref{s=infty}, respectively. 
The cases where bubbling happens at $t=1$ and $t=0$ correspond to \ref{t=1} and \ref{t=0}, respectively. 
The case $\theta =0$ corresponds to \ref{theta = 0}.  
The case $\theta = \infty$ corresponds to \ref{theta = infty}.   
The construction of a $G$-invariant Kuranishi structure is similar to the proof of Proposition \ref{GKur structure on moduli space} (See \cite{III} Section 4.3.). 
The boundary decomposition is similar to the proof of \cite{fukaya2017unobstructed} Proposition 15.22. 
\end{proof}

\begin{cor}
\label{g circ f is chain homotopy homotopic to identity}
 \[\g\circ \f - \mathbbm{1}^G =  \Theta \circ \delta_H^G +\delta_H^G \circ \Theta . \]   
\end{cor}
\begin{proof}
In Lemma \ref{Kuranishi boundary for cochain homotopic to identity}, 
the contributions of \ref{s=-infty} and \ref{s=infty} vanish. 
The terms \ref{t=1} and \ref{t=0} correspond to $(\Theta -d_G) \circ \delta_H^G$ and $\delta_H^G\circ (\Theta -d_G)$. 
And \ref{theta = infty} corresponds to $\g\circ \f$. 
The contribution of \ref{theta = infty} vanishes except the case when $B=0$, which corresponds to $\mathbbm{1}^G$. 
Then the corollary follows from Lemma \ref{Kuranishi boundary for cochain homotopic to identity}, Proposition \ref{composition formula}, and Theorem \ref{Stokes for Gkur}. 
\end{proof}

\begin{cor}
 $\f\circ \g $ is cochain homotopic to the identity map.   
\end{cor}
\begin{proof}
The proof is similar to that of Corollary \ref{g circ f is chain homotopy homotopic to identity}. 
\end{proof}

\begin{cor}
    If $HF_G((L,b),(L,b),\nove)\ne 0$, then $L$ is not displaceable by any $G$-equivariant Hamiltonian diffeomorphism that is the time-1 map of a $G$-equivariant Hamiltonian isotopy. 
\end{cor}
\begin{proof}
    If $L$ is displaceable by $\psi_H^1$, which is a $G$-equivariant Hamiltonian diffeomorphism that is the time-1 map of a $G$-equivariant Hamiltonian isotopy, then $\psi_H^1(L)\cap L = \emptyset$ implies $CF_G(L,b,H) = 0$. 
    Thus, 
    \[ 0 = HF_G(L,b,H) \cong HF_G((L,b),(L,b),\nove).\] 
\end{proof}

\section{\texorpdfstring{$G$}{G}-Equivariant Kuranishi structures}\label{GKur section}

\subsection{\texorpdfstring{$G$}{G}-equivariant Kuranishi structures}

We review related concepts in the orbifold theory in \S\ref{orbifolds section}. 
Moreover, we assume $G$ is a torus and $G$ acts on $\kur$ freely when necessary. 
For the general theory of Kuranishi structures, we refer the readers to the book \cite{kurbook}. 

\begin{defn}[$G$-equivariant Kuranishi chart]
\label{Kuranishi chart}
  Let $\mathcal{M}$ be a separable metrizable topological space with a topological action by a compact Lie group $G$. 
A \textbf{$G$-equivariant Kuranishi chart} on $\mathcal{M}$ is a quadruple
$\mathcal{U} = (U, \E,  \psi , s )$ that satisfies the following. 
    \begin{enumerate}[(a)]
        \item $U$ and $\E$ are oriented smooth effective orbifolds, possibly with corners, each equipped with a smooth $G$-action. 
        \item $\E \xrightarrow{\pi } U $ is a $G$-equivariant orbibundle. 
        \item $s : U  \to \E$ is a $G$-equivariant section of $\pi$. 
        \item $\psi: s^{-1}(0)\to \mathcal{M}$ is a $G$-equivariant continuous map, which is a homeomorphism onto an open subset in $\kur$. 
    \end{enumerate}
    We say $U$ is a \textbf{Kuranishi neighborhood}, $\E$ is an \textbf{obstruction bundle}, $\psi$ is a \textbf{parametrization} map, and $s$ is a \textbf{Kuranishi map}. 
    \end{defn}
    
\begin{defn}[Restriction of a $G$-equivariant Kuranishi chart]
Let $\mathcal{U}=(U, \E , \psi, s)$ be a $G$-equivariant Kuranishi chart and $U'\subset U$ be a $G$-invariant open subset of $U$. 
Then the restriction of $\mathcal{U}$ to $U'$ defines a Kuranishi chart $\mathcal{U}'=(U', \restr{\E}{U'} , \restr{\psi}{s^{-1}(0)\cap U'}, \restr{s}{U'})$. 
\end{defn}

\begin{defn}[Embedding of $G$-equivariant Kuranishi charts]
Let $\mathcal{M}$ be a separable metrizable topological space with a topological action by a compact Lie group $G$. 
An $G$-equivariant \textbf{embedding of $G$-equivariant Kuranishi charts}
\[ 
\Vec{\alpha}=(\alpha, \widehat{\alpha}): (U,\mathcal{E},\psi,s)\to (U',\mathcal{E}',\psi',s')
\]
is a $G$-equivariant embedding of the orbibundles $\left (U\xrightarrow{\alpha}U', \mathcal{E}\xrightarrow{\widehat{\alpha}}\mathcal{E}'\right)$ satisfying the following. 
   \begin{enumerate}[i)] 
        \item $\widehat{\alpha}\circ s = s' \circ \alpha$. 
        \item $\psi'\circ \restr{\alpha }{s^{-1}(0)}= \psi$. 
        \item $\forall x\in s^{-1}(0)$, 
        the derivative $D_{\alpha(x)}s'$ induces an isomorphism 
        \[ \frac{T_{\alpha 
    (x)} U'}{(D_x\alpha)(T_x U)} \cong \frac{\mathcal{E}'_{\alpha(x)}}{\widehat{\alpha}(\mathcal{E}_x)}
    \] 
    induced by  
    \[
    \xymatrix{
    \mathcal{E}_x \ar[r]^{\widehat{\alpha}} & \mathcal{E}_{\alpha(x)}'\\
    T_xU \ar[r]_{D_x\alpha} \ar[u]^{D_xs} & T_{\alpha(x)}U'
    \ar[u]_{D_{\alpha(x)}s'} } . 
    \]
    \end{enumerate}
\end{defn}

\begin{defn}[$G$-equivariant Kuranishi structure]
\label{Gkur structure}
Let $\kur$ be a separable metrizable topological space with a topological action by a compact Lie group $G$. 
A \textbf{$G$-equivariant Kuranishi structure} 
\begin{equation}
\label{Gkur structure notation}
\widehat{\mathcal{U}} =  \left(\set{ \mathcal{U}_p }{ p\in \kur},   
\set{\vec{\alpha}_{pq}}{ p\in \kur, q\in \image \psi_p } \right)  
\end{equation}
on $\kur$
consists of the following data. 
\begin{enumerate}[1)]
        \item \label{Kuranishi neighborhood of p} 
        $\forall p\in \kur$, 
        $\mathcal{U}_p = (U_p, \E_p, \psi_p, s_p)$ is a $G$-equivariant Kuranishi chart on $\mathcal{M}$ such that there exists a unique $o_p\in U_p$ satisfying $\psi_p(o_p)=p$. 
    \item $\forall p \in \kur, \forall g\in G$, $\mathcal{U}_p = \mathcal{U}_{gp}$. 
    \item 
\label{coordinate change}
$\forall p\in \kur$, $\forall q\in \image \psi_p$, 
we have a $G$-equivariant embedding of $G$-equivariant Kuranishi charts $\vec{\alpha}_{pq}=\left( U_{pq}\xrightarrow{\alpha_{pq}} U_p,  \restr{\E_{q}}{U_{pq}}\xrightarrow{\widehat{\alpha}_{pq}}\E_p\right)$ from the restriction of $\mathcal{U}_q$ to a $G$-invariant open subset $U_{pq}$ of $U_q$ with $q\in \psi_q(s_q^{-1}(0)\cap U_{pq})$.  
In particular, it satisfies the following. 
      \begin{enumerate}[a)]
        \item The following diagrams commute. 
        \[
        \xymatrix{
        \restr{\E_{q}}{U_{pq}}
        \ar[r]^{\widehat{\alpha}_{pq}}
        \ar[d]_{\pi_q}
        &  
        \E_p \ar[d]^{\pi_p}\\
        U_{pq}
        \ar[r]_{\alpha_{pq}}
        &  
        U_p
        }, 
        \;
        \xymatrix{
        \restr{\E_{q}}{U_{pq}}
        \ar[r]^{\widehat{\alpha}_{pq}} 
        &  
        \E_p \\
        U_{pq}
        \ar[u]^{s_q}
        \ar[r]_{\alpha_{pq}}
        &  
        U_p
        \ar[u]_{s_p}
        },
        \; 
        \xymatrix{ 
        s_q^{-1}(0)\cap U_{pq} \ar[r]^{\alpha_{pq}}
        \ar[dr]_{\psi_q}
        &  
       s_p^{-1}(0)\cap U_p\ar[d]^{\psi_p}
       \\
        &   \mathcal{M}
        } 
        \]
        \item If $x\in s_q^{-1}(0)\cap U_{pq}$ and $\alpha_{pq}(x)=y$, then $D_{y}s_p$ induces an isomorphism 
        \begin{equation}\label{existence of tangent bundle}         
        \end{equation}   
    \end{enumerate}
Such an embedding $\vec{\alpha}_{pq}$ is called a \textbf{$G$-equivariant Kuranishi coordinate change}. 
    \item For any $p\in \kur, q\in \image \psi_p$ and any $g, g'\in G$, we have $\Vec{\alpha}_{pq} = \Vec{\alpha}_{(gp)(g'q)}$.
    \item If $r\in \psi_q(s_q^{-1}(0)\cap U_{pq}) $ and $U_{pqr} = \alpha_{qr}^{-1}(U_{pq})\cap U_{qr}$, then the cocycle condition is satisfied in the following sense: 
        \[ \alpha_{pr}( x ) = \alpha_{pq} \circ\alpha_{qr} (x),  \qquad \forall x\in U_{pqr}, \]
        and 
        \[   \widehat{\alpha}_{pr}(v) 
        =  \widehat{\alpha}_{pq}\circ \widehat{\alpha}_{qr}(v),  \qquad \forall x\in U_{pqr}, \quad \forall v\in (\E_r)_x.   \]
\end{enumerate}  
A topological space $\mathcal{M}$ which satisfies the conditions above and has a $G$-equivariant Kuranishi structure is called a \textbf{$G$-equivariant Kuranishi space}. 
\end{defn}

\begin{defn}[$G$-equivariant good coordinate system]
   \label{GC}
    Let $\mathcal{M}$ be a separable metrizable topological space.
    A \textbf{$G$-equivariant good coordinate system} 
    \begin{equation}
    \label{good coordinate system notation}
     \widetriangle{\mathcal{U}} = \left( (\Pindex, \leq ) 
    \, , \, 
    \left\{ \mathcal{U}_{\p}\mid \p\in \Pindex \right\}
    \, , \, 
    \set{\Vec{\alpha}_{\pq}}{ \p, \q\in \Pindex, \q \leq \p } \right)   
    \end{equation}
 consists of 
    \begin{enumerate}
      [1)]
        \item \label{GC index set}
        a finite partially ordered set $(\Pindex, \leq )$; 
        \item  
        $\forall \p\in \Pindex$, a $G$-equivariant Kuranishi chart $\mathcal{U}_{\p}= (U_{\p}, \mathcal{E}_{\p }, \psi_{\p}, s_{\p})$ on $\mathcal{M}$  so that $\kur \subset \bigcup\limits_{\p \in \Pindex} \image \psi_{\p} $;  and 
        \item \label{GC coordinate change}
        $\forall \p,\q \in \Pindex$ with $\q\leq \p$,  an $G$-equivariant embedding $ (\alpha_{\pq},\widehat{\alpha}_{\pq})$ of $G$-equivariant Kuranishi charts 
        $\vec{\alpha}_{\pq}=\left( U_{\pq}\xrightarrow{\alpha_{\pq}} U_{\p},  \restr{\E_{\q}}{U_{\pq}}\xrightarrow{\widehat{\alpha}_{\pq}}\E_{\p}\right)$ from the restriction of $\mathcal{U}_{\q}$ to a $G$-invariant open subset $U_{\pq}$ of $U_{\q}$, called a \textbf{$G$-equivariant good coordinate change} from $\q$ to $\p$, satisfying
        \[ \image \psi_{\pq} = \image \psi_{\p}\cap \image \psi_{\q},  \]
\end{enumerate}
such that the following holds. 
\begin{enumerate}[i)]
    \item $\vec{\alpha}_{\mathfrak{pp}} = \left(\restr{\operatorname{id}}{U_{\p}}, \restr{\operatorname{id}}{\mathcal{E}_{\p}}\right)$.
    \item If $\image \psi_{\p} \cap \image \psi_{\q} \ne \emptyset$, then either $ \p \leq \q$ or $\q\leq \p$.  
    \item If $\mathfrak{r}\leq \q \leq \p$ and $U_{\mathfrak{pqr}} = \alpha_{\mathfrak{qr}}^{-1}(U_{\pq})\cap U_{\mathfrak{qr}}$, then  
        \[ \alpha_{\mathfrak{pr}}( x ) = \alpha_{\pq} \circ \alpha_{\mathfrak{qr}} (x), \qquad \forall x\in U_{\mathfrak{pqr}},\]
        and, 
        \[  \widehat{\alpha}_{\mathfrak{pr}}(v) 
        =  \widehat{\alpha}_{\pq} \circ \widehat{\alpha}_{\mathfrak{qr}} (v),  \qquad \forall x\in U_{\mathfrak{pqr}}, \quad \forall v\in (\E_{\mathfrak{r}})_x . \]
\end{enumerate} 
\end{defn}

\begin{defn}[Support system/support pair]
\label{support system} \label{support pair}
 Let $\left(\kur,  \widetriangle{\mathcal{U}}\right)$ be a space with a $G$-equivariant good coordinate system
\[ \widetriangle{\mathcal{U}} = \left( (\Pindex, \leq ) 
    \, , \, 
    \left\{ \mathcal{U}_{\p}\mid \p\in \Pindex \right\}
    \, , \, 
    \set{\Vec{\alpha}_{\pq}}{ \p, \q\in \Pindex, \q \leq \p } \right). 
\]
A \textbf{$G$-equivariant support system} $\widetriangle{\mathcal{K}} = \{ \mathcal{K}_{\p}\mid \p \in \Pindex\}$ on $\left(\kur,  \widetriangle{\mathcal{U}}\right)$ is a collection of sets satisfying the following. 
\begin{itemize}
    \item For each $\p\in \Pindex$,  $ \mathcal{K}_{\p} \subset \mathcal{U}_{\p}$ is a nonempty $G$-invariant compact subset, which is the closure of some open set 
    $\mathring{\mathcal{K}}_{\p} \subset \mathcal{U}_{\p}$. 
    \item $\bigcup\limits_{\p\in \Pindex} \psi_{\p} \left(  \mathring{\mathcal{K}}_{\p} \cap s_{\p}^{-1}(0)\right) = \kur$.  
\end{itemize} 
For any $G$-equivariant support system $\widetriangle{\mathcal{K}}$, 
we define 
\begin{equation}\label{heterodimensional compactum}
 |\mathcal{K}| = \lrp{\bigsqcup_{\p\in \Pindex} \mathcal{K}_{\p}}/\sim ,     
\end{equation}
where, for each $x\in \mathcal{K}_{\p}, y\in \mathcal{K}_{\q}$, we say $ x\sim y $ if and only if either $y=\alpha_{\mathfrak{qp}}(x)$ or 
$x= \alpha_{\mathfrak{pq}}(y)$. 
The $G$-action on the charts induces a $G$-action on $|\mathcal{K}|$. 

A \textbf{$G$-equivariant support pair} $\lrp{\widetriangle{\mathcal{K}}, \widetriangle{\mathcal{K}}^{++} }$ on $\left(\kur,  \widetriangle{\mathcal{U}}\right)$ consists of  $G$-equivariant support systems $\widetriangle{\mathcal{K}} = \{ \mathcal{K}_{\p}\mid \p \in \Pindex\}$, $\widetriangle{\mathcal{K}}^{++} = \{ \mathcal{K}_{\p}^{++} \mid \p \in \Pindex\}$ such that for all $\p\in \Pindex$ we have  $\mathcal{K}_{\p}\subset \mathring{\mathcal{K}}_{\p}^{++}$. We write $\widetriangle{\mathcal{K}}< \widetriangle{\mathcal{K}}^{++}$ for such a pair. 

Given a $G$-equivariant support pair $\lrp{\widetriangle{\mathcal{K}}, \widetriangle{\mathcal{K}}^{++}}$, a $G$-invariant metric on $|\mathcal{K}^{++}|$, and $\delta >0$,  
we can define another support system 
\[ \widetriangle{\mathcal{K}}(2\delta) = \{\mathcal{K}_{\p}(2\delta)\mid \p \in \Pindex \}, \]
where 
\[\mathcal{K}_{\p}(2\delta) = \{x\in \mathcal{K}_{\p}^{++} \mid d(x, \mathcal{K}_{\p})<2\delta \}. \]
\end{defn}

\begin{defn}[KG-embedding]\label{KG embedding}
    A  $G$-equivariant \textbf{strict KG-embedding} 
    \[ \set{ \Vec{\alpha}_{\mathfrak{p}p}=
    (\alpha_{\mathfrak{p}p}, \widehat{\alpha}_{\mathfrak{p}p}): \mathcal{U}_{p}\to  \mathcal{U}_{\p} }{ (p,\p ) \in \kur \times \Pindex, p\in \image \psi_{\p} }:  \widehat{\mathcal{U}} \longrightarrow \widetriangle{\mathcal{U}} \] 
    from a Kuranishi structure $\widehat{\mathcal{U}}$ on $\kur$ to a good coordinate system $\widetriangle{\mathcal{U}}$ on $\kur$ consists of one $G$-equivariant embedding of Kuranishi charts for each $(p,\p ) \in \kur \times \Pindex, p\in \image \psi_{\p}$ such that 
    the following holds. If $\p, \q\in \Pindex, \q\leq \p$ and $ p \in \image \psi_{\p}$, $q\in \image \psi_p \cap \image \psi_{\q}(U_{\pq}\cap s_{\q}^{-1}(0)) $, then on $\restr{\mathcal{U}_q}{U_{pq}\cap  \alpha _{\mathfrak{\q}q}^{-1}(U_{\mathfrak{pq}})}$ we have 
        \[ \vec{\alpha}_{\mathfrak{p}p} \circ  \vec{\alpha}_{pq} = \vec{\alpha}_{\mathfrak{pq}} \circ \vec{\alpha}_{\mathfrak{q}q}.  \]
    An \textbf{open substructure} of a $G$-equivariant Kuranishi structure $\widehat{\mathcal{U}}$ is a $G$-equivariant Kuranishi structure $\widehat{\mathcal{U}}'$ whose Kuranishi charts and coordinate changes are restrictions of those in $\widehat{\mathcal{U}}$ to $G$-invariant open subsets $U_p'$ of Kuranishi neighborhoods $U_p$. 
    
    A $G$-equivariant \textbf{KG-embedding} $\widehat{\mathcal{U}} \longrightarrow \widetriangle{\mathcal{U}}$ is a $G$-equivariant strict KG-embedding $\widehat{\mathcal{U}}_0 \longrightarrow \widetriangle{\mathcal{U}}$ from an open substructure $\widehat{\mathcal{U}}_0$ of the Kuranishi structure $\widehat{\mathcal{U}}$. 
\end{defn}

\begin{lem}
Suppose that $\kur$ is a space with a $G$-equivariant Kuranishi structure with corners, where the $G$-action on each Kuranishi chart is free. 
Then $\kur/G$ has a Kuranishi structure with corners. 

Similarly, if $\kur$ is a $G$-equivariant good coordinate system with corners, where the $G$-action on each Kuranishi chart is free, 
then $\kur/G$ has a good coordinate system with corners. 
\end{lem}
\begin{proof}
Suppose the $G$-equivariant Kuranishi structure on $\kur$ is given by 
\begin{equation*}
\widehat{\mathcal{U}} =  \left(\set{ \mathcal{U}_p }{ p\in \kur },   
\set{\vec{\alpha} _{pq}}{ p\in \kur, q\in \image \psi_p } \right), 
\end{equation*}
where the $G$-action on each chart $\mathcal{U}_p = (U_p, \E_p, \psi_p, s_p)$ is free. 

By $G$-equivariance and the freeness of the $G$-action, we obtain 
Kuranishi charts of the form $\mathcal{U}_p/G= (U_p/G, \E_p/G, [\psi_p], [s_p])$. Suppose $\vec{\alpha} _{pq}$ is given by 
a pair of $G$-equivariant orbifold embeddings  
\[ \alpha_{pq}: U_{pq}\to U_p, \quad \widehat{\alpha}_{pq}: \restr{\E_{q}}{U_{pq}}\to \E_p. \]
Then they induce maps  $[\alpha_{pq}]: U_{pq}/G \to U_p/G$ and 
$[\widehat{\alpha}_{pq}]: \restr{(\E_{q}/G)}{U_{pq}/G}\to \E_p/G$, which together define a 
Kuranishi coordinate change $[\vec{\alpha} _{pq}]$ from $\widehat{\mathcal{U}}_q/G$ to $\widehat{\mathcal{U}}_p/G$. 

The data $\widehat{\mathcal{U}}/G =  \left(\set{ \mathcal{U}_p/G }{ p\in \kur/G },   
\set{[\vec{\alpha} _{pq}]}{ p\in \kur/G, q\in \image [\psi_p] } \right)$ satisfy the definition of a Kuranishi structure on $\kur/G$.  
\end{proof}

\begin{defn}[Dimension/Orientation]

Let $\kur$ be a space with either a Kuranishi structure or a good coordinate system. 

    \begin{enumerate}[i)]
    \item  \label{dimension of a Kuranishi space/good coordinate system}
    We define the dimension of a Kuranishi chart of the form $\mathcal{U} = (U,\mathcal{E},\psi,s)$ by 
    \[\dim \mathcal{U} = \dim U - \rank \mathcal{E}.  \]
    We require the dimension of the Kuranishi charts in the Kuranishi structure (resp. good coordinate system) to be the same and define this common dimension to be the \textbf{dimension} of $\kur$. 
    \item \label{orientation of a Kuranishi space/good coordinate system}
    An orientation 
   of a Kuranishi chart $\mathcal{U} = (U,\mathcal{E},\psi,s)$ is given by an orientation on $U$ and an orientation on $\E$. 
    An \textbf{orientation} on $\kur$ is a choice of an orientation for each Kuranishi chart of the Kuranishi structure (resp. good coordinate system) such that the coordinate change maps are orientation-preserving. 
   \end{enumerate}
   
\end{defn}

\begin{defn}[$G$-equivariant strongly smooth map: orbifold $\to $ manifold]
\label{strongly smooth map on a chart}
Let $L$ be a smooth manifold with a $G$-action and $U$ be a smooth effective orbifold with a smooth $G$-action.     
A $G$-equivariant continuous map $g: U \to L$ is a \textbf{strongly smooth map} if $g \circ \varphi : V\to L$ is smooth for all orbifold charts $(V,\Gamma, \varphi)$ in the orbifold atlas of $U$. 
\end{defn}

\begin{defn}[$G$-equivariant strongly smooth map: Kuranishi $\to $ manifold]
\label{strongly smooth}
Let $L$ be a smooth manifold with a $G$-action and $\left(\mathcal{M},\widehat{\mathcal{U}}\right)$ be a $G$-equivariant Kuranishi space with Kuranishi structure
\[ 
\widehat{\mathcal{U}} =  \left(\set{ \mathcal{U}_p }{ p\in \mathcal{M}},   
\set{\vec{\alpha} _{pq}}{ p\in \mathcal{M}, q\in \image \psi_p } \right) 
\]
A \textbf{$G$-equivariant strongly smooth} map $\widehat{f}: (\kur,\widehat{U})\to L$ 
is a collection 
\[ \{f_{p}: U_{p}\to L \mid p\in \mathcal{M} \}\]
of $G$-equivariant strongly smooth maps satisfying the following. For all $p\in \mathcal{M}, q\in \image \psi_p$, the compatibility condition 
$f_{p} \circ  \alpha_{pq}
=  f_q|_{U_{pq}}$ is satisfied. 
Define the map associated with the strongly smooth map $\widehat{f}$ by 
\[ f: \mathcal{M}\to L, \quad f(p) = f_p(o_p) \quad  \forall p\in \mathcal{M}, \]
where $o_p$ is as in Definition \ref{Gkur structure} \ref{Kuranishi neighborhood of p}. 
We can define a $G$-equivariant strongly smooth map $\widetriangle{f} = \{f_{\p}: U_{\p}\to L \mid \p \in \Pindex\}$ from a space with good coordinate system to a manifold in a similar way. 
\end{defn}

\begin{defn}[$G$-equivariant differential forms on a $G$-equivariant Kuranishi space]\label{Gdiff forms}
Let $\mathcal{M}$ be a space with a $G$-equivariant Kuranishi structure as in Definition \ref{Gkur structure}.  
A \textbf{$G$-equivariant differential $k$-form} on $\left( \kur,\widehat{\mathcal{U}}\right)$ is given by a collection of differential forms 
\begin{equation}
 \widehat{\eta}= \set{ \eta_p \in \Omega_G^k(U_p) }{ p\in \kur }       
\end{equation}
such that 
\[ 
\restr{(\alpha_{pq}^G)^* \lrp{\eta_p
 }}{U_{pq}}
= \eta_q 
|_{U_{pq}}, \qquad \forall p \in \kur, \quad \forall q\in \image \psi_p, \]
where 
\[ (\alpha_{pq}^G)^*: \Omega_G^k(U_{p})  \to \Omega_G^k(U_{q}) \]
denotes the $G$-equivariant pullback via $U_q \xrightarrow{\text{restriction}} U_{pq}\xrightarrow{\alpha_{pq}}U_p$. 
We denote the set of $G$-equivariant differential $k$-forms on a $G$-equivariant Kuranishi space $\kur$ 
by $\Omega_G^k\left( \kur,\widehat{\mathcal{U}}\right)$ and denote $\Omega_G\left( \kur,\widehat{\mathcal{U}}\right) = \bigoplus\limits_{k\in \N} \Omega_G^k\left( \kur,\widehat{\mathcal{U}}\right)$.  
\end{defn}

\begin{defn}[$G$-equivariant differential forms on a good coordinate system]
  Let $\mathcal{M}$ be a space with a $G$-equivariant good coordinate system as in \autoref{good coordinate system notation}.  
 Let $\widetriangle{\mathcal{K}} = \{ \mathcal{K}_{\p}\mid \p \in \Pindex\}$ be a support system on $\left( \kur,\widetriangle{\mathcal{U}}\right)$. 
A \textbf{$G$-equivariant differential $k$-form} $\widetriangle{\eta}$ on $\left( \kur,\widetriangle{\mathcal{U}}\right)$ assigns a $G$-equivariant differential $k$-form $\eta_{\p}$ on $K_{\p}$ for each $\p\in \Pindex$
such that the following holds on a $G$-invariant open neighborhood of $\alpha_{\pq}^{-1}(K_{\p})\cap K_{\q}$: 
\[ 
(\alpha_{\pq}^G)^* \eta_{\p} 
= \eta_{\q} , \qquad \forall \p \in \kur, \quad \forall \q\in \image \psi_{\p}.\]
 
We denote the set of $G$-equivariant differential $k$-forms on a $G$-equivariant Kuranishi space $\kur$ 
by $\Omega_G^k\left( \kur,\widetriangle{\mathcal{U}}\right)$ and denote $\Omega_G\left( \kur,\widetriangle{\mathcal{U}}\right) = \bigoplus\limits_{k\in \N} \Omega_G^k\left( \kur,\widetriangle{\mathcal{U}}\right)$.    
\end{defn}

\begin{defn}[$G$-equivariant pullback map]
Let $\kur$ be a space with a $G$-equivariant Kuranishi structure $\widehat{U}$ as in \autoref{Gkur structure notation}.  
Let $\widehat{f}= \{f_{p}: U_{p}\to L \mid p\in \mathcal{M} \} : (\mathcal{M},\widehat{U})\to L$ be a $G$-equivariant strongly smooth map and 
\[ \eta= \set{ \eta_p \in \Omega_G^k(U_p) }{ p\in \kur }   \]
be a $G$-equivariant differential $k$-form on $(\kur,\widehat{U})$. 
Then the \textbf{$G$-equivariant pullback} of $\eta$ via $\widehat{f}$ is given by 
\[ \widehat{f}^*\widehat{\eta}= \set{ f_{p}^*\eta_p \in \Omega_G^k(U_p) }{ p\in \kur }. \]
We may also denote it by $f^*\widehat{\eta}$. 
Similarly, we can define the $G$-equivariant pullback $\widetriangle{f}^*\widetriangle{\eta}$ of a differential form $\widetriangle{\eta}$ on a good coordinate system via a $G$-equivariant strongly smooth map $\widetriangle{f}$. 
\end{defn}

\begin{lem}
\label{Existence of KG-embedding}
Let $\lrp{\kur,\widehat{\mathcal{U}}}$ be a space with a $G$-equivariant Kuranishi structure
\[
\widehat{\mathcal{U}} =  \left(\set{ \mathcal{U}_p }{ p\in \kur},   
\set{\vec{\alpha}_{pq}}{ p\in \kur, q\in \image \psi_p } \right).
\]   
Then there exists a $G$-equivariant good coordinate system $\widetriangle{\mathcal{U}}$ on $\kur$ and a $G$-equivariant KG-embedding $\widehat{\mathcal{U}} \longrightarrow \widetriangle{\mathcal{U}}$, given by a $G$-equivariant strict KG-embedding $ \Phi: \widehat{\mathcal{U}}_0 \longrightarrow \widetriangle{\mathcal{U}}$ from an open substructure $\widehat{\mathcal{U}}_0$ of $\widehat{\mathcal{U}}$. 

Moreover, the following holds. 
\begin{enumerate}[i)]
    \item Let $\widehat{h}$ be a $G$-equivariant differential form on $\lrp{\kur,\widehat{\mathcal{U}}}$. 
    Then there exists a $G$-equivariant differential form $\widetriangle{h}$ on $\lrp{\kur,\widetriangle{\mathcal{U}}}$ such that $\Phi_G^*(\widetriangle{h}) = \restr{\widehat{h}}{\widehat{\mathcal{U}}_0}$. 
    \item Let $\widehat{\mathcal{S}}$ be a $G$-equivariant CF-perturbation on $\lrp{\kur,\widehat{\mathcal{U}}}$. Then there exists a $G$-equivariant CF-perturbation $\widetriangle{\mathcal{S}}$ on $\lrp{ \kur,\widetriangle{\mathcal{U}}}$ 
        such that $\restr{\widehat{\mathcal{S}}}{\widehat{\mathcal{U}}_0}$, $\widetriangle{\mathcal{S}}$ are compatible with $\Phi$ and the following holds. 
    \begin{enumerate}[a)]
        \item If $\widehat{\mathcal{S}}$ is transverse to $0$, then $\widetriangle{\mathcal{S}}$ is also transverse to $0$. 
        \item If $\widehat{f}$ is strongly submersive with respect to $\widehat{\mathcal{S}}$, then $\widetriangle{f}$ is strongly submersive with respect to $\widetriangle{\mathcal{S}}$. 
        \item If $\widehat{f}$ is weakly transverse to $g: M\to N$ with respect to $\widehat{\mathcal{S}}$, then $\widetriangle{f}$ is strongly transverse to $g: M\to N$ with respect to  $\widetriangle{\mathcal{S}}$. 
    \end{enumerate}
\end{enumerate}
\end{lem}
The proof of Lemma \ref{Existence of KG-embedding} is similar to that of \cite{kurbook} 
Theorem 3.35 and Lemma 9.10. 
\begin{proof}[Proof sketch]
For each $d\in \N$, let
$S_{d}\kur =\{ p\in \kur \mid \dim U_p \geq d\}$. The proof is based on a downward induction on $d$. 
Suppose a $G$-equivariant good coordinate system 
\[   \widetriangle{\mathcal{U}_{d+1}} = \left( (\Pindex, \leq ) 
    \, , \, 
    \left\{ \mathcal{U}_{\p}\mid \p\in \Pindex \right\}
    \, , \, 
    \set{\Vec{\alpha}_{\pq}}{ \p, \q\in \Pindex, \q \leq \p } \right) \]
which covers $S_{d+1}\kur$ is constructed. 
Pick a collection $\{ K_{\p} \mid \p \in \Pindex\}$ of $G$-invariant compact subsets $K_{\p}\subset U_{\p}$ such that we have an open neighborhood of $S_{d+1}\kur$:
\[ S_{d+1}\kur \subset   \bigcup_{\p} \psi_{\p}\lrp{s_{\p}^{-1}(0) \cap \interior K_{\p} }.\] 
Let 
\[
B = S_d\kur \setminus \bigcup_{\p\in \Pindex} \psi_{\p}\lrp{s_{\p}^{-1}(0) \cap \interior K_{\p} } \subset S_d\kur \setminus S_{d+1}\kur. 
\]
Also pick $x_1,\ldots, x_n \in S_d\kur \setminus S_{d+1}\kur $ and $\{ K_{i} \mid 1\leq i \leq n\}$ such that $K_i\subset U_{x_i}$ are $G$-invariant subsets and 
\[ \bigcup_{i=1}^n \psi_{x_i}\lrp{s_{x_i}^{-1}(0) \cap \interior K_i } \supset B. \]
Then we can construct, as in the proof of \cite{kurbook} Theorem 3.35 and that of \cite{kurbook} Lemma 9.10, a $G$-equivariant good coordinate system that covers 
\[ \bigcup_{\p\in \Pindex} \psi_{\p}\lrp{s_{\p}^{-1}(0) \cap \interior K_{\p} } \cup  \bigcup_{i=1}^n \psi_{x_i}\lrp{s_{x_i}^{-1}(0) } \]
and satisfies the properties by induction on $n$. 
\end{proof}

\subsection{\texorpdfstring{$G$}{G}-equivariant CF-perturbations}

\begin{defn}
[$G$-equivariant CF-perturbation representative on a $G$-invariant subset]
\label{CF-perturbation on a suborbifold}
Let $\mathcal{U}=(U, \E , \psi, s)$ be a $G$-equivariant Kuranishi chart and $U_{\rindex}\subset U$ be a $G$-invariant open subset of $U$. 
A \textbf{$G$-equivariant CF-perturbation representative}\footnote{"CF'' stands for "continuous family''.} of $\mathcal{U}$ on $U_{\rindex}$ is a continuous family of data 
\[  \mathcal{S}_{\rindex} = \{ \mathcal{S}_{\rindex}^{\epsilon}= (W_{\rindex}\xrightarrow{\nu_{\rindex}}U_{\rindex},\tau_{\rindex}, \s^{\epsilon}_{\rindex})\mid \epsilon \in (0,1]\}\] such that the following holds.  
\begin{enumerate}[i)]
    \item $W_{\rindex}$ is an effective orbifold with a smooth $G$-action.\footnote{Note that we require $W_{\rindex}$ to be the total space of a $G$-orbibundle, unlike in the case of ordinary Kuranishi structures, $W$ denotes an open subset of some vector space. }
    \item $\nu_{\rindex}: W_{\rindex}\to U_{\rindex}$ is a smooth oriented $G$-equivariant orbibundle. 
     \item $\tau_{\rindex} \in \Omega_G(W_{\rindex})$ is a $G$-equivariant Thom form of $ \nu_{\rindex} : W_{\rindex}\to U_{\rindex}$.\footnote{The construction of a $G$-equivariant Thom form can be found in \cite{GS} Chapter 10. }

    \item Let $\nu_{\rindex}^*(\restr{\E}{U_{\rindex}})\to W_{\rindex}$ be the pullback bundle of $\restr{\E}{U_{\rindex}} \to U_{\rindex}$ via $\nu_{\rindex}$ and let 
    $pr_2: \nu_{\rindex}^*(\restr{\E}{U_{\rindex}}) \to \restr{\E}{U_{\rindex}}$ be the projection map.  
    $\forall 0< \epsilon\leq 1$, let $\tilde{\s}_{\rindex}^{\epsilon}: W_{\rindex} \to \nu_{\rindex}^*(\restr{\E}{U_{\rindex}})$ be a section of the bundle $\nu_{\rindex}^*(\restr{\E}{U_{\rindex}}) \to W_{\rindex}$ satisfying the following. 
    \begin{enumerate}[a)]
        \item $\s_{\rindex}^{\epsilon} = pr_2\circ \tilde{\s}_{\rindex}^{\epsilon} : W_{\rindex}\to \restr{\E}{U_{\rindex}}$ is a $G$-equivariant bundle map 
        and the family 
        $\{\s_{\rindex}^{\epsilon}\}_{\epsilon \in (0,1]}$ depends smoothly on $\epsilon$. 
        \item  Moreover, $\lim\limits_{\epsilon\to 0}\s_{\rindex}^{\epsilon} = s\circ \nu_{\rindex}$
    in the compact $C^1$-topology. 
    \[
    \xymatrix{
    \nu_{\rindex}^* (\restr{\E}{U_{\rindex}})
    \ar[d]_{pr_1}
    \ar[r]^{pr_2}
    & 
     \restr{\E}{U_{\rindex}}
     \ar[d]^{\restr{\pi}{ U_{\rindex}}}
     \\
    W_{\rindex}
    \ar[r]_{ \nu_{\rindex}}
    & 
    U_{\rindex}
    }
    \]
    \end{enumerate} 
\end{enumerate}
\end{defn}

\begin{defn}[Equivalent $G$-equivariant CF-perturbation representatives on subsets]
\label{equivalent local CF-perturbations}
Let $\mathcal{U}=(U, \E , \psi, s)$ be a $G$-equivariant Kuranishi chart and $U_{\rindex}\subset U$ be a $G$-invariant open subset of $U$. 
Let 
\[  
\mathcal{S}_{\rindex}  
= 
\set{ \mathcal{S}_{\rindex} ^{\epsilon}= (W_{\rindex}\xrightarrow{\nu_{\rindex}}U_{\rindex}, \s_{\rindex}^{\epsilon}, \tau_{\rindex})
}
{\epsilon \in (0,1]}.
\]
and 
\[  
\mathcal{S}_{\rindex}^i 
= 
\set{ \mathcal{S}_i ^{\epsilon}= (W_i\xrightarrow{\nu_{i}}U_{\rindex}, \s_i^{\epsilon}, \tau_i)}
{\epsilon \in (0,1] }
\quad \forall i\in \{1,2\}
\]
be $G$-equivariant CF-perturbation representatives of $\mathcal{U}=(U, \E , \psi, s)$ on $U_{\rindex}$.  
\begin{itemize}
    \item \label{projection}
    $ \mathcal{S}_{\rindex}^i $ is said to be a \textbf{projection} of $\mathcal{S}$ if 
    there exists a $G$-equivariant bundle map 
    $
    P:  
    W_{\rindex} \to W_i$ which fiberwise is a surjective linear map
    such that the following holds. 
    \begin{enumerate}[a)]
        \item $P_{G!}(\tau_{\rindex}) = \tau_i $.  
        \item  For each $\epsilon \in (0,1]$, 
       $ \s_i^{\epsilon} \circ P= \s_{\mathfrak{r}}^{\epsilon}
       $.  
    \end{enumerate}
    \item $\mathcal{S}_{\rindex}^1$ is said to be \textbf{equivalent} to $\mathcal{S}_{\rindex}^2$ if there exist $G$-equivariant CF-perturbation representatives $\Tilde{S}_j$ of $\mathcal{U}$ on $U_{\mathfrak{r}}$, 
    $j = 0, \ldots, 2N$ such that 
    \begin{enumerate}[a)]
        \item $\forall 0\leq k \leq N-1$, $\Tilde{\mathcal{S}}_{2k}$ and $\Tilde{\mathcal{S}}_{2k+2}$ are both projections of $\Tilde{\mathcal{S}}_{2k+1}$, and 
        \item $\Tilde{\mathcal{S}}_{0} = \mathcal{S}_{\rindex}^1$ and 
        $\Tilde{\mathcal{S}}_{2N} = \mathcal{S}_{\rindex}^2$. 
    \end{enumerate}
\end{itemize}
\end{defn}

\begin{defn}[$G$-equivariant CF-perturbation representative on a Kuranishi chart]
A \textbf{$G$-equivariant CF-perturbation representative} on a $G$-equivariant Kuranishi chart $\mathcal{U}=(U, \E , \psi, s)$ on $\kur$ is a collection of data
$ \{\mathcal{S}_{\rindex}  \mid \rindex \in \mathfrak{R}  \}$, where 
\[  \mathcal{S}_{\rindex} = \{ \mathcal{S}_{\rindex}^{\epsilon}= (W_{\rindex}\xrightarrow{\nu_{\rindex}}U_{\rindex}, \s^{\epsilon}_{\rindex}, \tau_{\rindex})\mid \epsilon \in (0,1]\}, \]
such that the following holds.  
\begin{enumerate}[i)]
    \item $\forall \rindex \in \Rindex$, $\mathcal{S}_{\rindex}$ is a $G$-equivariant CF-perturbation representative of $ \mathcal{U}$ on a $G$-invariant open subset $U_{\rindex}$ of $U$. 
    \item $\bigcup\limits_{\rindex \in \Rindex} U_{\rindex} = U$. 
    \item If $x\in U_{\rindex_1} \cap  U_{\rindex_2}$ for some $\rindex_1, \rindex_2 \in \Rindex$, 
    then there exists a $G$-invariant open subset 
    $U_{\rindex_{12}} \subset U_{\rindex_1} \cap  U_{\rindex_2}$ 
    such that the following holds. 
    For each $k\in \{1,2\}$, let $i_{k}: U_{\rindex_{12}} \hookrightarrow U_{\rindex_k}$ be the $G$-equivariant inclusion map 
    and let 
     \[ i_k^* \mathcal{S}_{\rindex_k}  = \set{\lrp{ i_k^*W_{\rindex_k} \to U_{\rindex_{12}}, i_k^*\tau_{\rindex_k}, i_k^*\s_{\rindex_k}^{\epsilon}}}{\epsilon \in (0,1] }  \]
     be the restriction of $\mathcal{S}_{\rindex_k} $ to $U_{\rindex_{12}}$. 
    Then we require the CF-perturbation representatives $ i_1^* \mathcal{S}_{\rindex_1},  i_2^* \mathcal{S}_{\rindex_2}$ of $\mathcal{U}$ on $U_{\rindex_{12}}$ to be equivalent as in Definition \ref{equivalent local CF-perturbations}. 
\end{enumerate}
\end{defn} 

\begin{defn}[Equivalent $G$-equivariant CF-perturbations on a Kuranishi chart]
\label{equivalent CF-perturbations on a chart}
Let $\mathcal{S}^{(1)} = \{ S_{\rindex}  \mid \rindex \in \Rindex \}$ and $\mathcal{S}^{(2)} = \{ S_{\mathfrak{j} }  \mid \mathfrak{j}\in \mathfrak{J} \} $ be $G$-equivariant CF-perturbation representatives on a $G$-equivariant Kuranishi chart $\mathcal{U}=(U, \E , \psi, s)$. Then the $G$-equivariant CF-perturbations $\mathcal{S}^{(1)}, \mathcal{S}^{(2)}$ on $\mathcal{U}$ are said to be \textbf{equivalent} if, whenever $x\in U_{\mathfrak{r}}\cap U_{\mathfrak{j}}$ for some $\mathfrak{r} \in \Rindex , \mathfrak{j}\in \mathfrak{J}$, 
there exists a $G$-invariant suborbifold $U_{\mathfrak{rj}} \subset U$ such that $U_{\mathfrak{rj}}  \subset U_{\mathfrak{r}}\cap U_{\mathfrak{j}}$ and the following holds. 
Let $i_{\rindex}: U_{\mathfrak{rj}} \hookrightarrow U_{\rindex}$ and $i_{\mathfrak{j}}: U_{\mathfrak{rj}} \hookrightarrow U_{\mathfrak{j}}$ be the $G$-equivariant inclusion maps. 
Then we require $i_{\rindex}^* \mathcal{S}_{\mathfrak{r}}, i_{\mathfrak{j}}^* \mathcal{S}_{\mathfrak{j}}$ to be equivalent on $ U_{\mathfrak{rj}}$. 
\end{defn}

\begin{defn}[$G$-equivariant CF-perturbation on a Kuranishi chart]
\label{CF-perturbation on a chart}
Let $\mathcal{S}$ be a $G$-equivariant CF-perturbation representative on a $G$-equivariant Kuranishi chart $\mathcal{U}$. 
A \textbf{$G$-equivariant CF-perturbation} on $\mathcal{U}$ represented by $\mathcal{S}$ is the class 
\[ [\mathcal{S}] = \set{ \mathcal{S}' }
{ 
\begin{aligned}
& \mathcal{S}' \text{ is a CF-perturbation representative on } \mathcal{U}, \\    
& \mathcal{S}' \text{ is equivalent to }  \mathcal{S}
\end{aligned}
} \] of $G$-equivariant CF-perturbation representatives on $\mathcal{U}$ that are equivalent to $\mathcal{S}$. 
\end{defn}

\begin{defn}
Let  $\mathcal{U}=(U, \E , \psi, s)$ be a $G$-equivariant Kuranishi chart and let 
   \[ \mathcal{S}_{\rindex} = \{ \mathcal{S}_{\rindex}^{\epsilon} = (W_{\rindex} \xrightarrow{\nu_{\rindex}} U_{\rindex}, \tau_{\rindex} , \s_{\rindex}^{\epsilon})\mid \epsilon \in (0,1]\} \] 
   be a CF-perturbation on a $G$-invariant open subset $U_{\rindex} \subset U$. 
\begin{enumerate}[i)]
    \item \label{local CF-perturbation transverse to zero}
   $\mathcal{S}_{\rindex}$ is said to be \textbf{transverse to zero} if, 
    $\forall 0<\epsilon\leq 1$, the map $\restr{\s_{\rindex}^{\epsilon}}{W_{\rindex}^{\epsilon}}$ is transverse to the zero section on some $G$-invariant neighborhood
    $ W_{\rindex}^{\epsilon}\subset W_{\rindex}$ of the support of $\tau_{\rindex}$. 
     \item \label{strongly submersive with respect to a local CF-perturbation} 
    Let $L$ be a smooth manifold. A $G$-equivariant smooth map $f_{\rindex}: U_{\rindex}\to L$ is said to be \textbf{strongly submersive} with respect to $\mathcal{S}_{\rindex}$ if $\mathcal{S}_{\rindex}$ is transverse to zero and, 
    $\forall 0<\epsilon\leq 1$, 
    , the map 
    \[ \restr{f_{\rindex}\circ \nu_{\rindex}}{(\s_{\mathfrak{r}}^{\epsilon})^{-1}(0)}: (\s_{\mathfrak{r}}^{\epsilon})^{-1}(0) \to L \] 
    is a submersion on some $G$-invariant neighborhood
    $ W_{\rindex}^{\epsilon}\subset W_{\rindex}$ of the support of $\tau_{\rindex}$. 
    \item \label{strongly transverse to a smooth map with respect to a local CF-perturbation} 
    Let $f_{\rindex}: U_{\rindex}\to L$ be strongly submersive with respect to $\mathcal{S}_{\rindex}$ and $g: N\to L$ be a smooth manifold between manifolds. We say $f_{\rindex}$ is \textbf{strongly transverse to $g$} if $\mathcal{S}_{\rindex}$ is transverse to zero and, for any $\epsilon\in(0,1]$ and any $x\in (\s_{\mathfrak{r}}^{\epsilon})^{-1}(0)$, the map $\restr{f_{\rindex}\circ \nu_{\rindex}}{(\s_{\mathfrak{r}}^{\epsilon})^{-1}(0)}$ is transverse to $g$. 
\end{enumerate}
\end{defn}

\begin{defn} 
Let  $\mathcal{U}=(U, \E , \psi, s)$ be a $G$-equivariant Kuranishi chart. Let $[\mathcal{S}]$ be a $G$-equivariant CF-perturbation on $\mathcal{U}$, where $\mathcal{S} = \{\mathcal{S}_{\rindex}\mid \rindex \in \Rindex\}$ and 
    \[ \mathcal{S}_{\rindex} = \{ \mathcal{S}_{\rindex}^{\epsilon} = (W_{\rindex} \xrightarrow{\nu_{\rindex}} U_{\rindex}, \tau_{\rindex} , \s_{\rindex}^{\epsilon})\mid \epsilon \in (0,1]\}.  \]
\begin{enumerate}[i)]
    \item \label{CF-perturbation transverse to zero on a chart}
   $[\mathcal{S}]$ is said to be \textbf{transverse to zero} if $ \mathcal{S}_{\rindex} $ is transverse to zero for all $\rindex \in \Rindex$. 
    \item \label{strongly submersive with respect to a CF-perturbation on a chart} 
    Let $L$ be a smooth manifold. Let $i_{\mathfrak{r}}: U_{\mathfrak{r}} \to U$ be the $G$-equivariant inclusion map. A $G$-equivariant smooth map $f: U\to L$ is said to be \textbf{strongly submersive} with respect to $[\mathcal{S}]$ if $[\mathcal{S}]$ is transverse to zero and $f\circ i_{\mathfrak{r}}$ is strongly submersive with respect to $S_{\rindex}$ for all $\rindex\in \Rindex$. 
    \item We can define strong transversality between a strongly submersive map from a Kuranishi chart and a smooth map from a smooth manifold to the same target smooth manifold similarly. 
\end{enumerate}
\end{defn}

\begin{defn}[$G$-equivariant CF-perturbation on a Kuranishi space] 
\label{CF-perturbation on a Kuranishi space}
 A \textbf{$G$-equivariant CF-perturbation} on a space $\kur$ with a Kuranishi structure with corners
\[ 
\widehat{\mathcal{U}} =  \left(\set{ \mathcal{U}_p }{ p\in \mathcal{M}},   
\set{\vec{\alpha} _{pq}}{ p\in \mathcal{M}, q\in \image \psi_p } \right) 
\]
 is a collection 
 \[ \widehat{\mathcal{S}} = \set{ [\mathcal{S}_p] }{p\in \kur}\]
 such that the following holds. 
 \begin{enumerate}[i)]
     \item 
     For each $p\in \kur$, $[\mathcal{S}_p]$ is a $G$-equivariant CF-perturbation on $\mathcal{U}_p$ represented by 
     \[ \mathcal{S}_p = \set{\mathcal{S}_p^{\epsilon} = (W_p , \s_p^{\epsilon},\tau_p ) }{\epsilon \in (0,1]}   \] 
     on $\mathcal{U}_p$. 
     \item For any $p\in \kur$, $q\in \image \psi_p$, the data $\left[\restr{\mathcal{S}_q}{U_{pq}}\right]$ and $[\mathcal{S}_p]$ are compatible with the $G$-equivariant Kuranishi coordinate change
     \[ \Vec{\alpha}_{pq} = \left(\alpha_{pq}: U_{pq} \to U_p, \, \widehat{\alpha}_{pq}: \restr{\E_q}{U_{pq}}\to \E_p\right) \] in the following sense. 
     \begin{enumerate}[a)]
         \item For each $x\in U_{pq}$, $y = \alpha_{pq}(x)$, 
         there exist a
         $G$-invariant open neighborhood $U_{q,x}\subset U_{pq}$ of $x$ and a $G$-invariant open neighborhood  $U_{p,y}= \alpha_{pq}(U_{q,x})$ of $y$ such that there exist $G$-equivariant
         CF-perturbation representatives 
        \[ 
        \mathcal{S}_{q,x} =\left(W_{q,x}, \s_{q,x}^{\epsilon}, \tau_{q,x}\right),\qquad 
        \mathcal{S}_{p,y} =\left(W_{p,y} , \s_{p,y} ^{\epsilon}, \tau_{p,y}  \right)
        \] 
         of $\left[\restr{\mathcal{S}_q}{U_{q,x}}\right]$, $\left[\restr{\mathcal{S}_p}{U_{p,y}}\right] $ satisfying the following. 
          \begin{itemize}
              \item $W_{q,x}\subset W_{q}$,  $W_{p,y}\subset W_{p}$ are $G$ invariant suborbifolds such that $W_{q,x} \xrightarrow{h_{pq}} W_{p,y}$ is a $G$-equivariant diffeomorphism.
            \item $(h_{pq})_{G!}(\tau_{q,x}) =  \tau_{p,y}$.
             \item 
              $\s^{\epsilon}_{q,x} = \restr{\s^{\epsilon}_{q}}{ W_{q,x}}$, 
              $\s^{\epsilon}_{p,y}  = \restr{\s^{\epsilon}_{p}}{W_{p,y}}$. 
              \item For each $\epsilon \in (0,1]$, the following diagram commutes. 
              \[
              \xymatrix{
             \restr{\E_q}{U_{q,x}}
              \ar[r]^{\widehat{\alpha}_{pq}}
              & 
             \restr{\E_p}{U_{p,x}}
               \\
             W_{q,x} 
            \ar[u]^{\s_{q,x}^{\epsilon}}
            \ar[r]_{ h_{pq}}
            & 
            W_{p,y} 
            \ar[u]_{\s_{p,y}^{\epsilon}}
              }
        \]
          \end{itemize}
     \end{enumerate}
 \end{enumerate}
\end{defn}
We can define a $G$-equivariant CF-perturbation 
 \[ \widetriangle{\mathcal{U}} = \left( (\Pindex, \leq ) 
    \, , \, 
    \left\{ \mathcal{U}_{\p}\mid \p\in \Pindex \right\}
    \, , \, 
    \set{\Vec{\alpha}_{\pq}}{ \p, \q\in \Pindex, \q \leq \p } \right)
    \]
on a $G$-equivariant good coordinate system similarly. 

\begin{defn}[$G$-invariant partition of unity on a good coordinate system]
\label{definition of G-invariant partition of unity}
Let $\left(\kur,  \widetriangle{\mathcal{U}}\right)$ be a space with $G$-equivariant good coordinate system
\[ \widetriangle{\mathcal{U}} = \left( (\Pindex, \leq ) 
    \, , \, 
    \left\{ \mathcal{U}_{\p}\mid \p\in \Pindex \right\}
    \, , \, 
    \set{\Vec{\alpha}_{\pq}}{ \p, \q\in \Pindex, \q \leq \p } \right).
\] 
Let $ \lrp{\widetriangle{\mathcal{K}}, \widetriangle{\mathcal{K}}^{++}}$ be a $G$-equivariant support pair and let $\delta > 0$. 
Take a $G$-invariant metric on $|\mathcal{K}^{++}|$. 
A collection of functions $\{\chi_{\p}\mid \p\in \Pindex\}$
is said to be a $G$-invariant \textbf{partition of unity} on $\lrp{\kur,  \widetriangle{\mathcal{U}}}$ with respect to the data 
 $\left(\widetriangle{\mathcal{K}}, \widetriangle{\mathcal{K}}^{++}, \delta \right)$
if it satisfies the following. 
\begin{enumerate}[i)] 
    \item For each $\p$, let 
     \[ \Omega_{\p}(\mathcal{K},\delta) := \{ x \in \mathcal{K}_{\p}^{++} \mid d(x, \mathcal{K}_{\p}) <\delta \}. \] 
    \item For each $\p$, $\chi_{\p}: |\mathcal{K}^{++}| \to [0,1]$ is a $G$-invariant strongly smooth function in the following sense: 
    \[\restr{\chi_{\p} }{\mathcal{K}_{\p}^{++} \cap \Omega_{\p}(\mathcal{K},\delta)} :  \mathcal{K}_{\p}^{++}  \cap \Omega_{\p}(\mathcal{K},\delta)   \to [0,1] \]
    is $G$-invariant and smooth. 
    \item For each $\p$, we require
    $\supp  \chi_{\p} \subset \Omega_{\p}(\mathcal{K},\delta)$.  
    \item There exists an open neighborhood $\mathcal{N}$ of $\kur$ in $|\mathcal{K}^{++}|$ such that 
    \[ \sum_{\p} \chi_{\p}(x) =1 \qquad \forall x\in \mathcal{N}.  \]
\end{enumerate}
\end{defn}

\begin{lem}
\label{existence of partition of unity}
      There exists a $G$-invariant partition of unity satisfying definition \ref{definition of G-invariant partition of unity} subordinate to the $G$-equivariant good coordinate system associated with the Kuranishi structure on $\Mq_{k+1}(L,J,\beta)$ in Proposition \ref{GKur structure on moduli space}.  
\end{lem}
\begin{proof}
Let $\lrp{\widetriangle{\mathcal{K}}, \widetriangle{\mathcal{K}}^{++}}$ be a $G$-equivariant support pair of the given good coordinate system. Take a $G$-invariant metric on $|\mathcal{K}^{++}|$. 
Then by \cite{kurbook} Proposition 7.68, if $\delta >0$ is sufficiently small, a partition of unity $\tilde{\chi}_{\p}$ associated with the data  $\lrp{\kur, \widetriangle{\mathcal{U}},  \widetriangle{\mathcal{K}}, \widetriangle{\mathcal{K}}^{++},\delta}$ that may not be $G$-invariant exists. 
We average $\tilde{\chi}_{\p}$ with respect to the $G$-action to obtain $\chi_{\p}$, which is now $G$-invariant. 
\end{proof}

We will define some sheaves of $G$-equivariant CF-perturbations via \'{e}tale spaces. 
\begin{defn}[\'{E}tale space]
An \textbf{\'{e}tale space} over a topological space $Y$ is a pair $\lrp{\mathcal{A}, p}$ consisting of a topological space $\mathcal{A}$ and a continuous map $p: \mathcal{A} \to Y$ such that 
$p$ is a local homeomorphism. 
\end{defn}

Given an \'{e}tale space over $Y$, one can construct a sheaf as follows. 
Let $\Omega$ be any open subset of $Y$. We assign $\Omega$ the set of sections
\[\Gamma(\Omega, \mathcal{A}) = \{s:\Omega\to \mathcal{A}\mid s \text{ is continuous, } p\circ s = \id_{\Omega} \}. \] 

\begin{defn}[Sheaf of $G$-equivariant CF-perturbations on a chart]
\label{Sheaf on a chart}
Let $\mathcal{U}_{\p}$ be a $G$-equivariant Kuranishi chart. For any open subset $\Omega /G \subset U_{\p}/G$, which is a quotient of a $G$-invariant open subset $\Omega$ of $U_{\p}$, let 
\[ \mathcal{CF}^{G, \mathcal{U}_{\p}}(\Omega)\] be the set of $G$-equivariant CF-perturbations of $\mathcal{U}_{\p}$ on  $\Omega$. 
Similar to \cite{kurbook} Proposition 7.22, $\mathcal{CF}^{G, \mathcal{U}_{\p}}$ defines a sheaf on $U_{\p}/G$. 
We define the \textbf{stalk} of $\mathcal{CF}^{G, \mathcal{U}_{\p}}$ at a point $x\in \mathcal{U}_{\p}$ by taking the direct limit
\begin{equation}
( \mathcal{CF}^{G, \mathcal{U}_{\p}} )_x = \varinjlim_{\Omega\ni x}     \mathcal{CF}^{G, \mathcal{U}_{\p}}(\Omega), 
\end{equation} 
where $\Omega$ runs through all $G$-invariant open subsets of $U_{\p}$ containing $x$. 
Indeed, the direct limit, up to isomorphism, can be constructed as 
\[ \bigsqcup_{\Omega\ni x} \mathcal{CF}^{G, \mathcal{U}_{\p}}(\Omega)/\sim,\]
where, if $[\mathcal{S}_1]\in  \mathcal{CF}^{G, \mathcal{U}_{\p}}(\Omega)$ and $[\mathcal{S}_2]\in  \mathcal{CF}^{G, \mathcal{U}_{\p}}(\Omega')$, 
then $[\mathcal{S}_1]\sim [\mathcal{S}_2]$ if and only if there exists a $G$-invariant $\Omega''\subset \mathcal{U}_{\p}$, contained in both $\Omega$ and $\Omega'$, such that $x\in \Omega''$ and $[\restr{\mathcal{S}_1}{\Omega''}] = [\restr{\mathcal{S}_2}{\Omega''}]$. 

For each $G$-invariant open subset $\Omega \subset U_{\p}$ containing $x$, there is a map
\begin{equation}
\label{direct limit map}
\mathcal{CF}^{G, \mathcal{U}_{\p}}(\Omega) \to ( \mathcal{CF}^{G, \mathcal{U}_{\p}} )_x, \qquad [\mathcal{S}]\mapsto [\mathcal{S}_x].   
\end{equation}
A member of $( \mathcal{CF}^{G, \mathcal{U}_{\p}} )_x$ is called a \textbf{germ}. 
\end{defn}

\begin{defn}
Let $\lrp{\kur, \widetriangle{\mathcal{U}}}$ be a space with a $G$-equivariant good coordinate system. 
Suppose $G$ acts on each Kuranishi chart freely. 
Let $\widetriangle{\mathcal{K}}$ be a support system on $\widetriangle{\mathcal{U}}$ and let $|\mathcal{K}|$ be as in \autoref{heterodimensional compactum}. 
Let $x \in |\mathcal{K}|$ and 
    $Q:  \bigsqcup K_{\p} \to |\mathcal{K}| $ be the map that identifies equivalent elements. 
    Suppose 
    \begin{equation}
    \label{Index of charts containing x}
    \Pindex (x) = \{\p \in \Pindex \mid Q^{-1}(x) \cap K_{\p} \ne \emptyset\} = \{\p_1 \leq \cdots \leq  \p_k\},     
    \end{equation}
where $\p_1 ,\ldots, \p_k$ are distinct. 
We denote the maximal such $\p_k$ by $\p(x)$. 
We introduce the following notations. 

 \begin{enumerate}[i)] 
     \item Define 
     \begin{equation}
    (\CFKG)_x 
    =
    \set{ [\mathcal{S}_x] \in ( \mathcal{CF}^{G, \mathcal{U}_{\p(x)}} )_x }{ \mathcal{S}_x \text{ is restrictable to } U_{\p}\quad \forall \p \in \Pindex(x)}. 
\end{equation} 
Let 
     \[ |\CFKG| = \bigcup_{[x]\in |\mathcal{K}|/G} \{[x]\}\times (\CFKG)_x \]
     and $p_{\mathcal{K}}^G: |\CFKG| \to |\mathcal{K}|/G, \quad ([x], [\mathcal{S}_x])\mapsto [x]$. 
     We can topologize $|\CFKG|$ in a way similar to \cite{kurbook} Definition 12.20 such that $p_{\mathcal{K}}^G$ becomes a local homeomorphism. 
     Then \[ \Omega/G \mapsto \CFKG(\Omega) =  \set{[\mathcal{S}]:\Omega \to|\CFKG|}{ [\mathcal{S}] \text{ is continuous}, p_{\mathcal{K}}^G \circ [\mathcal{S}] = \id_{\Omega/G}} \]
     defines a sheaf on $|\mathcal{K}|/G$. 
     \item 
     Define \begin{equation}
(\CFKGtranstozero)_x 
    =
    \set{ [\mathcal{S}_x] \in (\CFKG)_x }{ \mathcal{S}_x \text{ is transverse to zero}}. 
\end{equation}
One can similarly obtain a sheaf $\CFKGtranstozero$ on $|\mathcal{K}|/G$. 
\item  
Let $\widetriangle{f}: \left(\kur, \widetriangle{\mathcal{U}}\right) \to L$ 
be a $G$-equivariant strongly smooth map. Define
\begin{equation}
    (\mathcal{CF}_{\pitchfork f,\mathcal{K}}^{G})_x 
    =
    \set{ [\mathcal{S}_x] \in (\CFKG)_x }{ \widetriangle{f} \text{ is strongly submersive with respect to } \mathcal{S}_x}. 
\end{equation}
This defines a sheaf $\mathcal{CF}_{\pitchfork f,\mathcal{K}}^{G}$ on $|\mathcal{K}|/G$. 
\item 
Let $\widetriangle{f}: \left(\kur, \widetriangle{\mathcal{U}}\right) \to L$ 
be a $G$-equivariant strongly smooth map and $g: N\to L$ be a smooth map between manifolds. 
Define
\begin{equation}
    (\mathcal{CF}_{ f \pitchfork g, \mathcal{K}}^{G})_x 
    =
    \set{ [\mathcal{S}_x] \in (\CFKG)_x }
    { \begin{aligned}
    & \widetriangle{f} \text{ is strongly transverse to } g\\
    & \text{ with respect to } \mathcal{S}_x
    \end{aligned}} . 
\end{equation}
This defines a sheaf  $\mathcal{CF}_{ f \pitchfork g, \mathcal{K}}^{G}$ on $|\mathcal{K}|/G$.
 \end{enumerate}   
\end{defn}

\begin{defn}[Strongly transverse]
 \label{strongly transverse to 0}
Let $\mathcal{U}$ be a $G$-equivariant Kuranishi chart on $\kur$ and let $x\in \mathcal{U}$. 
Suppose $[\mathcal{S}_x]$ 
is a germ at $x$. 
Let $ (W_x\xrightarrow{\nu_x} U_x, \tau_x, \{\s_{x}^{\epsilon}\})$ be a representative of $[\mathcal{S}_x]$. 
Let $\lrp{V_x,\Gamma_x, F_x, \varphi_x, \widehat{\varphi}_x^{W}}$ and $\lrp{V_x,\Gamma_x, E_x, \varphi_x, \widehat{\varphi}_x^{\E}}$ 
be orbibundle charts of $W,\E$ at $x$, respectively. 
In particular, there exists a unique $o_x\in V_x$ such that $ \varphi_x(o_x) = x$.  
\begin{equation}
\label{local chart of CF-perturbations}
\xymatrix
{
V_x\times F_x 
\ar[r]^{\widehat{\varphi}_x^{W}}
\ar[d]_{pr_1}
&
W_x 
\ar[d]^{\nu_x}
\\
V_x 
\ar[r]^{\varphi_x} 
& 
U_x
}, 
\quad 
\xymatrix
{
V_x\times E_x   
\ar[r]^{\widehat{\varphi}_x^{\E}}
\ar[d]_{pr_1}
&
\restr{\E}{U_x} 
\ar[d]^{\pi}
\\
V_x 
\ar[r]^{\varphi_x} 
& 
U_x
}.  
\end{equation}
Let $\overline{\s}_x^{\epsilon} = pr_2\circ \lrp{\widehat{\varphi}_x^{\E}}^{-1}\circ  \s_x^{\epsilon}\circ \widehat{\varphi}_x^{W}:  V_x\times F_x \to E_x $. 
We say $[\mathcal{S}_x]$ is \textbf{strongly transverse} if, $\forall \epsilon\in (0,1]$, 
the derivative 
\[ \nabla_{(v,\xi)}^W \overline{\s}_x^{\epsilon} : T_{\xi}F_x \to T_{c}E_x \]
in the $F_x$-direction is surjective for all $(v,\xi)\in \lrp{\widehat{\varphi}_x^{W_x}}^{-1}\lrp{\nu_x^{-1}(o_x)\cap \supp(\tau_x)}$. 
     Define \begin{equation}
(\mathcal{CF}_{\pitchfork\pitchfork 0}^{G })_x
    =
    \set{ [\mathcal{S}_x] \in (\CFKG)_x }{ \mathcal{S}_x \text{ is strongly transverse}}. 
\end{equation}
\end{defn}

\begin{prop}
\label{soft sheaves}
Let $\lrp{\kur, \widetriangle{\mathcal{U}}}$ be a space with a $G$-equivariant good coordinate system. 
Suppose $G$ acts on each chart freely. 
Let $\widetriangle{\mathcal{K}}$ be a support system on $\widetriangle{\mathcal{U}}$ and let $|\mathcal{K}|$ be as in \autoref{heterodimensional compactum}.
\begin{enumerate}[i)]
\item 
The sheaves
$\CFKG, \CFKGtranstozero$ 
are soft. 
\item If $\widetriangle{f}$ is weakly submersive then the sheaf $\mathcal{CF}_{\pitchfork f,\mathcal{K}}^{G}$ is soft. 
If $\widetriangle{f}$ is weakly transverse to $g$, then $\mathcal{CF}_{ f \pitchfork g, \mathcal{K}}^{G}$ is soft.
\end{enumerate}
\end{prop}

\begin{proof}
    The proof is similar to that of \cite{kurbook} Theorem 12.24. 
    Let $\widetriangle{\mathcal{K}} =\{ K_{\p}\mid \p \in \Pindex\}$ be a $G$-equivariant support system on $\kur$ as in Definition \ref{support system}.  
    Let $|\mathcal{K}| = \bigsqcup K_{\p}/\sim $ 
    be the heterodimensional compactum as defined in \autoref{heterodimensional compactum}. 
  
    Let $x \in |\mathcal{K}|$ and 
    $Q:  \bigsqcup K_{\p} \to |\mathcal{K}| $ be the map that identifies equivalent elements. 
    Then, by the definition of a good coordinate system, we have a totally ordered set
    \begin{equation}
    \Pindex (x) = \{\p \in \Pindex \mid Q^{-1}(x) \cap K_{\p} \ne \emptyset\} = \{\p_1 \leq \ldots\leq \p_k\}     
    \end{equation}
for some positive integer $k$.  
We denote the maximal element $\p_k$ by $\p(x)$ and the minimal element $\p_1$ by $\p_-(x)$ 
\begin{lem}\label{pointwise germ}
There exists a germ $[\mathcal{S}_{x}]\in (\mathcal{CF}_{\mathcal{K}}^{G})_x $ 
such that 
$[\restr{\mathcal{S}_x}{U_{\p_1}}] \in (\mathcal{CF}_{\pitchfork\pitchfork 0}^{G, \mathcal{U}_{\p_1}})_x$. 
\end{lem}

\begin{proof}[Proof of Lemma \ref{pointwise germ}]
We construct,  for each $j$, a germ $[\mathcal{S}_{j,x}]\in (\mathcal{CF}_{\pitchfork 0}^{G, \mathcal{U}_{\p_j}})_x $ 
such that, 
$\forall 1\leq j \leq k$, 
$[\mathcal{S}_{j,x}]$ is restrictable to  $\mathcal{U}_{\p_i}$ for all $i<j$. 

We induct on $1\leq j\leq k$. 
Let $j=1$.
 Take a $G$-invariant neighborhood $U_1\subset U_{\p_1}$ of $x$ on which the coordinate change $\vec{\alpha}_{\p(x),\p_1}$ is defined. 
Let $W_1 =\restr{ \E_{\p_1}}{U_1} \xrightarrow{\nu_{1}}U_1 $ be the restriction of the obstruction bundle for $\p_1$ and $\tau_1$ be a $G$-equivariant Thom form. 
Define $\s_1^{\epsilon}: W_1 \to \restr{\E_{\p_1}}{U_1}$ by 
\begin{equation}
\s_1^{\epsilon} (w) 
= 
s_{\p_1}\circ \nu_1(w)
+
\epsilon w \qquad \forall w \in W_1.    
\end{equation}

Then 
\[ 
[\mathcal{S}_1] 
= 
\left[
\lrp{W_1\xrightarrow{\nu_{1}}U_1, \tau_1, \{\s_1^{\epsilon} \}}
\right ] 
\]
is strongly transverse. 
Then $[\mathcal{S}_{1,x}] \in (\mathcal{CF}_{\pitchfork \pitchfork 0}^{G, \mathcal{U}_{\p_1}})_x$. 
 Note that the following holds. 
\begin{itemize}
    \item $(\mathcal{CF}_{\pitchfork \pitchfork 0}^{G, \mathcal{U}_{\p_1}})_x \subset (\mathcal{CF}_{\pitchfork 0}^{G, \mathcal{U}_{\p_1}})_x$. 
    \item If $\widetriangle{f}$ is weakly submersive, then 
   $(\mathcal{CF}_{\pitchfork \pitchfork 0}^{G, \mathcal{U}_{\p_1}})_x \subset  (\mathcal{CF}_{\pitchfork f}^{G, \mathcal{U}_{\p_1}})_x $. 
   \item If $\widetriangle{f}$ is weakly transverse to $g$, then $(\mathcal{CF}_{\pitchfork \pitchfork 0}^{G, \mathcal{U}_{\p_1}})_x \subset  (\mathcal{CF}_{f\pitchfork g}^{G, \mathcal{U}_{\p_1}})_x $.
\end{itemize}

Suppose a germ 
$[\mathcal{S}_{j,x}]\in (\mathcal{CF}_{\pitchfork 0}^{G, \mathcal{U}_{\p_j}})_x$ is constructed. 
Then there exists some nonempty $G$-equivariant open subset $ U_{j,0} \subset U_{\p_j}$ such that the image of  
$[\mathcal{S}_{j}] 
= \left[ 
\lrp{W_{j} \xrightarrow{\nu_{j}} U_{j,0}, \tau_{j}, \{ \s_{j}^{\epsilon}\}}
\right]$
under the map $ \mathcal{CF}^{G, \mathcal{U}_{\p_j}}(U_j) \to ( \mathcal{CF}^{G, \mathcal{U}_{\p_j}} )_x$ is $[\mathcal{S}_{j,x}]$. 

We may take a nonempty $G$-invariant subset $U_{j}\subset U_{j,0}$ and a nonempty $G$-equivariant tubular neighborhood $U_{\p_{j+1},x} \subset U_{\p_{j+1}} $ of $\alpha_{\p_{j+1},\p_j}(U_j)$ such that the coordinate change map $\alpha_{\p_{j+2},\p_{j+1}}$ is defined on $U_{\p_{j+1},x}$. 
Without loss of generality, we assume $U_{j,0}=U_j$

Then there is a projection map $\pi_j: U_{j+1} \to U_{j}$ such that $\pi_j\circ \alpha_{\p_{j+1},\p_j} = \id$ on $U_j$.

Let $W_{j+1} = \pi_j^* W_{j} \xrightarrow{\nu_j} U_{j+1}$ be the pullback bundle and $ \tau_{j+1} = \pi_j^*\tau_{j}$. 

\begin{center}
    
\tikzset{every picture/.style={line width=0.75pt}} 

\begin{tikzpicture}[x=0.75pt,y=0.75pt,yscale=-1,xscale=1]

\draw (214.75,147.4) node [anchor=north west][inner sep=0.75pt]    {$U_{j}$};
\draw (161,91.4) node [anchor=north west][inner sep=0.75pt]    {$W_{j}$};
\draw (215.75,40.4) node [anchor=north west][inner sep=0.75pt]    {$\restr{\mathcal{E}_{\mathfrak{p}_{j} }}{U_j}$};
\draw (334.75,147.4) node [anchor=north west][inner sep=0.75pt]    {$U_{j+1}$};
\draw (287,89.4) node [anchor=north west][inner sep=0.75pt]    {$\pi _{j}^{*} W_{j}$};
\draw (336.25,40.4) node [anchor=north west][inner sep=0.75pt]    {$\restr{\mathcal{E}_{\mathfrak{p}_{j+1}}}{U_{j+1}}$};
\draw (286,162.4) node [anchor=north west][inner sep=0.75pt]    {$\pi _{j}$};
\draw (261,22.4) node [anchor=north west][inner sep=0.75pt]    {$\hat{\alpha }_{\mathfrak{p}_{j+1} ,\ \mathfrak{p}_{j}}$};
\draw (181,60.4) node [anchor=north west][inner sep=0.75pt]    {$\mathfrak{s}_{j}^{\epsilon }$};
\draw (253,77.4) node [anchor=north west][inner sep=0.75pt]    {$\tilde{\pi }_{j}$};
\draw (178,125.4) node [anchor=north west][inner sep=0.75pt]    {$\nu _{j}$};
\draw (288,117.4) node [anchor=north west][inner sep=0.75pt]    {$\nu _{j+1}$};
\draw    (254.75,156.5) -- (331.75,156.5) ;
\draw [shift={(252.75,156.5)}, rotate = 0] [color={rgb, 255:red, 0; green, 0; blue, 0 }  ][line width=0.75]    (10.93,-3.29) .. controls (6.95,-1.4) and (3.31,-0.3) .. (0,0) .. controls (3.31,0.3) and (6.95,1.4) .. (10.93,3.29)   ;
\draw    (204,100.32) -- (284,99.73) ;
\draw [shift={(202,100.34)}, rotate = 359.58] [color={rgb, 255:red, 0; green, 0; blue, 0 }  ][line width=0.75]    (10.93,-3.29) .. controls (6.95,-1.4) and (3.31,-0.3) .. (0,0) .. controls (3.31,0.3) and (6.95,1.4) .. (10.93,3.29)   ;
\draw    (251.75,49.5) -- (331.25,49.5) ;
\draw [shift={(333.25,49.5)}, rotate = 180] [color={rgb, 255:red, 0; green, 0; blue, 0 }  ][line width=0.75]    (10.93,-3.29) .. controls (6.95,-1.4) and (3.31,-0.3) .. (0,0) .. controls (3.31,0.3) and (6.95,1.4) .. (10.93,3.29)   ;
\draw    (193.83,87) -- (216.99,64.4) ;
\draw [shift={(218.42,63)}, rotate = 135.69] [color={rgb, 255:red, 0; green, 0; blue, 0 }  ][line width=0.75]    (10.93,-3.29) .. controls (6.95,-1.4) and (3.31,-0.3) .. (0,0) .. controls (3.31,0.3) and (6.95,1.4) .. (10.93,3.29)   ;
\draw    (192.6,114) -- (218.29,141.54) ;
\draw [shift={(219.65,143)}, rotate = 226.98] [color={rgb, 255:red, 0; green, 0; blue, 0 }  ][line width=0.75]    (10.93,-3.29) .. controls (6.95,-1.4) and (3.31,-0.3) .. (0,0) .. controls (3.31,0.3) and (6.95,1.4) .. (10.93,3.29)   ;
\draw    (327.25,85) -- (344.55,64.53) ;
\draw [shift={(345.84,63)}, rotate = 130.2] [color={rgb, 255:red, 0; green, 0; blue, 0 }  ][line width=0.75]    (10.93,-3.29) .. controls (6.95,-1.4) and (3.31,-0.3) .. (0,0) .. controls (3.31,0.3) and (6.95,1.4) .. (10.93,3.29)   ;
\draw    (325.75,114) -- (346.05,141.39) ;
\draw [shift={(347.24,143)}, rotate = 233.45] [color={rgb, 255:red, 0; green, 0; blue, 0 }  ][line width=0.75]    (10.93,-3.29) .. controls (6.95,-1.4) and (3.31,-0.3) .. (0,0) .. controls (3.31,0.3) and (6.95,1.4) .. (10.93,3.29)   ;
\draw    (232.25,63) -- (232.25,141) ;
\draw [shift={(232.25,143)}, rotate = 270] [color={rgb, 255:red, 0; green, 0; blue, 0 }  ][line width=0.75]    (10.93,-3.29) .. controls (6.95,-1.4) and (3.31,-0.3) .. (0,0) .. controls (3.31,0.3) and (6.95,1.4) .. (10.93,3.29)   ;
\draw    (357.25,63) -- (357.25,141) ;
\draw [shift={(357.25,143)}, rotate = 270] [color={rgb, 255:red, 0; green, 0; blue, 0 }  ][line width=0.75]    (10.93,-3.29) .. controls (6.95,-1.4) and (3.31,-0.3) .. (0,0) .. controls (3.31,0.3) and (6.95,1.4) .. (10.93,3.29)   ;

\end{tikzpicture}

\end{center}
For any $w \in W_{j+1}$, define 
\begin{equation}
\label{CF-perturbation at a point}
\s_{j+1}^{\epsilon} (w)
= 
s_{\p_{j+1}}\circ \nu_{j+1} (w) 
+ \widehat{\alpha}_{\p_{j+1},\p_j } \lrp{ \s_{j}^{\epsilon}(\tilde{\pi}_j(w)) - s_{\p_{j}} \circ \nu_{j} (\tilde{\pi}_j(w))}.  
\end{equation}
Then  
\[ [\mathcal{S}_{j+1}] 
= \left[ \lrp{W_{j+1} \xrightarrow{\nu_{j+1}} U_{j+1}, \tau_{j+1}, \{ \s_{j+1}^{\epsilon}\}} \right] 
\] 
defines a germ  $[\mathcal{S}_{j+1,x}] \in   ( \mathcal{CF}^{G, \mathcal{U}_{\p_{j+1}}} )_x$. 
 
By this construction, 
$[\mathcal{S}_{k}]$ defines an element  $[\mathcal{S}_{x}] \in  (\CFKG)_x$ satisfying Lemma \ref{pointwise germ}. 
\end{proof}

By forgetting the $G$-action, the construction coincides with the CF-perturbation constructed in \cite{kurbook} Lemma 12.12. 
In particular, the following holds. 
\begin{enumerate}[i)]
    \item \label{(CFKGtranstozero)_x} 
    $[\mathcal{S}_x] \in  (\CFK)_x$. 
    \item 
    \label{(CFKGtransf)_x}
    If $\widetriangle{f}$ is weakly submersive, then $[\mathcal{S}_{x}] \in  (\mathcal{CF}_{\pitchfork f,\mathcal{K}}^{G})_x$.
    \item 
    \label{(CFKGftranstog)_x}
    If $\widetriangle{f}$ is weakly transverse to $g$, then $[\mathcal{S}_{x}] \in  (\mathcal{CF}_{ f \pitchfork g, \mathcal{K}}^{G})_x$.
\end{enumerate}

Let $\bullet \in \{ \pitchfork 0, \pitchfork f, f\pitchfork g\}$. 
Suppose $K \subset |\mathcal{K}|$ is a $G$-invariant closed subset and $[\mathcal{S}_K]\in \mathcal{CF}_{\bullet, \mathcal{K}}^G(K)$. 
Let $\mathcal{K}^{++}$ be another $G$-equivariant support system such that $\mathcal{K}< \mathcal{K}^{++}$. 
Then $[\mathcal{S}_K]$ is the restriction of some $[\mathcal{S}_{\rindex_0}] \in \mathcal{CF}_{\bullet, \mathcal{K}}^G(U_{\rindex_0} \cap |\mathcal{K}|)$ for some $G$-invariant subset $U_{\rindex_0}$ of $|\mathcal{K}^{++}|$ containing $K$.  

For each $x\in |\mathcal{K}|\setminus K$, let $[\mathcal{S}_x]$ and $ [\mathcal{S}_{j,x}]$, $1\leq j \leq \p(x)$, be as in Lemma \ref{pointwise germ}.

Since $|\mathcal{K}|$ is compact, we may take finitely many points $\{x_{\rindex}\mid \rindex \in \Rindex'\}$ 
and representatives $[(W_{\rindex}\xrightarrow{\nu_{\rindex}} U_{\rindex} , \tau_{\rindex}, \{\s_{\rindex}^{\epsilon}\})]$, $[(W_{x_{\rindex},j}\xrightarrow{\nu_{x_{\rindex},j }} U_{x_{\rindex},j}, \tau_{x_{\rindex},j}, \{\s_{x_{\rindex},j}^{\epsilon}\})] $ of $[\mathcal{S}_{x_{\mathfrak{r}}}]$ and $[\mathcal{S}_{j,x_{\mathfrak{r}}}]$ such that 
\[ \bigcup_{\rindex \in \Rindex'} U_{\rindex} \supset |\mathcal{K}|  \]  
and the following holds.  

\begin{enumerate}[i)]
    \item $\rindex_0\in \Rindex'$ and, for all $\rindex \in  \Rindex'\setminus \{\rindex_0\}$, $U_{\rindex}\cap K =\emptyset$. 
    \item 
For any $\p_-(x)\leq \p_i\leq \p(x)$, we consider local charts of the form \autoref{local chart of CF-perturbations}.
    Let $F_x$ be the fiber of $\nu_x$ and $E_{\p_j}$ be the fiber of the orbibundle of $U_{\p_j}$.  
    By construction, we have a $\Gamma_x$-equivariant projection $\pi: V_{x,j}\to V_{x,1}$ and an embedding $I_j': \pi^*\E_1\to \E_j$. 
   Then, for any $y\in V_{x,j}$ and $\xi \in F_x$, the following diagram is commutative.   
    \[
    \xymatrix{
    0 
    \ar[r]
    & 
    T_{\xi} F_x 
    \ar[r]
    \ar[d]_{D_{(y,\xi)}\s_1^{\epsilon}}
    & 
    T_{(y,\xi)}(V_{x,j}\times F_x) 
    \ar[r]
    \ar[d]_{D_{(y,\xi)}\s_j^{\epsilon}}
    & 
    T_yV_{x,j} 
    \ar[r]
    \ar[d]
    & 
    0
    \\
    0 
    \ar[r]
    & 
    I_j'\lrp{\pi(y), E_{\p_-(x)}}
    \ar[r] 
    & 
    T_{c}(E_{\p_j})
    \ar[r]
    & 
   \dfrac{ T_{c}(E_{\p_j})}{I_j'\lrp{\pi(y), E_{\p_-(x)}}} 
   \ar[r]
    & 
    0
    }
    \]
    We may choose the CF-perturbations so that  
    \begin{enumerate}
        \item $D_{(y,\xi)}\s_1^{\epsilon}$ is surjective. 
        \item There exists a sufficiently small $\sigma_x>0$ such that, if $|\s_j^{\epsilon} - s_j|_{C^1}<\sigma_x$, then the third vertical map is surjective. 
    \end{enumerate}
    
\end{enumerate}

Let $\{\chi_{\rindex}\mid  \rindex \in \Rindex'\}$ be a $G$-invariant partition of unity subordinate to this covering. 
Suppose $y\in |\mathcal{K}|$. 
Let 
\[ 
I(y) = \{\rindex \in \Rindex' \mid \chi_{\rindex }(y)\ne 0\}. 
\]
Then $\vec{\alpha}_{\p(x_{\rindex})\p(y)}^*[\mathcal{S}_{\rindex}]$ is defined. 
Let $(W_{\rindex} , \tau_{\rindex}, \{ \s_{\rindex}^{\epsilon}\})$ be a representative of $\vec{\alpha}_{\p(x_i)\p(y)}^*[\mathcal{S}_{\rindex}]$. 
Let $W_y = \prod\limits_{\rindex\in \Rindex'} W_{\rindex}$, $\tau_y = \prod\limits_{\rindex\in \Rindex'} \tau_{\rindex}$.  
Let $ (w_{\rindex})_{\rindex \in I(y)}\in  \prod\limits _{\rindex\in \Rindex'} \left(\restr{W_{\rindex}}{z}\right)$. 
Let $\sigma<\sigma_x$ and define
\[ 
\s_{y}^{\epsilon} \lrp{w}  
=  s_y(z) +  \chi_{\rindex_0}(z) \lrp{ \s_{\rindex_0}^{\epsilon} (w_{\rindex_0})-s_y(z)} 
+ 
\sigma \sum_{\rindex \in I(y)} \chi_{\rindex}(z) \lrp{ \s_{\rindex}^{\epsilon}(w_{\rindex_0}) -s_y(z)}.
\]
If we forget the $G$-action, the construction is the same as 
that in \cite{kurbook} Theorem 12.24. 
Therefore, the transversality results follow from \cite{kurbook} Theorem 12.24. 
\end{proof}

Proposition \ref{soft sheaves} implies the following. 
\begin{prop}[Existence of a  $G$-equivariant CF-perturbation on a good coordinate system]
\label{existence of CF-perturbation}
Let $\lrp{\kur, \widetriangle{\mathcal{U}}}$ be a space with $G$-equivariant good coordinate system. 
Let $\widetriangle{\mathcal{K}}$ be a support system on $\widetriangle{\mathcal{U}}$. 
\begin{enumerate}[i)]
    \item There exists a $G$-equivariant CF-perturbation $\widetriangle{\mathcal{S}}$ of $\widetriangle{\mathcal{U}}$ such that $\widetriangle{\mathcal{S}}$ is transverse to zero. 
    \item If $\widetriangle{f}: (X, \widetriangle{\mathcal{U}}) \to L$ is a weakly submersive $G$-equivariant strongly smooth map, then $\widetriangle{\mathcal{S}}$ can be chosen so that $\widetriangle{f}$ is strongly submersive. 
\end{enumerate}
\end{prop}
This induces a CF-perturbation on some Kuranishi structure on the same space in the same way as \cite{kurbook} Lemma 9.9.

\subsection{\texorpdfstring{$G$}{G}-equivariant integration along the fiber}
We review the $G$-equivariant integration along the fiber in Appendix \ref{equivariant de rham} and refer the reader to \cite{GS} Chapter 10 for the detailed construction in the case of  smooth $G$-manifolds. 
The case of ordinary Kuranishi structures is explained in \cite{kurbook} Chapter 7--9. 
\begin{defn}[$G$-equivariant integration along the fiber on a chart] 
Let $[\mathcal{S}]$ be a $G$-equivariant CF-perturbation transverse to zero on a Kuranishi chart $\mathcal{U}=(U, \E , \psi, s)$, where $\mathcal{S} = \{\mathcal{S}_{\rindex}\mid \rindex \in \Rindex\}$ and 
    \[ 
    \mathcal{S}_{\rindex} = \{ \mathcal{S}_{\rindex}^{\epsilon} = (W_{\rindex} \xrightarrow{\nu_{\rindex}} U_{\rindex}, \tau_{\rindex} , \s_{\rindex}^{\epsilon})\mid \epsilon \in (0,1]\}.  \]
Let $f: U\to L$ be a
$G$-equivariant strongly smooth map on $\mathcal{U}$ such that $f: U\to L$ is strongly submersive with respect to $[\mathcal{S}]$. 
  We define the \textbf{$G$-equivariant integration along the fiber of $f$ via $[\mathcal{S}]$} as follows. 

  Suppose $h \in \Omega_{G,c}^{l}(U)$ and $0<\epsilon\leq \epsilon_0$. 
   Let $ \{\chi_{\mathfrak{r}}
   \mid \mathfrak{r} \in \Rindex\}$ be a partition of unity subordinate to the covering $\{ U_{\rindex} \mid \mathfrak{r} \in \Rindex\}$.
  Define $f_{G!} \left(h; \mathcal{S}^{\epsilon}\right)\in \Omega_{G,c}(L)$ as follows. Let 
  \[ f_{G!} \left(h; \mathcal{S}^{\epsilon}\right)(\xi)  
  =  
  \sum_{\rindex \in \Rindex} f_{!} \left(\chi_{\rindex}h(\xi); \mathcal{S}_{\mathfrak{r}}^{\epsilon}\right),  \qquad \forall \xi \in S(\g^*),  \] 
  where $f_!$ denotes the integration along the fiber in the case of ordinary Kuranishi structures. 
\end{defn}

\begin{defn}[$G$-equivariant integration along the fiber on a good coordinate system]  
\label{integration along the fiber on GC}
Let $\left(\kur,  \widetriangle{\mathcal{U}}\right)$ be a space with a good coordinate system 
\[ \widetriangle{\mathcal{U}} = \left( (\Pindex, \leq ) 
    \, , \, 
    \left\{ \mathcal{U}_{\p}\mid \p\in \Pindex \right\}
    \, , \, 
    \set{\Vec{\alpha}_{\pq}}{ \p, \q\in \Pindex, \q \leq \p } \right). 
\]
and let $L$ be a smooth manifold. 
Let $\widetriangle{f}: \left(\kur, \widetriangle{\mathcal{U}}\right) \to L$ 
be a $G$-equivariant strongly smooth map. 
Let $ \widetriangle{\mathcal{S}}$ be a CF-perturbation such that $\widetriangle{f}$ is strongly submersive with respect to $ \widetriangle{\mathcal{S}}$. Let 
$\lrp{\widetriangle{\mathcal{K}}, \widetriangle{\mathcal{K}}^{++}}$ be a $G$-equivariant support pair on $\left(\kur,  \widetriangle{\mathcal{U}}\right)$.

Let $\widetriangle{h}  = \{ h_{\p} \mid \p \in \Pindex\}$ be a compactly supported $G$-equivariant differential form on $\kur$. 
Define the \textbf{$G$-equivariant integration along the fiber} of $h$ with respect to $\widetriangle{f}, \widetriangle{\mathcal{S}}$ by 
\begin{align*}
(\widetriangle{f}_G)_{!} \left(\widetriangle{h} ,  \widetriangle{\mathcal{S}}^{\epsilon}\right) 
& = 
\sum_{\p\in \Pindex} (f_{\p})_{G!} \left(\chi_{\p}h_{\p} , \restr{\mathcal{S}_{\p}^{\epsilon}}{ K_{\p}(2\delta) 
\cap B_{\delta_2}(\kur)
 }\right).
\end{align*}
\end{defn} 
This definition is independent of the choices of the support pair and the partition of unity \cite{kurbook} Proposition 7.81.

\begin{defn}[$G$-equivariant integration along the fiber on a Kuranishi space] 
\label{integration along the fiber on Kuranishi}
Let $\left(\kur,\widehat{\mathcal{U}}\right)$ be a space with a $G$-equivariant Kuranishi structure $\widehat{\mathcal{U}}$ with corners
and let $N$ be a smooth manifold. 
Let $\widehat{f}: \left(\kur, \widehat{\mathcal{U}}\right) \to N$ 
be a $G$-equivariant strongly smooth map. 
Let $ \widehat{\mathcal{S}}$ be a CF-perturbation such that $\widehat{f}$ is strongly submersive with respect to $ \widehat{\mathcal{S}}$. 

Let $\widehat{h}  = \{h_p \mid p\in \kur \}$ be a compactly supported $G$-equivariant differential form on $\kur$. 

Then, by Lemma 9.10 of \cite{kurbook}, 
the Kuranishi data $\widehat{\mathcal{U}}, \widehat{f}, \widehat{\mathcal{S}}, \widehat{h}$ induce (non-uniquely) some compatible good coordinate system data $\widetriangle{\mathcal{U}}, \widetriangle{f}, \widetriangle{\mathcal{S}}, \widetriangle{h}$ on $\kur$ such that the conditions in Definition \ref{integration along the fiber on GC} are satisfied. 

Define the \textbf{$G$-equivariant integration along the fiber} of $h$ with respect to $\widehat{f}, \widehat{\mathcal{S}}$ by 
\begin{equation}
 \widehat{f}_{G!} \left(\widehat{h} ,  \widehat{\mathcal{S}}^{\epsilon}\right) 
= 
 \widetriangle{f} _{G!} \left(\widetriangle{h} ,  \widetriangle{\mathcal{S}}^{\epsilon}\right).     
\end{equation}

\end{defn}
This definition is independent of the choices by \cite{kurbook} Theorem 9.14. 

\subsection{\texorpdfstring{$G$}{G}-equivariant Stokes' theorem and the smooth correspondences}

\begin{defn}[Codimension-$k$ corner of a manifold]
Let $M$ be a manifold with corners and $k\in \N$. 
Define $S_k(M)$ to be the closure of the set
\[ \set{ x\in M}{
\begin{aligned}
& \text{ there exists a neighborhood } V \text{ of } x \text{ such that } \\ 
& V \text{ is diffeomorphic to } [0,\infty)^{k}\times \R^{n-k} 
 \end{aligned}
 }. 
\]
\end{defn}

\begin{lemdef}[Normalized boundary of a manifold with corners, \cite{kurbook} Lemma 8.2]
For every manifold with corners $V$, there exists a manifold with corners $\partial V$ and a smooth map $\pi: \partial V \to S_1(V)$ such that it induces a double covering map
\[ \restr{\pi}{S_1(V)\setminus S_2(V)}: S_1(V)\setminus S_2(V)\to S_1(V). \] 
\end{lemdef}
The Lemma is proved in \cite{kurbook} Lemma 8.2. We call $\partial V$ the normalized boundary of $V$. 

\begin{defn}[Normalized boundary of an orbifold with corners]
  Let $U$ be an orbifold with corners and $ \{(V_i,\Gamma_i, \varphi_i)\mid i\in I\}$ be an orbifold atlas on $U$. 
  Then the \textbf{normalized boundary} of $U$ is given by 
  \[\partial U = \bigcup_{i\in I}\varphi_i(\partial V_i/\Gamma_i). \]
  \end{defn}

\begin{defn}[Normalized boundary of a Kuranishi space]
Let $(\kur,\widehat{\mathcal{U}})$ be a $G$-equivariant Kuranishi space with corners as in Definition \ref{Gkur structure}. 
The normalized boundary 
$\partial \lrp{\kur, \widehat{\mathcal{U}}}: =\lrp{\partial \kur, \partial \widehat{\mathcal{U}}}$ of $\kur$ is a Kuranishi space with corners, where 
\[ \partial \kur = \bigcup_{p\in \kur} \psi_p \lrp{ s_p^{-1}(0) \cap \partial U_p } \]
and 
\[ \partial \widehat{\mathcal{U}} = 
 \left(\set{ \partial \mathcal{U}_p }{ p\in \partial \kur},   
\set{\restr{\vec{\alpha}_{pq}}{U_{pq} \cap\partial U_q }}{ p\in \partial \kur, q\in \psi_p(\partial U_p) } \right), 
\]
which consists of $G$-equivariant Kuranishi charts
\[ \partial \mathcal{U}_p = \lrp{\partial U_p, \restr{\E_p}{\partial U_p}, \restr{\psi_p}{\partial U_p}, \restr{s_p}{\partial U_p}}, \qquad \forall p\in   \partial \kur
 \]
 and $G$-equivariant Kuranishi coordinate change data
\[ \restr{\vec{\alpha}_{pq}}{U_{pq} \cap\partial U_q } = 
\lrp{\restr{\alpha _{pq}}{U_{pq} \cap\partial U_q }, \restr{\widehat{\alpha}_{pq}}{ \restr{\E_q}{U_{pq} \cap\partial U_q } } }, \qquad \forall p\in  \partial \kur, \quad q\in \psi_p ( \partial U_p). 
 \] 
\end{defn}

We can similarly define the normalized boundary of a good coordinate system. 
\begin{defn}[Normalized boundary of a good coordinate system]
Let $(\kur,\widetriangle{\mathcal{U}})$ be a $G$-equivariant space with good coordinate system with corners as in Definition \ref{GC}. 
The normalized boundary 
$\partial \lrp{\kur, \widetriangle{\mathcal{U}}}: =\lrp{\partial \kur, \partial \widetriangle{\mathcal{U}}}$ of $(\kur,\widetriangle{\mathcal{U}})$ is a good coordinate system with corners, where 
\[ \partial \kur = \bigcup_{\p\in \kur} \psi_{\p} \lrp{ s_{\p}^{-1}(0) \cap \partial U_{\p} } \]
and 
\[ \partial \widehat{\mathcal{U}} = 
 \left(\set{ \partial \mathcal{U}_{\p} }{ {\p}\in \partial \kur},   
\set{\restr{\vec{\alpha}_{\pq}}{U_{pq} \cap\partial U_{\q} }}{ {\p}\in \partial \kur, q\in \psi_{\p}(\partial U_{\p}) } \right), 
\]
which consists of $G$-equivariant Kuranishi charts
\[ \partial \mathcal{U}_{\p} = \lrp{\partial U_{\p}, \restr{\E_{\p}}{\partial U_{\p}}, \restr{\psi_{\p}}{\partial U_{\p}}, \restr{s_{\p}}{\partial U_{\p}}}, \qquad \forall \p\in   \partial \kur
 \]
 and $G$-equivariant Kuranishi coordinate change data
\[ \restr{\vec{\alpha}_{\pq}}{U_{\pq} \cap\partial U_{\q}  } = 
\lrp{\restr{\alpha_{\pq}}{U_{\pq} \cap\partial U_{\q} }, \restr{\widehat{\alpha}_{\pq}}{ \restr{\E_{\q}}{U_{\pq} \cap\partial U_{\q} } } }, \qquad \forall \p\in  \partial \kur, \quad \q\in \psi_p ( \partial U_{\p}). 
 \]    
\end{defn}

\begin{thm}[$G$-equivariant Stokes' Theorem on a good coordinate system with corners] 
\label{Stokes for G-equivariant GC}
Let $(\kur,\widetriangle{\mathcal{U}})$ be a $G$-equivariant space with good coordinate system with corners as in Definition \ref{GC}. 
  Let $N$ be a smooth $G$-manifold and $\widetriangle{f}:\left(\kur, \widetriangle{\mathfrak{\mathcal{U}}}\right)\to N$ be a $G$-equivariant strongly smooth map such that $\widetriangle{f}$ is strongly submersive with respect to some CF-perturbation $\widetriangle{\mathcal{S}}$ of $ \widetriangle{\mathfrak{\mathcal{U}}}$. 
  Then 
  \[ \forall \widetriangle{\eta} = \{ \eta_{\p}\in \Omega_G^l(U_{\p})\mid \p\in \Pindex\} \in \Omega_G^{l}\left(\kur, \widetriangle{\mathfrak{\mathcal{U}}}\right),\] the following equality holds for sufficiently small $\epsilon>0$: 
    \[ d_G\left(\widetriangle{f}_{G!} \left(\widetriangle{\eta} ; \widetriangle{\mathcal{S}}^{\epsilon}\right)\right) =\widetriangle{f}_{G!}\left(d_G\widetriangle{\eta} ; \widetriangle{\mathcal{S}}^{\epsilon}\right) + (-1)^{\dim(\kur,\widetriangle{\mathcal{U}})+l}(\widetriangle{f}_{\partial })_{G!} \left(\widetriangle{\eta}_{\partial } ; \widetriangle{\mathcal{S}}_{\partial } ^{\epsilon}\right), \]
    where $\widetriangle{f}_{\partial } , \widetriangle{\eta}_{\partial } ,  \widetriangle{\mathcal{S}}_{\partial } ^{\epsilon}$ are the restrictions of $\widetriangle{f}  , \widetriangle{\eta} ,  \widetriangle{\mathcal{S}}^{\epsilon}$ to $\partial \left(\kur,\widetriangle{U}\right)$. 
\end{thm}

\begin{proof}[Proof of Theorem \ref{Stokes for G-equivariant GC}] 
For any $\xi \in \g$, we have 
\begin{equation}
\label{Stokes LHS}
d_G\left(\widetriangle{f}_{G!} \left(\widetriangle{\eta} ; \widetriangle{\mathcal{S}}^{\epsilon}\right)\right) (\xi)
   = d\left(\widetriangle{f}_{!} \left(\widetriangle{\eta}(\xi) ; \widetriangle{\mathcal{S}}^{\epsilon}\right)\right)      
\end{equation}
   and 
\begin{align}\label{Stokes RHS}
& \widetriangle{f}_{G!}\left(d_G\widetriangle{\eta} ; \widetriangle{\mathcal{S}}^{\epsilon}\right) (\xi)+ (-1)^{\dim(\kur,\widetriangle{\mathcal{U}})+l}(\widetriangle{f}_{\partial })_{G!} \left(\widetriangle{\eta}_{\partial } ; \widetriangle{\mathcal{S}}_{\partial } ^{\epsilon}\right)(\xi)  \nonumber \\
 = 
 &
 \widetriangle{f}_{!}\left(d (\widetriangle{\eta} (\xi)); \widetriangle{\mathcal{S}}^{\epsilon}\right) + (-1)^{\dim(\kur,\widetriangle{\mathcal{U}})+l}(\widetriangle{f}_{\partial })_{!} \left(\widetriangle{\eta}_{\partial } (\xi) ; \widetriangle{\mathcal{S}}_{\partial } ^{\epsilon}\right) 
\end{align}
Then \autoref{Stokes LHS} equals \autoref{Stokes RHS} by \cite{kurbook} Theorem 8.11 (the usual Stokes' Theorem on a good coordinate system with corners). 
\end{proof}

We can prove the following theorem in a similar way by applying \cite{kurbook} Theorem 9.28 (the usual Stokes' Theorem on a Kuranishi space). 
\begin{thm}[$G$-equivariant Stokes' Theorem on a Kuranishi space with corners] \label{Stokes for Gkur}
Let $(\kur,\widehat{\mathcal{U}})$ be a $G$-equivariant Kuranishi space with corners as in \ref{Gkur structure} and $\widehat{\mathcal{S}}$ be a CF-perturbation of $ \widehat{\mathcal{U}}$. 
Let $L$ be a smooth $G$-manifold and $\widehat{f}: \left(\kur, \widehat{\mathfrak{\mathcal{U}}}\right)\to L$ be a $G$-equivariant map that is strongly submersive with respect to $\widehat{\mathcal{S}}$. 
  Then $\forall \widehat{\eta} \in \Omega_G^{l}\left(\kur, \widehat{\mathfrak{\mathcal{U}}}\right)$, the following equality holds for sufficiently small $\epsilon>0$: 
\[d_G\left(\widehat{f}_{G!}\left(\widehat{\eta }, \widehat{\mathcal{S}}^{\epsilon} \right)\right) 
=
\widehat{f}_{G!}\left(d_G\widehat{\eta }, \widehat{\mathcal{S}}^{\epsilon} \right) + (-1)^{\dim(\kur,\widehat{\mathcal{U}})+l}(\widehat{f}_{\partial })_{G!}\left(\widehat{\eta }_{\partial }, \widehat{\mathcal{S}}_{\partial }^{\epsilon} \right),    \]
where $\widehat{f}_{\partial } , \widehat{\eta}_{\partial } ,  \widehat{\mathcal{S}}_{\partial } ^{\epsilon}$ are the restrictions of $\widehat{f}  , \widehat{\eta} ,  \widehat{\mathcal{S}}^{\epsilon}$ to $\partial \left(\kur, \widehat{\mathfrak{\mathcal{U}}}\right)$. 
\end{thm}

\begin{defn}[Weakly transverse to a smooth manifold map]\label{weakly transverse to a smooth manifold map}
 Let $\left(\kur,\widehat{\mathcal{U}}\right)$ be a $G$-equivariant Kuranishi space with Kuranishi structure 
\[ 
\widehat{\mathcal{U}} =  \left(\set{ \mathcal{U}_p }{ p\in \mathcal{M}},   
\set{\vec{\alpha} _{pq}}{ p\in \mathcal{M}, q\in \image \psi_p } \right) 
\]
Let 
\[\widehat{f} = \set{f_p : U_p\to L }{p\in \kur} :\left(\kur,\widehat{\mathcal{U}}\right) \to L \] be a $G$-equivariant strongly smooth map to a smooth $G$-manifold $N$. 
Let $g: N\to L$ be a $G$-equivariant smooth map between smooth manifolds. 
$\widehat{f}$ is said to be \textbf{weakly transverse to} $g$ if $f_p$ is transverse to $ g$ for each $p\in \kur$. 
\end{defn}
\begin{defn}[Weakly transverse strongly smooth maps]\label{weakly transverse strongly smooth maps}
Let $(\kur_1,\widehat{\mathcal{U}}_1)$, $(\kur_2,\widehat{\mathcal{U}}_2)$ be Kuranishi spaces and $N$ be a smooth manifold. 
Let $\widehat{f}_i: (\kur_i,\widehat{\mathcal{U}}_i) \to N$ be a strongly smooth map for $i\in \{1,2\}$. 
Let $\Delta_L: L\to L\times L$ be the diagonal map. 
Then we say $\widehat{f}_1$ and $\widehat{f}_2$ are \textbf{weakly transverse} if 
the map 
\[\widehat{f}_1 \times \widehat{f}_2:  (\kur_1,\widehat{\mathcal{U}}_1)\times (\kur_2,\widehat{\mathcal{U}}_2) \to L\times L \]
is weakly transverse to $\Delta_L$ as in Definition \ref{weakly transverse to a smooth manifold map}. 
\end{defn} 

\begin{defn}[Fiber product of a Kuranishi structure with a smooth manifold for maps in \ref{weakly transverse to a smooth manifold map}]
\label{Fiber product Kuranishi structure for maps weakly transverse to a smooth manifold map}
Let $\widehat{f}$ be weakly transverse to $g$ as defined in Definition \ref{weakly transverse to a smooth manifold map}. Let $f$ be the map associated with $\widehat{f}$ as in Definition \ref{strongly smooth}. 
We can define 
a \textbf{$G$-equivariant Kuranishi structure on the fiber product} 
\[
\kur\times_L N = \{ (p,m)\in \kur\times N \mid f(p)=g(m) \}. 
\]
Let $(p,m)\in \kur\times_L N$ and $(U_p,\mathcal{E}_{p}, \psi_p, s_p)$ be the Kuranishi neighborhood assigned to $p$ in $\widehat{\mathcal{U}}$. 
Let 
\[  U_p\times_L N = \{ (x,m)\in U_p\times N \mid f_p(x)=g(m)\}.\]
Let $\pi_p: \mathcal{E}_p\to U_p$ be the obstruction bundle of $\kur$ for $p$, $\pr_1: U_p\times_L N \to U_p$ be the projection map to the first factor, and 
\[ 
\pi_{(p,m)}: \pr_1^*\mathcal{E}_p = \set{ ((w,z), e) \in (U_p\times_L N) \times E_p }{ w = \pi_{p}(e)}
\to U_p\times_L N
\] 
be the pullback orbibundle. 
Then $s_p$ induces a section of the pullback orbibundle by
\begin{equation}\label{pullback orbibundle section}
s_{(p,m)}(w,z) = ((w,z), s_p(w)) \qquad \forall (w,z)\in U_p\times_L N.  
\end{equation}
Let
\[ 
\mathcal{U}_{(p,m)} = 
\left(
U_{(p,m)} :=    U_p\times_L N , \quad
\mathcal{E}_{(p,m)}: = pr_1^*\mathcal{E}_p, \quad s_{(p,m)}, \quad \psi_{(p,m)}:=\psi_p\times \id_N 
\right).
\]

Let $ (q, z)\in \psi_{(p,m)}(x,z)$ for some $(x,z) \in U_{(p,m)} $. 
Then we define 
\begin{itemize}
    \item $U_{(p,m), (q,z)} = U_{pq}\times_L N$; 
    \item $\alpha_{(p,m), (q,z)} = \alpha_{pq}\times_N\id_M:  U_{pq}\times_L N \to U_p\times_L N$
    ; and 
 \item $\widehat{\alpha}_{(p,m), (q,z)}:  = \widehat{\alpha}_{pq}\times_L \id_N:  \mathcal{E}_{pq}\times_L N  \to \mathcal{E}_p\times_L N$. 
\end{itemize}
Then the fiber product $\kur\times_L N$ induced by $\widehat{f}$ and $g$ is a $G$-equivariant Kuranishi space with Kuranishi structure 
\[ \widehat{\mathcal{U}} \times_L N = \left( 
\begin{aligned}
& \set{\mathcal{U}_{(p,m)} }
{(p,m)\in \kur\times_L N}, \\    
& \set{\Vec{\alpha}_{pq} = (\alpha_{(p,m), (q,z)},\widehat{\alpha}_{(p,m), (q,z)} )}
{
\begin{aligned}
 & (p,m) \in \kur\times_L N,  \\ 
 & (q,z)\in \image\psi_{(p,m)}
\end{aligned}
}  
\end{aligned}
\right), \]
where 
\[ \mathcal{U}_{(p,m)} 
= 
\left(
 U_p\times_L N,\; \mathcal{E}_{(p,m)}: = pr_1^*\mathcal{E}_p,  \; s_{(p,m)} \text{is given by Eq. \autoref{pullback orbibundle section} }, \; \psi_{(p,m)}:=\psi_p\times \id_M
\right). \]
\end{defn} 
\begin{defn}[Fiber product of Kuranishi structures]
Let $\widehat{f}_1: (\kur_1,\widehat{\mathcal{U}}_1) \to L$ and $\widehat{f}_2: (\kur_2,\widehat{\mathcal{U}}_2) \to L$ be $G$-equivariant weakly transverse strongly smooth maps as in Definition \ref{weakly transverse strongly smooth maps}. 
We define 
\[ (\kur_1,\widehat{\mathcal{U}}_1) 
_{\widehat{f}_1}\times_{\widehat{f}_2}
(\kur_2,\widehat{\mathcal{U}}_2)  
\]
to be the fiber product 
\[ 
\left(
(\kur_1,\widehat{\mathcal{U}}_1) 
\times 
(\kur_2,\widehat{\mathcal{U}}_2) \right)\times_L 
\left(L \times L\right) 
\]
induced by the weakly transverse maps
$\widehat{f}_1\times \widehat{f}_2$ and $\Delta_L$ as in Definition \ref{Fiber product Kuranishi structure for maps weakly transverse to a smooth manifold map}. 
\end{defn}

 \begin{defn}[$G$-equivariant smooth correspondences]
\label{smooth correspondence data}
Let $N_s,N_t$ be oriented compact smooth $G$-manifolds without boundary. A \textbf{$G$-equivariant smooth correspondence} from $N_s$ to $N_t$ is a collection of data 
\[ \mathfrak{X} = \lrp{\kur,\widehat{\mathcal{U}},\widehat{f}_s, \widehat{f}_t},\] where
\begin{itemize}
    \item $\left(\kur,\widehat{\mathcal{U}}\right)$ is an oriented $G$-equivariant Kuranishi space with corners, 
    \item $\widehat{f}_s:\left(\kur,\widehat{\mathcal{U}}\right)\to N_s$ is a $G$-equivariant strongly smooth map as in \ref{strongly smooth}, and 
    \item $\widehat{f}_t:\left(\kur,\widehat{\mathcal{U}}\right)\to N_t$ is a $G$-equivariant strongly smooth and weakly submersive map. 
\end{itemize}
A \textbf{perturbed $G$-equivariant smooth correspondence} from $N_s$ to $N_t$ is a pair $\left(\mathfrak{X}, \widehat{\mathcal{S}}\right)$, where
$\mathfrak{X} = \lrp{\kur,\widehat{\mathcal{U}},\widehat{f}_s, \widehat{f}_t}$ is a smooth correspondence from $N_s$ to $N_t$ and $\widehat{\mathcal{S}}$ is a $G$-equivariant CF-perturbation of $\widehat{\mathcal{U}}$ with respect to which $\widehat{f}_t$ is strongly submersive. 
\end{defn}

\begin{defn}[$G$-equivariant correspondence map]
Let $ \lrp{\mathfrak{X} =( \kur,\widehat{\mathcal{U}}, \widehat{f}_s, \widehat{f}_t), \widehat{\mathcal{S}}}$
be a perturbed smooth correspondence from $N_s$ to $N_t$. 
For $\epsilon>0$ sufficiently small, 
we define the \textbf{$G$-equivariant correspondence map} 
 by 
\[ \gcorr{\lrp{\mathfrak{X}, \widehat{\mathcal{S}} }}: \Omega_G^{\bullet}(N_s)\to \Omega_G^{\bullet+\dim N_t - \dim\left(\kur,\widehat{\mathcal{U}}\right)}(N_t)   \]
associated with $\left(\mathfrak{X},\widehat{\mathcal{S}}\right)$ by 
\[\gcorr{\lrp{\mathfrak{X}, \widehat{\mathcal{S}}}}(\eta) = (\widehat{f}_t)_{G!} \lrp{(\widehat{f}_s)^*_G\eta; \mathcal{S}^{\epsilon} } \qquad \forall \eta \in \Omega_G^{\bullet}(N_s).\]
\end{defn}
Then Stokes' Theorem \ref{Stokes for Gkur} implies the following. 
\begin{prop}[Compare with \cite{kurbook} Proposition 26.16]
\label{Stokes corollary}
\[ d_G\circ \gcorr{\lrp{\mathfrak{X}, \widehat{\mathcal{S}}}} = \gcorr{\lrp{\mathfrak{X}, \widehat{\mathcal{S}}}}\circ  d_G 
+ 
(-1)^{\dim \mathfrak{X} + \deg (\cdot )}\gcorr{\partial \lrp{\mathfrak{X}, \widehat{\mathcal{S}}}}. \]
\end{prop}

\begin{defn}[Composition of smooth correspondences]\label{Composition of smooth correspondences}
Let $N_s,N_t$ be oriented compact smooth manifolds without boundary. For each bi-index $ji\in \{21, 32\}$, let 
$\mathfrak{X}_{ji} = \lrp{\kur_{ji} ,\widehat{\mathcal{U}}_{ji} ,\widehat{f}_{s; _{ji} }, \widehat{f}_{t; _{ji} }}$ be a smooth correspondence from $N_i$ to $N_j$. 
Assume $\widehat{f}_{t,21}$ and $\widehat{f}_{s, 32}$ are weakly submersive. 
We define the \textbf{composition} $ \mathfrak{X}_{31} = \lrp{\kur_{31} ,\widehat{\mathcal{U}}_{31} ,\widehat{f}_{s, {31} }, \widehat{f}_{t, {31} }}$ of $\mathfrak{X}_{21}$ with $ \mathfrak{X}_{32}$ by 
\[ \lrp{\kur_{31} ,\widehat{\mathcal{U}}_{31}} = \lrp{\kur_{32} ,\widehat{\mathcal{U}}_{32}} 
{}_{\widehat{f}_{s, 32}} \times_{\widehat{f}_{t, 21}}\lrp{\kur_{21} ,\widehat{\mathcal{U}}_{21}},  
\]
\[ \widehat{f}_{s,31}: \lrp{\kur_{31} ,\widehat{\mathcal{U}}_{31}} \lra \kur_{21}\xrightarrow{ \widehat{f}_{s, 21}} N_1, \]
\[ \widehat{f}_{s,32}: \lrp{\kur_{31} ,\widehat{\mathcal{U}}_{31}} \lra \kur_{32}\xrightarrow{ \widehat{f}_{t,32}} N_2.   \]
Then $\scorr_{31}$ is again a smooth correspondence. 
\end{defn}
\begin{defn}
    [Composition of perturbed $G$-equivariant smooth correspondences]
\label{Correspondence of perturbed smooth correspondences}
Let $N_s,N_t$ be oriented compact smooth manifolds without boundary. 
For each bi-index $ji\in \{21, 32\}$, let 
$\left(\mathfrak{X}_{ji}, \widehat{\mathcal{S}}_{ji}\right)$ be a perturbed $G$-equivariant smooth correspondence from $N_i$ to $N_j$, where  
$\mathfrak{X}_{ji} = \lrp{ \kur_{ji},\widehat{\mathcal{U}}_{ji},  \widehat{f}_{s,ji}, \widehat{f}_{t,ji}}$ 
Assume $\widehat{f}_{t,21}$ and $\widehat{f}_{s, 32}$ are weakly submersive. 
One can define the composition $ \left(\mathfrak{X}_{31} ,\widehat{\mathcal{S}}_{31}\right)$ of $ \left(\mathfrak{X}_{32},\widehat{\mathcal{S}}_{32}\right) $ with $\left(\mathfrak{X}_{21},\widehat{\mathcal{S}}_{21}\right)$  so that 
\begin{itemize}
    \item $\mathfrak{X}_{31} $ is the composition of $\mathfrak{X}_{21}$ and $\mathfrak{X}_{32}$ as in Definition \ref{Composition of smooth correspondences}, and that 
    \item $\widehat{f}_{t,31}$ is strongly submersive with respect to $\widehat{\mathcal{S}}_{31} = \widehat{\mathcal{S}}_{21}{}_{\widehat{f}_{t,21}}\times_{\widehat{f}_{s,32}} \widehat{\mathcal{S}}_{32}$, whose construction we refer to Definition 10.13 of \cite{kurbook}.   
\end{itemize}
Then $\left(\scorr_{31}, \widehat{\mathcal{S}}_{31}\right)$ is again a perturbed $G$-equivariant smooth correspondence.  
\end{defn}

\begin{prop}
[Equivariant composition formula]
\label{composition formula}
In the case of Definition \ref{Correspondence of perturbed smooth correspondences}, we have 
\[ \gcorr{\left(\mathfrak{X}_{32}, \widehat{\mathcal{S}}_{32}\right)} \circ \gcorr{\left(\mathfrak{X}_{21}, \widehat{\mathcal{S}}_{21}\right)}  
= 
\gcorr{\left(\mathfrak{X}_{31}, \widehat{\mathcal{S}}_{31}\right)}. \]
\begin{center}
\tikzset{every picture/.style={line width=0.75pt}} 

\begin{tikzpicture}[x=0.75pt,y=0.75pt,yscale=-1,xscale=1]
\draw (300,14.4) node [anchor=north west][inner sep=0.75pt]    {$\mathfrak{X}_{31}$};
\draw (351,69.4) node [anchor=north west][inner sep=0.75pt]    {$\mathfrak{X}_{32}$};
\draw (252,69.4) node [anchor=north west][inner sep=0.75pt]    {$\mathfrak{X}_{21}$};
\draw (206,120.4) node [anchor=north west][inner sep=0.75pt]    {$N_{1}$};
\draw (302,120.4) node [anchor=north west][inner sep=0.75pt]    {$N_{2}$};
\draw (400,120.4) node [anchor=north west][inner sep=0.75pt]    {$N_{3}$};
\draw    (302.16,36) -- (278.16,63.5) ;
\draw [shift={(276.84,65.01)}, rotate = 311.11] [color={rgb, 255:red, 0; green, 0; blue, 0 }  ][line width=0.75]    (10.93,-3.29) .. controls (6.95,-1.4) and (3.31,-0.3) .. (0,0) .. controls (3.31,0.3) and (6.95,1.4) .. (10.93,3.29)   ;
\draw    (253.26,91.01) -- (231.11,114.54) ;
\draw [shift={(229.74,116)}, rotate = 313.27] [color={rgb, 255:red, 0; green, 0; blue, 0 }  ][line width=0.75]    (10.93,-3.29) .. controls (6.95,-1.4) and (3.31,-0.3) .. (0,0) .. controls (3.31,0.3) and (6.95,1.4) .. (10.93,3.29)   ;
\draw    (325.55,36) -- (351.09,63.54) ;
\draw [shift={(352.45,65.01)}, rotate = 227.16] [color={rgb, 255:red, 0; green, 0; blue, 0 }  ][line width=0.75]    (10.93,-3.29) .. controls (6.95,-1.4) and (3.31,-0.3) .. (0,0) .. controls (3.31,0.3) and (6.95,1.4) .. (10.93,3.29)   ;
\draw    (376.48,91.01) -- (398.16,114.53) ;
\draw [shift={(399.52,116)}, rotate = 227.33] [color={rgb, 255:red, 0; green, 0; blue, 0 }  ][line width=0.75]    (10.93,-3.29) .. controls (6.95,-1.4) and (3.31,-0.3) .. (0,0) .. controls (3.31,0.3) and (6.95,1.4) .. (10.93,3.29)   ;
\draw    (277.74,91.01) -- (299.89,114.54) ;
\draw [shift={(301.26,116)}, rotate = 226.73] [color={rgb, 255:red, 0; green, 0; blue, 0 }  ][line width=0.75]    (10.93,-3.29) .. controls (6.95,-1.4) and (3.31,-0.3) .. (0,0) .. controls (3.31,0.3) and (6.95,1.4) .. (10.93,3.29)   ;
\draw    (351.5,91.01) -- (327.92,114.59) ;
\draw [shift={(326.5,116)}, rotate = 315] [color={rgb, 255:red, 0; green, 0; blue, 0 }  ][line width=0.75]    (10.93,-3.29) .. controls (6.95,-1.4) and (3.31,-0.3) .. (0,0) .. controls (3.31,0.3) and (6.95,1.4) .. (10.93,3.29)   ;
\end{tikzpicture}
\end{center}
\end{prop}

\begin{proof}
$\forall \eta \in \Omega_G(N_1)$, $\forall \xi \in \g$, 
\begin{align*}
& \gcorr{\left(\mathfrak{X}_{32}, \widehat{\mathcal{S}}_{32}\right)} \circ \gcorr{\left(\mathfrak{X}_{21}, \widehat{\mathcal{S}}_{21}\right)}(\eta)(\xi)  \\
= 
& 
\corr{\left(\mathfrak{X}_{32}, \widehat{\mathcal{S}}_{32}\right)} \circ \corr{\left(\mathfrak{X}_{21}, \widehat{\mathcal{S}}_{21}\right)}(\eta(\xi) ) \\
= 
& 
\corr{\left(\mathfrak{X}_{31}, \widehat{\mathcal{S}}_{31}\right)}(\eta(\xi)) \qquad \text{ by \cite{kurbook} Theorem 10.21} \\
= 
& \gcorr{\left(\mathfrak{X}_{31}, \widehat{\mathcal{S}}_{31}\right)}(\eta)(\xi) . 
\end{align*}
\end{proof}

\newpage 
\appendix
\section{Equivariant de Rham theory} 
\label{equivariant de rham}
Since the construction of a $G$-equivariant $\Ainf$-algebra is built upon the equivariant de Rham theory, we will some of the concepts here. The interested reader should refer to \cite{GS} and \cite{Tu} for more detailed discussion. 

Let $G$ be an $r$-dimensional compact connected Lie group. 
Let $\g$ be the Lie algebra of $G$ and $\g^*$ be its dual. Fix a basis $\{u_1,\ldots, u_r\}$ for $\g^*$. Let $S(\g^*)$ be the symmetric algebra on $\g^*$. Then we can identify it with the polynomial algebra in $r$-variables: $S(\g^*) \cong \R[u_1,\ldots, u_r]$,  which is a graded algebra in the following sense: $\forall i\in \N$, 
\[ S^i(\g^*)\cong \{ P\in \R[u_1,\ldots, u_r]\mid P \text{ is a homogeneous polynomial of degree } i\}.\] 
By a $G$-manifold we mean a smooth manifold with a smooth $G$-action, which induces a diffeomorphism $l_g: L\to L$ for each $g\in G$. 

\begin{defn}[Fundamental vector fields]
Define a Lie algebra homomorphism $\sigma: \g \to \Gamma(TL)$, which assigns every $X\in \g$ its \textbf{fundamental vector field}\index{fundamental vector field} $\underline{X}$ on $L$, as follows. 
 \[\sigma(X)_p=\underline{X}_p = \frac{d}{dt}\eval_{t=0} (e^{-tX}\cdot p) \qquad \forall p\in L.\]
Equivalently, 
by identifying $\g$ with $T_eG$,  we define, 
$\forall p\in L$,
\[ \underline{X}_p = (dj_p)_e(-X),\] where $j_p: G\to L$ is the map $j_p(g)=g\cdot p$ and $e\in G$ is the identity element.
\end{defn}

\begin{defn}[Cartan model for a manifold] \label{Cartan model for a manifold}
Let $G$ be a Lie group. 
The \textbf{Cartan complex} 
for $G$-equivariant differential forms of a smooth $G$-manifold  $L$
is given by 
\[ \Omega_G(L) : = \left( \Omega(L)\otimes S(\g^*) \right)^G ,
\]
where $(\Omega(L),d)$ is the usual de Rham complex of $L$. 
We may regard $ \Omega_G(L)$ as the set of equivariant polynomial maps $\alpha : \g\to \Omega(L)$. 

We consider the grading
\[ \Omega_G(L) : = \bigoplus_{j\in \N}  \Omega_G^j(L), \qquad \quad \text{where } \Omega_G^j(L) = \bigoplus_{0\leq 2l \leq \dim L} \Bigp{\Omega^{j-2l}(L)\otimes S^l(\g^*)}^G. \]
Define the \textbf{equivariant de Rham differential} $d_G: \Omega_G^*(L)\to \Omega_G^{*+1}(L) $, 
\[  (d_G\alpha) (X) = d(\alpha(X)) - \iota_{\underline{X}}( \alpha(X) )\qquad \forall X\in \g, \quad \forall \alpha \in \Omega_G^*(L).\]
We call $(\Omega_G(L;\R), d_G)$ the \textbf{Cartan model} for $L$. 
It is a differential graded algebra. Thus, we can define the \textbf{equivariant de Rham cohomology} to be 
\[ H_G(L) : = \frac{\ker \Omega_G(L)}{\image \Omega_G(L)}.  \]
\end{defn} 
\begin{rmk}
When the group $G$ is abelian, the adjoint action of $G$ on $S(\g^*)$ is trivial and thus $\Omega_G(L) = \Omega(L)^G\otimes S(\g^*)$. 
Moreover, the Cartan complex $\Omega_G(L)$ is a module over the ring $S(\g^*)^G$. 
\end{rmk}

\begin{prop}\label{pullback, pushforward}
 Let $f:M\to N$ be an oriented $G$-equivariant smooth map between oriented $G$-manifolds. 
\begin{enumerate}[i)]
    \item There exists an \textbf{equivariant pullback}\index{equivariant! pullback} map 
    \[ f_G^*: \Omega_G^\bullet(N)\to \Omega_G^\bullet(M), \qquad (f_G^*\alpha)(X) = f^*(\alpha(X)) \qquad \forall \alpha\in \Omega_G(N), \quad \forall X\in \g  \] 
    and $f_G^*$ commutes with $d_G$. 
    \item If $M,N$ are both compact and oriented smooth manifolds with $\dim M - \dim N = d$, then there is an \textbf{equivariant integration along the fiber map}\index{equivariant! integration along the fiber}, also called an \textbf{equivariant pushforward} map
    \[ 
    f_{G!}: \Omega_{G,c}^\bullet(M)\to \Omega_{G}^{\bullet-d}(N), \qquad (f_{G!}\alpha)(X) = f_!(\alpha(X)) \qquad \forall \alpha\in \Omega_{G,c}(M), \quad \forall X\in \g,  
    \] 
The idea is that one can decompose any $G$-equivariant smooth map $f:M\to N$ into the composition of a projection map $\pi: M\times N \to N$ with the graph map $\gamma: M \to M\times N$.  
Fiber integration defines the ordinary pushforward $\pi_!:  \Omega_c^{\bullet}(M\times N) \to \Omega_c^{\bullet- \dim M}(N)$ by $\pi$, which induces $\pi_{G!}$ by $(\pi_{G!}\alpha)(X) := \pi_!(\alpha (X))$. 
For the inclusion map $\gamma: M\to M\times N$, we define $\gamma_{G!}: \Omega_{G,c}^{\bullet}(M) \to \Omega_{G}^{\bullet + \dim N}(M\times N)$ by 
\[ \gamma_{G!}(\alpha) = p_G^* \alpha \wedge \tau \qquad \forall \alpha \in  \Omega_{G,c}^{\bullet}(M), \]
where $p: U\to M $ is a $G$-equivariant tubular neighborhood of $M$ in $M\times N$ and $\tau$ is an equivariant Thom form. 
We refer the reader to \cite{GS} Chapter 10 for a detailed discussion. 
\end{enumerate}
\end{prop} 

\section{Orbifolds}
\label{orbifolds section}
General orbifold theory is discussed in  \cite{RuanOrbifolds} and \cite{kurbook} Chapter 23. The interested reader can read more about equivariant orbifolds and equivariant Kuranishi charts in \cite{FukayaLieGroupoids}. 

\begin{defn}[Orbifold chart]\label{orbifold charts}
Let $U$ be a paracompact Hausdorff topological space. 
An $n$-dimensional effective \textbf{orbifold chart}\index{orbifold!charts} of $U$ is a triple $(V,\Gamma, \varphi)$ such that 
\begin{enumerate}[i)]
    \item $V$ is a smooth $n$-dimensional manifold (possibly with corners);
    \item $\Gamma$ is a finite group acting smoothly and effectively on $V$; 
    \item $\varphi: V\to U$ is a continuous map which induces a homeomorphism $\Bar{\varphi}: V/\Gamma \to \varphi(V)$ onto an open subset $\varphi(V)$ of $U$. 
\end{enumerate} 
Let $x\in U$. We say $(V,\Gamma, \varphi)$ is an \textbf{orbifold chart at $x$} if there exists a point $o_x\in V$ such that $\varphi(o_x)=x$ and $\Gamma \cdot o_x = \{o_x\}$. 
Given an orbifold chart at $x$, the \textbf{tangent space} of the orbifold $U$ at $x$ is given by $T_xU = (T_{o_x}V)/\Gamma$. 

Let $(V,\Gamma, \varphi)$ be an orbifold chart and $p\in V$. An \textbf{orbifold subchart} of $(V,\Gamma, \varphi)$ relative to $p$ is an orbifold chart $(V_p, \Gamma_p,\restr{\varphi}{V_p})$ such that $\Gamma_p$ is the isotropy group of $\Gamma$ at $p$,  $V_p$ is a $\Gamma_p$-invariant open neighborhood of $p$ in $V$, 
and $\restr{\varphi}{V_p}$ induces an injective map $V_p/\Gamma_p\to U$. 
\end{defn} 

\begin{defn}[Embedding of orbifold charts]
Let $f:U_1\to U_2$ be a continuous map between paracompact topological spaces.
     An \textbf{embedding} \index{orbifold!charts!embedding of} $(h,\lambda): (V_1,\Gamma_1,\varphi_1 ) \to (V_2,\Gamma_2,\varphi_2) $ from an orbifold chart of $U_1$ to an orbifold chart of $U_2$ \textbf{relative to $f$} consists of 
   \begin{itemize}
       \item a group isomorphism $h: \Gamma_1\to \Gamma_2$ and 
       \item an $h$-equivariant embedding of manifolds $\lambda: V_1\to V_2$ 
   \end{itemize}
   such that $f\circ \varphi_1 = \varphi_2 \circ \lambda$. 

   An embedding of effective orbifold charts is an \textbf{isomorphism} if $\lambda$ is also a  diffeomorphism. 
   If an isomorphism of two orbifold charts on the same topological space $U$ is taken relative to the identity map, we may simply say it is an isomorphism of orbifold charts. 
\end{defn}

\begin{defn}[Orbifold]
Let $U$ be a paracompact Hausdorff topological space.  An $n$-dimensional (effective) \textbf{orbifold atlas} is a collection 
\[ \{ (V_i,\Gamma_i,\varphi_i) \mid i\in I\}\]
of $n$-dimensional effective orbifold charts such that the following holds. 
\begin{enumerate}[i)]
    \item $ \bigcup\limits _{i\in I}\varphi_i(V_i) = U$. 
    \item If   $\varphi_i(p)=\varphi_j(q) = x$ for some $p\in V_i$, $q\in V_j$, 
    then there exists an isomorphism of orbifold charts 
    \[ (h_{qp},\lambda_{qp}): (V_{i,p}, (\Gamma_{i})_p, \restr{\varphi_{i}}{V_{i,p}}) \to (V_{j,q}, (\Gamma_{j})_q,\restr{\varphi_{j}}{V_{j,q}}), \] 
    called a \textbf{transition map},
    from some orbifold subchart of $(V_i,\Gamma_i,\varphi_i)$ relative to $p$ to some orbifold subchart of $(V_i,\Gamma_i,\varphi_i)$ relative to $q$.  
\end{enumerate}
A $n$-dimensional \textbf{maximal orbifold atlas} $\mathcal{A}$ on $U$ is an $n$-dimensional orbifold atlas such that: $\mathcal{B} \subset \mathcal{A}$ whenever $\mathcal{B}$ and $\mathcal{A}\cup \mathcal{B}$ are both orbifold atlases on $U$. 
An $n$-dimensional effective \textbf{orbifold} $(U, \mathcal{A})$ is a paracompact Hausdorff topological space $U$ equipped with an $n$-dimensional maximal orbifold atlas $\mathcal{A}$ on $U$.  
\end{defn}

\begin{defn}[Embedding of orbifolds]
A topological map $f: U_1\to U_2$ between effective orbifolds 
is an \textbf{embedding of orbifolds} if, $\forall x\in U_1$, 
there exists an embedding of effective orbifold charts 
$(h,\lambda)$ 
from an orbifold chart $(V_x,\Gamma_x,\varphi_x)$ of $U_1$ at $x$ to an orbifold chart $(V_y',\Gamma_y',\varphi_y')$ of $U_2$ at $y=f(x)$ relative to $f$.  
An embedding of orbifolds is a \textbf{diffeomorphism} of orbifolds if it is a homeomorphism. 
Let $U$ be an effective orbifold. 
We denote the group of diffeomorphisms from $U$ to itself by $\Diff(U)$, which is a topological group under the compact open topology. 
\end{defn}

\begin{defn}[Orbibundle chart]
Let $U,\mathcal{E}$ be orbifolds and $\pi: \mathcal{E} \to U$ be a continuous surjective map between the underlying topological spaces. 
An \textbf{orbibundle chart} is a quintuple $(V,E,\Gamma, \varphi, \widehat{\varphi})$, where 
\begin{itemize}
    \item $(V,\Gamma, \varphi)$ is an orbifold chart on $U$,
    \item  $E$ is a finite-dimensional vector space with a linear $\Gamma$-action, and 
    \item $(V\times E,\Gamma, \widehat{\varphi})$ is an orbifold chart on $\mathcal{E}$, where $\Gamma$ acts on $V\times E$ diagonally  
\end{itemize}  
such that
\begin{enumerate}[i)]
    \item $\pi \circ \widehat{\varphi} = \varphi \circ \pr_1$. 
\item $\widehat{\varphi}$ induces a homeomorphism on the quotients $\overline{\widehat{\varphi}}: (V\times E )/\Gamma\to \pi^{-1}(\overline{\varphi}(V/\Gamma))$ such that 
$\overline{\varphi}^{-1}\circ \pi \circ \overline{\widehat{\varphi}} = \overline{\pr}_1$.  
\[ 
\xymatrix{
   &  V\times E\ar[r]^{\widehat{\varphi}} \ar[d]_{\pr_1} & \mathcal{E} \ar[d]^{\pi}\\ 
   & V \ar[r]_{\varphi} & U
    }  \quad
\xymatrix{
   &  (V\times E)/\Gamma\ar[rr]^{\overline{\widehat{\varphi}}} \ar[dr]_{\overline{\pr}_1} & & \pi^{-1}(\overline{\varphi}(V/\Gamma)) \ar[dl]^{\overline{\varphi}^{-1}\circ \pi}\\ 
   &  & V/\Gamma &
    }. 
\]
\end{enumerate}

\end{defn}
\begin{defn}[Embedding of orbibundle charts]
Let $\mathcal{E}_1, U_1,\mathcal{E}_2, U_2$ be orbifolds and $\pi_1 :\mathcal{E}_1\to U_1$, $\pi_2 :\mathcal{E}_2\to U_2$ be continuous surjective maps between the underlying topological spaces. 
Let $f: U_1\to U_2,\widehat{f}: \mathcal{E}_1\to \mathcal{E}_2 $ be continuous maps such that $ f  \circ \pi_1= \pi_2\circ \widehat{f}$. 
    An \textbf{embedding of orbibundle charts} is a triple 
    \[ (h, \lambda, \widehat
{\lambda}): (V_1,E_1,\Gamma_1, \varphi_1, \widehat{\varphi}_1) \to  (V_2,E_2,\Gamma_2, \varphi_2, \widehat{\varphi}_2)  \] 
from an orbibundle chart on $\pi_1: \mathcal{E}_1\to U_1$ 
to an orbibundle chart on $\pi_2: \mathcal{E}_2\to U_2$ 
relative to $(f,\widehat{f})$ such that the following holds.   
\begin{enumerate}[i)]
    \item $(h,\lambda): (V_1,\Gamma_1, \varphi_1)\to  (V_2,\Gamma_2, \varphi_2)$ is an embedding of orbifold charts relative to $f$.
    \item $(h,\widehat{\lambda}):  (V_1\times E_1,\Gamma_1,  \widehat{\varphi}_1) \to (V_2\times E_2,\Gamma_2,  \widehat{\varphi}_2)$ is an embedding of orbifold charts relative to $\widehat{f}$. 
    \item If $\pi_{V_1}:V_1\times E_1\to V_1$ and $\pi_{V_2}:V_2\times E_2\to V_2$ are projection maps to their first factors, then 
    \begin{enumerate}[a)]
        \item $ \lambda \circ \pi_{V_1}  = \pi_{V_2}\circ \widehat{\lambda}$ and 
        \item for each $x\in V_1$, $\restr{\widehat{\lambda}}{\{x\}\times E_1} : \{x\}\times E_1  \to \{\lambda(x)\}\times E_2$ is a linear embedding.
    \end{enumerate}
\end{enumerate}

An embedding of orbibundle charts is an \textbf{isomorphism} if $(h,\lambda)$ and $(h,\widehat{\lambda})$ define isomorphisms of orbifold charts and, 
for each $x\in V$, $\restr{\widehat{\lambda}}{\{x\}\times E_1} : \{x\}\times E_1 \to \{\lambda(x)\}\times E_2$ is a linear isomorphism.

\end{defn}

\begin{defn}[Orbibundle subchart at a point]
    Let $(V,E,\Gamma, \varphi, \widehat{\varphi})$ be an orbibundle chart. If $(V_p,\Gamma_p, \restr{\varphi}{V_p})$ is an orbifold subchart of $ (V,\Gamma, \varphi)$ relative to $p\in V$, then 
    $(V_p,E,\Gamma_p, \restr{\varphi}{V_p}, \restr{\widehat{\varphi}}{V_p\times E})$ 
    is also an orbibundle chart, called an \textbf{orbibundle subchart of $(V,E,\Gamma, \varphi, \widehat{\varphi})$ at $p$}. \index{orbibundle!chart!subchart at $p$} 
\end{defn}

\begin{defn}[Orbibundle atlas]
    Let $U,\mathcal{E}$ be orbifolds and $\pi: \mathcal{E} \to U$ be a continuous surjective map between the underlying topological spaces. 
An \textbf{orbibundle atlas} is a locally finite collection of orbibundle charts 
\[\left \{ (V_i,E_i,\Gamma_i, \varphi_i, \widehat{\varphi}_i) \,\middle|\, i\in I \right \}\]
such that 
\begin{enumerate}[i)]
    \item $\left \{ (V_i,\Gamma_i, \varphi_i) \,\middle|\, i\in I \right \}$ is an orbifold atlas on $U$; 
    \item $\left \{ (V_i\times E_i,\Gamma_i, \widehat{\varphi}_i) \,\middle|\, i\in I \right \}$ is an orbifold atlas on $\mathcal{E}$;
    \item for any $i,j\in I$, 
    if $p\in V_i$, $q\in V_j$ satisfy $ \varphi_i(p) = \varphi_j(q)$, then there exist an isomorphism 
    \[ (h_{qp}, \lambda_{qp}, \widehat{\lambda}_{qp}): \left(V_{i,p},E_i,\Gamma_{i,p}, \restr{\varphi}{V_{i,p}}, \restr{\widehat{\varphi}_i}{V_{i,p}\times E_i}\right)\to \left(V_{j,q},E_j,\Gamma_{j,q}, \restr{\varphi}{V_{j,q}}, \restr{\widehat{\varphi}_j}{V_{j,q}\times E_j}\right)\] between
    an orbibundle subchart of  $(V_i,E_i,\Gamma_i, \varphi_i, \widehat{\varphi}_i) $ at $p$ and an orbibundle subchart of  $(V_j,E_j,\Gamma_j, \varphi_j, \widehat{\varphi}_j) $ at $q$.
\end{enumerate}
\end{defn}
\begin{defn}[Embedding of orbibundles]
For $\alpha \in \{1,2\}$, let 
\[\left (\mathcal{E}_\alpha \xrightarrow{\pi_{\alpha}} U_\alpha, \quad \widehat{\mathcal{V}}_{\alpha}=\left \{ \left (V_i^{\alpha},E_i^{\alpha},\Gamma_i^{\alpha}, \varphi_i^{\alpha}, \widehat{\varphi}_i^{\alpha}\right)\,\middle|\, i\in I_{\alpha}\right \}\right)\] 
be a pair such that $\widehat{\mathcal{V}}_{\alpha}$ is an orbibundle atlas on $\mathcal{E}_\alpha \xrightarrow{\pi_{\alpha}} U_\alpha$. 
An \textbf{embedding of orbibundles} $\left (f, \widehat{f}\right): \left(\mathcal{E}_1\xrightarrow{\pi_{1}} U_1, \widehat{\mathcal{V}}_1\right)\to \left(\mathcal{E}_2 \xrightarrow{\pi_2} U_2, \widehat{\mathcal{U}}_2\right)$ consists of two orbifold embeddings $U_1\xrightarrow{f}U_2$ and $\mathcal{E}_1\xrightarrow{\widehat{f}}\mathcal{E}_2$ such that the follows holds. 
\begin{enumerate}[i)]
    \item For any $i\in I_{1}, j\in I_2$ and $p\in V_i^1$, $q\in V_j^2$ with $f(\varphi_i^1(p)) = \varphi_j^2(q)$, there exists an embedding $\left(h_{qp}, f_{qp}, \widehat{f}_{qp}\right)$
   relative to 
    $\left(f, \widehat{f}\right)$, 
     from an orbibundle subchart of $\left(V_i^{1},E_i^{1},\Gamma_i^{1}, \varphi_i^{1}, \widehat{\varphi}_i^{1}\right)$ at $p$ 
     to an orbibundle subchart of $\left (V_j^{2},E_j^{2},\Gamma_j^{2}, \varphi_j^{2}, \widehat{\varphi}_j^{2}\right)$ at $q$. 
    \item $ \pi_2\circ \widehat{f} = f\circ \pi_1$. 
\end{enumerate}

Two orbibundle atlases $\widehat{\mathcal{V}}_1, \widehat{\mathcal{V}}_2$ on  $\E\xrightarrow{\pi} U$ are \textbf{equivalent} if the pair of identity maps 
$(\id, \widehat{\id}): \left(\mathcal{E}\xrightarrow{\pi} U, \widehat{\mathcal{V}}_1\right)\to \left(\mathcal{E} \xrightarrow{\pi} U, \widehat{\mathcal{V}}_2\right)$ is an embedding of orbibundles and $\id, \widehat{\id}$ are diffeomorphisms of orbifolds.
\end{defn}

\begin{defn}[Orbibundle]
    An \textbf{orbibundle} $(\E\xrightarrow{\pi} U,  [\widehat{\mathcal{V}}])$ consists of a continuous surjective map $\E\xrightarrow{\pi} U$ between the underlying topological spaces of two orbifolds and an equivalence class of orbibundle atlases on $\pi$. 
\end{defn}

\begin{defn}[$G$-action on an orbifold]
\label{G-orbifold}
Let $G$ be a compact connected Lie group and $U$ be an effective orbifold. 
A continuous group homomorphism 
\[ 
\alpha: G\to \Diff (U), 
\qquad 
\alpha(g)(x) = g\cdot x \quad \forall x\in U,
\]
is a \textbf{smooth action} of $G$ on $U$,
if $\forall g\in G, \forall x\in U$ there exist 
\begin{itemize}
     \item an open neighborhood $R$ of $g$ in $G$, 
    \item orbifold charts $(V_x,\Gamma_x,\varphi_x )$ at $x$ and $(V_y',\Gamma_y',\varphi_y' )$ at $y=g\cdot x$ of $U$, 
    \item a group isomorphism $h_{g,x}: \Gamma_x\to \Gamma_y'$  , and  
    \item a smooth map $f_{g,x}: R\times V_x\to V_y'$
\end{itemize}
    such that the following holds. 
    \begin{enumerate}[i)]
        \item $f_{g,x}$ is $h_{g,x}$-equivariant: 
        \[ f_{g,x}(\gamma \cdot p) = h_{g,x}(\gamma)  f_{g,x}(p) \quad \forall p\in V_x.  \]
        \item $\varphi_y'(g\cdot p) = g\cdot \varphi_x(p)$ for all $p\in V_x$. 
    \end{enumerate}
An effective orbifold equipped with a smooth $G$-action is called a \textbf{$G$-orbifold}. 
\end{defn}

\begin{defn}[$G$-equivariant orbibundle] 
Let $\pi: \E \to U$
be an orbibundle between $G$-orbifolds and
\[  \left \{ (V_i,E_i,\Gamma_i, \varphi_i, \widehat{\varphi}_i) \,\middle|\, i\in I \right \}  \]
be an orbibundle atlas on $\pi$. 
Then, in particular, $\forall g \in G, x\in \E$, the $G$-action on $\E$ induces some smooth map $f_{g,x}: R\times (V_x\times E_x)\to V_{g\cdot x}\times E_{g\cdot x}$. 
We say $\pi$ is a \textbf{$G$-equivariant orbibundle} if the following holds. 
\begin{enumerate}[i)]
\item $\pi$ is $G$-equivariant:  
$ \pi(g\cdot x) = g\cdot \pi(x)  \qquad \forall g\in G, \quad  \forall x\in U.$
\item For each $g\in G, p\in V_1$, the map
\[ E_x\to E_{g\cdot x}, \qquad v \mapsto \pr_2\circ f_{g,x}(g,x,v)  \]
is linear. 
\end{enumerate}

A \textbf{$G$-equivariant section} of a $G$-equivariant orbibundle $\pi: \E \to U$ is an orbifold embedding $s: U\to \E$ such that 
$\pi \circ s = \id_{U}$ and $g\cdot s(x) = s(g\cdot x)$ for all $g\in G$, $x\in U$. 

An \textbf{embedding of $G$-equivariant orbibundles} $\left (f, \widehat{f}\right): \left(\mathcal{E}_1\xrightarrow{\pi_{1}} U_1, \widehat{\mathcal{V}}_1\right)\to \left(\mathcal{E}_2 \xrightarrow{\pi_2} U_2, \widehat{\mathcal{U}}_2\right)$ 
is an embedding of orbibundles such that $f, \widehat{f}$ are both $G$-equivariant. 
\end{defn} 

\begin{defn}[Differential form on an orbifold]
        A \textbf{differential form} on an orbifold $\lrp{U, \left \{ (V_i,\Gamma_i, \varphi_i) \,\middle|\, i\in I \right \} }$ 
        is a collection 
        \[ \eta = \{ \eta_i \in \Omega(V)^{\Gamma_i}\mid i\in I \}\]
       which associates each orbifold chart $(V_i,\Gamma_i, \varphi_i) $ with a $\Gamma_i$-invariant differential form $\eta_i \in \Omega(V_i)$ such that the following holds. 
        \begin{enumerate}[i)]
            \item If $(h,\lambda): (V_i,\Gamma_i, \varphi_i)\to  (V_j,\Gamma_j, \varphi_j)$ is an isomorphism of orbifold charts, then 
            $ \lambda^*\eta_j =\eta_i  $.
            \item If $\mathfrak{B}_j =(V_j,\Gamma_j, \varphi_j)$ is an orbifold subchart of $(V_i,\Gamma_i, \varphi_i)$, then $\eta_j =\restr{\eta_i}{\mathfrak{B}_j}$. 
        \end{enumerate}
    Denote the set of differential forms on $U$ by $\Omega(U)$.  
    An \textbf{orientation} on an orbifold $U$ is a choice of a differential form $\eta \in \Omega(U)$ such that $\eta_i$ never vanishes. 
\end{defn}

\begin{defn}[Equivariant differential forms an orbifold]
Let $U$ be a $G$-orbifold and $\Omega(U)$ be the space of differential forms on $U$. 
The set of \textbf{$G$-equivariant differential forms} on $U$ is given by 
\[
\Omega_G(U):= \lrp{ \Omega(U) \otimes S(\g^*)}^G.
\]
\end{defn}

\section{Some rigid analytic geometry}
In this appendix, we review some basic definitions and properties of rigid analytic geometry. 
We refer the interested reader to \cite{bgr}, \cite{bosch},  \cite{fresnel2012rigid}, \cite{lutkebohmert}, \cite{nicaise}, \cite{boschlutkebohmertstable}, and \cite{conrad}.  
For tropical analytic geometry and polyhedral domains (of the form $\trop^{-1}(\Delta)$, where $\Delta$ is a polyhedron), we refer the 

We will work over an algebraically closed field $\Lambda$, which is a non-Archimedean field, namely, a field that is complete with respect to a non-Archimedean absolute value (see Definition \ref{NA absolute value}). 
Note that the general rigid analytic geometry concerns a non-Archimedean field $K$, which may not be algebraically closed. 

\begin{defn}[Non-Archimedean valuation]
\label{NA valuation}
A function $\val: \Lambda \to \R \cup \{\infty\}$ 
is a \textbf{non-Archimedean valuation} on $\Lambda$ if the following holds. 
\begin{enumerate}[i)]
    \item $\val(a) = \infty$ if and only if $a=0$. 
    \item $\val(ab) = \val (a) + \val(b)$ for all $a,b\in \Lambda$. 
    \item $\val(a+b)\geq \min \{\val (a) , \val(b)\}$ for all $a,b\in \Lambda$.  
\end{enumerate}
\end{defn}
\begin{defn}[Non-Archimedean absolute value]
\label{NA absolute value}
A function $|\cdot|: \Lambda \to \R_{\geq 0}$ 
is a \textbf{non-Archimedean absolute value} on $\Lambda$ if the following holds. 
\begin{enumerate}[i)]
    \item $|a| = 0$ if and only if $a=0$. 
    \item $|ab|  = |a|  + |b| $ for all $a,b\in \Lambda$. 
    \item \label{NA property} $|a+b| \leq  \max\{|a| , |b|\}$ for all $a,b\in \Lambda$.  
\end{enumerate}    
\end{defn}
From \ref{NA property} one can see that $|n\cdot a| \leq |a|$ for any $n\in \N$, $|a|>0$, which shows that the Archimedean property does not hold for a non-Archimedean absolute value. 
We can associate a non-Archimedean absolute value to a field with non-Archimedean valuation by defining 
\[ |a| = e^{-\val(a)},\] where $e$ is Euler's number.  

We now introduce some major players of rigid analytic geometry. 
\begin{defn}[Closed unit polydisc $B_{\Lambda}^n$]
The \textbf{closed unit polydisc} $B_{\Lambda}^n$ is defined by
   \[B_{\Lambda}^n = \set{(x_1, \ldots, x_n) \in \Lambda^n}{|x_i|\leq 1 \quad \forall 1\leq i \leq n}. \]
\end{defn} 
The set of all power series that converge on $B_{\Lambda}^n$ is called the Tate algebra. 
\begin{defn}[Tate algebra]
    Let $n\geq 1$. The \textbf{Tate algebra} in $n$ variables is defined by
    \[ T_n = \set{\sum_{\pmb{c}\in \N^n} a_{\pmb{c}}x^{\pmb{c}}\in \Lambda\pseries{ x_1, \ldots, x_n} }{a_{\pmb{c}}\in \Lambda, \lim_{|\pmb{c}|\to \infty}|a_{\pmb{c}}| =0 }, \]
    where 
    if $\pmb{c}=(c_1,\ldots, c_n)$ is a multi-index, then $x^{\pmb{c}} = x_1^{c_1}\cdots x_n^{c_n}$ and 
    $|\pmb{c}| = c_1+\cdots + c_n$.  
    Equivalently, 
    \[ T_n = \set{\sum_{\pmb{c}\in \N^n} a_{\pmb{c}}x^{\pmb{c}} \in \Lambda\pseries{ x_1, \ldots, x_n} }{ \lim_{|\pmb{c}|\to \infty} \val (a_{\pmb{c}})=\infty }.  \]
    We denote it by $\Lambda \pair{x_1,\ldots, x_n}$. 
    In particular, $T_0 = \Lambda$. 
\end{defn} 

\begin{prop}[Properties of the Tate algebra]
\label{Tate algebra properties}
    Let $n\geq 1$. 
    The Tate algebra $T_n$ is normal and is a Noetherian integral domain. 
\end{prop}

\begin{defn}[$\Lambda$-affinoid algebra and $\Lambda$-affinoid space]
    A $\Lambda$-algebra that is isomorphic to $T_n/I$ for some ideal $I$ in $T_n$ is called a \textbf{$\Lambda$-affinoid algebra} and the maximal spectrum of $A$, denoted by
    \[ \Sp A = \{ \mathfrak{a}\subset A\mid \mathfrak{a} \text{ is a maximal ideal of }A \},\]
    is called a \textbf{$\Lambda$-affinoid space}. 
    A $\Lambda$-algebra morphism $f: A\to B$ between two $\Lambda$-affinoid algebras is called a \textbf{morphism of $\Lambda$-affinoid algebras}. 
    If $f^{\#}: \Sp B\to \Sp A$ is induced by a morphism $f: A\to B$ of $\Lambda$-affinoid algebras such that 
    \[ f^{\#} (\m) = f^{-1}(\m) \quad \forall \m \in \Sp B,\]
    then $f^{\#}$ is called a \textbf{morphism of $\Lambda$-affinoid spaces}. 
\end{defn}

\begin{defn}[$\Lambda$-affinoid subdomain] 
Let $f : A\to B$ be a $\Lambda$-affinoid algebra homomorphism.  
    Then $f^{\#}(\Sp B)$ is a \textbf{$\Lambda$-affinoid subdomain} of $\Sp A$ if the following universal property holds. Whenever $g: A\to C$ is a $\Lambda$-affinoid algebra homomorphism such that $g^{\#}(\Sp C) \subset f^{\#}(\Sp B)$, there exists a unique morphism $h: A\to C$ of $\Lambda$-affinoid algebras such that $g^{\#} = f^{\#} \circ h^{\#}$. 
    
    We call such a morphism $f^{\#}$ of $\Lambda$-affinoid spaces an \textbf{open immersion} if it satisfies the universal property as above.  
\end{defn}
By \cite{bgr} Proposition  7.7.2/1, if $f^{\#}:\Sp B\to \Sp A$ satisfies the universal property, then it is injective. 
So the affinoid subdomain $f^{\#}(\Sp B) \subset \Sp A$ can be identified with the affinoid space $\Sp B$. 
\begin{defn}[Grothendieck topology]
  A \textbf{Grothendieck topology (G-topology)}
 consists of 
 \begin{itemize}
     \item a category $\Cat$, called the admissible open subsets, and, 
     \item for each $U\in \operatorname{ob} \Cat$, a set 
     $ \operatorname{Cov} U$, 
     called the set of  admissible coverings of $U$, which consists of 
    families of the form $(U_i\xrightarrow{\Phi_i}U)_{i\in I}$, where each $\Phi_i$ is a morphism in $\Cat$,
 \end{itemize}
 such that the following holds. 
 \begin{enumerate}[i)]
     \item If $\Phi: U'\to U$ is an isomorphism in $\operatorname{mor} \Cat$, then the family $(U' \xrightarrow{\Phi} U) \in \operatorname{Cov} U$. 
     \item If $(U_i \xrightarrow{\Phi_i} U)_{i\in I} \in \operatorname{Cov} U$ and $(U_{ij} \xrightarrow{\Phi_{ij}} U_i)_{j\in J} \in \operatorname{Cov} U_i$, 
     then $( U_{ij} \xrightarrow{\Phi_{ij}} U_i\xrightarrow{\Phi_i}U)_{i\in I, j\in J} \in \operatorname{Cov} U$. 
     \item \label{intersection of admissible opens}
     If $(U_i \xrightarrow{\Phi_i} U)_{i\in I} \in \operatorname{Cov} U$  and $V\to U$ belongs to $\operatorname{mor} \Cat$, 
     then the fiber products $U_i\times _U V$ exist in $\Cat$ and 
     $(U_i\times _U V\to V)_{i\in I}$ belongs to $\operatorname{Cov} V$. 
 \end{enumerate}
\end{defn}
A category with a Grothendieck topology is called a \textbf{site}. 
It's a generalization of a topological space. A topological space can be viewed as a site whose admissible open subsets are open subsets of $X$ and the set of admissible coverings of an open subset $U$ of $X$ consists of the open covers of $U$.  
 
\begin{defn}[Weak Grothendieck topology on a $\Lambda$-affinoid space]
   The \textbf{weak G-topology} on a $\Lambda$-affinoid space $X$ is defined as follows.  
   The \textit{admissible open} subsets are the affinoid subdomains of $X$ and the \textit{admissible coverings} of an affinoid subdomain $U\subset X$ are the coverings of $U$ by finitely many affinoid subdomains of $X$. 
\end{defn}

\begin{defn}[Structure sheaf]
Let $X$ be an affinoid space. 
We define a structure presheaf $\hol_X$ on the site $X$ as follows.  
Let $U\subset X$ be an affinoid subdomain of $X$ which is the image of a morphism $f^{\#}: \Sp B\to \Sp A$ of affinoid spaces. 
 $\hol_X(U) = B$. 
 By Tate's acyclicity theorem \cite{bosch} 4.3/Theorem 1, $\hol_X$ is a sheaf. 
\end{defn}

\begin{defn}[Strong Grothendieck topology on a $\Lambda$-affinoid space]
     The \textbf{strong G-topology} on a $\Lambda$-affinoid space $X$ is defined as follows. 
\begin{itemize}
    \item  A subset $U\subset X$ is \textit{admissible open} if
    there is a covering 
    $U = \bigcup\limits_{i\in I} U_i$ of $U$ by (not necessarily finitely many) affinoid subdomains $U_i$ of $X$ such that,
   for any $\Lambda$-affinoid space morphism $\varphi: Y\to X$ with $\varphi(Y)\subset U$, 
   the covering $\{ \varphi^{-1}(U_i)\}_{i\in I}$ 
   has a refinement which is a covering by finitely many affinoid subdomains of $Y$. 
   \item A covering $U = \bigcup\limits_{j\in J}U_j$ of an admissible open subset $U$ of $X$ is \textit{admissible} if, for each affinoid space morphism $\varphi: Y\to X$ with $\varphi(Y)\subset U$, 
  the covering $\{ \varphi^{-1}(U_i)\}_{i\in I}$ has a refinement which is a covering by finitely many affinoid subdomains of $Y$. 
\end{itemize}
\end{defn}

Any sheaf defined with respect to the weak G-topology extends uniquely to a sheaf with respect to the strong G-topology. 
In particular, $\hol_X$ extends to a sheaf with respect to the strong G-topology.  

\begin{defn}[$\Lambda$-rigid analytic space]
    A $\Lambda$-rigid analytic space is a pair $(X,\hol_X)$ such that the following holds. 
\begin{enumerate}[i)]
\item $X$ can be endowed with a G-topology which satisfies the completeness conditions in the following sense. 
\begin{enumerate}[i)]
        \item [(G0)] $\emptyset, X$ are admissible open. 
        \item [(G1)]
        If $U = \bigcup_{i\in I} U_i$ is an admissible covering and $V\subset U$ is a subset such that $V\cap U_i$ is admissible open for all $i$, then $V$ is admissible open in $X$. 
        \item [(G2)] 
        If $U, U_i$ are admissible open for all $i\in I$ and $U = \bigcup\limits_{i\in I} U_i$ is an admissible covering which admits an admissible refinement, 
        then $(U_i)_{i\in I}$ is an admissible covering of $U$. 
        \end{enumerate}
\item $\hol_X$ is a sheaf of $\Lambda$-algebras such that there exists an admissible covering $X = \bigcup\limits_{i\in I} X_i$ where each $(X_i, \hol_X|_{X_i})$ is a $\Lambda$-affinoid space.
    \end{enumerate}
\end{defn} 
\begin{prop}[\cite{bosch} 5.1/Proposition 7] 
\label{strict norm inequalities define admissible opens}
Suppose $X$ is an affinoid space with the strong Grothendieck topology and $f\in \hol_X(X)$. 
Let 
\[ U = \{x\in X \mid |f(x)|<1\}, \quad U' = \{ x\in X \mid |f(x)|> 1\}, \quad U'' =  \{ x\in X \mid |f(x)|> 0\}. \]
Then any finite union of the sets of these types is admissible open in the strong G-topology. 
Any finite covering by finite unions of such sets is admissible.
\end{prop}

\begin{defn}[Dimension of a rigid analytic space]
Let $X$ be a rigid analytic space and $x\in X$. 
Define 
\[ \hol_{X,x} :=\lim_{\substack{\longrightarrow\\ U\ni x }} \hol_X(U), \]
where the direct limit is taken over all admissible open subsets $U$ of $X$ containing $x$. 

The dimension $\dim_x X$ of a rigid analytic space $X$ at a point $x\in X$ is defined to be the Krull dimension of $\hol_{X,x}$. 
The \textbf{dimension} $\dim X$ of a rigid analytic space $X$ is defined to be 
\[ \dim X: = \sup_{x\in X} \dim_x X = \sup _{x\in X} \dim \hol_{X,x}. \]
In particular, the dimension of a $\Lambda$-affinoid space $X =\Sp A$ is the Krull dimension of $A$. 
\end{defn} 

For any $\Lambda$-scheme $X$ of finite type, one can associate a $\Lambda$-rigid analytic space $X^{an}$ with $X$ such that the underlying set of $X^{an}$ is the set of closed points of $X$. 
    
\begin{defn}[Analytification]
Let $(X,\hol_X)$ be a $\Lambda$-scheme of finite type. 
$X^{an}$ is the rigid analytification of $X$ if there exists a natural morphism of locally ringed sites 
$i: (X^{an}, \hol_{X^{an}})\to (X,\hol_X)$ 
which maps the underlying set of $X^{an}$ bijectively to the set of closed points of $X$ such that the following universal property holds. 
For any $\Lambda$-rigid analytic space $(Y, \hol_Y)$ and a morphism of locally ringed sites $j: Y\to X$  there exists a unique morphism of rigid analytic spaces $j_Y: Y\to X^{an}$ such that $j = i\circ j_Y$. 

Let $r\in \Lambda$ be an element with $|r|>1$. 
One can describe the rigid analytification of a scheme of the form 
\[ X=\Spec (\Lambda [x_1,\ldots, x_n]/(f_1,\ldots, f_k))\] 
explicitly by 
\[ X^{an} = \bigcup_{i\in \N}\Sp  \frac{\Lambda \pair{ r^{-i} x_1,\ldots, r^{-i}x_n} }{(f_1,\ldots, f_k) }. \]
\end{defn} 
Let $\Lambda_0$ be the valuation ring of $\Lambda$. Let $\mathfrak{a}$ be its ideal of definition.

\begin{defn}[Genus of a rigid analytic curve]
Let $W$ be a smooth proper rigid analytic variety over $\Lambda$ of dimension $1$ which is geometrically connected. 
Then the \textbf{genus} of $W$ is defined to be $g : = \dim_{\Lambda}H^1(W,\hol_W)$. 
\end{defn}

\section{Some tropical geometry}
Let $\Lambda$ be a field with a valuation and $\Lambda^* = \Lambda\setminus \{0\}$. Let 
\[ \Lambda_0 = \{ y \in \Lambda \mid \val(y)\geq 0 \},\quad  \Lambda_+ = \{ y \in \Lambda \mid \val(y) >  0 \}, \]
and $k = \Lambda_0/\Lambda_+$ be the residue field. 

Recall there is a tropicalization map defined on the algebraic $n$-torus $(\Lambda^*)^n$: 
\[ \trop: (\Lambda^*)^n \to \R^n, \qquad (y_1,\ldots, y_n)\mapsto (\val(y_1),\ldots, \val(y_n)). \]
\begin{defn}[Tropicalization of a Laurent polynomial]
For any Laurent polynomial 
\[ f = \sum_{c\in \Z^n} a_c y^c \in \Lambda[y_1^{\pm 1}, \ldots, y_n^{\pm 1}], \]
where $y^c:= y_1^{c_1}\cdots y_n^{c_n}$, $a_c\in \Lambda$, 
define the \textbf{tropicalization} 
$\trop f: \R^n\to \R$ of $f$ by 
\[ (\trop f)(u) = \min \{ \val(a_c) +\pair{u,c}\mid c\in \Z^n\} \qquad \forall u\in \R^n. \]
\end{defn}

\begin{defn}[Tropical variety]
We define the \textbf{tropical hypersurface} $V(\trop f)$ associated with a Laurent polynomial
\[ f = \sum_{c\in \Z^n} a_c y^c \in \Lambda[y_1^{\pm 1}, \ldots, y_n^{\pm 1}] \]
such that $u\in \R^n$ is an element of $V(\trop f)$ if and only if there exist at least two $c',c''\in \Z^n$ such that 
\[ \trop(f)(u) =\val(a_{c'}) + \pair{u,c'} = \val(a_{c''}) + \pair{u,c''} . \]
Let $I\subset  \Lambda[y_1^{\pm 1}, \ldots, y_n^{\pm 1}]$ be an ideal. 
The tropical variety associated with $V(I)$ is defined to be 
\[ V(\trop (I)) = \bigcap_{f\in I}\trop (V(f)). \]
\end{defn}

\begin{defn}[Initial form/ideal]
Let \[ f = \sum_{c\in \Z^n} a_c y^c \in \Lambda[y_1^{\pm 1}, \ldots, y_n^{\pm 1}], \quad u\in \R^n. \]
The Laurent polynomial 
\[ in_u(f) = \sum_{ \substack{c\in \Z^n\\ \trop(f)(u) = \val(a_c) + \pair{u,c} } } \overline{T^{-\val(a_c)}a_c} \cdot y^c \in k[y_1^{\pm 1}, \ldots, y_n^{\pm 1}], \]
where $T\in \Lambda$ is an element satisfying \[ \val(T^\lambda) = \lambda \qquad \forall \lambda\in \val(\Lambda^*) \]
and $\overline{\phantom{A}}: \Lambda^* \to k^*$ is the reduction map,  
is called the \textbf{initial form} of the Laurent polynomial $f$ at $u$. 
Similarly, if $I$ is an ideal in $\Lambda[y_1^{\pm 1}, \ldots, y_n^{\pm 1}]$ and $u\in \R^n$, the \textbf{initial ideal} $in_u(I)$ is given by
\[in_u(I) = \left(in_u(f)\mid f\in I\right).  \]
\end{defn}

\begin{thm}[Kapranov's Theorem, \cite{maclagan} Theorem 3.1.3]
\label{Kapranov's Theorem}
Let $\Lambda$ be an algebraically closed field with a non-trivial valuation. 
Let
\[ f = \sum_{c\in \Z^n} a_c y^c \in \Lambda[y_1^{\pm 1}, \ldots, y_n^{\pm 1}] \]
be a Laurent polynomial. Then 
\[ V(\trop f) = \{u\in \R^n\mid in_u(f) \text{ is not a monomial}\}= \overline{\trop(V(f))},\]
where the last set is the closure of the image $\trop(V(f))$ of $V(f)\subset (\Lambda^*)^n$ under the coordinate-wise valuation map in $\R^n$. 
\end{thm}

\begin{thm}[Fundamental theorem of tropical algebraic geometry,  \cite{maclagan} Theorem 3.2.3]
Let $\Lambda$ be an algebraically closed field with a non-trivial valuation. 
Let
$I\subset \Lambda[y_1^{\pm 1}, \ldots, y_n^{\pm 1}]$
be an ideal. Then 
\[ \bigcap_{f\in I}\trop (V(f)) = \{u\in \R^n\mid in_u(I)\ne \pair{1}\}= \overline{\trop(V(I))},\]
where the last set is the closure of the image $\trop(V(I))$ of $V(I)\subset (\Lambda^*)^n$ under the coordinate-wise valuation map in $\R^n$. 
\end{thm}

\newpage

\addcontentsline{toc}{section}{\refname}
\printbibliography

\end{document}

%% file: header.tex
\DeclareMathAlphabet\mathchorus{T1}{qzc}{m}{n}

\usepackage[title]{appendix}
\usepackage{subfiles}
\usepackage{graphicx}
\usepackage{caption}
\usepackage{subcaption}
\graphicspath{ {./images/} }
\usepackage[utf8]{inputenc}
\usepackage{adjustbox}
\usepackage{amsmath}
\usepackage{amssymb, amsthm, physics,leftidx, bbm} 
\usepackage{yhmath}
\usepackage{mathrsfs}

\usepackage[maxnames=99]{biblatex}
\usepackage{mathrsfs}
\usepackage{tikz-cd}

\usepackage{txfonts}
\usepackage[T1]{fontenc}

\usepackage{enumerate}
\usepackage{soul}
\setstcolor{red}
\usepackage[normalem]{ulem}
\usepackage{dashrule}

\bibliography{refs}

\usepackage[all]{xy}
\usepackage{mathtools}
\usepackage{stmaryrd}
\usepackage[shortlabels]{enumitem}
\usepackage{graphicx} 
\graphicspath{ {./images/} }

\usepackage{tikz-cd}
\usepackage[
 letterpaper, 
 margin = 1.5 in
 ]{geometry}

\theoremstyle{definition}

\newtheorem{thm}{Theorem}[section]
\newtheorem{defn}{Definition}[section]	
\newtheorem{eg}{Example}[section]
\newtheorem{prop}{Proposition}[section]
\newtheorem{lem}{Lemma}[section]
\newtheorem{cor}{Corollary}[section]
\newtheorem{rmk}{Remark}[section]

\newtheorem{lemdef}{Lemma/Definition}[section]

\newtheoremstyle{dotlessS}{}{}{}{}{\color{blue}\bfseries}{}{ }{}
\theoremstyle{dotlessS}

\newtheoremstyle{dotlessP}{}{}{\color{blue}
}{}{\color{blue}\bfseries
}{}{ }{}
\theoremstyle{dotlessP}

\DeclareMathOperator{\act}{\pmb{\cdot}}
\DeclareMathOperator{\delb}{\bar{\partial}}

\DeclareMathOperator{\N}{\mathbb{N}}
\DeclareMathOperator{\C}{\mathbb{C}}
\DeclareMathOperator{\Cat}{\mathcal{C}}

\DeclareMathOperator{\CFKG}{\mathcal{CF}_{\mathcal{K}}^G}
\DeclareMathOperator{\CFKGtranstozero}{\mathcal{CF}_{\pitchfork 0,\mathcal{K}}^{G}}

\DeclareMathOperator{\CFK}{\mathcal{CF}_{\mathcal{K}}}

\DeclareMathOperator{\Ainf}{\mathchorus{A}_{\infty}}

\DeclareMathOperator{\Z}{\mathbb{Z}}
\DeclareMathOperator{\B}{\mathbb{B}}
\DeclareMathOperator{\R}{\mathbb{R}}
\DeclareMathOperator{\Q}{\mathbb{Q}}

\DeclareMathOperator{\s}{\mathfrak{s}}
\DeclareMathOperator{\rindex}{\mathfrak{r}}
\DeclareMathOperator{\Rindex}{\mathfrak{R}}
\DeclareMathOperator{\f}{\mathfrak{f}}
\DeclareMathOperator{\g}{\mathfrak{g}}

\DeclareMathOperator{\Aut}{Aut}

\DeclareMathOperator{\scorr}{\mathfrak{X}}

\DeclareMathOperator{\evaluation}{ev}

\DeclareMathOperator{\E}{\mathcal{E}}

\DeclareMathOperator{\novringe}{\Lambda_{0,nov}}
\DeclareMathOperator{\nove}{\Lambda_{nov}}

\DeclareMathOperator{\Spec}{Spec}

\DeclareMathOperator{\m}{\mathfrak{m}}
\DeclareMathOperator{\n}{\mathfrak{n}}
\DeclareMathOperator{\p}{{\mathfrak{p}}}
\DeclareMathOperator{\tn}{\mathfrak{t}}
\DeclareMathOperator{\pq}{{\mathfrak{pq}}}
\DeclareMathOperator{\q}{{\mathfrak{q}}}

\DeclareMathOperator{\Hom}{Hom}

\DeclareMathOperator{\fil}{\mathcal{F}}

\DeclareMathOperator{\Mq}{\mathchorus{M}}
\DeclareMathOperator{\kur}{\mathcal{M}}

\DeclareMathOperator{\pr}{pr}

\DeclareMathOperator{\id}{Id}

\DeclareMathOperator{\hol}{\mathscr{O}}

\DeclareMathOperator{\lra}{\longrightarrow}
 
\DeclareMathOperator{\Pindex}{\mathfrak{P}}
\DeclareMathOperator{\image}{im}

\DeclareMathOperator{\val}{val}

\DeclareMathOperator{\trop}{trop}

\DeclareMathOperator{\supp}{supp}

\DeclareMathOperator{\Crit}{Crit}
\DeclareMathOperator{\Sp}{Sp}

\DeclareMathOperator{\Span}{span}

\DeclareMathOperator{\mlag}{\mathchorus{MLag}}

\DeclareMathOperator{\Diff}{Diff}
\DeclareMathOperator{\po}{\mathfrak{PO}}

\DeclareMathOperator{\interior}{int}

\newcommand{\proj}{\mathbb{CP}}

\newcommand{\corr}[1]{\text{Corr}_{#1}^{\epsilon}}
\newcommand{\gcorr}[1]{\text{Corr}_{#1}^{G,\epsilon}}

\newcommand{\complete}[1]{\widehat{#1}}

\newcommand{\pair}[1]{\left\langle #1 \right\rangle }

\newcommand{\pseries}[1]{ \llbracket  #1 \rrbracket }

\newcommand{\set}[2]{\left\{ #1 \, \middle|\, #2\right\}}
\newcommand{\lrp}[1]{\left(#1\right)}
\newcommand{\Bigp}[1]{\Big(#1\Big)}

\usepackage{imakeidx}
\makeindex[intoc]
\usepackage[colorlinks=true]{hyperref}
\hypersetup{
  colorlinks = false,
  linkcolor  = blue,
  citecolor = purple
}
\def\equationautorefname~#1\null{(#1)\null}

\usepackage{xurl}
\hypersetup{breaklinks=true}

\newcommand\restr[2]{{
  \left.\kern-\nulldelimiterspace 
  #1 
  \littletaller 
  \right|_{#2} 
  }}

\newcommand{\littletaller}{\mathchoice{\vphantom{\big|}}{}{}{}}